\documentclass[11pt, reqno]{amsart}
\usepackage[margin=1in]{geometry}
\usepackage{txfonts}
\usepackage{amsmath,amssymb,amsthm,mathrsfs,enumerate,bm,xcolor,multirow,pbox}
\usepackage{graphicx,color,framed,tikz,caption,subcaption}
\usepackage{enumitem}
\setlist{leftmargin=10mm}
\usepackage[colorlinks,linkcolor=black,citecolor=black,urlcolor=black]{hyperref}

\allowdisplaybreaks[4]
\numberwithin{equation}{section}
\newcommand{\N}{\mathbb{N}}
\newcommand{\R}{\mathbb{R}}

\newcommand{\E}{\mathbb{E}}
\newcommand{\Prob}{\mathbb{P}}

\newcommand{\pnorm}[2]{\lVert#1\rVert_{#2}}
\newcommand{\bigpnorm}[2]{\big\lVert#1\big\rVert_{#2}}
\newcommand{\biggpnorm}[2]{\bigg\lVert#1\bigg\rVert_{#2}}
\newcommand{\abs}[1]{\lvert#1\rvert}
\newcommand{\bigabs}[1]{\big\lvert#1\big\rvert}
\newcommand{\biggabs}[1]{\bigg\lvert#1\bigg\rvert}
\newcommand{\iprod}[2]{\left\langle#1,#2\right\rangle}
\renewcommand{\epsilon}{\varepsilon}

\renewcommand{\d}[1]{\mathrm{d}#1}

\newcommand{\equald}{\stackrel{d}{=}}

\renewcommand{\tilde}{\widetilde}

\DeclareMathOperator{\tr}{tr}
\DeclareMathOperator{\var}{Var}

\DeclareMathOperator{\dcov}{dCov}
\DeclareMathOperator{\dcor}{dCor}

\DeclareMathOperator{\diag}{diag}
\DeclareMathOperator{\op}{op}

\let\limsup\relax
\DeclareMathOperator*\limsup{\overline{lim}}

\newcommand{\beq}{\begin{equation}}
\newcommand{\eeq}{\end{equation}}
\newcommand{\beqa}{\begin{equation} \begin{aligned}}
\newcommand{\eeqa}{\end{aligned} \end{equation}}
\newcommand{\beqas}{\begin{equation*} \begin{aligned}}
\newcommand{\eeqas}{\end{aligned} \end{equation*}}

\newcommand{\bit}{\begin{itemize}}
	\newcommand{\eit}{\end{itemize}}
\newcommand{\bmat}{\begin{bmatrix}}
	\newcommand{\emat}{\end{bmatrix}}

\theoremstyle{definition}\newtheorem{problem}{Problem}[section]
\theoremstyle{definition}
\theoremstyle{remark}\newtheorem{assumption}{Assumption}

\theoremstyle{remark}\newtheorem{remark}[problem]{Remark}
\theoremstyle{definition}
\theoremstyle{plain}\newtheorem{theorem}[problem]{Theorem}
\theoremstyle{plain}
\theoremstyle{plain}\newtheorem{lemma}[problem]{Lemma}
\theoremstyle{plain}\newtheorem{proposition}[problem]{Proposition}
\theoremstyle{plain}\newtheorem{corollary}[problem]{Corollary}
\theoremstyle{plain}

\AtBeginDocument{%
	\def\MR#1{}
}

\begin{document}

\title[Generalized kernel distance covariance in high dimensions]{Generalized kernel distance covariance in high dimensions: non-null CLTs and power universality}
\thanks{The research of Q. Han is partially supported by NSF grants DMS-1916221 and DMS-2143468.}

\author[Q. Han]{Qiyang Han}

\address[Q. Han]{
Department of Statistics, Rutgers University, Piscataway, NJ 08854, USA.
}
\email{qh85@stat.rutgers.edu}

\author[Y. Shen]{Yandi Shen}

\address[Y. Shen]{
Department of Statistics, University of Chicago, Chicago, IL 60615, USA.
}
\email{ydshen@uchicago.edu}

\date{\today}

\begin{abstract}
Distance covariance is a popular dependence measure for two random vectors $X$ and $Y$ of possibly different dimensions and types. Recent years have witnessed concentrated efforts in the literature to understand the distributional properties of the sample distance covariance in a high-dimensional setting, with an exclusive emphasis on the null case that $X$ and $Y$ are independent. This paper derives the first non-null central limit theorem for the sample distance covariance, and the more general sample (Hilbert-Schmidt) kernel distance covariance in high dimensions, in the distributional class of $(X,Y)$ with a separable covariance structure. The new non-null central limit theorem yields an asymptotically exact first-order power formula for the widely used generalized kernel distance correlation test of independence between $X$ and $Y$. The power formula in particular unveils an interesting universality phenomenon: the power of the generalized kernel distance correlation test is completely determined by $n\cdot \dcor^2(X,Y)/\sqrt{2}$ in the high dimensional limit, regardless of a wide range of choices of the kernels and bandwidth parameters. Furthermore, this separation rate is also shown to be optimal in a minimax sense. The key step in the proof of the non-null central limit theorem is a precise expansion of the mean and variance of the sample distance covariance in high dimensions, which shows, among other things, that the non-null Gaussian approximation of the sample distance covariance involves a rather subtle interplay between the dimension-to-sample ratio and the dependence between $X$ and $Y$. 
\end{abstract}

\maketitle

\sloppy
\setcounter{tocdepth}{1}

\vspace{-3em}
\tableofcontents
\vspace{-3em}

\section{Introduction}

\subsection{Overview}
Given samples from a random vector $(X,Y)$ in $\R^{p + q}$, it is of fundamental statistical interest to test whether $X$ and $Y$ are independent. The long history of this problem has given rise to a large number of dependence measures targeting at different types of dependence structure. Notable examples include the classical Pearson correlation coefficient \cite{pearson1895notes}, rank-based correlation coefficients \cite{spearman1904proof, kendall1938new, hoeffding1948nonparametric, blum1961distribution, yanagimoto1970measures, rosenblatt1975quadratic, feuerverger1993consistent}, Cram\'er-von Mises-type measures \cite{de1980cramer}, measure based on characteristic functions \cite{szekely2007measuring, szekely2009brownian}, kernel-based measures \cite{gretton2005kernel, gretton2007kernel}, sign covariance \cite{bergsma2014consistent,weihs2018symmetric}. We refer to the classical textbooks \cite[Chapter 9]{anderson1958introduction} and \cite[Chapter 11]{muirhead1982aspects} for a systematic exposition on this topic.

Among the plentiful dependence measures for such a purpose, the distance covariance metric and its generalizations \cite{szekely2007measuring,szekely2014partial} have attracted much attention in recent years. In one of its many equivalent forms, the distance covariance between $X$ and $Y$ can be defined as (cf. \cite[Theorems 7 and 8]{szekely2009brownian})
\begin{align}\label{intro:dcov}
\dcov^2(X,Y)&\equiv \E\big(\pnorm{X_1-X_2}{}\pnorm{Y_1 - Y_2}{}\big) - 2\E\big(\pnorm{X_1-X_2}{}\pnorm{Y_1 - Y_3}{}\big)\nonumber\\
&\quad\quad\quad + \E\big(\pnorm{X_1-X_2}{})\E\big(\pnorm{Y_1 - Y_2}{}\big).
\end{align}
Here $(X_i,Y_i)$, $i=1,2,3$ are independent copies from the joint distribution of $(X,Y)$, and $\pnorm{\cdot}{}$ is the Euclidean norm. The distance covariance metric $\dcov^2(X,Y)$ is particularly appealing for several nice features. First, $X$ and $Y$ are independent if and only if $\dcov^2(X,Y) = 0$. Second, $\dcov^2(X,Y)$ can be used in cases where $X$ and $Y$ are of different dimensions and data type (discrete, continuous or mixed). Third, several estimators of $\dcov^2(X,Y)$ are known to allow for efficient calculation. Due to these reasons, the distance covariance has been utilized in a wide range of both methodological and applied contexts, see e.g. \cite{kong2012using, li2012feature, shao2014martingale, matteson2017independent, yao2018testing, zhang2018conditional} for an incomplete list of references.

An estimator of $\dcov^2(X,Y)$ based on $n$ i.i.d. samples $(X_1,Y_1),\ldots,(X_n,Y_n)$ from the distribution of $(X,Y)$ is first proposed in \cite{szekely2007measuring}, with its bias corrected version proposed in \cite{szekely2014partial}, which is now known as the sample distance covariance $\dcov_\ast^2(\bm{X},\bm{Y})$. The finite-sample distribution of $\dcov_\ast^2(\bm{X},\bm{Y})$ is generally intractable; so the literature has focused on deriving its asymptotic distribution in different growth regimes of $(n,p,q)$, cf. \cite{szekely2013distance,huo2016fast,zhu2020distance,gao2021asymptotic}. In the fixed dimensional asymptotic regime when $p,q$ are fixed and $n$ diverges to infinity, \cite{huo2016fast} showed that $\dcov_\ast^2(\bm{X},\bm{Y})$ converges in distribution to a mixture of chi-squared distributions. This is complemented by the result of \cite{szekely2013distance},  where a $t$-distribution limit was derived in the so-called `high dimensional low sample size' regime when $n$ remains fixed and both $p,q$ diverge to infinity. The high dimensional regime where both the sample size $n$ and the data dimension $p,q$ diverge was recently studied in \cite{zhu2020distance,gao2021asymptotic}, where $\dcov_\ast^2(\bm{X},\bm{Y})$ was shown to obey a central limit theorem (CLT). We also refer to \cite{gao2021two, yan2023kernel} for some related distributional results in the problem of two-sample distribution testing. Except for some non-null results in \cite[Proposition 2.2.2]{zhu2020distance} under the fixed $n$ regime, all these results are derived under the null scenario where $X$ and $Y$ are independent. This leaves open the more challenging but equally important issue of non-null limiting distributions of $\dcov_\ast^2(\bm{X},\bm{Y})$, which are the key to a complete characterization of the power behavior of distance covariance based tests. Bridging this significant theoretical gap is one of the main motivations of the current paper. 

\subsection{Non-null CLTs}

For the majority of the paper, we work with distributions of $(X,Y)$ with a separable covariance structure (see Section \ref{subsec:nonnull_clt_I} ahead for details), and perform an exact analysis of the distributional properties of the sample distance covariance $\dcov_*^2(\bm{X},\bm{Y})$ in the following high dimensional regime:
\begin{align}\label{def:hd_regime}
\min\{n,p,q\} \rightarrow \infty.
\end{align}
Our first main result is the following non-null CLT (see Theorem \ref{thm:non_null_clt} below for the formal statement): Uniformly over the covariance matrix $\Sigma$ of $(X^\top, Y^\top)^\top$ with a compact spectrum in $(0,\infty)$,
\begin{align}\label{intro:non_null_clt}
\frac{\dcov_*^2(\bm{X},\bm{Y}) - \dcov^2(X,Y)}{\var^{1/2}\big(\dcov_*^2(\bm{X},\bm{Y})\big)} \text{ converges in distribution to $\mathcal{N}(0,1)$}
\end{align} 
in the regime (\ref{def:hd_regime}). Here $\mathcal{N}(0,1)$ denotes the standard normal distribution. For simplicity, we have presented here the asymptotic version of the result; the more complete Theorem \ref{thm:non_null_clt} below is non-asymptotic in nature and gives an error bound with explicit dependence on the problem parameters $(n,p,q)$. Furthermore, we show that an analogue of (\ref{intro:non_null_clt}) also holds for $\dcov_*^2(\bm{X})$, the `marginal' analogous unbiased estimator of $\dcov^2(X)$. To the best of our knowledge, (\ref{intro:non_null_clt}) is the first non-null CLT for $\dcov_\ast^2(\bm{X},\bm{Y})$ in the literature.

Let us now give some intuition why one would expect a  non-null CLT (\ref{intro:non_null_clt}) that holds for a general $\Sigma$ in the regime (\ref{def:hd_regime}). It is well-known that the sample distance covariance $\dcov_\ast^2(\bm{X},\bm{Y})$ admits a $U$-statistics representation (cf. Proposition \ref{prop:hoef_decomp}) with first-order degeneracy under the null. By classical theory in the fixed dimensional asymptotics (i.e., $p,q$ fixed with $n\to \infty$), a CLT holds for $\dcov_\ast^2(\bm{X},\bm{Y})$ under any fixed alternative $\Sigma \neq I_{p+q}$; while a non-Gaussian limit holds under the null $\Sigma = I_{p+q}$. In such fixed dimensional asymptotics, the Gaussian limit is due to the \emph{non-degeneracy} of $\dcov_\ast^2(\bm{X},\bm{Y})$ under the alternative, while the non-Gaussian limit is due to the \emph{degeneracy} of $\dcov_\ast^2(\bm{X},\bm{Y})$ under the null. Now, as high dimensionality also enforces a Gaussian approximation of $\dcov_\ast^2(\bm{X},\bm{Y})$ under the null with degeneracy  (cf. \cite{zhu2020distance,gao2021asymptotic}), one would naturally expect the finite-sample distribution of the centered $\dcov_\ast^2(\bm{X},\bm{Y})$ under a general $\Sigma$,  to be approximately a `mixture' of a centered Gaussian component due to non-degeneracy and another centered Gaussian component due to degeneracy, which is again Gaussian. The non-null CLT (\ref{intro:non_null_clt}) formalizes this intuition in the regime (\ref{def:hd_regime}).

To formally implement the above intuition, an important step in the proof of (\ref{intro:non_null_clt}) is to obtain precise mean and variance expansions for the sample distance covariance $\dcov_\ast^2(\bm{X},\bm{Y})$  in the regime (\ref{def:hd_regime}). In particular, we show in Theorem \ref{thm:dcov_expansion} that the mean can be expanded as
\begin{align}\label{intro:mean_expansion}
\dcov^2(X,Y) = \frac{\pnorm{\Sigma_{XY}}{F}^2}{2\sqrt{\tr(\Sigma_X)\tr(\Sigma_Y)}}\big(1+\mathfrak{o}(1)\big),
\end{align}
and in Theorem \ref{thm:variance_expansion} that the variance \emph{under the null} $\sigma^2_{\textrm{null}}$ can be expanded as
\begin{align}\label{intro:var_expansion}
\sigma^2_{\textrm{null}} = \frac{\pnorm{\Sigma_X}{F}^2\pnorm{\Sigma_Y}{F}^2}{2n(n-1)\tr(\Sigma_X)\tr(\Sigma_Y)}\big(1+\mathfrak{o}(1)\big).
\end{align}
Here $\Sigma_X$, $\Sigma_Y$, and $\Sigma_{XY}$ are sub-blocks of the covariance matrix $\Sigma = [\Sigma_X, \Sigma_{XY};\Sigma_{YX},\Sigma_Y]$, $\tr(\cdot)$ denotes the trace and $\pnorm{\cdot}{F}$ denotes the matrix Frobenious norm, and $\mathfrak{o}(1)$ is the standard small-o notation representing a vanishing term under the asymptotics (\ref{def:hd_regime}).  The variance formula for general $\Sigma$ (explicit form see Theorem \ref{thm:variance_expansion}) is rather complicated so is not presented here, but as explained above, it is expected to contain two parts that are contributed individually by the non-degenerate and the degenerate components of $\dcov_\ast^2(\bm{X},\bm{Y})$. Notably, the contributions of these two components to the Gaussian approximation in (\ref{intro:non_null_clt}) depend on the dimension-to-sample ratio in a fairly subtle way. In `very high dimensions', the non-null CLT is entirely driven by the degeneracy of $\dcov_\ast^2(\bm{X},\bm{Y})$ regardless of the degree of dependence between $X$ and $Y$. On the other hand, in `moderate high dimensions', dependence between $X$ and $Y$ plays a critical role in determining the contributions of the (non-)degeneracy in the non-null CLT. See the discussion after Theorem \ref{thm:non_null_clt} for details.

\subsection{Independent test via distance covariance: power asymptotics}

A major application of the non-null CLT derived in (\ref{intro:non_null_clt}) is a precise power formula for the following popular distance correlation test of independence between $X$ and $Y$, first considered in \cite{szekely2013distance}:
\begin{align}\label{def:test_t}
\Psi(\bm{X},\bm{Y};\alpha) \equiv \bm{1}\bigg( \biggabs{\frac{n\cdot\dcov_\ast^2(\bm{X},\bm{Y})}{\sqrt{2\dcov_\ast^2(\bm{X})\cdot \dcov_\ast^2(\bm{Y})}}}>z_{\alpha/2}\bigg).
\end{align}
The above independence test and the null part of (\ref{intro:non_null_clt}) is connected by the mean and variance expansions in (\ref{intro:mean_expansion}) and (\ref{intro:var_expansion}). In fact, as will be detailed in Section \ref{section:power_formula} below, the above test is asymptotically (in the regime (\ref{def:hd_regime})) equivalent to the (infeasible) $z$-test built from the null part of (\ref{intro:non_null_clt}). As a direct consequence, $\Psi(\bm{X},\bm{Y};\alpha)$ will also have an asymptotic size of $\alpha$. The null behavior of (a variant of) $\Psi(\bm{X},\bm{Y};\alpha)$ was first studied in \cite{szekely2013distance} in the regime of fixed $n$ and $\min\{p,q\}\rightarrow \infty$, and then in \cite{gao2021asymptotic} in a high dimensional regime slightly broader than ours (\ref{def:hd_regime}).

Having understood the behavior of the test (\ref{def:test_t}) under the null, we now turn to the more challenging question of its behavior under a generic alternative covariance $\Sigma$. Using again the non-null CLT in (\ref{intro:non_null_clt}), we show that the test statistic in (\ref{def:test_t}) is asymptotically normal with a mean shift (formal statement see Theorem \ref{thm:dcov_power}):
\begin{align}\label{intro:power}
\E_\Sigma \Psi(\bm{X},\bm{Y};\alpha) = \Prob(|\mathcal{N}(m_n(\Sigma), 1)| > z_{\alpha/2}) + \mathfrak{o}(1).
\end{align}
Here $\E_\Sigma$ denotes expectation under the data distribution with covariance $\Sigma$ so the left side is the power of the test $\Psi(\bm{X},\bm{Y};\alpha)$, and the mean shift parameter $m_n(\Sigma)$ can be either
\begin{align*}
\frac{n\cdot \dcor^2(X,Y)}{\sqrt{2}}\equiv \frac{n\dcov^2(X,Y)}{\sqrt{2\dcov^2(X)\dcov^2(Y)}} \quad \text{or} \quad \frac{n\pnorm{\Sigma_{XY}}{F}^2}{\sqrt{2}\pnorm{\Sigma_X}{F}\pnorm{\Sigma_Y}{F}}.
\end{align*}
Here the (rescaled) left side is known as the \textit{distance correlation} between $X$ and $Y$, and its asymptotic equivalence to the right side follows again from the mean expansion in (\ref{intro:mean_expansion}). It follows directly from (\ref{intro:power}) that if the spectra of $\Sigma_X$ and $\Sigma_Y$ are appropriately bounded, $\Psi(\bm{X},\bm{Y};\alpha)$ has asymptotically full power if and only if $n\cdot \dcor^2(X,Y)\rightarrow \infty$. A complementary minimax lower bound in Theorem \ref{thm:minimax_lower} shows that this separation rate cannot be further improved from an information theoretic point of view.

Power results for tests based on distance covariance (correlation) are scarce, particularly in high dimensions when both $n$ and $p$ and/or $q$ diverge to infinity. \cite{zhu2020distance} gives a relatively complete power characterization for a related studentized test in the regime of fixed $n$ and $p\wedge q\rightarrow \infty$, followed by some partial results in the regime $\min\{p,q\}/n^2\rightarrow \infty$. The same test $\Psi(\bm{X},\bm{Y};\alpha)$ as in (\ref{def:test_t}) is recently analyzed in \cite{gao2021asymptotic} in the slightly broader regime $\min\{n, \max\{p,q\}\}\rightarrow\infty$, but their analysis requires a much stronger condition $\sqrt{n}\cdot \dcor^2(X,Y)\rightarrow \infty$ for power consistency (see their Theorem 5 and subsequent discussion). In contrast, under the distributional Assumption \ref{assump:sepa}, (\ref{intro:power}) gives a much more precise characterization of the power behavior of the distance correlation test (\ref{def:test_t}), even when consistency does not hold.

\subsection{Kernel generalizations and power universality}

Following \cite{gretton2005kernel,gretton2007kernel}, the distance covariance in (\ref{intro:dcov}) can be naturally generalized to the so-called \textit{Hilbert-Schmidt covariance}:
\begin{align}\label{intro:dcov_kernel}
\dcov^2(X,Y;f,\gamma) &\equiv \E \Big[f_X\big(\pnorm{X_1-X_2}{}/\gamma_X\big)f_Y\big(\pnorm{Y_1-Y_2}{}/\gamma_Y\big)\Big]\nonumber\\
&\qquad -2 \E\Big[f_X\big(\pnorm{X_1-X_2}{}/\gamma_X\big)f_Y\big(\pnorm{Y_1-Y_3}{}/\gamma_X\big)\Big]\nonumber\\
&\qquad + \E\Big[f_X\big(\pnorm{X_1-X_2}{}/\gamma_X\big)\Big] \E \Big[f_Y\big(\pnorm{Y_1-Y_2}{}/\gamma_Y\big)\Big].
\end{align}
Here $f = (f_X,f_Y)$ are kernel functions, and $\gamma = (\gamma_X,\gamma_Y)\in \R^2_{\geq 0}$ are the bandwidth parameters for $X$ and $Y$ respectively. $\dcov^2(X;f,\gamma)$ and $\dcov^2(Y;f,\gamma)$ are defined analogously. The above definition reduces to the (rescaled) standard distance covariance in (\ref{intro:dcov}) when the kernel is taken to be the identity function. Let $\dcov_*^2(\bm{X},\bm{Y};f,\gamma)$, $\dcov^2(\bm{Y};f,\gamma)$, $\dcov_*^2(\bm{X};f,\gamma)$ be the sample kernel distance covariance. As a key step of the universality results presented below, we show that these quantities can be related to the standard sample distance covariance: under mild conditions on kernels $f=(f_X,f_Y)$ and bandwidths $\gamma=(\gamma_X,\gamma_Y)$, 
\begin{align}\label{intro:kernel_expansion}
\dcov_*^2(\bm{X},\bm{Y};f,\gamma) =  \varrho(\gamma)\dcov_*^2(\bm{X},\bm{Y})\big(1+\mathfrak{o}_{\mathbf{P}}(1)\big).
\end{align}
Here $\mathfrak{o}_{\mathbf{P}}(1)$ is again under the asymptotics (\ref{def:hd_regime}) and $\varrho(\gamma)$, whose exact definition is given in (\ref{def:varrho}) below, depends on the kernels $f$, bandwidths $\gamma$, and population covariance $\Sigma$. Similar expansions hold for the marginal quantities $\dcov^2(X;f,\gamma)$ and $\dcov^2(Y;f,\gamma)$. 

Relation (\ref{intro:kernel_expansion}) implies that as long as the scaling factor $\varrho(\gamma)$ stabilizes away from zero and infinity, $\dcov_*^2(\bm{X},\bm{Y};f,\gamma)$ (up to a scaling) shares the same limiting distribution as $\dcov_*^2(\bm{X},\bm{Y})$, which has been studied in detail in the previous subsection. In particular, both the non-null CLT in (\ref{intro:non_null_clt}) and the power expansion in (\ref{intro:power}) hold for the kernelized distance covariance as well, upon changing the test (\ref{def:test_t}) to its kernelized version in the latter result; see Theorems \ref{thm:non_null_clt_kernel} and \ref{thm:dcov_power} for formal statements. In other words, the power behavior in (\ref{intro:power}) exhibits \textit{universality} with respect to both the choice of kernels and bandwidth parameters; see Section \ref{sec:simulation} for numerical evidence.

\subsection{Organization}

The rest of the paper is organized as follows. Section \ref{section:nonnull_clt} starts with some background knowledge of the distance covariance metric and then states the main non-null CLTs for both the canonical sample distance covariance and its kernel generalizations. Section \ref{section:power_formula} studies the power behavior of a class of generalized kernel distance correlation tests and discusses their minimax optimality. Some numerical simulations for the main results in the preceding two sections are presented in Section \ref{sec:simulation}, with some concluding remarks in Section \ref{sec:discussion}. Section \ref{section:proof_outline} is devoted to a proof outline for the non-null CLTs. Details of the important steps are then presented in Sections \ref{section:preliminaries} - \ref{section:clt_truncated_dcov}, followed by the main proof in Section \ref{sec:proof_clts}. The rest of the technical proofs are deferred to the supplement.

\subsection{Notation}\label{section:notation}

For any positive integer $n$, let $[n]$ denote the set $\{1,\ldots,n\}$. For $a,b \in \R$, $a\vee b\equiv \max\{a,b\}$ and $a\wedge b\equiv\min\{a,b\}$. For $a \in \R$, let $a_+\equiv a\vee 0$ and $a_- \equiv (-a)\vee 0$. For $x \in \R^n$, let $\pnorm{x}{p}=\pnorm{x}{\ell_p(\R^n)}$ denote its $p$-norm $(0\leq p\leq \infty)$ with $\pnorm{x}{2}$ abbreviated as $\pnorm{x}{}$. Let $B_p(r;x)\equiv \{z \in \R^p: \pnorm{z-x}{}\leq r\}$ be the unit $\ell_2$ ball in $\R^{p}$. By $\bm{1}_n$ we denote the vector of all ones in $\R^n$. For a matrix $M \in \R^{n\times n}$, let $\pnorm{M}{\op}$ and $\pnorm{M}{F}$ denote the spectral and Frobenius norms of $M$ respectively. For two matrices $M,N$ of the same size, let $M\circ N$ denote their Hadamard product. We use $\{e_j\}$ to denote the canonical basis, whose dimension should be self-clear from the context.

We use $C_{x}$ to denote a generic constant that depends only on $x$, whose numeric value may change from line to line unless otherwise specified. Notations $a\lesssim_{x} b$ and $a\gtrsim_x b$ mean $a\leq C_x b$ and $a\geq C_x b$ respectively, and $a\asymp_x b$ means $a\lesssim_{x} b$ and $a\gtrsim_x b$. The symbol $a\lesssim b$ means $a\leq Cb$ for some absolute constant $C$. For two nonnegative sequences $\{a_n\}$ and $\{b_n\}$, we write $a_n\ll b_n$ (respectively~$a_n\gg b_n$) if $\lim_{n\rightarrow\infty} (a_n/b_n) = 0$ (respectively~$\lim_{n\rightarrow\infty} (a_n/b_n) = \infty$). We write $a_n\sim b_n$ if $\lim_{n\rightarrow\infty} (a_n/b_n) = 1$. We follow the convention that $0/0 = 0$. 

Let $\varphi,\Phi$ be the density and the cumulative distribution function of a standard normal random variable. For any $\alpha \in (0,1)$, let $z_\alpha$ be the normal quantile defined by $\Prob(\mathcal{N}(0,1)> z_\alpha) = \alpha$. For two random variables $X,Y$ on $\R$, we use 
$d_{\mathrm{Kol}}(X,Y)$  to denote their Kolmogorov distance defined by
\begin{align}\label{def:intro_distance}
d_{\mathrm{Kol}}(X,Y)&\equiv \sup_{t \in \R}\bigabs{\Prob\big(X \leq t\big)-\Prob\big(Y \leq t\big)}.
\end{align}
Here $\mathcal{B}(\R)$ denotes the Borel $\sigma$-algebra of $\R$.

\section{Non-null central limit theorems}\label{section:nonnull_clt}

\subsection{Distance covariance: a review}

We start with a review for the distance covariance (correlation). For two random vectors $X \in \R^p$ and $Y \in \R^q$, the squared distance covariance \cite{szekely2007measuring} is originally defined by
\begin{align*}
\dcov^2(X,Y)\equiv \int_{\R^{p+q}} \frac{\abs{\varphi_{(X,Y)}(t,s)-\varphi_X(t)\varphi_Y(s)}^2}{c_pc_q \pnorm{t}{}^{p+1}\pnorm{s}{}^{q+1}}\, \d{t}\d{s}.
\end{align*}
Here $c_p = \pi^{(p+1)/2}/\Gamma\big((p+1)/2\big)$ with $\Gamma(\cdot)$ denoting the gamma function, and $\varphi(\cdot)$ is the characteristic function. The marginal quantities $\dcov^2(X,X)$ and $\dcov^2(Y,Y)$ are defined analogously, and we will shorthand them as $\dcov^2(X)$ and $\dcov^2(Y)$ in the sequel. It is well-known that $X$ and $Y$ are independent if and only if $\dcov(X,Y)=0$, hence $\dcov^2(X,Y)$ captures any kind of dependence between $X$ and $Y$ including non-linear and non-monotone ones. Analogous to the standard notion of covariance and correlation, the squared distance correlation is defined by
\begin{align*}
\dcor^2(X,Y) \equiv \frac{\dcov^2(X,Y)}{\sqrt{\dcov^2(X)\dcov^2(Y)}},
\end{align*}
with convention $\dcor^2(X,Y) \equiv 0$ if $\dcov^2(X)\dcov^2(Y) = 0$.

The distance covariance can be equivalently characterized in a number of different ways. 
In addition to (\ref{intro:dcov}), another useful representation that will be particularly relevant for our purpose is through the double centered distances:
\begin{align}\label{def:UV}
U(x_1,x_2)&\equiv\pnorm{x_1-x_2}{}-\E\pnorm{x_1-X}{}-\E\pnorm{X-x_2}{}+ \E \pnorm{X-X'}{},\nonumber\\
V(y_1,y_2)&\equiv\pnorm{y_1-y_2}{}-\E\pnorm{y_1-Y}{}-\E\pnorm{Y-y_2}{}+ \E \pnorm{Y-Y'}{}.
\end{align}
Then by (\ref{intro:dcov}) or \cite[pp. 3287]{lyons2013distance}, we have the identity
\begin{align}\label{def:dcov_population}
\dcov^2(X,Y) = \E U(X_1,X_2) V(Y_1,Y_2).
\end{align}
Now we define the sample distance covariance.
For $n$ copies of i.i.d. observations $(X_1,Y_1),\ldots,(X_n,Y_n)$, define two symmetric matrices $A,B\in\R^{n\times n}$ entrywise by
\begin{align}\label{def:ab}
A_{k\ell}\equiv \pnorm{X_k-X_\ell}{},\quad B_{kl}\equiv \pnorm{Y_k-Y_\ell}{},\quad 1\leq k,\ell \leq n.
\end{align}
Following \cite{szekely2014partial}, the bias-corrected sample distance covariance is defined by
\begin{align}\label{def:dcov_empirical}
\dcov_\ast^2(\bm{X},\bm{Y})& = \frac{1}{n(n-3)}\sum_{k\neq \ell} A_{k\ell}^\ast B_{k\ell}^\ast,
\end{align}
where $A^\ast, B^\ast\in\R^{n\times n}$ are $U$-centered versions of $A,B$ defined by
\begin{align}\label{def:AB}
A^\ast& = A- \frac{\bm{1}\bm{1}^\top A+A\bm{1}\bm{1}^\top}{n-2}+\frac{\bm{1}\bm{1}^\top A \bm{1}\bm{1}^\top}{(n-1)(n-2)},\nonumber\\
B^\ast &= B- \frac{\bm{1}\bm{1}^\top B+B\bm{1}\bm{1}^\top}{n-2}+\frac{\bm{1}\bm{1}^\top B \bm{1}\bm{1}^\top}{(n-1)(n-2)}.
\end{align}
Marginal quantities $\dcov_\ast^2(\bm{X},\bm{X})$ and $\dcov_\ast^2(\bm{Y},\bm{Y})$ are defined analogously, and will be shorthanded as $\dcov_\ast^2(\bm{X})$ and $\dcov_\ast^2(\bm{Y})$ in the sequel. 

The definition of the sample distance covariance $\dcov_\ast^2(\bm{X},\bm{Y})$ in (\ref{def:dcov_empirical}) above looks a bit mysterious at first sight, but the following representation via a 4-th order $U$-statistic makes it clear why the definition is indeed natural. Recall the definitions of $U,V$ in (\ref{def:UV}).

\begin{proposition}[\cite{yao2018testing,gao2021asymptotic}]\label{prop:hoef_decomp}
	The following holds:
	\begin{align*}
	\dcov_\ast^2(\bm{X},\bm{Y}) & = \frac{1}{\binom{n}{4}}\sum_{i_1<\cdots<i_4} k\big(Z_{i_1},Z_{i_2},Z_{i_3},Z_{i_4}\big),
	\end{align*}
	where the symmetric kernel can be either 
	\begin{align}\label{eqn:kernel_rep1}
	&k(Z_1,Z_2,Z_3,Z_4) =\frac{1}{4!}\sum_{(i_1,\ldots,i_4)\in \sigma(1,2,3,4)} \bigg[\pnorm{X_{i_1}-X_{i_2}}{}\pnorm{Y_{i_1}-Y_{i_2}}{}\nonumber\\
	&\qquad\qquad + \pnorm{X_{i_1}-X_{i_2}}{}\pnorm{Y_{i_3}-Y_{i_4}}{}-2\pnorm{X_{i_1}-X_{i_2}}{}\pnorm{Y_{i_1}-Y_{i_3}}{}\bigg],
	\end{align}
	or
	\begin{align}\label{eqn:kernel_rep2}
	&k(Z_1,Z_2,Z_3,Z_4) =\frac{1}{4!}\sum_{(i_1,\ldots,i_4)\in \sigma(1,2,3,4)} \bigg[ U(X_{i_1},X_{i_2})V(Y_{i_1},Y_{i_2})\nonumber\\
	&\qquad\qquad+U(X_{i_1},X_{i_2})V(Y_{i_3},Y_{i_4})- 2 U(X_{i_1},X_{i_2})V(Y_{i_1},Y_{i_3})\bigg].
	\end{align}
	Here $Z_i=(X_i,Y_i)$ for $i\in \N$, and $\sigma(1,2,3,4)$ denotes the set of all ordered permutation of $\{1,2,3,4\}$.
\end{proposition}

It is a direct consequence of the above result that $\dcov_\ast^2(\bm{X},\bm{Y})$ is an unbiased estimator for $\dcov^2(X,Y)$. The fact that $\dcov_\ast^2(\bm{X},\bm{Y})$ can be represented as a $U$-statistic is first validated in \cite[Section 3.2]{huo2016fast}. The kernel representation (\ref{eqn:kernel_rep1}) (proved in e.g. \cite[Lemma 2.1]{yao2018testing}) is quite natural in that it gives an unbiased estimate for the population in the form (\ref{intro:dcov}). The double centered version (\ref{eqn:kernel_rep2}), which turns out to be more convenient and useful for the purpose of theoretical developments in this paper, is essentially proved in \cite[Lemma 5]{gao2021asymptotic} in a different form. For the convenience of the reader, we provide a self-contained proof in Appendix \ref{appendix:kernel_rep}.

\subsection{General non-null CLTs I: distance covariance}\label{subsec:nonnull_clt_I}

Throughout the paper, we work with the following distributional family of $(X,Y)$ with a \emph{separable covariance} structure:
\begin{assumption}\label{assump:sepa}
	Suppose 
	\begin{align}\label{def:main_dist}
	(X^\top, Y^\top)^\top \stackrel{d}{=} \Sigma^{1/2} Z.
	\end{align} 
	Here $\Sigma$, in its block form $[\Sigma_X, \Sigma_{XY};\Sigma_{YX},\Sigma_Y]$, is a covariance matrix in $\R^{(p+q)\times (p+q)}$, $Z \in \R^{p+q}$ has i.i.d. components with mean $0$, variance $1$ and satisfies the following:
	\begin{enumerate}
		\item[(A1)] $Z_1$ is symmetric around $0$ with excess kurtosis $\kappa \equiv \E Z_1^4 - 3$.
		\item[(A2)] $Z_1$ satisfies a Poincar\'{e} inequality: for some $c_*>0$, we have $
		\var f(Z_1) \leq c_\ast \E (f'(Z_1))^2$
		for any absolutely continuous $f$ such that $\E (f'(Z_1))^2<\infty$.
		\item[(A3)] $Z_1$ has a Lebesgue density $f_Z(\cdot)$ with $\sup_{x\in(-\epsilon, \epsilon)} f_Z(x) \leq C_\epsilon$ for some small $\epsilon > 0$ and positive $C_\epsilon$. For future purpose, let $\epsilon_0 \equiv \sup\big\{\epsilon>0: 2\pi \epsilon\cdot \sup_{x\in(-\epsilon,\epsilon)} f_Z(x)\leq 1\big\}$.
	\end{enumerate}
\end{assumption}
The distribution class of $(X,Y)$ with separable covariance is quite common in the literature, in particular in the study of non-null behavior of statistics related to large random matrices; the readers are referred to the recent papers \cite{tao2012random,knowles2017anisotropic, ding2018necessary,ding2021spiked} and monographs \cite{bai2010spectral,erdos2017dynamical} for more backgrounds and results under separable covariance in this direction.

The major assumption on the distribution of $Z_1$ is the requirement of a Poincar\'e inequality in condition (A2). It is well-known that the existence of a Poincar\'e inequality as in (A2) is equivalent to exponential mixing of a Markov semigroup with stationary distribution $Z_1$ and Dirichlet form $\mathcal{E}(f,g)=\E f'(Z_1)g'(Z_1)$, cf. \cite{bakry2014analysis}. An important example fulfilling Assumption \ref{assump:sepa} is the family of symmetric, strongly log-concave distributions, which are known to satisfy a Poincar\'e inequality (cf. \cite{bobkov1999isoperimetric,saumard2014logconcavity}) and contain the Gaussian distribution as a special case. It is easy to further weaken condition (A2) to a \emph{weighted Poincar\'e inequality} as in \cite{bobkov2009weighted}; we shall not pursue these formal refinements here. Condition (A3) above is purely technical, and can be further weakened at the cost of a more involved mathematical expression. We choose to work under this condition for clean presentation.

We mention two important implications of Assumption \ref{assump:sepa}: (i) since $Z_1$ has a Lebesgue density, $\kappa \geq c_0 - 2$ for some $c_0 > 0$ depending only on the distribution of $Z_1$; (ii) by \cite[Theorem 4.1]{bobkov2009weighted}, $Z_1$ has moments of any order with $(\E|Z_1|^p)^{1/p}\leq p\cdot\sqrt{c_\ast/2}$ for any $p\geq 1$.

Some notation that will be used throughout the paper:
\begin{align}
\tau_X^2 \equiv \E \pnorm{X-X'}{}^2 = 2\tr(\Sigma_X),\quad \tau_Y^2\equiv \E\pnorm{Y-Y'}{}^2 = 2\tr(\Sigma_Y).
\end{align}
We also reserve $\kappa$ for the excess kurtosis of $Z_1$:
\begin{align}\label{def:kappa}
\kappa \equiv \E Z_1^4 - 3.
\end{align}
Since $Z_1$ has a Lebesgue density, we have $\kappa \geq c_0 - 2$ for some $c_0 > 0$ that only depends on the distribution of $Z_1$.

Let $I_{[ij]} \equiv (\bm{1}_{(i',j')=(i,j)})_{1\leq i',j'\leq 2}$ be the indicator of the block matrix, and $\Sigma_{[ij]}\equiv (\Sigma_{(i'j')}\bm{1}_{(i',j')=(i,j)})_{1\leq i',j'\leq 2}$. Let
\begin{align}\label{def:GH}
G_{[ij]} \equiv \Sigma^{1/2}\Sigma_{[ij]}\Sigma^{1/2},\quad H_{[ij]} \equiv \Sigma^{1/2}I_{[ij]} \Sigma^{1/2}.
\end{align}
Then $G_{[ij]}^\top = G_{[ji]}$ and $H_{[ij]}^\top = H_{[ji]}$. Let $\bar{G}_{[1,2]}\equiv (G_{[12]}+G_{[21]})/2$.

The following non-null central limit theorem is the first main result of this paper; its proof can be found in Section \ref{sec:proof_clts}.
\begin{theorem}\label{thm:non_null_clt}
	Suppose that Assumption \ref{assump:sepa} holds, and that the spectrum of $\Sigma$ is contained in $[1/M,M]$ for some $M>1$. Then there exists some $C=C(M,Z_1)>0$ such that
	\begin{align*}
	d_{\mathrm{Kol}}\bigg(\frac{ \dcov_\ast^2(\bm{X},\bm{Y})-\dcov^2(X,Y)}{\sigma_n(X,Y)},\mathcal{N}(0,1)\bigg)\leq \frac{C}{(n\wedge p\wedge q)^{1/6}}.
	\end{align*}
	Here $\sigma_n(X,Y)$ can be either $\var^{1/2}(\dcov_\ast^2(\bm{X},\bm{Y}))$ or $\bar{\sigma}_n(X,Y)$, where
	\begin{align*}
	\bar{\sigma}_n^2(X,Y) &\equiv 	\bar{\sigma}_{n,1}^2(X,Y) + \bar{\sigma}_{n,2}^2(X,Y),
	\end{align*}
	with 
	\begin{align*}
	\bar{\sigma}_{n,1}^2(X,Y) &\equiv \frac{4}{n\tau_X^2\tau_Y^2}\bigg[ \pnorm{\Sigma_{XY}\Sigma_{YX}}{F}^2 + \tr(\Sigma_{XY}\Sigma_{Y}\Sigma_{YX}\Sigma_X) +\frac{\pnorm{\Sigma_{XY}}{F}^4\pnorm{\Sigma_X}{F}^2}{2\tau_X^4}\\
	&\qquad\qquad + \frac{\pnorm{\Sigma_{XY}}{F}^4\pnorm{\Sigma_Y}{F}^2}{2\tau_Y^4} -\frac{2\pnorm{\Sigma_{XY}}{F}^2}{\tau_X^2}\tr(\Sigma_{XY}\Sigma_{YX}\Sigma_X)\\
	&\qquad\qquad-\frac{2\pnorm{\Sigma_{XY}}{F}^2}{\tau_Y^2}\tr(\Sigma_{YX}\Sigma_{XY}\Sigma_Y) + \frac{\pnorm{\Sigma_{XY}}{F}^6}{\tau_X^2\tau_Y^2} +\kappa\cdot\Big(\tr(G_{[12]} \circ G_{[12]}) \\
	&\qquad  +  \frac{\pnorm{\Sigma_{XY}}{F}^4}{4\tau_X^4}\tr(H_{[11]}\circ H_{[11]}) + \frac{\pnorm{\Sigma_{XY}}{F}^4}{4\tau_Y^4}\tr(H_{[22]}\circ H_{[22]})\\
	&\qquad -\frac{\pnorm{\Sigma_{XY}}{F}^2}{2\tau_X^2}\tr(H_{[11]}\circ G_{[12]})-\frac{\pnorm{\Sigma_{XY}}{F}^2}{2\tau_Y^2}\tr(H_{[22]}\circ G_{[12]})\\
	&\qquad  + \frac{\pnorm{\Sigma_{XY}}{F}^4}{4\tau_X^2\tau_Y^2}\tr(H_{[11]}\circ H_{[22]})\Big)\bigg],\\
	\bar{\sigma}_{n,2}^2(X,Y)&\equiv \frac{2}{n(n-1) \tau_X^2\tau_Y^2}\big(\pnorm{\Sigma_X}{F}^2\pnorm{\Sigma_Y}{F}^2 +  \pnorm{\Sigma_{XY}}{F}^4\big).
	\end{align*}
	Here $\circ$ is the Hadamard product and $H_{[\cdot\cdot]}$ and $G_{[\cdot\cdot]}$ are given in (\ref{def:GH}) above.
\end{theorem}

\begin{remark}[Variance formula]
	The variance formula $\var\big(\dcov_\ast^2(\bm{X},\bm{Y})\big) = \big(1+\mathfrak{o}(1)\big)\bar{\sigma}^2_n(X,Y)$ is valid in the high dimensional limit $n\wedge p\wedge q\rightarrow\infty$; see Section \ref{section:hoeffding_decomp} ahead for details. The same is true for the variance formula in Theorem \ref{thm:non_null_clt_X} below. 
\end{remark}

\begin{remark}[Convergence rate]
The convergence rate $(n \wedge p \wedge q)^{-1/6}$ results from the application of Chatterjee's second-order Poincar\'e inequality \cite{chatterjee2008new} (see Section \ref{section:clt_truncated_dcov} for details) and is likely to be sub-optimal. Such rate, however, is common in the literature. For example,  \cite[Proposition 3]{gao2021asymptotic} assumed $(X,Y)$ to be jointly Gaussian and employed a martingale CLT theorem to prove the normal limit under the null, but the convergence rate is also only $(pq)^{-1/5} \vee n^{-1/5}$. It is an interesting but challenging task to pin down the optimal Gaussian approximation rate in terms of $(n,p,q)$.
\end{remark}

As mentioned in the introduction, most of the CLTs in the literature so far are derived under the null case that $X$ and $Y$ are independent, and to the best of our knowledge, Theorem \ref{thm:non_null_clt} is the first non-null CLT that applies to a general class of alternatives. Due to the challenging nature of non-null analysis, the proof of the above theorem requires several technically involved and intertwined steps, so an outline will be provided in Section \ref{section:proof_outline} that discusses the relevance of the groundwork laid in Sections \ref{section:preliminaries}-\ref{section:clt_truncated_dcov}, which culminates in the proof of Theorem \ref{thm:non_null_clt} in Section \ref{sec:proof_clts}.

Let us examine the variance structure in more detail. Roughly speaking, in the high dimensional regime $n\wedge p\wedge q\to \infty$, the variance $\bar{\sigma}_n^2(X,Y)$ of $\dcov_\ast^2(\bm{X},\bm{Y})$ only contains two possibly different sources---the first part $\bar{\sigma}_{n,1}^2(X,Y)$ comes from the contribution of the non-degenerate first-order kernel, while the second part $\bar{\sigma}_{n,2}^2(X,Y)$ comes from the contribution of the degenerate second-order kernel, in the Hoeffding decomposition of $\dcov_\ast^2(\bm{X},\bm{Y})$ to be detailed in Section \ref{section:hoeffding_decomp} ahead.

One notable complication of $\bar{\sigma}_n^2(X,Y)$ is the existence of terms with negative signs in the first-order variance $\bar{\sigma}_{n,1}^2(X,Y)$. These terms may contribute to the same order of the leading quantities $\pnorm{\Sigma_{XY}\Sigma_{YX}}{F}^2$ and $\tr(\Sigma_{XY}\Sigma_{Y}\Sigma_{YX}\Sigma_X)$, but a lower bound in Lemma \ref{lem:first_variance_lower} ahead shows that their contributions do not lead to `cancellations' of the main terms. In fact, under the spectrum condition of Theorem \ref{thm:non_null_clt}, the second claim of Lemma \ref{lem:first_variance_lower} indicates the order of $\bar{\sigma}_n^2(X,Y)$ with terms of positive signs only:
\begin{align*}
\bar{\sigma}_n^2(X,Y)\asymp_M  \max\bigg\{\frac{\pnorm{\Sigma_{XY}}{F}^2}{npq}, \frac{1}{n^2} \bigg\}.
\end{align*}
Here the first term in the above maximum is contributed by $\bar{\sigma}_{n,1}^2(X,Y)$ while the second term is contributed by $\bar{\sigma}_{n,2}^2(X,Y)$. Now we consider two regimes:
\begin{itemize}
	\item (\emph{Ultra high-dimensional regime $\sqrt{pq}\gg n$}). In this regime, as $\pnorm{\Sigma_{XY}}{F}^2\lesssim_M p\wedge q\leq \sqrt{pq}$ via Lemma \ref{lem:SigmaXY_F_bound} in the appendix, the variance $\bar{\sigma}_n^2(X,Y)$ in this ultra high-dimensional regime is completely determined by the contribution from the degenerate second-order kernel $\bar{\sigma}_{n,2}^2(X,Y)$:
	\begin{align*}
	\bar{\sigma}_n^2(X,Y) =\big(1+\mathfrak{o}(1)\big) \bar{\sigma}_{n,2}^2(X,Y).
	\end{align*}
	\item (\emph{Moderate high-dimensional regime $\sqrt{pq}\lesssim n$}). In this regime, there are three possibilities:
	\begin{itemize}
		\item If $\pnorm{\Sigma_{XY}}{F}^2\ll (pq)/n$, which includes the null $\Sigma_{XY} = 0$ as a special case, the variance $\bar{\sigma}_{n}^2(X,Y)$ is again completely determined by the degenerate second-order kernel $\bar{\sigma}_{n,2}^2(X,Y)$.
		\item If $\pnorm{\Sigma_{XY}}{F}^2\gg (pq)/n$,  then the variance $\bar{\sigma}_{n}^2(X,Y)$ is completely determined by the non-degenerate first-order kernel $\bar{\sigma}_{n,1}^2(X,Y)$:
		\begin{align*}
		\bar{\sigma}_n^2(X,Y) &= \big(1+\mathfrak{o}(1)\big) \bar{\sigma}_{n,1}^2(X,Y).
		\end{align*}
		If furthermore $\pnorm{\Sigma_{XY}}{F}^2\ll p\wedge q$ (i.e., excluding the critical regime $\pnorm{\Sigma_{XY}}{F}^2\asymp p\wedge q$), then the first-order variance $\bar{\sigma}_{n,1}^2(X,Y) $ can be simplified to be
		\begin{align*}
		\bar{\sigma}_n^2(X,Y) &= \frac{4\big(1+\mathfrak{o}(1)\big)}{n\tau_X^2\tau_Y^2}\Big[ \pnorm{\Sigma_{XY}\Sigma_{YX}}{F}^2+ \tr(\Sigma_{XY}\Sigma_{Y}\Sigma_{YX}\Sigma_X)\Big].
		\end{align*}
		\item If $\pnorm{\Sigma_{XY}}{F}^2\asymp (pq)/n$, the variance $\bar{\sigma}_n^2(X,Y)$ is contributed by both the non-degenerate first-order kernel $\bar{\sigma}_{n,1}^2(X,Y)$ and the degenerate second-order kernel $\bar{\sigma}_{n,2}^2(X,Y)$ so the general variance expression in Theorem \ref{thm:non_null_clt} cannot be further simplified.
	\end{itemize}
\end{itemize}

The smallest eigenvalue condition in Theorem \ref{thm:non_null_clt} excludes the case $X=Y$, but a slight variation of the proof can cover this case as well. We record formally the result below.

\begin{theorem}\label{thm:non_null_clt_X}
	Suppose that Assumption \ref{assump:sepa} holds, and that the spectrum of $\Sigma_X$ is contained in $[1/M,M]$ for some $M>1$. Then there exists some $C=C(M,Z_1)>0$ such that
	\begin{align*}
	d_{\mathrm{Kol}}\bigg(\frac{ \dcov_\ast^2(\bm{X})-\dcov^2(X)}{\sigma_n(X)},\mathcal{N}(0,1)\bigg)\leq \frac{C}{(n\wedge p)^{1/6}}.
	\end{align*}
	Here $\sigma_n(X)$ can be either $\var^{1/2}\big(\dcov_\ast^2(\bm{X})\big)$ or $\bar{\sigma}_n(X,X)$, where
	\begin{align*}
	\bar{\sigma}_n^2(X) &\equiv 	\bar{\sigma}_{n,1}^2(X) +	\bar{\sigma}_{n,2}^2(X),
	\end{align*}
	with 
	\begin{align*}
	\bar{\sigma}_{n,1}^2(X) &\equiv \frac{1}{n \tr^2(\Sigma_X)}\bigg[2 \pnorm{\Sigma_X^2}{F}^2  +\frac{\pnorm{\Sigma_X}{F}^6}{2\tr^2(\Sigma_X)} -\frac{2\pnorm{\Sigma_X}{F}^2 \tr(\Sigma_X^3)}{\tr(\Sigma_X)}\\
	&\quad +\kappa\cdot\Big(\tr(\Sigma_X^2\circ \Sigma_X^2) + \frac{3\pnorm{\Sigma_X}{F}^4}{4\tau_X^4}\tr(\Sigma_X\circ \Sigma_X) - \frac{\pnorm{\Sigma_X}{F}^2}{\tau_X^2}\tr(\Sigma_X^2\circ \Sigma_X)\Big)\bigg],\\
	\bar{\sigma}_{n,2}^2(X)&\equiv \frac{ \pnorm{\Sigma_X}{F}^4}{n(n-1) \tr^2(\Sigma_X)}.
	\end{align*}
\end{theorem}

The variance $\bar{\sigma}_n^2(X)$ in Theorem \ref{thm:non_null_clt_X} is simpler than that in Theorem \ref{thm:non_null_clt}. In fact, a similar consideration using the variance lower bound in Lemma \ref{lem:first_variance_lower}, we may obtain the order of $\bar{\sigma}_n^2(X)$:
\begin{align*}
\bar{\sigma}_n^2(X)&\asymp_M \max \bigg\{\frac{1}{np},\frac{1}{n^2}\bigg\}. 
\end{align*}
Through the Hoeffding decomposition of $\dcov_\ast^2(\bm{X})$, the first and second terms in the maximum are contributed by the variance of the non-degenerate first-order kernel and the the degenerate second-order kernel respectively. Therefore:
\begin{itemize}
	\item In the ultra high-dimensional regime $p\gg n$, the variance $\bar{\sigma}_n^2(X)$ is completely determined by the degenerate second-order kernel $\bar{\sigma}_{n,2}^2(X)$.
	\item In the strictly moderate high-dimensional regime $p\ll n$, the variance $\bar{\sigma}_{n}^2(X)$ is completely determined by the non-degenerate first-order kernel $\bar{\sigma}_{n,1}^2(X)$.
	\item In the critical regime $p\asymp n$, the variance $\bar{\sigma}_{n}^2(X)$ is determined jointly by the first- and second-order kernels $\bar{\sigma}_{n,1}^2(X)$, $\bar{\sigma}_{n,2}^2(X)$ so cannot be in general simplified.
\end{itemize}

\subsection{General non-null CLTs II: generalized kernel distance covariance}
The sample distance covariance $\dcov_\ast^2(\bm{X},\bm{Y})$ can be generalized using kernel functions as follows. Given functions $f_X,f_Y: \R_{\geq 0}\to \R$, and bandwidth parameters $\gamma_X,\gamma_Y>0$, let for $1\leq k,\ell\leq n$ 
\begin{align*}
A_{k\ell}(f_X,\gamma_X)\equiv f_X\big(\pnorm{X_k-X_\ell}{}/\gamma_X\big)\bm{1}_{k\neq \ell},\quad B_{k\ell}(f_Y,\gamma_Y)\equiv f_Y\big(\pnorm{Y_k-Y_\ell}{}/\gamma_Y\big)\bm{1}_{k\neq \ell}.
\end{align*}
It is essential to set the diagonal terms $\{A_{kk}(f_X,\gamma_X)\}_k, \{B_{kk}(f_Y,\gamma_Y)\}_k$ to be $0$, so that the generalized kernel distance covariance to be introduced below can be analyzed in a unified manner; see Proposition \ref{prop:hoef_decomp_kernel} in the appendix for details. Now with $A_{k\ell}^\ast(f_X,\gamma_X), B_{k\ell}^\ast(f_Y,\gamma_Y)$ defined similarly as in (\ref{def:AB}) by replacing $A_{k\ell},B_{k\ell}$ with $A_{k\ell}(f_X,\gamma_X),B_{k\ell}(f_Y,\gamma_Y)$, we may define the generalized sample distance covariance with kernels $f = (f_X, f_Y)$  and bandwidth parameters $\gamma=(\gamma_X,\gamma_Y) \in \R^2_{>0}$ by
\begin{align}\label{def:dcov_empirical_kernel}
\dcov_\ast^2(\bm{X},\bm{Y};f,\gamma)&\equiv \frac{1}{n(n-3)}\sum_{k\neq \ell} A_{k\ell}^\ast (f_X,\gamma_X) B_{k\ell}^\ast(f_Y,\gamma_Y),
\end{align}
and its population version $\dcov^2(X,Y;f,\gamma)$ as in (\ref{intro:dcov_kernel}). Marginal quantities $\dcov^2(X;f,\gamma)$ and $\dcov^2(Y;f,\gamma)$ are defined analogously, and similar to distance correlation, the kernelized distance correlation is defined as
\begin{align*}
\dcor^2(X,Y;f,\gamma) \equiv \frac{\dcov^2(X,Y;f,\gamma)}{\sqrt{\dcov^2(X;f,\gamma)\dcov^2(Y;f,\gamma)}},
\end{align*} 
with convention $\dcor^2(X,Y;f,\gamma) \equiv 0$ if $\dcov^2(X;f,\gamma) \dcov^2(Y;f,\gamma) = 0$.

A more general formulation, when replacing $f_X(\pnorm{X_\ell-X_k}{}/\gamma_X)$ (resp. $f_Y(\pnorm{Y_\ell-Y_k}{}/\gamma_X)$) by some generic bivariate kernel $k_X(X_\ell,X_k)$ (resp. $k_Y(Y_\ell,Y_k)$), is also known as the \emph{Hilbert-Schmidt independence criteria}, see e.g. \cite{gretton2005kernel,gretton2007kernel}, which can in fact be written as the maximum mean discrepancy between the joint distribution and the marginal distributions of $X$ and $Y$; see e.g. \cite[Section 3.3]{sejdinovic2013equivalence} for an in-depth discussion. Two particular important choices for $f$ are the Laplace and Gaussian kernels:
\begin{itemize}
	\item (Laplace kernel) $f(w)=e^{-w}$;
	\item (Gaussian kernel) $f(w) = e^{-w^2/2}$.
\end{itemize}
These kernels have been considered in, e.g., \cite{gretton2012kernel,zhu2020distance}.

\begin{assumption}[Conditions on the kernel $f$]\label{assump:kernel}
	Suppose that $f \in\{f_X,f_Y\}$ is four times differentiable on $(0,\infty)$ such that: 
	\begin{enumerate}
		\item $f$ is bounded on $[0,M)$ for any $M>0$.
		\item For any $\epsilon>0$, $\max_{1\leq\ell\leq 4}\sup_{x\geq \epsilon}|f^{(\ell)}(x)|\leq C_\epsilon$ for some $C_\epsilon>0$.
		\item For any $\epsilon>0$, there exists some $c_\epsilon>0$ such that $\inf_{x\in(\epsilon, \epsilon^{-1})} |f'(x)| \geq c_\epsilon$.
		\item There exists some $\mathfrak{q}>0$ such that $\limsup_{x \downarrow 0} \max_{1\leq \ell \leq 4} x^{\mathfrak{q}} \abs{f^{(\ell)}(x)}<\infty$.
	\end{enumerate}
\end{assumption}

In words, Assumption \ref{assump:kernel}-(1)(2) require the kernel functions $f_X,f_Y$ and its derivatives to be appropriately bounded, (3) requires the first derivative to be bounded from below on any compacta in $(0,\infty)$, and finally (4)  regulates that the derivatives of $f_X,f_Y$ up to the fourth order can only blow up at $0$ with at most a polynomial rate of divergence.  It is easy to check that both the Laplace/Gaussian kernels, and the canonical choice $f(x)=x$ that recovers the distance covariance (up to a scaling factor) all satisfy Assumption \ref{assump:kernel}.

Let $\rho_X\equiv \tau_X/\gamma_X$, $\rho_Y\equiv \tau_Y/\gamma_Y$ and
\begin{align}\label{def:varrho}
\varrho(\gamma)\equiv \frac{f_X'(\rho_X)f_Y'(\rho_Y)}{\gamma_X\gamma_Y}.
\end{align}
We are now ready to state the following non-null CLT for the generalized kernel distance covariance $\dcov_\ast^2(\bm{X},\bm{Y};f,\gamma)$; its proof can be found in Appendix \ref{sec:proof_kernel}.

\begin{theorem}\label{thm:non_null_clt_kernel}
	Suppose that Assumptions \ref{assump:sepa} and \ref{assump:kernel} hold, and that (i) the spectrum of $\Sigma$ (ii) $\rho_X,\rho_Y$ are contained in $[1/M,M]$ for some $M>1$. Then there exists some $C=C(f,M,Z_1)>0$ such that
	\begin{align*}
	d_{\mathrm{Kol}}\bigg(\frac{ \dcov_\ast^2(\bm{X},\bm{Y};f,\gamma)-\dcov^2(X,Y;f,\gamma)}{\sigma_n(X,Y;f,\gamma)},\mathcal{N}(0,1)\bigg)\leq \frac{C}{(n\wedge p\wedge q)^{1/6}}.
	\end{align*}
	Here $\sigma_n(X,Y;f,\gamma)=\varrho(\gamma) \sigma_n(X,Y)$, where $\sigma_n(X,Y)$ is defined in Theorem \ref{thm:non_null_clt}.  In the case $X=Y$, the conclusion continues to hold if (i) is replaced by: (i') the spectrum of $\Sigma_X$ is contained in $[1/M,M]$ for some $M>1$.
\end{theorem}

We note that the conditions posed in the above theorem are not the weakest possible; for instance one may relax all the conditions on the kernels $f=(f_X,f_Y)$ and $(\rho_X,\rho_Y)$ to some growth conditions involving $n,p,q$ at the cost of a more involved error bound, but we have stated the current formulation to avoid unnecessary digressions. 

The key step in the proof of Theorem \ref{thm:non_null_clt_kernel} is to reduce the analysis of $\dcov_\ast^2(\bm{X},\bm{Y};f,\gamma)$ with general kernels $f$ to that of the canonical $\dcov_\ast^2(\bm{X},\bm{Y})$. Analogous to the quantities $U,V$ defined in (\ref{def:UV}) for the canonical distance covariance, let
\begin{align*}
U_{f_X,\gamma_X}(x_1,x_2)&\equiv f_X\big(\pnorm{x_1-x_2}{}/\gamma_X\big)-\E f_X\big(\pnorm{x_1-X}{}/\gamma_X\big)\\
&\qquad\qquad-\E f_X\big(\pnorm{X-x_2}{}/\gamma_X\big)+ \E f_X\big(\pnorm{X-X'}{}/\gamma_X\big),\\
V_{f_Y,\gamma_Y}(y_1,y_2)&\equiv f_Y\big(\pnorm{y_1-y_2}{}/\gamma_Y\big)-\E f_Y\big(\pnorm{y_1-Y}{}/\gamma_Y\big)\\
&\qquad\qquad-\E f_Y\big(\pnorm{Y-y_2}{}/\gamma_Y\big)+ \E f_Y\big(\pnorm{Y-Y'}{}/\gamma_Y\big).
\end{align*}
Then it is shown in Lemma \ref{lem:UV_expansion_kernel} that
\begin{align*}
U_{f_X,\gamma_X} \approx -\frac{f_X'(\rho_X)}{\gamma_X}U,\quad V_{f_Y,\gamma_Y} \approx -\frac{f_Y'(\rho_Y)}{\gamma_Y}V,
\end{align*} 
hence with appropriate control on the remainder terms, it follows from the $U$-statistic representation in (\ref{eqn:kernel_rep2}) that 
\begin{align*}
\dcov_\ast^2(\bm{X},\bm{Y};f,\gamma) \approx \frac{f_X'(\rho_X)}{\gamma_X}\frac{f_Y'(\rho_Y)}{\gamma_Y}\dcov_\ast^2(\bm{X},\bm{Y}) = \varrho(\gamma)\dcov_\ast^2(\bm{X},\bm{Y}).
\end{align*}
Following this line of arguments, we are then able to study the asymptotics of $\dcov_\ast^2(\bm{X},\bm{Y};f,\gamma)$ and $\dcov_\ast^2(\bm{X},\bm{Y})$ in a unified manner; see Appendix \ref{sec:proof_kernel} for detailed arguments.

\subsection{Local CLTs}

As a corollary of the non-null CLT in Theorem \ref{thm:non_null_clt}, we state below a local CLT that will be important to obtain the power formula for the distance correlation test introduced in (\ref{def:test_t}). Its proof is presented in Section \ref{section:proof_local_clt}.

\begin{theorem}\label{thm:local_clt}
	Suppose that Assumption \ref{assump:sepa} holds, and that the spectrum of $\Sigma$ is contained in $[1/M,M]$ for some $M>1$. Let
	\begin{align}\label{def:A}
	A(\Sigma)\equiv \frac{n\pnorm{\Sigma_{XY}}{F}^2}{\pnorm{\Sigma_X}{F}\pnorm{\Sigma_Y}{F}}.
	\end{align}
	Then there exists some constant $C=C(M,Z_1)>0$ such that
	\begin{align*}
	&d_{\mathrm{Kol}}\bigg(\frac{n\big(\tau_X\tau_Y\dcov_\ast^2(\bm{X},\bm{Y})-\pnorm{\Sigma_{XY}}{F}^2\big)}{\sqrt{2} \pnorm{\Sigma_X}{F}\pnorm{\Sigma_Y}{F}}, \mathcal{N}(0,1)\bigg)\leq C\bigg[1\bigwedge \bigg(\frac{1\vee A(\Sigma)^2}{n\wedge p\wedge q}\bigg)^{1/6}\bigg].
	\end{align*}
	If Assumption \ref{assump:kernel} holds and $\rho_X,\rho_Y$ are contained in $[1/M,M]$ for some $M>1$, then the above conclusion holds with $\dcov_\ast^2(\bm{X},\bm{Y})$ replaced by $\dcov_\ast^2(\bm{X},\bm{Y};f,\gamma)/\varrho(\gamma)$ and $C$ replaced by $C' = C'(M,f, Z_1)$.
\end{theorem}

The definition of the local (contiguity) parameter $A(\Sigma)$ is motivated by the critical parameter in the power expansion formula of the (generalized kernel) distance correlation test in Theorem \ref{thm:dcov_power} below. The interesting phenomenon in the local central limit theorem above is that in the local (contiguity) regime $\limsup A(\Sigma)<\infty$, a central limit theorem holds for $\dcov_\ast^2(\bm{X},\bm{Y})$ and $\dcov_\ast^2(\bm{X},\bm{Y};f,\gamma)/\varrho(\gamma)$ with the \emph{null variance} $\sigma^2_{\text{null}}$ in (\ref{intro:var_expansion}), i.e., the variance under $\Sigma_{XY}=0$. Of course, this necessarily implies that (recall $\bar{\sigma}_n^2\equiv \bar{\sigma}_n^2(X,Y)$ defined in Theorem \ref{thm:non_null_clt}) 
\begin{align}\label{ineq:ratio_var_contiguity}
\frac{\bar{\sigma}^2_n}{\sigma^2_{\textrm{null}}}\rightarrow 1,\quad\quad \textrm{if }\quad \limsup A(\Sigma)<\infty,
\end{align}
which can be verified via elementary calculations (see e.g. (\ref{ineq:local_clt_1}) ahead). This fact will be crucial in Theorem \ref{thm:dcov_power} ahead, where we obtain the asymptotic exact power formula for the distance correlation test using the distance correlation itself (or equivalently, $A(\Sigma)$) as the critical parameter.

\section{Generalized kernel distance correlation tests}\label{section:power_formula}

In this section, we study the performance of the distance correlation test $\Psi(\bm{X},\bm{Y};\alpha)$ in (\ref{def:test_t}) and its kernel generalizations for the null hypothesis $H_0: X \textrm{ is independent of } Y$, or equivalently under our Gaussian assumption, $\Sigma_{XY} = 0$. 

Let us start with a motivation for the test $\Psi(\bm{X},\bm{Y};\alpha)$ in (\ref{def:test_t}) by explaining its connection to the non-null CLT derived in Theorem \ref{thm:non_null_clt} for the sample distance covariance $\dcov_*^2(\bm{X},\bm{Y})$. The null part of Theorem \ref{thm:non_null_clt} (i.e., the case of independent $X$ and $Y$) motivates the following `oracle' independence test: for any prescribed size $\alpha\in(0,1)$,
\begin{align}\label{def:test_oracle}
\tilde{\Psi}(\bm{X},\bm{Y};\alpha) \equiv \bm{1}\bigg(\biggabs{\frac{\dcov_\ast^2(\bm{X},\bm{Y})}{\sigma_{\textrm{null}}}}>z_{\alpha/2}\bigg),
\end{align}
where $\sigma^2_{\textrm{null}}$ is the variance of $\dcov_*^2(\bm{X},\bm{Y})$ under the null in (\ref{intro:var_expansion}). Since $\dcov^2(X,Y) = 0$ under the null, Theorem \ref{thm:non_null_clt} implies immediately that the above test has an asymptotic size of $\alpha$. The test $\tilde{\Psi}(\bm{X},\bm{Y};\alpha)$, however, is not practical because even under the null, $\sigma^2_{\textrm{null}}$ might still depend on the unknown marginal distributions of $X$ and $Y$. To see the connection between $\tilde{\Psi}(\bm{X},\bm{Y};\alpha)$ in (\ref{def:test_oracle}) and $\Psi(\bm{X},\bm{Y};\alpha)$ in (\ref{def:test_t}), note that by some preliminary variance bounds, $\dcov_*^2(\bm{X})$ and $\dcov_*^2(\bm{Y})$ appearing in the denominator of $\Psi(\bm{X},\bm{Y};\alpha)$ in (\ref{def:test_t}) will concentrate around their mean values $\dcov^2(X)$ and $\dcov^2(Y)$ respectively (cf.  Lemma \ref{lem:ratio_dcov}). Furthermore, by the mean and variance formula in (\ref{intro:mean_expansion}) and (\ref{intro:var_expansion}) (cf. Theorems \ref{thm:dcov_expansion} and \ref{thm:variance_expansion}), 
\begin{align*}
\sigma_{\textrm{null}}^2 = \frac{2\big(1+\mathfrak{o}(1)\big)}{n^2}\dcov^2(X)\dcov^2(Y). 
\end{align*}
The above identity implies the asymptotic equivalence between $\tilde{\Psi}(\bm{X},\bm{Y};\alpha)$ in (\ref{def:test_oracle}) and $\Psi(\bm{X},\bm{Y};\alpha)$ in (\ref{def:test_t}) under the null, showing in particular that $\Psi(\bm{X},\bm{Y};\alpha)$ will also have an asymptotic size of $\alpha$. The rest of the section is devoted to studying the power asymptotics of $\Psi(\bm{X},\bm{Y};\alpha)$ and its kernel generalizations.

\subsection{Power universality}
Recall the generalized kernel distance covariance in (\ref{def:dcov_empirical_kernel}) with kernel functions $f = (f_X,f_Y)$ and bandwidth parameters $\gamma = (\gamma_X,\gamma_Y)$. Let the generalized kernel distance correlation test, i.e., a kernelized version of $\Psi(\bm{X},\bm{Y};\alpha)$, be defined by
\begin{align}\label{def:dist_cor_kernel}
\Psi_{f,\gamma}(\bm{X},\bm{Y};\alpha)\equiv \bm{1}\bigg(  \biggabs{\frac{n\cdot\dcov_\ast^2(\bm{X},\bm{Y};f,\gamma)}{\sqrt{2\dcov_\ast^2(\bm{X};f,\gamma)\cdot \dcov_\ast^2(\bm{Y};f,\gamma)}}}>z_{\alpha/2}\bigg).
\end{align}
The factor $n$ in the definition above is sometimes replaced by $\sqrt{n(n-1)}$ (e.g. \cite{gao2021asymptotic}), but this will make no difference in the theory below. Using the (local) central limit theorems derived in Theorem \ref{thm:local_clt}, the following result gives a unified power expansion formula for the distance correlation test $\Psi(\bm{X},\bm{Y};\alpha)$ and the generalized kernel distance correlation test $\Psi_{f,\gamma}(\bm{X},\bm{Y};\alpha)$. Its proof can be found in Section \ref{section:proof_power_formula}.
\begin{theorem}\label{thm:dcov_power}
	Suppose that Assumption \ref{assump:sepa} holds, and that the spectrum of $\Sigma$ is contained in $[1/M,M]$ for some $M>1$. Then there exists some constant $C=C(\alpha,M,Z_1)>0$ such that
	\begin{align*}
	&\biggabs{\E_\Sigma \Psi(\bm{X},\bm{Y};\alpha)-\Prob\big(\bigabs{\mathcal{N}\big(m_n(\Sigma),1\big)}> z_{\alpha/2}\big)}\leq \frac{C}{(n\wedge p\wedge q)^{1/7}}.
	\end{align*}
	Here $m_n(\Sigma)$ can be either
	\begin{align*}
	\frac{n\dcov^2(X,Y)}{\sqrt{2 \dcov^2(X,X)\dcov^2(Y,Y)}} = \frac{n\dcor^2(X,Y)}{\sqrt{2}}\textrm{ or }\frac{n\pnorm{\Sigma_{XY}}{F}^2}{\sqrt{2}\pnorm{\Sigma_X}{F}\pnorm{\Sigma_Y}{F}} = \frac{A(\Sigma)}{\sqrt{2}}.
	\end{align*}
	If Assumption \ref{assump:kernel} holds and $\rho_X,\rho_Y$ are contained in $[1/M,M]$ for some $M>1$, then the above conclusion holds with $\Psi(\bm{X},\bm{Y};\alpha)$ replaced by $\Psi_{f,\gamma}(\bm{X},\bm{Y};\alpha)$ and $C$ replaced by $C' = C'(\alpha,M,f,Z_1) > 0$.
\end{theorem}

A direct message of the above theorem is that, interestingly, for a large class of kernels $f=(f_X,f_Y)$ and bandwidth parameters $\gamma = (\gamma_X,\gamma_Y)$, the generalized kernel distance correlation test $\Psi_{f,\gamma}(\bm{X},\bm{Y};\alpha)$ in (\ref{def:dist_cor_kernel}) exhibits exactly the same power behavior with the distance correlation test $\Psi(\bm{X},\bm{Y};\alpha)$ in (\ref{def:test_t}) in the high dimensional limit $n\wedge p\wedge q\to \infty$.  Here we have focused on deterministic choices of $\gamma_X,\gamma_Y$ merely for simplicity of exposition, but following \cite{zhu2020distance}, analogous results for data driven choices of $\gamma_X,\gamma_Y$ can also be proved with further concentration arguments, for instance for the popular choice 
\begin{align*}
\gamma_X&\equiv \mathrm{median}\big\{\pnorm{X_s-X_t}{}:s\neq t\big\},\quad \gamma_Y\equiv \mathrm{median}\big\{\pnorm{Y_s-Y_t}{}:s\neq t\big\}.
\end{align*}
We omit here formal developments along these lines.

The proof of Theorem \ref{thm:dcov_power} crucially depends on the local central limit theorem in Theorem \ref{thm:local_clt}. An interesting feature of Theorem \ref{thm:dcov_power} is that although one may expect that the power formula of the distance correlation test $\Psi(\bm{X},\bm{Y};\alpha)$ in (\ref{def:test_t}) and the generalized kernel distance correlation test $\Psi_{f,\gamma}(\bm{X},\bm{Y};\alpha)$ in (\ref{def:dist_cor_kernel}) involves the complicated expression of the variance $\bar{\sigma}_n^2$ in Theorem \ref{thm:non_null_clt}, in fact only the null variance plays a role as in Theorem \ref{thm:local_clt}. The main reason for this phenomenon to occur is due to the fact that \emph{the regime in which the local central theorem in Theorem \ref{thm:local_clt} with the null variance holds covers the entire local contiguity regime $\limsup A(\Sigma)<\infty$}. In other words:
\begin{itemize}
	\item In the contiguity regime $\limsup A(\Sigma)<\infty$, the ratio of the non-null variance and the null variance is asymptotically $1$, cf. (\ref{ineq:ratio_var_contiguity}). 
	\item In the large departure regime $A(\Sigma)\to \infty$, both the distance correlation test $\Psi(\bm{X},\bm{Y};\alpha)$ in (\ref{def:test_t}) and the generalized kernel distance correlation test $\Psi_{f,\gamma}(\bm{X},\bm{Y};\alpha)$ in (\ref{def:dist_cor_kernel}) achieve asymptotically full power. 
\end{itemize}
As a result, the `driving parameter' $n\dcor^2(X,Y)/\sqrt{2}$ (or equivalently $A(\Sigma)/\sqrt{2}$) in the power formula for the distance correlation test $\Psi(\bm{X},\bm{Y};\alpha)$ in (\ref{def:test_t}) and its kernelized version $\Psi_{f,\gamma}(\bm{X},\bm{Y};\alpha)$ in (\ref{def:dist_cor_kernel}), inherited from the local CLT in Theorem \ref{thm:local_clt} is in a similar form of the test itself, although its proof to reach such a conclusion is far from being obvious.

\subsection{Minimax optimality}
Theorems \ref{thm:dcov_power} directly implies a separation rate for the (generalized) distance correlation test in the Frobenius norm $\pnorm{\cdot}{F}$ in a minimax framework. To formulate this, for any $\zeta >0$, $M > 1$, and $\Sigma_0 \equiv \diag(\Sigma_X,\Sigma_Y)$, consider the alternative class
\begin{align*}
\Theta(\zeta,\Sigma_0;M)\equiv& \Big\{\Sigma=
\begin{pmatrix}
\Sigma_X & \Sigma_{XY}\\
\Sigma_{YX} & \Sigma_Y
\end{pmatrix}
\in\R^{(p+q)\times(p+q)}: \pnorm{\Sigma_{XY}}{F}^2 \geq \zeta\sqrt{pq}/n, M^{-1} \leq \lambda_{\min}(\Sigma) \leq \lambda_{\max}(\Sigma) \leq M\Big\}.
\end{align*}
A direct consequence of Theorem \ref{thm:dcov_power} is the following (for simplicity we only state the result for the distance correlation test $\Psi(\bm{X},\bm{Y};\alpha)$ in  (\ref{def:test_t})).
\begin{corollary}
	Fix $\alpha \in (0,1)$. Suppose that Assumption \ref{assump:sepa} holds. 
	Then there exists some constant $C=C(\alpha,M,Z_1)>0$ such that the distance correlation test (\ref{def:test_t}) satisfies
	\begin{align*}
	&\sup_{\Sigma \in \Theta(\zeta,\Sigma_0;M)} \big(\E_{\Sigma_0}\Psi(\bm{X},\bm{Y};\alpha) + \E_{\Sigma}(1-\Psi(\bm{X},\bm{Y};\alpha))\big) \leq \alpha+  C\bigg[e^{-\zeta^2/C}+\frac{1}{(n\wedge p\wedge q)^{1/7}}\bigg].
	\end{align*}
\end{corollary}

In particular, the above corollary shows that the distance correlation test (\ref{def:test_t}) gives a separation rate in $\pnorm{\cdot}{F}$ of order $(pq)^{1/4}/n^{1/2}$, i.e., the testing error (Type I $+$ Type II error) on the left side is bounded by any prescribed $\alpha$ for $\zeta\rightarrow\infty$ in the regime (\ref{def:hd_regime}). In view of the power universality derived in Theorem \ref{thm:dcov_power}, the above results continues to hold when the distance correlation test $\Psi(\bm{X},\bm{Y};\alpha)$ in (\ref{def:test_t}) is replaced by the generalized kernel distance correlation test $\Psi_{f,\gamma}(\bm{X},\bm{Y};\alpha)$ in (\ref{def:dist_cor_kernel}) under the assumption that Assumption \ref{assump:kernel} holds and $\rho_X,\rho_Y$ are contained in $[1/M,M]$ for some $M>1$.

The separation rate $(pq)^{1/4}/n^{1/2}$ in $\pnorm{\cdot}{F}$, as will be shown in the following theorem, cannot be improved in a minimax sense. While previous covariance testing literature has mostly focused on the likelihood-ratio test \cite{jiang2013central, jiang2015likelihood, qi2019limiting, dette2020likelihood, dornemann2023likelihood}, this implies that the (generalized) distance correlation tests (\ref{def:test_t}) and (\ref{def:dist_cor_kernel}) are rate-optimal in this minimax sense. We prove the lower bound in the special case of Gaussian distribution in (\ref{def:main_dist}).

\begin{theorem}\label{thm:minimax_lower}
	Suppose that $Z_1\stackrel{d}{=}\mathcal{N}(0,1)$, and $\sqrt{pq}/n\leq M$ for some $M>1$. Then for any small $\delta \in (0,1)$, there exists some positive constant $\zeta = \zeta(\delta,M)$ such that 
	\begin{align*}
	\inf_{\psi}\sup_{\Sigma \in \Theta(\zeta,\Sigma_0;M)} \big(\E_{\Sigma_0}\psi(\bm{X},\bm{Y}) + \E_{\Sigma}(1-\psi(\bm{X},\bm{Y}))\big)\geq 1-\delta,
	\end{align*}
	where the infimum is taken over all measurable test functions. 
\end{theorem}

The above theorem improves \cite[Theorem 1]{ramdas2016minimax} by requiring $\sqrt{pq}/n\lesssim 1$ rather than $(p+q)/n\lesssim 1$ therein. Note that this improvement in terms of a single condition on $\sqrt{pq}/n$ is particularly compatible with alternative class $\Theta(\zeta,\Sigma_0,M)$ defined above. 

The proof of Theorem \ref{thm:minimax_lower} follows a standard minimax reduction in that we only need to find a prior $\Pi$ on $\Theta(\zeta,\Sigma_0;M)$ with sufficient separation from $\Sigma_0$, while at the same time the chi-squared divergence between the posterior density corresponding to $\Pi$ and the density corresponding to $\Sigma_0$ is small. For $\Sigma_0=I$, the prior $\Pi$ we construct takes the form 
\begin{align*}
\Sigma_{u,v}(a) = 
\begin{pmatrix}
I_p & a\tilde{u}\tilde{v}^\top\\
a\tilde{v}\tilde{u}^\top & I_q
\end{pmatrix},
\end{align*}
with component-wise independent priors $\tilde{u}_i\sim_{\textrm{i.i.d.}}\sqrt{q}\cdot\textrm{Unif}\{\pm 1\}$ for $\tilde{u}\in \R^p$ and $\tilde{v}_j\sim_{\textrm{i.i.d.}}\sqrt{p}\cdot\textrm{Unif}\{\pm 1\}$ for $\tilde{v}\in \R^q$, and some $a>0$ to be chosen in the end. The calculations of the chi-squared divergence require an exact evaluation of the eigenvalues of certain inverse of $\Sigma_{u,v}(a)$, which eventually leads to a bound of order $a^4 n^2p^3q^3$. So under the constraint that the chi-squared distance is bounded by some sufficiently small constant, the maximal choice $a=a_\ast \asymp n^{-1/2}p^{-3/4}q^{-3/4}$ leads to a minimax separation rate in $\pnorm{\cdot}{F}^2$ of the order $\pnorm{\Sigma_{u,v}(a_\ast)-I}{F}^2\asymp \pnorm{a_\ast \tilde{u}\tilde{v}^\top}{F}^2 =a_\ast^2 \pnorm{\tilde{u}}{}^2\pnorm{\tilde{v}}{}^2=a_\ast^2 p^2q^2\asymp (pq)^{1/2}/n$. Details of the arguments can be found in Section \ref{section:proof_minimax} in the supplement.

\begin{remark}
	In Theorem \ref{thm:minimax_lower} above, the growth condition $\sqrt{pq}/n\leq M$ for some $M>1$ is similar to the condition `$p/n\leq M$' in the covariance testing literature (e.g. \cite{cai2013optimal}), under which the lower bound construction mentioned above is valid. Whether this condition can be removed (and similarly the condition in \cite{cai2013optimal}) remains open.
\end{remark}

\section{Simulation studies}\label{sec:simulation}

In this section, we perform a small scale simulation study to validate the theoretical results established in previous sections. We consider the balanced case $p = q$ under the following data-generating scheme: i.i.d. across $j \in [p]$, 
\begin{align}\label{eq:sim_setup}
(X_j, Y_j) \stackrel{d}{=} (\sqrt{\rho}Z_1 + \sqrt{1-\rho}Z_2, \sqrt{\rho}Z_1 + \sqrt{1-\rho}Z_3),
\end{align}
where $\rho\in(0,1)$ is the dependence parameter, and $Z_1,Z_2,Z_3$ are independent variables with mean zero and variance one. We carry out the simulation primarily in the case where $Z_1$-$Z_3$ are standard normal, which is a special case of (\ref{def:main_dist}) with $\Sigma_X = \Sigma_Y = I_p$ and $\Sigma_{XY} = \rho I_p$. Non-Gaussianity is examined in Figure \ref{fig:unif} below.

\begin{figure}[h]
	\begin{minipage}[b]{0.32\textwidth}
		\includegraphics[width=\textwidth]{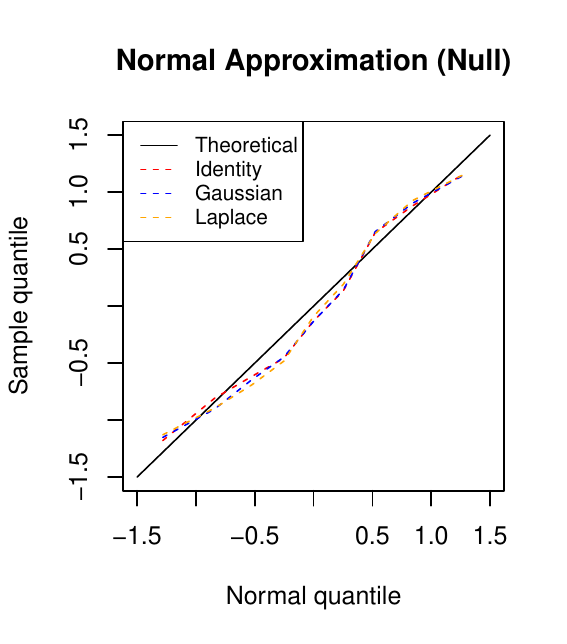}
	\end{minipage}
	\begin{minipage}[b]{0.32\textwidth}
		\includegraphics[width=\textwidth]{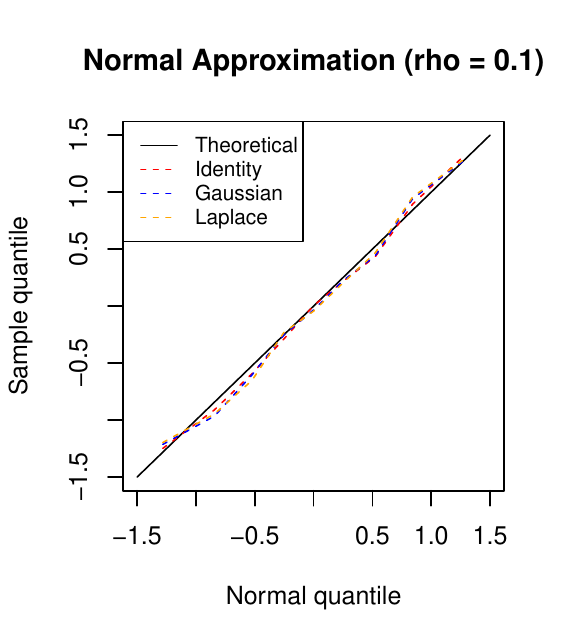}
	\end{minipage}
	\begin{minipage}[b]{0.32\textwidth}
		\includegraphics[width=\textwidth]{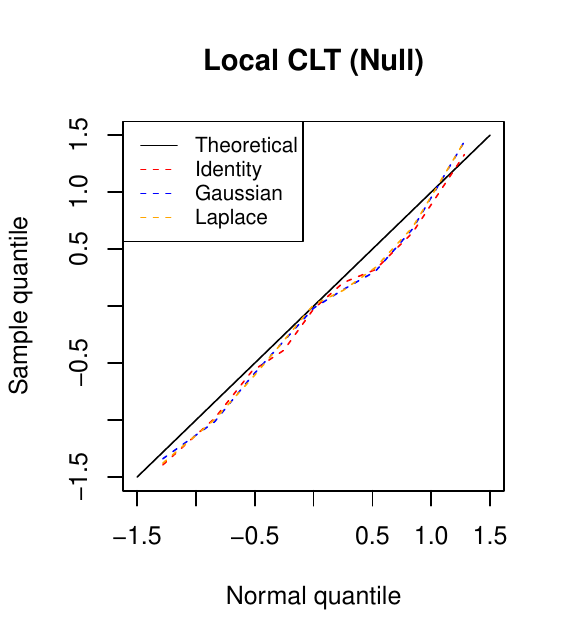}
	\end{minipage}
	\caption{Verification of CLTs. Solid lines correspond to the standard normal quantiles, and dashed lines correspond to sample quantiles with the identity, Gaussian, and Laplace kernels, respectively. Simulation parameters: $(n,p,q) = (1000,100,100)$, $B = 200$ replications, bandwidth choices $\rho_X = \rho_Y = \sqrt{2}$ for both Gaussian and Laplace kernels.}
	\label{fig:gau_clt}
\end{figure}

We start by verifying the CLTs derived in Theorems \ref{thm:non_null_clt} and \ref{thm:local_clt} in Figure \ref{fig:gau_clt}. We take $\rho = 0$ for the null case and $\rho = 0.1$ for the non-null case, and compare the normal quantiles with the corresponding sample quantiles. Normal approximation appears to be accurate in all three cases. 

\begin{figure}[h]
	\begin{minipage}[b]{0.32\textwidth}
		\includegraphics[width=\textwidth]{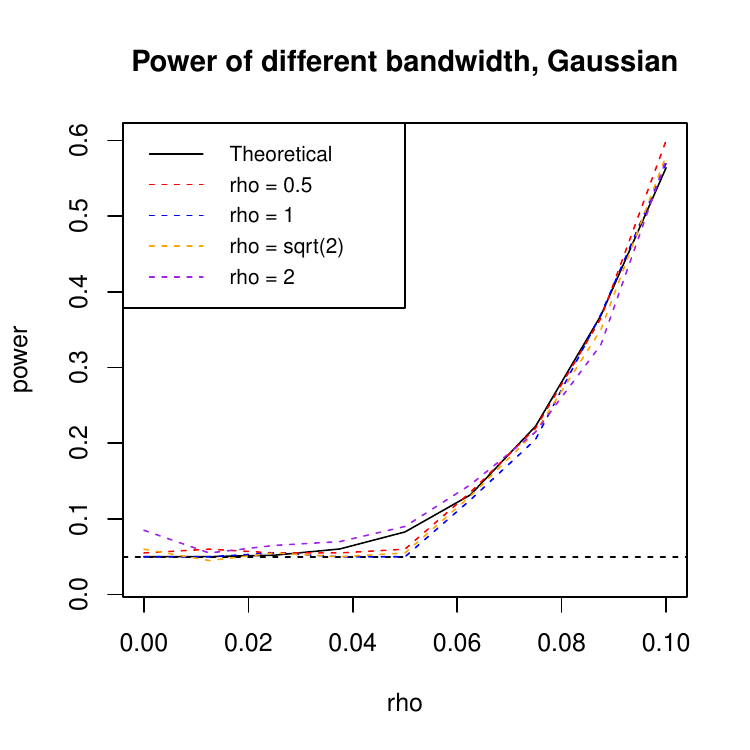}
	\end{minipage}
	\begin{minipage}[b]{0.32\textwidth}
		\includegraphics[width=\textwidth]{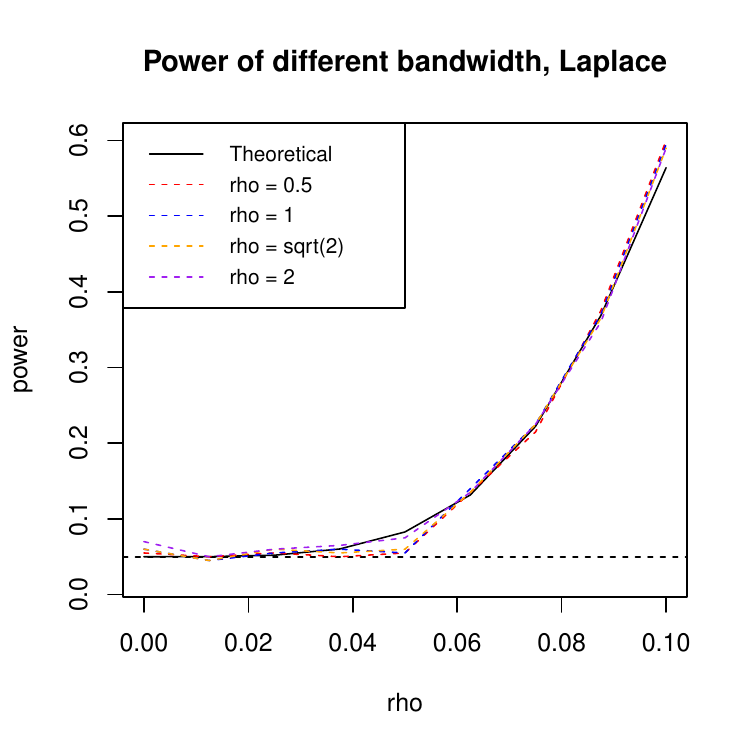}
	\end{minipage}
	\begin{minipage}[b]{0.32\textwidth}
		\includegraphics[width=\textwidth]{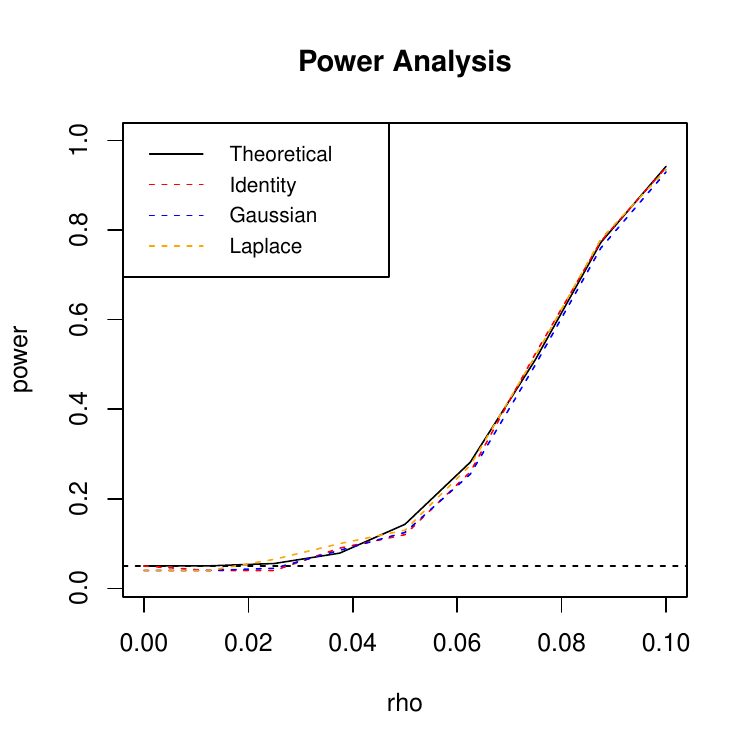}
	\end{minipage}
	\caption{Verification of power universality in choice of bandwidth parameter (left and middle) and choice of kernel (right). Solid lines correspond to the standard normal quantiles, and dashed lines correspond to sample quantiles.} 
	\label{fig:gau_universality}
\end{figure}

Figure \ref{fig:gau_universality} verifies power universality demonstrated via Theorem \ref{thm:dcov_power} in two aspects: (i) the choice of kernel; (ii) the choice of bandwidth parameters $\gamma_X,\gamma_Y$ when using the Gaussian and Laplace kernels. The first two figures illustrates the second point, where the Gaussian and Laplace kernels are used with different bandwidth parameters $\rho_X = \rho_Y \in\{0.5,1,\sqrt{2},5\}$. The third figure uses a fixed bandwidth $\rho_X = \rho_Y = \sqrt{2}$ for both Gaussian and Laplace kernels and compares the performances of different kernels. 

\begin{figure}[h]
	\begin{minipage}[b]{0.32\textwidth}
		\includegraphics[width=\textwidth]{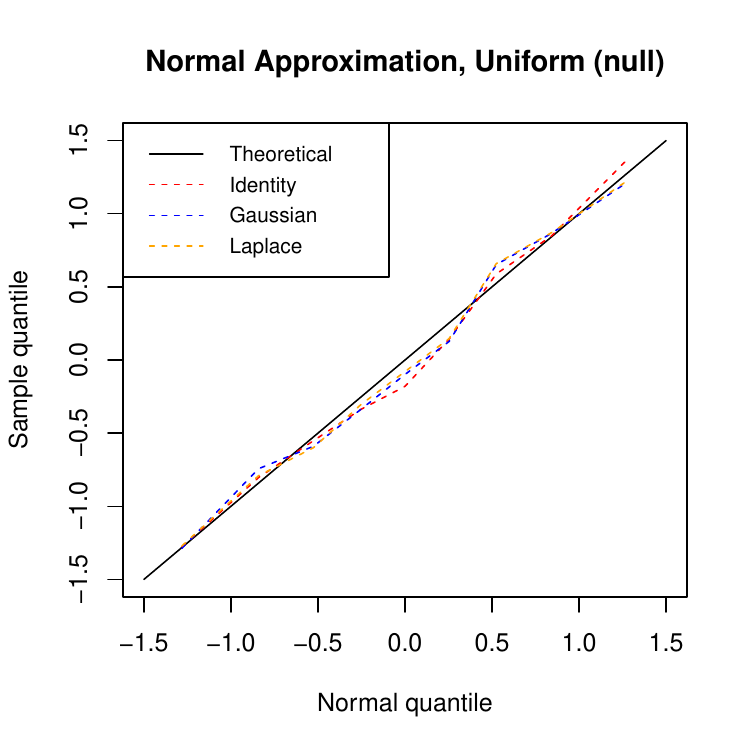}
	\end{minipage}
	\begin{minipage}[b]{0.32\textwidth}
		\includegraphics[width=\textwidth]{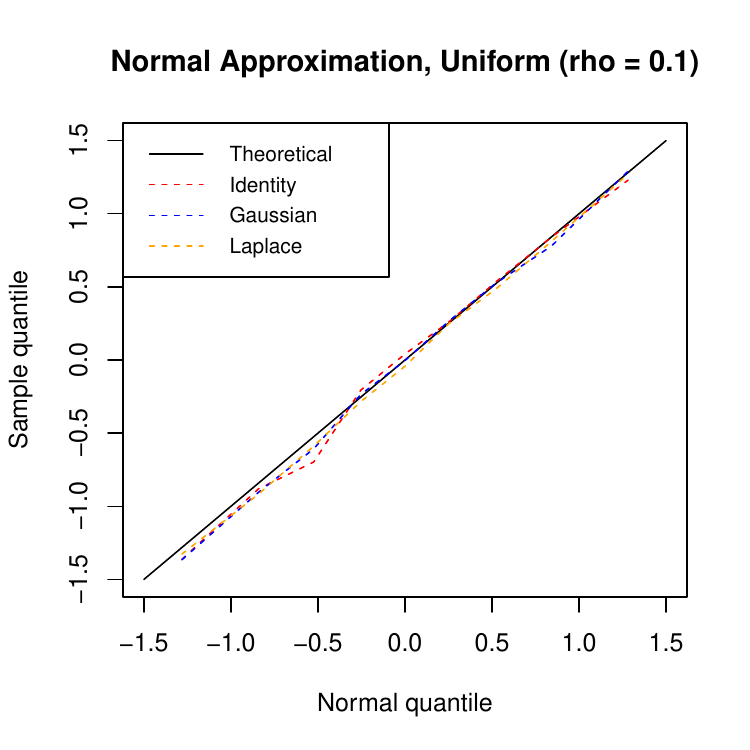}
	\end{minipage}
	\begin{minipage}[b]{0.32\textwidth}
		\includegraphics[width=\textwidth]{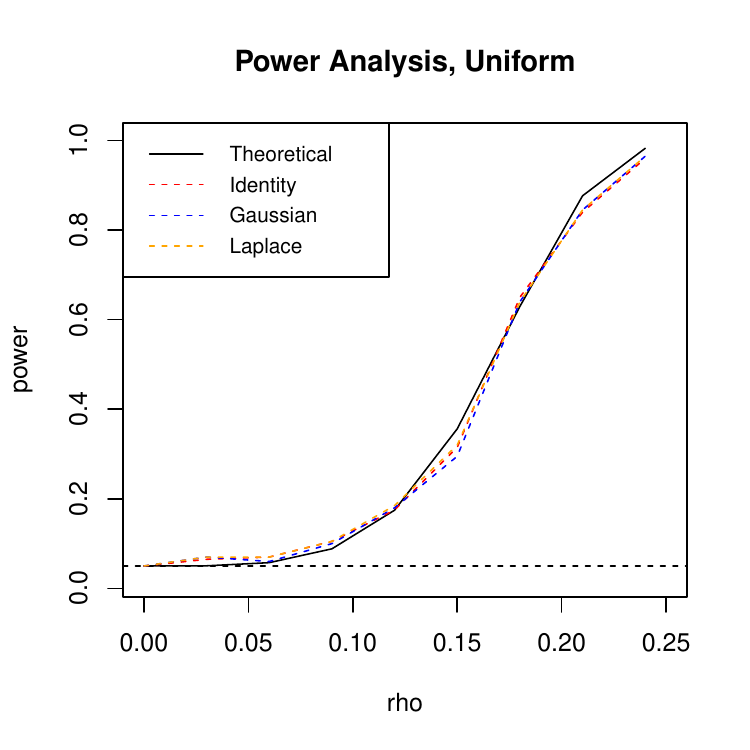}
	\end{minipage}
	
	\begin{minipage}[b]{0.32\textwidth}
		\includegraphics[width=\textwidth]{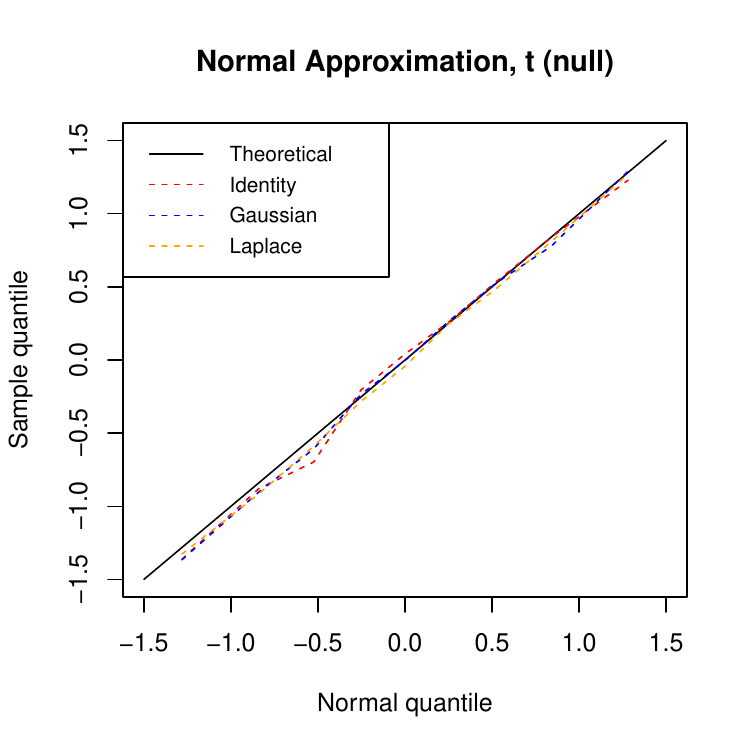}
	\end{minipage}
	\begin{minipage}[b]{0.32\textwidth}
		\includegraphics[width=\textwidth]{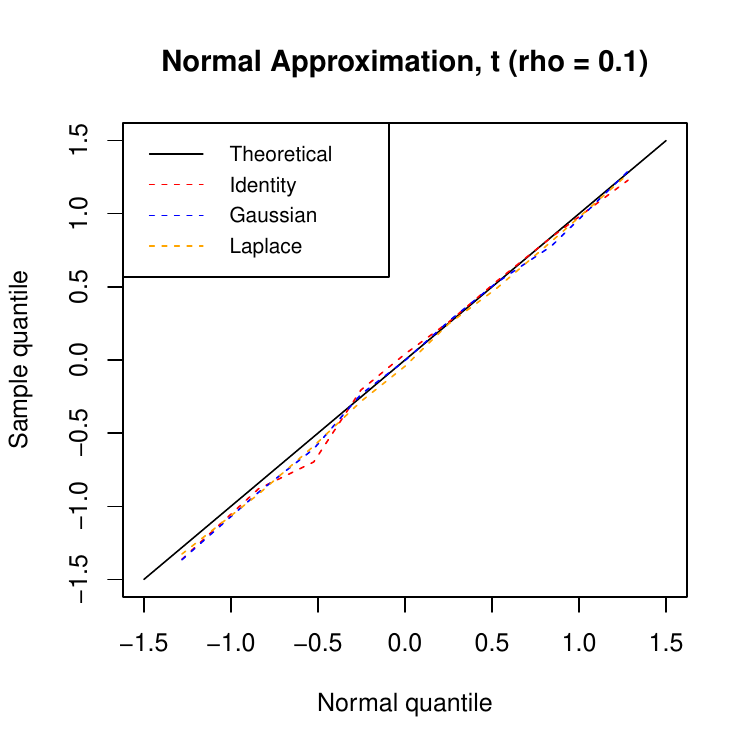}
	\end{minipage}
	\begin{minipage}[b]{0.32\textwidth}
		\includegraphics[width=\textwidth]{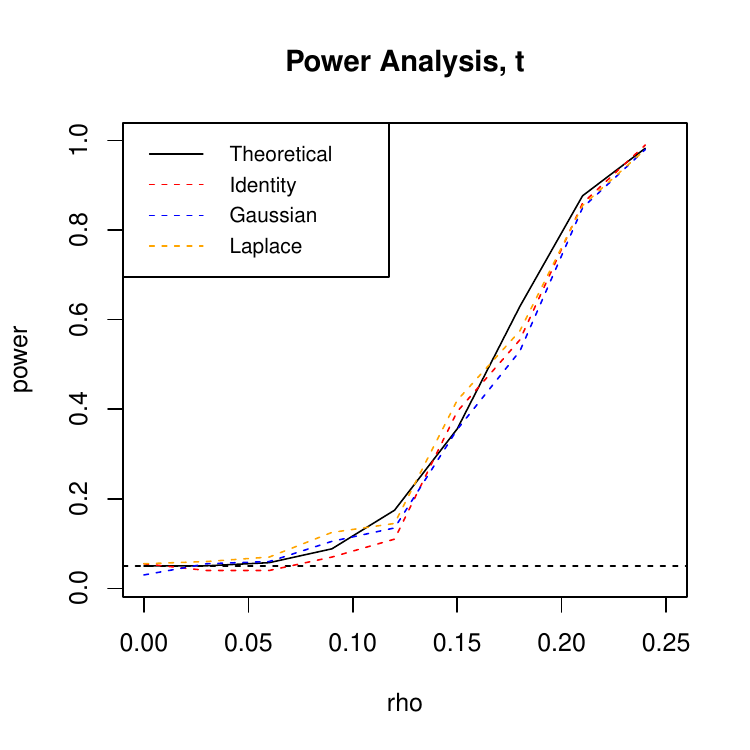}
	\end{minipage}
	\caption{Verification of CLTs and power expansion for uniform (top 3 figures) and $t$- (bottom 3 figures) distributed data. Simulation parameters: $(n,p,q) = (100,100,100)$, $B = 200$ replications, bandwidth choices $\rho_X = \rho_Y = \sqrt{2}$ for both Gaussian and Laplace kernels.} 
	\label{fig:unif}
\end{figure}

Finally we examine the robustness of our theory for non-Gaussian data. We take two choices for $Z_1$-$Z_3$ in the set-up (\ref{eq:sim_setup}): (i) uniform distribution on $[-\sqrt{3},\sqrt{3}]$; (ii) $t$-distribution with $4$ degrees of freedom scaled by $\sqrt{2}$. These parameters are chosen such that $Z_1$-$Z_3$ have mean zero and variance one. Normal approximation and power universality are examined in Figure \ref{fig:unif} for the uniform distribution and the (rescaled) $t$-distribution. These figures suggest that our theory continues to hold for a broader class of data distributions.

\section{Concluding remarks and open questions}\label{sec:discussion}

In this paper, we establish in Theorem \ref{thm:non_null_clt} a general non-null central limit theorem (CLT) for the sample distance covariance $\dcov_\ast^2(\bm{X},\bm{Y})$ in the high dimensional regime $n\wedge p\wedge q \to \infty$ under a separable covariance structure of $(X,Y)$ and a spectral condition on its covariance $\Sigma$. The non-null CLT then applies to obtain a first-order power expansion for the distance correlation test $\Psi(\bm{X},\bm{Y};\alpha)$ in (\ref{def:test_t}):
\begin{align}\label{ineq:power_conc}
\E \Psi(\bm{X},\bm{Y};\alpha) \sim  \Prob\bigg(\biggabs{\mathcal{N}\bigg(\frac{n\dcor^2(X,Y)}{\sqrt{2}},1\bigg) }>z_{\alpha/2}\bigg).
\end{align}
The non-null CLT and the power expansion (\ref{ineq:power_conc}) are also established for a more general class of Hilbert-Schmidt kernel distance covariance, and the associated generalized kernel distance correlation test $\Psi_{f,\gamma}(\bm{X},\bm{Y};\alpha)$ in (\ref{def:dist_cor_kernel}), under mild conditions on the kernels and the bandwidth parameters. This result in particular implies that the generalized kernel distance correlation test admits a universal power behavior with respect to a wide range of choices of kernels and bandwidth parameters.

An important open question is the universality of the power expansion formula (\ref{ineq:power_conc}) and the non-null CLT with respect to more general distributions of $(X,Y)$. The first question of power universality, recorded below, is motivated by the fact that the power expansion formula in the form of (\ref{ineq:power_conc}) does not explicitly involve any specific form of $(X,Y)$ assumed in (\ref{def:main_dist}).

\begin{problem}
	Establish the universality of (\ref{ineq:power_conc}) for general data distributions. 
\end{problem}

Note that in this paper we established (\ref{ineq:power_conc}) by a precise mean and variance expansion for the sample distance covariance $\dcov_\ast^2(X,Y)$. This is the place where the specific form of the data generating distribution in (\ref{def:main_dist}) is crucially used. As a preliminary step towards both power asymptotics and non-null CLT, it is therefore of great interest to investigate:

\begin{problem}
	Obtain (asymptotic) mean and variance formulae for $\dcov_\ast^2(\bm{X},\bm{Y})$ for general data distributions in the \emph{entire} high dimensional regime $n\wedge p\wedge q\to \infty$. 
\end{problem}

In principle, one can obtain `some' mean and variance formulae for general data distributions by expanding the square root $\pnorm{X_\ell-X_k}{},\pnorm{Y_\ell-Y_k}{}$ to further sufficient many `higher-order terms' in (\ref{ineq:sqrt_norm_expansion}) below. However, it seems likely this approach will suffer from significant deficiencies in certain regimes within $n\wedge p\wedge q\to \infty$ as a cost of handling the `residual term' of the highest order. In fact, even with the current distributional form in (\ref{def:main_dist}), it is already a fairly complicated task to handle the residual terms sharply enough to allow a mean and variance expansion in the entire regime $n\wedge p\wedge q \to \infty$; see Section \ref{section:proof_outline} below for an outline of the complications involved.  A new approach may be needed for this problem.

\section{Proof road-map of Theorem \ref{thm:non_null_clt}}\label{section:proof_outline}
We give a road-map for the proof of the main Theorem \ref{thm:non_null_clt}. The basic strategy is to identify the `main terms' in the Hoeffding decomposition of the 4-th order $U$-statistics representation in Proposition \ref{prop:hoef_decomp}. An immediate problem is that the $U,V$ functions in (\ref{def:UV}) involve the square root of the squared $\ell_2$ norm which causes differentiability problems. A simple idea is to use the expansion
\begin{align}\label{ineq:sqrt_norm_expansion}
\frac{\pnorm{X_1-X_2}{}}{\tau_X} &= \bigg(1+\frac{\pnorm{X_1-X_2}{}^2-\tau_X^2}{\tau_X^2}\bigg)^{1/2}\approx \frac{\pnorm{X_1-X_2}{}^2-\tau_X^2}{2\tau_X^2},\nonumber\\
\frac{\pnorm{Y_1-Y_2}{}}{\tau_Y} &= \bigg(1+\frac{\pnorm{Y_1-Y_2}{}^2-\tau_Y^2}{\tau_Y^2}\bigg)^{1/2}\approx \frac{\pnorm{Y_1-Y_2}{}^2-\tau_Y^2}{2\tau_Y^2},
\end{align} 
as the fluctuation of $\pnorm{X_1-X_2}{}^2$ (resp. $\pnorm{Y_1-Y_2}{}^2$) around $\tau_X^2$ (resp. $\tau_Y^2$) is expected to be of smaller order than $\tau_X^2$ (resp. $\tau_Y^2$) in high dimensions. Proceeding with this heuristic, with some book-keeping calculations, we may obtain the following approximation of $U,V$ functions:
\begin{align}\label{ineq:UV_approx}
U(x_1,x_2)\approx-x_1^\top x_2/\tau_X, \quad V(y_1,y_2)\approx -y_1^\top y_2/\tau_Y.
\end{align}
Now if we replace $U,V$ in the $k$ function defined in (\ref{eqn:kernel_rep2}) with the above approximation (\ref{ineq:UV_approx}), the approximate first- and second-order kernels $g_{1,\ast},g_{2,\ast}$ associated with the $k$ function may be computed as follows:
\begin{align*}
g_{1,\ast}(x_1,y_1) &\equiv \frac{1}{2\tau_X\tau_Y}\bigg[\big(x_1^\top\Sigma_{XY}y_1 - \pnorm{\Sigma_{XY}}{F}^2\big)\bigg],\\
g_{2,\ast}\big((x_1,y_1),(x_2,y_2) \big) & \equiv \frac{1}{6\tau_X\tau_Y}\bigg[\Big(x_1^\top x_2 y_1^\top y_2 - x_1^\top \Sigma_{XY} y_1-x_2^\top \Sigma_{XY} y_2 + \pnorm{\Sigma_{XY}}{F}^2\Big)\nonumber\\
&\qquad -\Big(x_1^\top \Sigma_{XY} y_2+x_2^\top \Sigma_{XY} y_1\Big) \bigg].
\end{align*}
Although the heuristic so far seems plausible, it turns out that the above approximation falls short of fully capturing the behavior of the sample distance covariance, even pretending that the effect of higher order kernels can be neglected. In fact, the approximation (\ref{ineq:UV_approx}) is not good enough, in a somewhat subtle way, in the entire high-dimensional regime $p\wedge q \to \infty$: The first-order kernel $g_{1,\ast}$ requires the following correction
\begin{align*}
&\tilde{g}_{1,\ast}(x_1,y_1)\equiv g_{1,\ast}(x_1,y_1)\\
&\qquad\qquad - \frac{1}{2\tau_X\tau_Y}\bigg[\frac{\pnorm{\Sigma_{XY}}{F}^2 }{2\tau_X^2}(\pnorm{x_1}{}^2-\tr(\Sigma_X))+\frac{\pnorm{\Sigma_{XY}}{F}^2 }{2\tau_Y^2}(\pnorm{y_1}{}^2-\tr(\Sigma_Y))\bigg],
\end{align*}
whereas, interestingly, no correction is required for the second-order kernel $g_{2,\ast}$. The underlying reason for the correction terms in the first-order kernel appears to be non-negligible interaction of the approximation of $U,V$ in (\ref{ineq:sqrt_norm_expansion}), while such interaction is of a strict smaller order in the second-order kernel approximation. In fact, the correction terms in $\tilde{g}_{1,\ast}(x_1,y_1)$ above contributes to the difficult terms of negative signs in the first-order variance $\bar{\sigma}_{n,1}^2(X,Y)$ in Theorem \ref{thm:non_null_clt}. As a consequence, the variance of $g_{1,\ast}$ and $\tilde{g}_{1,\ast}$ are of the same order but not asymptotically equivalent. Of course, at this point there is no apriori reason to explain why the correction terms must take this form---they come out of exact calculations. 

From here, a road-map of the proof of Theorem \ref{thm:non_null_clt} can be outlined:
\begin{enumerate}
	\item Derive sharp enough estimates for the approximation errors of $U,V$ in (\ref{ineq:UV_approx}) and their interactions. This will be detailed in Section \ref{section:res_mean}. These sharp enough estimates will immediately give a mean expansion for the sample distance covariance in Theorem \ref{thm:dcov_expansion}. 
	\item Using the estimates in (1), validate that the corrected first-order kernel $\tilde{g}_{1,\ast}$ and the vanilla second-order kernel $g_{2,\ast}$ are indeed `good enough main terms' to approximate the sample distance covariance. This is done via variance considerations detailed in Section \ref{section:hoeffding_decomp}. As a result, a sharp variance expansion of the sample distance covariance is obtained in Theorem \ref{thm:variance_expansion}.
	\item Using the mean and variance expansion established in (1)-(2), we establish a non-null CLT for the `good enough main terms' involving the kernels $\tilde{g}_{1,\ast}$ and $g_{2,\ast}$. The main tool is Chatterjee's discrete second-order Poincar\'e inequality \cite{chatterjee2008new}. This is accomplished in Section \ref{section:clt_truncated_dcov}.
\end{enumerate}
Finally Section \ref{sec:proof_clts} assemblies all these steps to complete the proof for Theorem \ref{thm:non_null_clt}. In the next Section \ref{section:preliminaries} below, we record some further notations and preliminary results that will be used throughout the proofs.

\section{Proof preliminaries}\label{section:preliminaries}

\subsection{$G_{[\cdot\cdot]}$ and $H_{[\cdot\cdot]}$}
Recall the definition of the matrices $G_{[\cdot\cdot]},H_{[\cdot\cdot]}$ in (\ref{def:GH}). We summarize some basic properties of these matrices below.
\begin{lemma}\label{lem:property_GH}
	The following hold.
	\begin{enumerate}
		\item $H_{[11]}^2 = G_{[11]}, H_{[22]}^2=G_{[22]}$, $H_{[11]}H_{[22]}=G_{[12]}$, $H_{[22]}H_{[11]}=G_{[21]}$.
		\item $\pnorm{G_{[11]}}{F}^2= \pnorm{\Sigma_X^2}{F}^2$, $\pnorm{G_{[22]}}{F}^2= \pnorm{\Sigma_Y^2}{F}^2$, $\pnorm{G_{[12]}}{F}^2=\pnorm{G_{[21]}}{F}^2= \pnorm{\Sigma_{XY}\Sigma_{YX}}{F}^2$.
		\item $\tr(G_{[11]}) = \pnorm{\Sigma_X}{F}^2$, $\tr(G_{[22]}) = \pnorm{\Sigma_Y}{F}^2$, $\tr(G_{[12]})=\tr(G_{[21]})=\pnorm{\Sigma_{XY}}{F}^2$.
		\item $\pnorm{H_{[11]}}{F}^2= \pnorm{\Sigma_X}{F}^2$, $\pnorm{H_{[22]}}{F}^2= \pnorm{\Sigma_Y}{F}^2$, $\pnorm{H_{[12]}}{F}^2=\pnorm{H_{[21]}}{F}^2= \pnorm{\Sigma_{XY}}{F}^2$.
		\item $\tr(H_{[11]}) = \tr(\Sigma_X)$, $\tr(H_{[22]}) = \tr(\Sigma_Y)$, $\tr(G_{[12]})=\tr(G_{[21]})=\tr(\Sigma_{XY})$.
		\item $\tr(G_{[11]}G_{[22]})=\tr(G_{[12]}G_{[21]})=\tr(\Sigma_X\Sigma_{XY}\Sigma_Y\Sigma_{YX})$.
	\end{enumerate}
\end{lemma}
\begin{proof}
	These claims follow from direct calculations so we omit the details.
\end{proof}

The following lemma will be useful in the second moment part of Proposition \ref{prop:res_psi_bound} ahead.
\begin{lemma}\label{lem:GH_hadamard}
Suppose that the spectrum of $\Sigma$ is contained in $[M^{-1},M]$ for some $M>1$, then the following hold.
\begin{enumerate}
\item $\tr(G_{[12]}\circ G_{[12]}) \vee \tr(H_{[11]}^2\circ H_{[22]}^2) \vee \pnorm{H_{[11]}\circ H_{[22]}}{F}^2 \lesssim_M \pnorm{\Sigma_X}{F}^2\pnorm{\Sigma_{Y}}{F}^2/(\tau_X\wedge \tau_Y)^2$.
\item $\pnorm{G_{[11]}\circ G_{[22]}}{F} \lesssim_M \pnorm{\Sigma_{XY}}{F}^2$.
\end{enumerate}
\end{lemma}
\begin{proof}[Proof of Lemma \ref{lem:GH_hadamard}]
See Appendix \ref{appendix:proof_prel}.
\end{proof}

\subsection{The function $h$}
Let for $u>-1$
\begin{align}\label{def:h}
h(u) &\equiv \sqrt{1+u}-1-\frac{u}{2} = -\frac{u^2}{4} \int_0^1 \frac{(1-s)}{(1+su)^{3/2}}\,\d{s}.
\end{align}
We summarize below some basic properties of $h$.
\begin{lemma}\label{lem:property_h_fcn}
	We have $\abs{h(u)}\lesssim u^2$ and $\abs{h'(u)}\lesssim {\abs{u}}/{(1+u)^{1/2}}$.
	Furthermore,
	\begin{align*}
	h(u)&=-\frac{u^2}{8}+ u^3 \int_0^1 \frac{3(1-s)^2}{16(1+su)^{5/2}}\,\d{s}\\
	&=-\frac{u^2}{8}+\frac{u^3}{16}-u^4 \int_0^1 \frac{5(1-s)^3}{32(1+su)^{7/2}}\,\d{s}\equiv -\frac{u^2}{8}+h_3(u).
	\end{align*}
\end{lemma}
The proof of the above lemma can be found in Appendix \ref{appendix:proof_prel}.

\subsection{$L_X,L_Y$ and $R_X,R_Y$}
Let
\begin{align}\label{def:L_X}
L_X(x_1,x_2)&\equiv \frac{\pnorm{x_1-x_2}{}^2-\tau_X^2}{\tau_X^2}\geq -1,\quad R_X(x_1,x_2)\equiv h\big(L_X(x_1,x_2)\big)\nonumber\\
L_Y(y_1,y_2)&\equiv \frac{\pnorm{y_1-y_2}{}^2-\tau_Y^2}{\tau_Y^2}\geq -1,\quad R_Y(y_1,y_2)\equiv h\big(L_Y(y_1,y_2)\big),
\end{align}
and the double centered quantities
\begin{align*}
\bar{R}_X(x_1,x_2)&\equiv R_X(x_1,x_2)-\E\big[R_X(x_1,X)\big] -\E\big[R_X(X,x_2)\big]+\E\big[R_X(X,X')\big],\\
\bar{R}_Y(y_1,y_2)&\equiv R_Y(y_1,y_2)-\E\big[R_Y(y_1,Y)\big] -\E\big[R_Y(Y,y_2)\big]+\E\big[R_Y(Y,Y')\big].
\end{align*}

Using these quantities, we may represent the square root of the Euclidean distance as follows.
\begin{lemma}\label{lem:sqrt_norm_LX}
	The following hold:
	\begin{align*}
	\frac{\pnorm{x_1-x_2}{}}{\tau_X}& \equiv 1+\frac{L_X(x_1,x_2)}{2}+R_X(x_1,x_2) = 1+\frac{L_X(x_1,x_2)}{2}+ h\big(L_X(x_1,x_2)\big),\\
	\frac{\pnorm{y_1-y_2}{}}{\tau_Y}& \equiv 1+\frac{L_Y(y_1,y_2)}{2}+R_Y(y_1,y_2) = 1+\frac{L_Y(y_1,y_2)}{2}+ h\big(L_Y(y_1,y_2)\big),
	\end{align*}
	and
	\begin{align}\label{eqn:UV}
	U(x_1,x_2)& = - \frac{1}{\tau_X}\Big(x_1^\top x_2-\tau_X^2 \bar{R}_X(x_1,x_2)\Big),\nonumber\\
	V(y_1,y_2) &= - \frac{1}{\tau_Y}\Big(y_1^\top y_2 -\tau_Y^2 \bar{R}_Y(y_1,y_2)\Big).
	\end{align}
\end{lemma}

The following moment estimate will be used repeatedly. 
\begin{lemma}\label{lem:moment_R}
	Suppose that the spectrum of $\Sigma$ lies in $[M^{-1},M]$ for some $M > 1$. Fix any positive integer $s\in\mathbb{N}$, there exists some $C = C(s,M,Z_1)> 0$ such that the following moment estimates hold.
	\begin{enumerate}
		\item For any positive integer $s \in \N$,
		\begin{align*}
		\E L_X^s (X_1,X_2)\lesssim_s \tau_X^{-2s}\pnorm{\Sigma_X}{F}^{s},\quad	\E L_Y^s (Y_1,Y_2)\lesssim_s \tau_Y^{-2s}\pnorm{\Sigma_Y}{F}^{s}.
		\end{align*}
		\item For any positive integer $s \in \N$,
		\begin{align*}
		\E R_X^s (X_1,X_2)\lesssim_s \tau_X^{-4s}\pnorm{\Sigma_X}{F}^{2s},\quad	\E R_Y^s (Y_1,Y_2)\lesssim_s \tau_Y^{-4s}\pnorm{\Sigma_Y}{F}^{2s}.
		\end{align*}
		Consequently the same estimates hold with $\E R_X^s (X_1,X_2), \E R_Y^s (Y_1,Y_2)$ replaced by their double-centered analogues $\E \bar{R}_X^s (X_1,X_2), \E \bar{R}_Y^s (Y_1,Y_2)$.
		\item Suppose the spectrum of $\Sigma_X,\Sigma_Y$ is contained in $[1/M,M]$ for some $M>1$. Then for any positive integer $s\in \N$, for $p\wedge q\geq 2s+1$, 
		\begin{align*}
		\E h'(L_X(X_1,X_2))^s\lesssim_{M,s} \tau_X^{-s},\quad \E h'(L_Y(Y_1,Y_2))^s\lesssim_{M,s} \tau_Y^{-s}.
		\end{align*}
	\end{enumerate} 
\end{lemma}

The proofs of the above lemmas can be found in Appendix \ref{appendix:proof_prel}.

\section{Residual estimates and mean expansion}\label{section:res_mean}

\subsection{Residual estimates}
Let 
\begin{align}\label{def:psi}
\psi_X(x_1,y_1)&\equiv \E_{X_2,Y_2}[\bar{R}_X(x_1,X_2)Y_2^\top y_1],\nonumber\\
\psi_Y(x_1,y_1)&\equiv \E_{X_2,Y_2}[\bar{R}_Y(y_1,Y_2)X_2^\top x_1],\nonumber\\
\psi_{X,Y}(x_1,y_1)& \equiv \E_{X_2,Y_2}[\bar{R}_X(x_1,X_2)\bar{R}_Y(y_1,Y_2)].
\end{align}
In view of Lemma \ref{lem:sqrt_norm_LX}, these terms appear naturally as the interaction error terms when $U(x_1,x_2)V(y_1,y_2)$ is approximated using (\ref{ineq:UV_approx}). As mentioned in Section \ref{section:proof_outline}, sharply controlling these `residual terms' constitutes the first crucial step in the proof of Theorem \ref{thm:non_null_clt}.

First, we have the following representation of $\psi_X(x_1,y_1),\psi_Y(x_1,y_1)$.
\begin{lemma}\label{lem:psi_decomp}
	The following decomposition holds:
	\begin{align*}
	\psi_X(x_1,y_1)&=A_{1,X}(x_1,y_1)+A_{2,X}(x_1,y_1),\, 	\psi_Y(x_1,y_1)=A_{1,Y}(x_1,y_1)+A_{2,Y}(x_1,y_1).
	\end{align*}
	Here
	\begin{align*}
	A_{1,X}(x_1,y_1)&=\frac{1}{2\tau_X^4}\Big[\big(\pnorm{x_1}{}^2-\tr(\Sigma_X)\big)x_1^\top \Sigma_{XY} y_1 +2 x_1^\top \Sigma_X \Sigma_{XY} y_1 + \kappa\tr(H_{[11]}\circ Q_X) \Big],\\
	A_{1,Y}(x_1,y_1)&=\frac{1}{2\tau_Y^4}\Big[\big(\pnorm{y_1}{}^2-\tr(\Sigma_Y)\big)x_1^\top \Sigma_{XY} y_1 +2 x_1^\top \Sigma_{XY} \Sigma_{Y} y_1 + \kappa\tr(H_{[22]}\circ Q_Y)\Big],
	\end{align*}
	and
	
	\begin{align*}
	A_{2,X}(x_1,y_1)&\equiv \E\big[h_3\big(L_X(x_1,X)\big)(Y^\top y_1)\big],\,	 A_{2,Y}(x_1,y_1)\equiv \E\big[h_3\big(L_Y(y_1,Y)\big)(X^\top x_1)\big],
	\end{align*}
	with $h_3$ defined in Lemma \ref{lem:property_h_fcn} and
	\begin{align}\label{def:PQ}
	Q_X = H_{[11]}z_1z_1^\top H_{[22]}, \quad Q_Y = H_{[22]}z_1z_1^\top H_{[11]},
	\end{align}
	with $z_1 = (x_1^\top, y_1^\top)^\top$ and $H_{[\cdot\cdot]}$ given in (\ref{def:GH}).
\end{lemma}

\begin{proposition}\label{prop:res_psi_bound}
	Suppose that the spectrum of $\Sigma$ is contained in $[1/M,M]$ for some $M>1$, and that $p, q$ are larger than a big enough absolute constant. 
	\begin{enumerate}
		\item (First moments) The following hold:
		\begin{align*}
		\tau_X^4 \abs{\E \psi_X(X_1,Y_1)}\bigvee  \tau_Y^4 \abs{\E \psi_Y(X_1,Y_1)} \bigvee \tau_X^2\tau_Y^2(\tau_X\wedge \tau_Y)\abs{\E\psi_{X,Y}(X_1,Y_1)}\lesssim \pnorm{\Sigma_{XY}}{F}^2.
		\end{align*}
		\item (Second moments) The following hold:
		\begin{align*}
		&\tau_X^6\E \psi_X^2(X_1,Y_1)\bigvee \tau_Y^6 \E \psi_Y^2(X_1,Y_1)\\
		&\qquad \bigvee \tau_X^4\tau_Y^4 (\tau_X\wedge \tau_Y)^2 \E \psi_{X,Y}^2(X_1,Y_1) \lesssim   \pnorm{\Sigma_{XY}}{F}^2\big(1\vee \pnorm{\Sigma_{XY}}{F}^2\big).
		\end{align*}
	\end{enumerate}
	The constants in $\lesssim$ only depend on $M$ and the distribution of $Z_1$ via its Poincar\'e constant $c_\ast$ and $\epsilon_0$ prescribed in Assumption \ref{assump:sepa}. The claims remain valid with $X=Y$ when the spectrum of $\Sigma_X=\Sigma_Y$ is contained in $[1/M,M]$ for some $M>1$.
\end{proposition}

The role and sharpness of these bounds will be gradually clear in later sections. In particular, these bounds will be essential in the proof of the mean expansion Theorem \ref{thm:dcov_expansion} and the variance expansion Theorem \ref{thm:variance_expansion} ahead.

Note that here the first moment bounds in Proposition \ref{prop:res_psi_bound} do not follow directly by the stated second moment bounds, as the `first moments' here are obtained by first taking expectation followed by the absolute value. In fact, these first moment estimates are stronger by those derived directly from the second moment estimates, indicating the essential role of the order of taking expectation and absolute value in this setting.

An important feature of the bounds in Proposition \ref{prop:res_psi_bound} above is that when $\Sigma_{XY}=0$, all estimates reduce to $0$. Furthermore the exponent in $\pnorm{\Sigma_{XY}}{F},\tau_X,\tau_Y$ also need be correct to allow precise mean and variance expansions in Theorems \ref{thm:dcov_expansion} and \ref{thm:variance_expansion}, and therefore the non-null CLT in Theorem \ref{thm:non_null_clt}, under the entire high dimensional regime $n\wedge p\wedge q\to \infty$. It is for this reason that the proof of Proposition \ref{prop:res_psi_bound} is rather involved, the details of which can be found in Appendix \ref{appendix:proof_res}. The following lemma is representative in terms of an interpolation technique in proving Proposition \ref{prop:res_psi_bound} and may be of broader interest.

\begin{lemma}\label{lem:cross_XY_var}
	Suppose the  spectrum of $\Sigma$ is contained in $[1/M,M]$ for some $M>1$. Let $\mathfrak{h}_X, \mathfrak{h}_Y: \R\to \R$ be smooth functions. For any $k,k',\ell,\ell' \in \{1,2\}$, define
	\begin{align*}
	\psi_{\mathfrak{h}_X,\mathfrak{h}_Y}(\Sigma_{XY})\equiv \E\Big[\mathfrak{h}_X\big(L_X(X_k,X_\ell)\big)\mathfrak{h}_Y\big(L_Y(Y_{k'},Y_{\ell'})\big)\big\lvert X_1,Y_1\Big].
	\end{align*}
	Then
	\begin{align*}
	&\E \big(\psi_{\mathfrak{h}_X,\mathfrak{h}_Y}(\Sigma_{XY})-\psi_{\mathfrak{h}_X,\mathfrak{h}_Y}(0)\big)^2\lesssim  \pnorm{\Sigma_{XY}}{F}^2(1\vee \pnorm{\Sigma_{XY}}{F}^2)\\
	&\quad \times  \bigg( \tau_X^{-4} \cdot \E^{1/4} (\mathfrak{h}_X'\circ L_X)^8\cdot \E^{1/4} (\mathfrak{h}_Y\circ L_Y)^8 + \tau_Y^{-4} \cdot \E^{1/4} (\mathfrak{h}_Y'\circ L_Y)^8\cdot \E^{1/4} (\mathfrak{h}_X\circ L_X)^8\bigg).
	\end{align*}
	The constants in $\lesssim$ only depend on $M$. The claims remain valid with $X=Y$ when the spectrum of $\Sigma_X=\Sigma_Y$ is contained in $[1/M,M]$ for some $M>1$.
\end{lemma}

Lemma \ref{lem:cross_XY_var} gives a general recipe of bounding the second moment, in terms of the dependence measure $\pnorm{\Sigma_{XY}}{F}$. Details see Appendix \ref{appendix:proof_res}.

\subsection{Mean expansion}

As a quick application of the residual estimates in Proposition \ref{prop:res_psi_bound}, we may get the following mean expansion.

\begin{theorem}\label{thm:dcov_expansion}
	Suppose that the spectrum of $\Sigma$ is contained in $[1/M,M]$ for some $M>1$. Then the following expansion holds for the distance covariance:
	\begin{align*}
	m_\Sigma =\E_\Sigma \dcov_\ast^2(\bm{X},\bm{Y})=\dcov^2(X,Y)= \frac{\pnorm{\Sigma_{XY}}{F}^2}{\tau_X\tau_Y}\Big[1+\mathcal{O}\big((\tau_X\wedge \tau_Y)^{-1}\big)\Big].
	\end{align*}
	The constants in $\mathcal{O}$ only depend on $M$ and the distribution of $Z_1$ via its Poincar\'e constant $c_\ast$ and $\epsilon_0$ prescribed in Assumption \ref{assump:sepa}. The claims remain valid with $X=Y$ when the spectrum of $\Sigma_X=\Sigma_Y$ is contained in $[1/M,M]$ for some $M>1$.
\end{theorem}

\begin{proof}
	Note that
	\begin{align*}
	&\dcov^2(X,Y)= \E\big[U(X_1,X_2)V(Y_1,Y_2)\big]\\
	& = \frac{1}{\tau_X\tau_Y}\Big[\pnorm{\Sigma_{XY}}{F}^2
	-\tau_X^2 \E\big(\bar{R}_X(X_1,X_2)Y_1^\top Y_2\big)\\
	&\qquad-
	\tau_X^2\E\big(\bar{R}_Y(Y_1,Y_2)X_1^\top X_2\big)+\tau_X^2\tau_Y^2 \E\big( \bar{R}_X(X_1,X_2)\bar{R}_Y(Y_1,Y_2)\big)\Big]\\
	& \equiv \frac{1}{\tau_X\tau_Y}\bigg[\pnorm{\Sigma_{XY}}{F}^2-\tau_X^2 \E \psi_X(X_1,X_2)-\tau_Y^2 \E \psi_Y(Y_1,Y_2)+\tau_X^2\tau_Y^2 \E \psi_{X,Y}(X_1,Y_1)\bigg].
	\end{align*}
	The claim now follows by invoking Proposition \ref{prop:res_psi_bound}-(1).
\end{proof}

A stochastic version of the above theorem was previously derived in \cite[Theorem 2.1.1]{zhu2020distance}, where the main term $\pnorm{\Sigma_{XY}}{F}^2/(\tau_X\tau_Y)$ was replaced by an unbiased estimator and the remainder term was controlled at the order $(p\wedge q)^{-1/2}$. In comparison, thanks to the sharp residual estimates in Proposition \ref{prop:res_psi_bound}, our bound for the remainder term is much more refined in that it contains an important multiplicative factor $\pnorm{\Sigma_{XY}}{F}^2$, which makes it asymptotically negligible in the null case as well.

\section{Hoeffding decomposition and variance expansion}\label{section:hoeffding_decomp}

We first review the basics of Hoeffding decomposition that will be relevant to our purpose. Following \cite[Section 5.1.5, pp. 177]{serfling1980approximation}, for a generic 4-th order $U$-statistic with symmetric kernel $k:\mathcal{Z}^4\to \R$, let
\begin{align*}
k_c(z_1,\ldots,z_c)\equiv \E[k(z_1,\ldots,z_c,Z_{c+1},\ldots,Z_4)],\quad z_1,\ldots,z_c \in \mathcal{Z},
\end{align*}
and for any $z_1,z_2,z_3,z_4 \in \mathcal{Z}$,
\begin{align}\label{def:hoeffding_kernel}
\notag g_0&\equiv \E k(\bm{Z}),\quad g_1(z_1)\equiv k_1(z_1)-\E k(\bm{Z}),\\
\notag g_2(z_1,z_2)&\equiv k_2(z_1,z_2)-k_1(z_1)-k_2(z_2)+\E k(\bm{Z}),\\
\notag g_3(z_1,z_2,z_3)&\equiv  k_3(z_1,z_2,z_3)-\E k(\bm{Z}) -\sum_{\ell=1}^3 g_1(z_\ell)-\sum_{1\leq \ell_1<\ell_2\leq 3} g_2\big(z_{\ell_1},z_{\ell_2}\big),\\
\notag g_4(z_1,z_2,z_3,z_4) &\equiv k_4\big(z_1,z_2,z_3,z_4\big)-\E k(\bm{Z}) -\sum_{\ell=1}^3 g_1(z_\ell)\\
&\qquad-\sum_{1\leq \ell_1<\ell_2\leq 3} g_2\big(z_{\ell_1}, z_{\ell_2}\big) - \sum_{1\leq \ell_1<\ell_2<\ell_3\leq 4} g_3\big(z_{\ell_1},z_{\ell_2},z_{\ell_3}\big).
\end{align}
Then the Hoeffding decomposition says that
\begin{align*}
U_n(k) = \sum_{c=0}^4 \binom{4}{c} U_n(g_c). 
\end{align*}
Here for a generic symmetric kernel $g: \mathcal{Z}^c\to \R$, 
\begin{align*}
U_n(g)\equiv 
\begin{cases}
\binom{n}{c}^{-1}\sum_{i_1<\ldots<i_c} g\big(z_{i_1},\ldots,z_{i_c}\big), & c\geq 1;\\
g & c=0.
\end{cases}
\end{align*}
For $c=0$, $g$ is understood as a real number. In what follows, we will take $\mathcal{Z}\equiv \mathcal{X}\times \mathcal{Y}$, and $k$ as the kernel defined in Proposition \ref{prop:hoef_decomp}. We will evaluate the variance of $\dcov_\ast^2(\bm{X},\bm{Y})=U_n(k)$ by evaluating the variance of $g_1,g_2,g_3,g_4$ associated with $k$ as defined above. 

\subsection{Hoeffding decomposition: 1st order}

The goal of this subsection is to prove the following variance expansion for the first-order kernel associated with $k$.
\begin{proposition}\label{prop:var_first_order}
	Suppose the spectrum of $\Sigma$ is contained in $[1/M,M]$ for some $M>1$. Then for any $\epsilon > 0$, the first-order variance is given by
	\begin{align*}
	&\binom{4}{1}^2\binom{n}{1}^{-1}\E g_1^2(X_1,Y_1) = (1\pm \epsilon)\cdot\bar{\sigma}_{n,1}^2(X,Y)\cdot\bigg[1+\mathcal{O}\Big(\frac{1}{\epsilon\cdot (\tau_X\wedge \tau_Y)}\Big)\bigg].
	\end{align*}
	Here $\bar{\sigma}_{n,1}^2(X,Y)$ is as defined in Theorem \ref{thm:non_null_clt}, and the constants in $\mathcal{O}$ only depend on $M$ and the distribution of $Z_1$ only  via its Poincar\'e constant $c_\ast$, excess kurtosis $\kappa$, and $\epsilon_0$ prescribed in Assumption \ref{assump:sepa}.
	The claim remains valid with $X=Y$ when the spectrum of $\Sigma_X=\Sigma_Y$ is contained in $[1/M,M]$ for some $M>1$.
\end{proposition}

The proof of the above proposition will be presented towards the end of this subsection. First, we may compute:
\begin{lemma}\label{lem:h1}
	The first order kernel is given by
	\begin{align*}
	k_1(z_1) &= \E k(z_1,Z_2,Z_3,Z_4) =  \frac{1}{2}\bigg[\E U(x_1,X)V(y_1,Y)+\dcov^2(X,Y)\bigg],\\
	g_1(z_1) &= k_1(z_1) - \E k(Z) = \frac{1}{2}\bigg[\E U(x_1,X)V(y_1,Y)-\dcov^2(X,Y)\bigg].
	\end{align*}
\end{lemma}
We will use the above lemma to devise an expansion for $g_1$. From the approximation of $U,V$ in (\ref{ineq:UV_approx}), one may hope that the main term for $g_1$ would be $2^{-1}\E U(x_1,X) V(y_1,Y) \approx (x_1^\top\Sigma_{XY}y_1 - \pnorm{\Sigma_{XY}}{F}^2)/2\tau_X\tau_Y$. As announced in Section \ref{section:proof_outline}, this is however not  the case.  Let the `main term' be defined as 
	\begin{align}\label{def:g1_bar}
	\bar{g}_1(x_1,y_1) \equiv \frac{1}{2\tau_X\tau_Y}\bigg[\big(x_1^\top\Sigma_{XY}y_1 - \pnorm{\Sigma_{XY}}{F}^2\big) + \mathscr{A}_{1,X}(x_1,y_1) +  \mathscr{A}_{1,Y}(x_1,y_1)\bigg],
	\end{align}
where
	\begin{align*}
	\mathscr{A}_{1,X}(x_1,y_1) &\equiv -\frac{\pnorm{\Sigma_{XY}}{F}^2}{2\tau_X^2}\big(\pnorm{x_1}{}^2 - \tr(\Sigma_X)\big),\quad\,\mathscr{A}_{1,Y}(x_1,y_1) \equiv -\frac{\pnorm{\Sigma_{XY}}{F}^2}{2\tau_Y^2}\big(\pnorm{y_1}{}^2 - \tr(\Sigma_Y)\big).
	\end{align*}

The terms $\mathscr{A}_{1,X}(x_1,y_1),\mathscr{A}_{1,Y}(x_1,y_1)$ are essential to correct the naive approximation (\ref{ineq:UV_approx}), in that these terms contribute to the somewhat difficult terms of negative sign in the variance expansion of $\bar{g}_1$ in Lemma \ref{lem:bar_g1_second_moment}, which cannot be neglected as they may have the same order as that of the leading terms.

With the main term defined above, let the `residual term' be defined by
\begin{align}\label{def:R1_bar}
\bar{R}_1(x_1,y_1) \equiv -\tau_X^2 \bar{\psi}_X (x_1,y_1)-\tau_Y^2\bar{\psi}_Y(x_1,y_1)+\tau_X^2\tau_Y^2 \psi_{X,Y}(x_1,y_1),
\end{align}
where $\psi_{X,Y}$ is defined in (\ref{def:psi}), and
\begin{align}\label{def:bar_psi}
\bar{\psi}_X(x_1,y_1) &\equiv \frac{1}{2\tau_X^4}\bigg[\big(\pnorm{x_1}{}^2-\tr(\Sigma_X)\big)\big(x_1^\top \Sigma_{XY} y_1-\pnorm{\Sigma_{XY}}{F}^2\big)+2x_1^\top \Sigma_X\Sigma_{XY} y_1\nonumber\\
&\qquad\qquad + \kappa \tr(H_{[11]}\circ Q_X)\bigg]+ A_{2,X}(x_1,y_1),\nonumber\\
\bar{\psi}_Y(x_1,y_1) &\equiv \frac{1}{2\tau_Y^4}\bigg[\big(\pnorm{y_1}{}^2-\tr(\Sigma_Y)\big)\big(x_1^\top \Sigma_{XY} y_1-\pnorm{\Sigma_{XY}}{F}^2\big)+2x_1^\top \Sigma_{XY}\Sigma_{Y} y_1\nonumber\\
&\qquad\qquad + \kappa\tr(H_{[22]}\circ Q_Y)\bigg]A_{2,Y}(x_1,y_1).
\end{align}
Here $Q_{X}, Q_Y, A_{2,X},A_{2,Y}$ are defined in Lemma \ref{lem:psi_decomp}. Using $\bar{g}_1,\bar{R}_1$ defined above, we may expand $g_1$ into the sum of main and residual terms as follows. 
\begin{lemma}\label{lem:g1}
	The following expansion holds:
	\begin{align}\label{eqn:g1}
	g_1(x_1,y_1) = \bar{g}_1(x_1,y_1)+ \frac{1}{2\tau_X\tau_Y}\big(\bar{R}_1(x_1,y_1)-\E \bar{R}_1(X_1,Y_1)\big).
	\end{align}
\end{lemma}

Now we will evaluate the variance of $\bar{g}_1$ and $\bar{R}_1$. The variance of $\bar{g}_1$ is given by the following.

\begin{lemma}\label{lem:bar_g1_second_moment}
We have $\E\bar{g}_1^2(X,Y) = 4^{-2} n\cdot \bar{\sigma}_{n,1}^2$, where $\bar{\sigma}_{n,1}^2$ is given in Theorem \ref{thm:non_null_clt}.
\end{lemma}

As mentioned above, the variance of $\bar{g}_1$ as above involves terms with a negative sign that are contributed by the `correction terms' $\mathscr{A}_{1,X}(x_1,y_1),\mathscr{A}_{1,Y}(x_1,y_1)$. These terms can be of the same order as the main terms. It is therefore important to have a lower bound on this quantity.

\begin{lemma}\label{lem:first_variance_lower}
	\begin{enumerate}
		\item Suppose $\pnorm{\Sigma^{-1}}{\op}\leq M$ for some $M>1$. Then $\E \bar{g}_1^2(X,Y) \gtrsim \tau_X^{-2}\tau_Y^{-2} \pnorm{\Sigma_{XY}}{F}^2$. If furthermore $\pnorm{\Sigma}{\op}\leq M$, then $\E \bar{g}_1^2(X,Y) \asymp \tau_X^{-2}\tau_Y^{-2} \pnorm{\Sigma_{XY}}{F}^2$.
		\item Suppose $X=Y$, and $\pnorm{\Sigma_X^{-1}}{\op}\leq M$ for some $M>1$. Then $\E \bar{g}_1^2(X,X) \gtrsim \tau_X^{-8}\cdot p \pnorm{\Sigma_X}{F}^4$. If furthermore $\pnorm{\Sigma_X}{\op}\leq M$, then $\E \bar{g}_1^2(X,X) \asymp \tau_X^{-2}$.
	\end{enumerate}
	The constants in $\gtrsim,\asymp$ only depend on $M$ and the distribution of $Z_1$ via its excess kurtosis $\kappa$ in Assumption \ref{assump:sepa}.
\end{lemma}

Lemma \ref{lem:first_variance_lower} above is an important result, showing that the negative contributions of the `correction terms' $\mathscr{A}_{1,X}(x_1,y_1),\mathscr{A}_{1,Y}(x_1,y_1)$ will not affect the order the variance $\bar{g}_1$. In other words, these terms will contribute a non-vanishing but small proportion of the main terms.

Next to the variance of the main term $\bar{g}_1$, an important step to obtain variance bound for the residual term $\bar{R}_1$ is to obtain variance bounds for $\bar{\psi}_X,\bar{\psi}_Y$ defined in (\ref{def:bar_psi}). 
\begin{lemma}\label{lem:var_bar_psi}
	Suppose that the spectrum of $\Sigma_X,\Sigma_Y$ is contained in $[1/M,M]$ for some $M>1$. Then
	\begin{align*}
	\tau_X^6 \var\big( \bar{\psi}_X(X,Y)\big)\bigvee \tau_Y^6 \var\big( \bar{\psi}_Y(X,Y)\big)&\lesssim_{M,Z_1}  \pnorm{\Sigma_{XY}}{F}^2.
	\end{align*}
	Here the dependence of $\lesssim$ on $Z_1$ is via its Poincar\'e constant $c_\ast$ and $\epsilon_0$ prescribed by Assumption \ref{assump:sepa}.
\end{lemma}

This variance bound plays an important role to keep the residual terms small when $\pnorm{\Sigma_{XY}}{F}$ is large. In particular, if one uses the vanilla versions $\psi_{X},\psi_Y$ defined in (\ref{def:psi}), the right hand side of the above display scales as $\pnorm{\Sigma_{XY}}{F}^4$ that would lead to essential difficulties in controlling the residuals. In other words, the reduction from $\pnorm{\Sigma_{XY}}{F}^4$ to $\pnorm{\Sigma_{XY}}{F}^2$ is made possible by the `correction terms'  $\mathscr{A}_{1,X}(x_1,y_1),\mathscr{A}_{1,Y}(x_1,y_1)$ that, in a certain sense, `center' the vanilla versions $\psi_{X},\psi_Y$ to reduce the variance.

Detailed proofs of Lemmas \ref{lem:h1}-\ref{lem:var_bar_psi} are deferred to Appendix \ref{appendix:proof_hoeff}. Now we are in a good position to prove Proposition \ref{prop:var_first_order}.

\begin{proof}[Proof of Proposition \ref{prop:var_first_order}]
	By (\ref{def:R1_bar}),
	\begin{align*}
	&\var\big(\bar{R}_1(X_1,Y_1)\big)\\
	&\lesssim \tau_X^4 \var \big(\bar{\psi}_X(X_1,Y_1)\big)+\tau_Y^4 \var \big(\bar{\psi}_Y(X_1,Y_1)\big)+\tau_X^4\tau_Y^4 \E\psi_{X,Y}^2(X_1,Y_1).
	\end{align*}
	The first two terms can be handled by Lemma \ref{lem:var_bar_psi}, while the last term can be bounded by
	\begin{align*}
	\E\psi_{X,Y}^2(X_1,Y_1)\lesssim_M \tau_X^{-4}\tau_Y^{-4}\bigg[\frac{\pnorm{\Sigma_{XY}}{F}^2+\pnorm{\Sigma_{XY}}{F}^4}{(\tau_X\wedge \tau_Y)^2}\bigwedge 1\bigg].
	\end{align*}
	This follows by Proposition \ref{prop:res_psi_bound}-(2) and the simple bound
	\begin{align*}
	\E\psi_{X,Y}^2(X_1,Y_1)\leq \E \bar{R}_X^2\cdot \E \bar{R}_Y^2\lesssim_M \tau_X^{-4}\tau_Y^{-4}
	\end{align*}
	by using Lemma \ref{lem:moment_R}. Summarizing the estimates, we have
	\begin{align*}
	\var\big(\bar{R}_1(X_1,Y_1)\big) \lesssim \frac{\pnorm{\Sigma_{XY}}{F}^2}{(\tau_X\wedge \tau_Y)^2}+ \frac{\pnorm{\Sigma_{XY}}{F}^4}{(\tau_X\wedge \tau_Y)^2}\bigwedge 1.
	\end{align*}
	As $\var(g_1) = (1\pm \epsilon)\var(\bar{g}_1) + \mathcal{O}\big(\epsilon^{-1}\cdot\tau_X^{-2}\tau_Y^{-2}\var(\bar{R}_1({X,Y})\big)$ for any $\epsilon> 0$, the proof is now complete by noting that
	\begin{align*}
	\frac{\var(\bar{R}_1({X,Y}))}{\tau_X^2\tau_Y^2\var(\bar{g}_1)} \lesssim_M \frac{\frac{\pnorm{\Sigma_{XY}}{F}^2}{(\tau_X\wedge \tau_Y)^2}+ \frac{\pnorm{\Sigma_{XY}}{F}^4}{(\tau_X\wedge \tau_Y)^2}\bigwedge 1}{\pnorm{\Sigma_{XY}}{F}^2}\lesssim \frac{1}{\tau_X\wedge \tau_Y},
	\end{align*}
	using Lemma \ref{lem:first_variance_lower} in the first inequality.
\end{proof}

\subsection{Hoeffding decomposition: 2nd order} 

The goal of this subsection is to prove the following variance expansion for the second-order kernel associated with $k$.

\begin{proposition}\label{prop:var_second_order}
	Suppose that the spectrum of $\Sigma$ is contained in $[1/M,M]$ for some $M>1$. For any $\epsilon > 0$, the second-order variance is given by
	\begin{align*}
	&\binom{4}{2}^2\binom{n}{2}^{-1} \E g_2^2\big((X_1,Y_1),(X_2,Y_2) \big)= (1\pm \epsilon)\cdot\bar{\sigma}_{n,2}^2(X,Y)\cdot\Big[1+\mathcal{O}\Big(\frac{1}{\epsilon\cdot (\tau_X\wedge \tau_Y)^2}\Big)\Big].
	\end{align*}
	Here $\bar{\sigma}_{n,2}^2(X,Y)$ is as defined in Theorem \ref{thm:non_null_clt}, the constant in $\mathcal{O}$ depends on $M$ and the distribution of $Z_1$ only via its Poincar\'e constant $c_\ast$, excess kurtosis $\kappa$, and $\epsilon_0$ prescribed by Assumption \ref{assump:sepa}. The claim remains valid with $X=Y$ when the spectrum of $\Sigma_X=\Sigma_Y$ is contained in $[1/M,M]$ for some $M>1$.
\end{proposition}
The prove Proposition \ref{prop:var_second_order}, we will first get an expansion for $g_2$, which requires a calculation of $k_2$:

\begin{lemma}\label{lem:h2}
	The second order kernel is given by
	\begin{align*}
	&k_2(z_1,z_2)= \E k(z_1,z_2,Z_3,Z_4)\\
	& = \frac{1}{6}\bigg[U(x_1,x_2) V(y_1,y_2)+2\E U(x_1,X)V(y_1,Y)+2\E U(x_2,X)V(y_2,Y)+ \dcov^2(X,Y)\\
	&\qquad\qquad -\E U(x_1,X) V(y_2,Y)-\E U(x_2,X)V(y_1,Y)\bigg].
	\end{align*}
\end{lemma}
We will use the above lemma to devise an expansion for $g_2$. In the first-order expansion in the previous subsection, we have seen that the approximation of $U,V$ in (\ref{ineq:UV_approx}) is \emph{not} enough to get a precise variance expansion of $g_1$. Somewhat interestingly, as announced in Section \ref{section:proof_outline} such approximation is good enough in the second-order expansion. Formally, let the `main term' of $g_2$ be defined by
\begin{align}\label{def:g2_bar}
\bar{g}_2\big((x_1,y_1),(x_2,y_2) \big) & \equiv \frac{1}{6\tau_X\tau_Y}\bigg[\Big(x_1^\top x_2 y_1^\top y_2 - x_1^\top \Sigma_{XY} y_1-x_2^\top \Sigma_{XY} y_2 + \pnorm{\Sigma_{XY}}{F}^2\Big)\nonumber\\
&\qquad -\Big(x_1^\top \Sigma_{XY} y_2+x_2^\top \Sigma_{XY} y_1\Big) \bigg],
\end{align}
and the `residual term' be defined by [recall the definitions of $\bar{R}_X,\bar{R}_Y$ after (\ref{def:L_X})]
\begin{align*}
&\bar{R}_2\big((x_1,y_1),(x_2,y_2)\big) \\
& = -\tau_Y^2 x_1^\top x_2 \bar{R}_Y(y_1,y_2)-\tau_X^2 y_1^\top y_2 \bar{R}_X(x_1,x_2) + \tau_X^2\tau_Y^2 \bar{R}_X(x_1,x_2)\bar{R}_Y(y_1,y_2)\\
&\qquad -R_1(x_1,y_1)-R_1(x_2,y_2) - R_1(x_1,y_2)-R_1(x_2,y_1),
\end{align*}
with $R_1$ defined by [recall the definitions of $\psi_X,\psi_Y,\psi_{X,Y}$ in (\ref{def:psi})]
\begin{align}\label{def:R1}
R_1(x_1,y_1) \equiv -\tau_X^2 \psi_X (x_1,y_1)-\tau_Y^2\psi_Y(x_1,y_1)+\tau_X^2\tau_Y^2 \psi_{X,Y}(x_1,y_1).
\end{align}

The following lemma gives an expansion of $g_2$ into the sum of the main term $\bar{g}_2$ and the centered residual term $\bar{R}_2$.

\begin{lemma}\label{lem:g2}
	The following expansion holds:
	\begin{align*}
	&g_2\big((x_1,y_1),(x_2,y_2) \big)\\
	& = \bar{g}_2\big((x_1,y_1),(x_2,y_2) \big)+\frac{1}{6\tau_X\tau_Y}\Big[\bar{R}_2\big((x_1,y_1),(x_2,y_2)\big) -\E \bar{R}_2\big((X_1,Y_1),(X_2,Y_2)\big)\Big].
	\end{align*}
\end{lemma}

Proofs of the proceeding lemmas can be found in Appendix \ref{appendix:proof_hoeff}. Using the above decomposition, we only need to compute the variance for the two terms $\bar{g}_2$, $\bar{R}_2$ on the right hand side of the above display for the proof of Proposition \ref{prop:var_second_order}. Clearly the variance of $\bar{g}_2$ can be evaluated by a book-keeping calculation, and the variance of $\bar{R}_2$ can be handled by the residual estimates in Proposition \ref{prop:res_psi_bound}. The proof below illustrates the strength of the bounds obtained in Proposition \ref{prop:res_psi_bound}.

\begin{proof}[Proof of Proposition \ref{prop:var_second_order}]
	First note that we may expand $36\tau_X^2\tau_Y^2\E \bar{g}_2^2\big((X_1,Y_1),(X_2,Y_2) \big)$ as
	\begin{align*}
	&\E\Big[X_1^\top X_2 Y_1^\top Y_2 - X_1^\top \Sigma_{XY} Y_1-X_2^\top \Sigma_{XY} Y_2 + \pnorm{\Sigma_{XY}}{F}^2\Big]^2 + \E\Big[X_1^\top\Sigma_{XY} Y_2+X_2^\top \Sigma_{XY} Y_1\Big]^2 \equiv E_1+E_2.
	\end{align*}
	This follows as the cross term has expectation $0$ due to the symmetry of the distribution of $(X,Y)$. After expansion, we have.
	\begin{align*}
	E_1 &= \E(X_1^\top X_2 Y_1^\top Y_2)^2+2\E(X_1^\top \Sigma_{XY} Y_1)^2+ \pnorm{\Sigma_{XY}}{F}^4 - 4 \E(X_1^\top X_2 Y_1^\top Y_2X_1^\top \Sigma_{XY} Y_1)\\
	&\qquad+2 \pnorm{\Sigma_{XY}}{F}^4+ 2 \E(X_1^\top \Sigma_{XY} Y_1X_2^\top \Sigma_{XY} Y_2) - 4 \pnorm{\Sigma_{XY}}{F}^4.
	\end{align*}
	The above terms can be calculated using Lemma \ref{lem:gaussian_4moment}:
	\begin{itemize}
	\item The first term is
	\begin{align*}
	\E(X_1^\top X_2 Y_1^\top Y_2)^2 &= 2\Big(\pnorm{\Sigma_{XY}}{F}^4 + \tr(\Sigma_{XY}\Sigma_{YX}\Sigma_{XY}\Sigma_{YX}) + 2\tr(\Sigma_{XY}\Sigma_Y\Sigma_{YX}\Sigma_X)\Big)\\
	&+ \pnorm{\Sigma_X}{F}^2\pnorm{\Sigma_Y}{F}^2 + 2\kappa\cdot\big[2\tr(G_{[12]}\circ G_{[12]}) + \tr(H_{[11]}^2\circ H_{[22]}^2)\big] + \kappa^2\cdot\pnorm{H_{[11]}\circ H_{[22]}}{F}^2.
	\end{align*}
	\item The second term is
	\begin{align*}
	2\E(X_1^\top \Sigma_{XY} Y_1)^2 = 2\Big(\pnorm{\Sigma_{XY}}{F}^4 + \pnorm{\Sigma_{XY}\Sigma_{YX}}{F}^2 + \tr(\Sigma_{XY}\Sigma_Y\Sigma_{YX}\Sigma_X) + \kappa\cdot\tr(G_{[12]}\circ G_{[12]})\Big).
	\end{align*}
	\item The fourth term is
	\begin{align*}
	 &- 4 \E(X_1^\top X_2 Y_1^\top Y_2X_1^\top \Sigma_{XY} Y_1) = -4\E(X_1^\top \Sigma_{XY} Y_1)^2\\
	 &=-4\Big(\pnorm{\Sigma_{XY}}{F}^4 + \pnorm{\Sigma_{XY}\Sigma_{YX}}{F}^2 + \tr(\Sigma_{XY}\Sigma_Y\Sigma_{YX}\Sigma_X) + \kappa\cdot\tr(G_{[12]}\circ G_{[12]})\Big).
	\end{align*}
	\item The sixth term is $ 2 \E(X_1^\top \Sigma_{XY} Y_1X_2^\top \Sigma_{XY} Y_2) = 2\pnorm{\Sigma_{XY}}{F}^4$.
	\end{itemize}
	In summary, we have
	\begin{align*}
	E_1 &= \pnorm{\Sigma_X}{F}^2\pnorm{\Sigma_Y}{F}^2 + \pnorm{\Sigma_{XY}}{F}^4 + 2\tr(\Sigma_{XY}\Sigma_Y\Sigma_{YX}\Sigma_X)\\
	&+2\kappa\cdot\big(\tr(G_{[12]}\circ G_{[12]}) + \tr(H_{[11]}^2\circ H_{[22]}^2)\big) + \kappa^2\cdot\pnorm{H_{[11]}\circ H_{[22]}}{F}^2.
	\end{align*}
	Similarly, we have
	\begin{align*}
	E_2 &=2\E(X_1^\top \Sigma_{XY} Y_2)^2 + 2\E(X_1^\top \Sigma_{XY} Y_2X_2^\top \Sigma_{XY} Y_1) = 2\tr(\Sigma_{XY}\Sigma_Y \Sigma_{YX}\Sigma_{X}) +2\pnorm{\Sigma_{XY}\Sigma_{YX}}{F}^2.
	\end{align*}
	Combining the identities and applying Lemmas \ref{lem:GH_hadamard} and \ref{lem:SigmaXY_F_bound} yields that
	\begin{align*}
	&\E \bar{g}_2^2\big((X_1,Y_1),(X_2,Y_2) \big)\\
	&= \frac{1}{36\tau_X^2\tau_Y^2}\Big[\Big( \pnorm{\Sigma_X}{F}^2\pnorm{\Sigma_Y}{F}^2+ 4\tr(\Sigma_{XY}\Sigma_Y \Sigma_{YX}\Sigma_X)+ \pnorm{\Sigma_{XY}}{F}^4 + 2\pnorm{\Sigma_{XY}\Sigma_{YX}}{F}^2\Big)\\
	&\qquad + 2\kappa\cdot\big(\tr(G_{[12]}\circ G_{[12]}) + \tr(H_{[11]}^2\circ H_{[22]}^2)\big) + \kappa^2\cdot\pnorm{H_{[11]}\circ H_{[22]}}{F}^2\Big]\\
	&= \frac{1}{36\tau_X^2\tau_Y^2}\Big( \pnorm{\Sigma_X}{F}^2\pnorm{\Sigma_Y}{F}^2+ \pnorm{\Sigma_{XY}}{F}^4)\bigg[1+\mathcal{O}_{M,\kappa}\Big(\frac{1}{(\tau_X\wedge \tau_Y)^2}\Big)\bigg].
	\end{align*}
	For the residual term, it can be bounded as follows:
	\begin{align*}
	\E \bar{R}_2^2\big((X_1,Y_1),(X_2,Y_2)\big)&\lesssim \tau_Y^4 \E^{1/2} (X_1^\top X_2)^4 \cdot\E^{1/2} \bar{R}_Y^4+ \tau_X^4 \E^{1/2} (Y_1^\top Y_2)^4 \cdot\E^{1/2}\bar{R}_X^4\\
	&\qquad + \tau_X^4\tau_Y^4 \E^{1/2} \bar{R}_X^4\cdot \E^{1/2}\bar{R}_Y^4 + \E R_1^2.
	\end{align*}
	Using Proposition \ref{prop:res_psi_bound}-(2), it follows that
	\begin{align*}
	\E R_1^2 \lesssim \tau_X^4 \E\psi_X^2+\tau_Y^4 \E \psi_Y^2+\tau_X^4\tau_Y^4 \E\psi_{X,Y}^2 \lesssim_M \frac{\pnorm{\Sigma_{XY}}{F}^2(1\vee \pnorm{\Sigma_{XY}}{F}^2)}{(\tau_X\wedge \tau_Y)^2}. 
	\end{align*}
	This, combined with Lemma \ref{lem:moment_R}-(2) and an easy calculation that $\E (X_1^\top X_2)^4\lesssim \pnorm{\Sigma_X}{F}^4$ and $\E (Y_1^\top Y_2)^4\lesssim \pnorm{\Sigma_Y}{F}^4$, shows that
	\begin{align*}
	&\E \bar{R}_2^2\big((X_1,Y_1),(X_2,Y_2)\big)\lesssim_M \tau_Y^4\cdot \pnorm{\Sigma_X}{F}^2 \cdot \tau_Y^{-8}\pnorm{\Sigma_Y}{F}^4 + \tau_X^4\cdot \pnorm{\Sigma_Y}{F}^2 \cdot \tau_X^{-8}\pnorm{\Sigma_X}{F}^4\\
	&\qquad + \tau_X^4\tau_Y^4\cdot  \tau_Y^{-8}\pnorm{\Sigma_Y}{F}^4\cdot\tau_X^{-8}\pnorm{\Sigma_X}{F}^4 +\frac{\pnorm{\Sigma_{XY}}{F}^2(1\vee \pnorm{\Sigma_{XY}}{F}^2)}{(\tau_X\wedge \tau_Y)^2}\\
	&\lesssim_M \frac{\pnorm{\Sigma_X}{F}^2\pnorm{\Sigma_Y}{F}^2+\pnorm{\Sigma_{XY}}{F}^2(1\vee \pnorm{\Sigma_{XY}}{F}^2)}{(\tau_X\wedge \tau_Y)^2}.
	\end{align*}	
	As $\var(g_2) = (1\pm \epsilon)\var(\bar{g}_2) + \mathcal{O}\big(\epsilon^{-1}\cdot\tau_X^{-2}\tau_Y^{-2}\var(\bar{R}_2({X,Y})\big)$ for any $\epsilon> 0$, the proof is now complete by noting that
	\begin{align*}
	\frac{\var\big(\bar{R}_2({X,Y})\big)}{\tau_X^2\tau_Y^2\var(\bar{g}_2)} \lesssim_M \frac{\frac{\pnorm{\Sigma_X}{F}^2\pnorm{\Sigma_Y}{F}^2+\pnorm{\Sigma_{XY}}{F}^2 + \pnorm{\Sigma_{XY}}{F}^4}{(\tau_X\wedge \tau_Y)^2}}{\pnorm{\Sigma_X}{F}^2\pnorm{\Sigma_Y}{F}^2+ \pnorm{\Sigma_{XY}}{F}^4}\lesssim_M \frac{1}{(\tau_X\wedge \tau_Y)^2},
	\end{align*}
	as desired.
	\end{proof}

\subsection{Hoeffding decomposition: higher orders}

The goal of this section is to prove the following.
\begin{proposition}\label{prop:var_higher_order}
	Suppose that the spectrum of $\Sigma$ is contained in $[1/M,M]$ for some $M>1$. Then the third- and fourth- order variance are bounded by
	\begin{align*}
	\E g_3^2+\E g_4^2\lesssim \tau_X^{-2}\tau_Y^{-2}\Big(\pnorm{\Sigma_X}{F}^2\pnorm{\Sigma_Y}{F}^2  + \pnorm{\Sigma_{XY}}{F}^2+  \pnorm{\Sigma_{XY}}{F}^4\Big)\lesssim \E g_1^2+\E g_2^2.
	\end{align*}
	Here the constants in $\lesssim$ depend on $M$ and the distribution of $Z_1$ only via its Poincar\'e constant $c_\ast$, excess kurtosis $\kappa$, and $\epsilon_0$ prescribed by Assumption \ref{assump:sepa}. The claims remain valid with $X=Y$ when the spectrum of $\Sigma_X=\Sigma_Y$ is contained in $[1/M,M]$ for some $M>1$.
\end{proposition}

To prove this proposition, we need to evaluate $k_3$ and $k_4$. $k_4 = k$ is already given by Proposition \ref{prop:hoef_decomp}, so we only need to compute $k_3$ as follows. 
\begin{lemma}\label{lem:h3}
	The third order kernel is given by 
	\begin{align*}
	&k_3(z_1,z_2,z_3) = \E k(z_1,z_2,z_3,Z_4)\\
	& = \frac{1}{12}\bigg[2\sum_{1\leq i_1< i_2\leq 3}U(x_{i_1},x_{i_2})V(y_{i_1},y_{i_2})+2\sum_{1\leq i\leq 3}\E U(X,x_{i}) V(Y,y_i)\\
	&\qquad -\sum_{(i_1,i_2,i_3) \in \sigma (1,2,3)} U(x_{i_1},x_{i_2}) V(y_{i_1},y_{i_3})-\sum_{1\leq i_1\neq i_2\leq 3} \E U(X,x_{i_1})V(Y,y_{i_2})\bigg].
	\end{align*}
\end{lemma}

The proof of the above lemma can be found in Appendix \ref{appendix:proof_hoeff}.
\begin{proof}[Proof of Proposition \ref{prop:var_higher_order}]
	For the second moment of $g_3$, we each term in its definition (\ref{def:hoeffding_kernel}) can be bounded as follows:
	\begin{itemize}
		\item First we have
		\begin{align*}
		&\E k_3^2\big((X_1,Y_1),(X_2,Y_2),(X_3,Y_3)\big)\nonumber\\
		&\lesssim \E^{1/2} U^4 (X_1,X_2)\cdot \E^{1/2} V^4(Y_1,Y_2)\lesssim  \tau_X^{-2}\tau_Y^{-2} \pnorm{\Sigma_X}{F}^2\pnorm{\Sigma_Y}{F}^2.
		\end{align*}
		The last inequality follows as
		\begin{align*}
		\E U^4 (X_1,X_2)&\lesssim \tau_X^{-4} \Big( \E(X_1^\top X_2)^4+\tau_X^8 \E \bar{R}_X^4(X_1,X_2) \Big)\\
		&\stackrel{(\ast)}{\lesssim} \tau_X^{-4}\big(\pnorm{\Sigma_X}{F}^4+\tau_X^{-8}\pnorm{\Sigma_X}{F}^8\big)\lesssim \tau_X^{-4} \pnorm{\Sigma_X}{F}^4,
		\end{align*}
		and similarly $\E V^4(Y_1,Y_2)\lesssim \tau_Y^{-4}\pnorm{\Sigma_Y}{F}^4$, using Lemma \ref{lem:moment_R} in $(*)$.
		
		\item By Theorem \ref{thm:dcov_expansion},
		\begin{align*}
		\big(\dcov^2(X,Y)\big)^2\lesssim_M \tau_X^{-2}\tau_Y^{-2}\pnorm{\Sigma_{XY}}{F}^4.
		\end{align*}
		
		\item By Proposition \ref{prop:var_first_order} and Lemma \ref{lem:SigmaXY_F_bound},
		\begin{align*}
		\E g_1^2(X_1,Y_1) &\lesssim_M \tau_X^{-2}\tau_Y^{-2}\pnorm{\Sigma_{XY}}{F}^2.
		\end{align*}
		
		\item Byy Proposition \ref{prop:var_second_order},
		\begin{align*}
		&\E g_2^2\big((X_1,Y_1),(X_2,Y_2) \big)\lesssim_M \tau_X^{-2}\tau_Y^{-2}\big(\pnorm{\Sigma_X}{F}^2\pnorm{\Sigma_Y}{F}^2+\pnorm{\Sigma_{XY}}{F}^4\big).
		\end{align*}
	\end{itemize} 
	
	Collecting the above bounds and using the variance lower bound in Lemma \ref{lem:first_variance_lower},
	\begin{align*}
	\E g_3^2\lesssim_M \tau_X^{-2}\tau_Y^{-2}\Big(\pnorm{\Sigma_X}{F}^2\pnorm{\Sigma_Y}{F}^2  + \pnorm{\Sigma_{XY}}{F}^2+  \pnorm{\Sigma_{XY}}{F}^4\Big)\lesssim_M \E g_1^2+\E g_2^2.
	\end{align*}
	The second moment bound for $g_4$ can be obtained in a similar way so we omit the proof.
\end{proof}

\subsection{Variance expansion}

With the groundwork laid above, we are now able to prove the following variance expansion formula.

\begin{theorem}\label{thm:variance_expansion}
	Suppose that the spectrum of $\Sigma$ is contained in $[1/M,M]$ for some $M>1$. Then
	\begin{align*}
	\biggabs{\frac{\var\big(\dcov_\ast^2(\bm{X},\bm{Y})\big)}{\bar{\sigma}_n^2(X,Y)}-1}\lesssim n^{-1/2}+(p\wedge q)^{-1/4}.
	\end{align*}
	Here $\bar{\sigma}_{n}^2(X,Y)$ is defined in Theorem \ref{thm:non_null_clt}, and the constant in $\lesssim$ depends on $M$ and the distribution of $Z_1$ only via its Poincar\'e constant $c_\ast$, excess kurtosis $\kappa$, and $\epsilon_0$ prescribed by Assumption \ref{assump:sepa}. The claims remain valid with $X=Y$ when the spectrum of $\Sigma_X=\Sigma_Y$ is contained in $[1/M,M]$ for some $M>1$.
\end{theorem}
\begin{proof}
	By Hoeffding decomposition $
	\dcov_\ast^2(\bm{X},\bm{Y}) = \sum_{c=0}^4 \binom{4}{c} U_n(g_c)$,  so 
	\begin{align*}
	\sigma_\Sigma^2 \equiv \var_\Sigma\big(\dcov_\ast^2(\bm{X},\bm{Y})\big)= \sum_{c=1}^4 \binom{4}{c}^2 \binom{n}{c}^{-1} \E g_c^2.
	\end{align*}
	Now we may apply Propositions \ref{prop:var_first_order}, \ref{prop:var_second_order}, and \ref{prop:var_higher_order} to conclude that the left hand side of the desired inequality is bounded by
	\begin{align*}
	\inf_{\epsilon>0}\biggabs{(1+\epsilon)\bigg(1+ \frac{\mathcal{O}_M(n^{-1}+(p\wedge q)^{-1/2})}{\epsilon}\bigg)-1}\asymp n^{-1/2}+(p\wedge q)^{-1/4}.
	\end{align*}
	The claim follows.
\end{proof}

\section{Normal approximation of truncated $\dcov_\ast^2$}\label{section:clt_truncated_dcov}

Let $\bar{g}_0\equiv \E\dcov_*^2(\bm{X},\bm{Y})=\dcov^2(X,Y)$. Recall $\bar{g}_1,\bar{g}_2$ defined in (\ref{def:g1_bar}) and (\ref{def:g2_bar}). Define the truncated sample distance covariance:
\begin{align}\label{def:bar_T}
\bar{T}_n\big(\bm{X},\bm{Y}\big)& = \sum_{c=0}^2 \binom{4}{c} U_n(\bar{g}_c).
\end{align}

The goal of this section is to prove the following non-null central limit theorem for $\bar{T}_n(\bm{X},\bm{Y})$.

\begin{theorem}\label{thm:clt_Tn_bar}
	Suppose that the spectrum of $\Sigma$ lies in $[1/M,M]$ for some $M > 1$. Then there exists some $C=C(M,Z_1)>0$ such that
	\begin{align*}
	\mathrm{err}_n\equiv d_{\mathrm{Kol}}\bigg(\frac{\bar{T}_n(\bm{X},\bm{Y})-\E \bar{T}_n(\bm{X},\bm{Y})}{\var^{1/2}(\bar{T}_n(\bm{X},\bm{Y}))},\mathcal{N}(0,1)\bigg)\leq C\bigg(\frac{1}{n}+\frac{1}{pq}\bigg)^{1/4}.
	\end{align*}
	Here $\var(\bar{T}_n(\bm{X},\bm{Y}))=\bar{\sigma}_n^2(X,Y)$ is defined in Theorem \ref{thm:non_null_clt}, and $C$ depends on $Z_1$ only via its Poincar\'e constant $c_\ast$, excess kurtosis $\kappa$, and $\epsilon_0$ prescribed by Assumption \ref{assump:sepa}.
\end{theorem}

The major tool to prove the CLT in Theorem \ref{thm:clt_Tn_bar} is the following discrete second-order Poincar\'e inequality proved by Chatterjee \cite{chatterjee2008new}.

\begin{lemma}[Discrete second-order Poincar\'e inequality]\label{lem:second_poincare}
	Let $X = (X_1,\ldots,X_n)$ be a vector of independent $\mathcal{X}$-valued random variables, and $X'=(X_1',\ldots,X_n')$ be an independent copy of $X$. For any $A\subset [n]$, define the random variable
	\begin{align*}
	X_i^A\equiv
	\begin{cases}
	X_i',& \text{if } i\in A,\\
	X_i,& \text{if } i\notin A.
	\end{cases}
	\end{align*}
	Define $\Delta_j f \equiv f(X)- f(X^{\{j\}})$, $T_{A}\equiv \sum_{j\notin A} \Delta_jf(X)\Delta_jf(X^A)$, and 
	\begin{align*}
	T\equiv \frac{1}{2}\sum_{A\subsetneq [n]}\frac{T_A}{{n\choose |A|}(n-|A|)}.
	\end{align*}
	Then with $W\equiv f(X)$ admitting finite variance $\sigma^2$, 
	\begin{align*}
	d_{\mathrm{Kol}}\bigg(\frac{W-\E(W)}{\var^{1/2}(W)},\mathcal{N}(0,1)\bigg) \leq 2 \bigg[\frac{\var^{1/2}\big(\E(T|W)\big)}{\sigma^2} + \frac{1}{2\sigma^3}\sum_{j=1}^n\E|\Delta_j f(X)|^3\bigg]^{1/2}.
	\end{align*}
\end{lemma}
\begin{proof}
	This follows from \cite[Theorem 2.2]{chatterjee2008new} and Lemma \ref{lem:dkov_dW}.
\end{proof}

We start with the following decomposition of $\bar{T}_n(\bm{X},\bm{Y})$. Its proof will be presented in Appendix \ref{subsec:truncate_decomposition}.

\begin{lemma}\label{lem:rep_Tbar}
	Let
	\begin{align*}
	\psi_1(\bm{X},\bm{Y}) &\equiv \sum_{I_n^2}\big(X_{i_1}^\top X_{i_2} Y_{i_1}^\top Y_{i_2} - \pnorm{\Sigma_{XY}}{F}^2\big),\\
	\psi_2(\bm{X},\bm{Y}) &\equiv \sum_{i\neq j} \big(X_i^\top\Sigma_{XY}Y_j+ X_j^\top\Sigma_{XY}Y_i\big),\\
	\psi_3(\bm{X},\bm{Y}) &\equiv \sum_{i=1}^n\Big[\frac{\pnorm{\Sigma_{XY}}{F}^2}{\tau_X^2}\big(\pnorm{X_i}{}^2 - \tr(\Sigma_X)\big)+ \frac{\pnorm{\Sigma_{XY}}{F}^2}{\tau_Y^2}\big(\pnorm{Y_i}{}^2 - \tr(\Sigma_Y)\big)\Big].
	\end{align*} 
	Then
	\begin{align*}
	&\bar{T}_n\big(\bm{X},\bm{Y}\big) = \dcov^2(X,Y)+ \frac{1}{\tau_X\tau_Y\cdot 2 \binom{n}{2}}\Big(\psi_1\big(\bm{X},\bm{Y}\big)-\psi_2\big(\bm{X},\bm{Y}\big)\Big)-\frac{2}{\tau_X\tau_Yn}\psi_{3}(\bm{X},\bm{Y}).
	\end{align*}
\end{lemma}

\begin{proof}[Outline of the proof of Theorem \ref{thm:clt_Tn_bar}]
	Define $T_{\psi_1}(\bm{X},\bm{Y})$-$T_{\psi_3}(\bm{X},\bm{Y})$ and $\Delta_i\psi_1(\bm{X},\bm{Y})$-$\Delta_i\psi_3(\bm{X},\bm{Y})$ as in the discrete second-order Poincar\'e inequality (cf. Lemma \ref{lem:second_poincare}). The following three propositions give variance and third moment bounds for these quantities.
	
	\begin{proposition}[Analysis of $\psi_1$]\label{prop:var_psi_1}
		Assume the conditions in Theorem \ref{thm:clt_Tn_bar}. Then the following hold:
		\begin{enumerate}
			\item (Variance bound)
			\begin{align*}
			\var\big[\E(T_{\psi_1}|\bm{X},\bm{Y})\big] &\lesssim n^3\cdot \pnorm{\Sigma_X}{F}^4\pnorm{\Sigma_Y}{F}^4 + n^4\cdot(1\vee\pnorm{\Sigma_{XY}}{F}^2)\pnorm{\Sigma_X}{F}^2\pnorm{\Sigma_Y}{F}^2 \\
			&\quad\quad\quad+ n^5\cdot(1\vee\pnorm{\Sigma_{XY}}{F}^2) \pnorm{\Sigma_{XY}}{F}^2.
			\end{align*}
			\item (Third moment bound)
			\begin{align*}
			\sum_{i=1}^n \E|\Delta_i \psi_1(\bm{X},\bm{Y})|^3 \lesssim n^{5/2} \tr^{3/2}(\Sigma_X)\tr^{3/2}(\Sigma_Y) + n^4\pnorm{\Sigma_{XY}}{F}^3.
			\end{align*}
		\end{enumerate}
		The constants in $\lesssim$ depend on $M$ and the distribution of $Z_1$ only.
	\end{proposition}
	
	\begin{proposition}[Analysis of $\psi_2$]\label{prop:var_psi_2}
		Assume the conditions in Theorem \ref{thm:clt_Tn_bar}. Then the following hold.
		\begin{enumerate}
			\item $\var\big[\E(T_{\psi_2}|\bm{X},\bm{Y})\big] \lesssim n^4\pnorm{\Sigma_{XY}}{F}^4$.
			\item $\sum_{i=1}^n \E|\Delta_j \psi_2(\bm{X},\bm{Y})|^3 \lesssim n^{5/2}\pnorm{\Sigma_{XY}}{F}^3$.
		\end{enumerate}
		The constants in $\lesssim$ depend on $M$ and the distribution of $Z_1$ only.
	\end{proposition}
	
	\begin{proposition}[Analysis of $\psi_3$]\label{prop:var_psi_3} 
		Assume the conditions in Theorem \ref{thm:clt_Tn_bar}. Then the following hold.
		\begin{enumerate}
			\item $\var\big[\E(T_{\psi_3}|\bm{X},\bm{Y})\big]\lesssim n\cdot \pnorm{\Sigma_{XY}}{F}^8(\tau_X^{-4} + \tau_Y^{-4})$.
			\item $\sum_{i=1}^n \E|\Delta_i \psi_3(\bm{X},\bm{Y})|^3 \lesssim n\cdot \pnorm{\Sigma_{XY}}{F}^6(\tau_X^{-3} + \tau_Y^{-3})$.
		\end{enumerate}
		The constants in $\lesssim$ depend on $M$ and the distribution of $Z_1$ only.
	\end{proposition}
	
	The proofs of these propositions will be detailed in Appendix \ref{appendix:proof_clt_Tn_bar}. By the proceeding propositions and Lemma \ref{lem:SigmaXY_F_bound}, we have 
	\begin{align*}
	D_1&\equiv \frac{\var^{1/2}\big(\E(T_{\psi_1}|\bm{X},\bm{Y})\big) + \var^{1/2}\big(\E(T_{\psi_2}|\bm{X},\bm{Y})\big)}{n^4\tau_X^2\tau_Y^2} + \frac{\var^{1/2}\big(\E(T_{\psi_3}|\bm{X},\bm{Y})\big)}{n^2\tau_X^2\tau_Y^2}\\
	&\lesssim_M \bigg[\frac{1}{n^{5/2}} + \frac{1\vee \pnorm{\Sigma_{XY}}{F}}{n^2\tau_X\tau_Y} + \frac{(1\vee\pnorm{\Sigma_{XY}}{F})\pnorm{\Sigma_{XY}}{F}}{n^{3/2}\tau_X^2\tau_Y^2}\bigg]+\frac{\pnorm{\Sigma_{XY}}{F}^4}{n^{3/2}\tau_X^2\tau_Y^2(\tau_X^2\wedge \tau_Y^2)}\\
	&\asymp_M \frac{1}{n^2\tau_X^2\tau_Y^2} \Big[n^{-1/2}\tau_X^2\tau_Y^2+\tau_X\tau_Y(1\vee \pnorm{\Sigma_{XY}}{F})+n^{1/2} (1\vee \pnorm{\Sigma_{XY}}{F})\pnorm{\Sigma_{XY}}{F}\Big],
	\end{align*}
	and
	\begin{align*}
	D_2&\equiv \frac{\sum_{i=1}^n \E|\Delta_i\psi_1(\bm{X},\bm{Y})|^3 + \E|\Delta_i\psi_2(\bm{X},\bm{Y})|^3}{n^6\tau_X^3\tau_Y^3} + \frac{\sum_{i=1}^n \E|\Delta_i \psi_3(\bm{X},\bm{Y})|^3}{n^3\tau_X^3\tau_Y^3}\\
	&\lesssim_M \bigg[\frac{1}{n^{7/2}} + \frac{\pnorm{\Sigma_{XY}}{F}^3}{n^2\tau_X^3\tau_Y^3}\bigg]+\frac{n\pnorm{\Sigma_{XY}}{F}^6(\tau_X^{-3}+\tau_Y^{-3})}{n^3\tau_X^3\tau_Y^3}\\
	&\asymp_M \frac{1}{n^3\tau_X^3\tau_Y^3}\Big[n^{-1/2}\tau_X^3\tau_Y^3+ n\pnorm{\Sigma_{XY}}{F}^3\Big].
	\end{align*}
	Using Theorem \ref{thm:variance_expansion} with the lower bound Lemma \ref{lem:first_variance_lower}, we have
	\begin{align}\label{ineq:var_lower_bound}
	\bar{\sigma}_n^2= \var\big(\bar{T}_n(\bm{X},\bm{Y})\big)\gtrsim_M \frac{1}{n^2\tau_X^2\tau_Y^2}\Big[n \pnorm{\Sigma_{XY}}{F}^2+ \tau_X^2\tau_Y^2 + \pnorm{\Sigma_{XY}}{F}^4\Big].
	\end{align}
	This entails that
	\begin{align*}
	\frac{D_1}{\bar{\sigma}^2}&\lesssim_M  \frac{n^{-1/2}\tau_X^2\tau_Y^2+\tau_X\tau_Y(1\vee \pnorm{\Sigma_{XY}}{F})+n^{1/2} (1\vee \pnorm{\Sigma_{XY}}{F})\pnorm{\Sigma_{XY}}{F}}{n \pnorm{\Sigma_{XY}}{F}^2+ \tau_X^2\tau_Y^2 + \pnorm{\Sigma_{XY}}{F}^4}\\
	&\lesssim \frac{1}{n^{1/2}}+\frac{1}{\tau_X\tau_Y}+ \frac{\tau_X\tau_Y\pnorm{\Sigma_{XY}}{F}}{n \pnorm{\Sigma_{XY}}{F}^2+ \tau_X^2\tau_Y^2} +\frac{n^{1/2}\pnorm{\Sigma_{XY}}{F}}{n \pnorm{\Sigma_{XY}}{F}^2+ \tau_X^2\tau_Y^2} \asymp \frac{1}{n^{1/2}}+\frac{1}{\tau_X\tau_Y},\\
	\frac{D_2}{\bar{\sigma}^3}&\lesssim_M \frac{n^{-1/2}\tau_X^3\tau_Y^3+ n\pnorm{\Sigma_{XY}}{F}^3}{n^{3/2} \pnorm{\Sigma_{XY}}{F}^3+ \tau_X^3\tau_Y^3 + \pnorm{\Sigma_{XY}}{F}^6}\lesssim \frac{1}{n^{1/2}}.
	\end{align*}
	The claim follows by using Lemma \ref{lem:second_poincare} in the form
	\begin{align*}
	\mathrm{err}_n \leq C\cdot\bigg[\frac{D_1}{\bar{\sigma}^2} + \frac{D_2}{\bar{\sigma}^3}\bigg]^{1/2},
	\end{align*}
	as desired.
\end{proof}

As outlined above, the major step in the proof is to obtain good enough variance and third moment bounds for $\E [T_{\psi_1}|\bm{X},\bm{Y}]$-$\E[T_{\psi_3}|\bm{X},\bm{Y}]$ and $\Delta_i\psi_1(\bm{X},\bm{Y})$-$\Delta_i\psi_3(\bm{X},\bm{Y})$, as claimed in Propositions \ref{prop:var_psi_1}-\ref{prop:var_psi_3}. The proofs to these propositions are fairly delicate and involved. The most complicated case appears to be the control for $\E [T_{\psi_1}|\bm{X},\bm{Y}], \Delta_i\psi_1(\bm{X},\bm{Y})$ associated with the first term $\psi_1(\bm{X},\bm{Y})$ due to its highest polynomial order by definition. The structure of the bounds in Propositions \ref{prop:var_psi_1}-\ref{prop:var_psi_3} also reveals a careful balance among the power in the terms $n,\pnorm{\Sigma_X}{F}\pnorm{\Sigma_Y}{F}, \pnorm{\Sigma_{XY}}{F}^2$. Such a balance turns out to be crucial to reach the announced error bound in Theorem \ref{thm:clt_Tn_bar} that requires no more than a bounded spectrum condition. See Appendix \ref{appendix:proof_clt_Tn_bar} for proof details.

\section{Proof of Theorem \ref{thm:non_null_clt}}\label{sec:proof_clts}

\noindent (\textbf{Step 1}) By definition of $\bar{T}_n(\bm{X},\bm{Y})$ in (\ref{def:bar_T}), we have
\begin{align*}
\Delta_n\equiv \dcov_\ast^2(\bm{X},\bm{Y})-\bar{T}_n(\bm{X},\bm{Y}) = \sum_{c=1}^2 U_n(g_c-\bar{g}_c)+ \sum_{c=3}^4 U_n(g_c). 
\end{align*}
This means
\begin{align*}
\frac{\var(\Delta_n)}{\bar{\sigma}_n^2(X,Y)}&\lesssim \sum_{c=1}^2 \frac{\var \big[U_n\big(g_c-\bar{g}_c\big)\big]}{\bar{\sigma}_n^2(X,Y)}+ \sum_{c=3}^4 \frac{\var\big[U_n(g_c)\big]}{\bar{\sigma}_n^2(X,Y)}.
\end{align*}
For $c=1$, by Lemma \ref{lem:g1} and the proof of Proposition \ref{prop:var_first_order}, we have
\begin{align*}
\frac{\var \big[U_n(g_1-\bar{g}_1)\big]}{\bar{\sigma}_n^2(X,Y)}\asymp \frac{\var(\bar{R}_1(X,Y))}{n\tau_X^2\tau_Y^2 \bar{\sigma}_n^2(X,Y)}&\lesssim \frac{\var(\bar{R}_1)}{\tau_X^2\tau_Y^2 \var(\bar{g}_1)}\lesssim_M \frac{1}{\tau_X\wedge \tau_Y}.
\end{align*}
For $c=2$, by Lemma \ref{lem:g2} and the proof of Proposition \ref{prop:var_second_order}, we have
\begin{align*}
\frac{\var \big[U_n(g_2-\bar{g}_2)\big]}{\bar{\sigma}_n^2(X,Y)}\lesssim_M \frac{\var(\bar{R}_2)}{n^2\tau_X^2\tau_Y^2 \bar{\sigma}_n^2(X,Y)}\lesssim_M \frac{\var(\bar{R}_2)}{\tau_X^2\tau_Y^2\var(\bar{g}_2)}\lesssim_M \frac{1}{(\tau_X\wedge \tau_Y)^2}.
\end{align*}
For $c=3,4$, the proof of Proposition \ref{prop:var_higher_order} yields that
\begin{align*}
&\var\big[U_n(g_3)\big]+\var\big[U_n(g_4)\big]\lesssim_M \frac{1}{n^3 \tau_X^2\tau_Y^2}\big(\pnorm{\Sigma_X}{F}^2\pnorm{\Sigma_Y}{F}^2  + \pnorm{\Sigma_{XY}}{F}^2+  \pnorm{\Sigma_{XY}}{F}^4\big),
\end{align*}
so using the variance lower bound in (\ref{ineq:var_lower_bound}), we have 
\begin{align*}
\frac{\var\big[U_n(g_3)\big]+\var\big[U_n(g_4)\big]}{\bar{\sigma}_n^2(X,Y)} \lesssim_M \frac{1}{n}.
\end{align*}
Collecting the bounds, we have 
\begin{align}\label{ineq:non_null_clt_1}
\frac{\var(\Delta_n)}{\bar{\sigma}_n^2(X,Y)}\lesssim_M \frac{1}{n}+\frac{1}{\tau_X\wedge \tau_Y}.
\end{align}
\noindent (\textbf{Step 2}) First consider normalization by $\bar{\sigma}_n(X,Y)=\var^{1/2}(\bar{T}_n(\bm{X},\bm{Y}))$. Using the decomposition
\begin{align*}
\bar{L}_n\equiv \frac{ \dcov_\ast^2(\bm{X},\bm{Y})-\dcov^2(X,Y)}{\bar{\sigma}_n(X,Y)} &= \frac{\Delta_n}{\bar{\sigma}_n(X,Y)}+ \frac{\bar{T}_n(\bm{X},\bm{Y})-\E \bar{T}_n(\bm{X},\bm{Y})}{\var^{1/2}(\bar{T}_n(\bm{X},\bm{Y}))},
\end{align*}
by Lemma \ref{lem:kolmogorov_res} and Theorem \ref{thm:clt_Tn_bar} (note that $\E\Delta_n = 0$),
\begin{align*}
&d_{\mathrm{Kol}}\bigg(\frac{ \dcov_\ast^2(\bm{X},\bm{Y})-\dcov^2(X,Y)}{\bar{\sigma}_n(X,Y)} ,\mathcal{N}(0,1)\bigg)\\
&\leq d_{\mathrm{Kol}}\bigg(\frac{\bar{T}_n(\bm{X},\bm{Y})-\E \bar{T}_n(\bm{X},\bm{Y})}{\var^{1/2}(\bar{T}_n(\bm{X},\bm{Y}))} ,\mathcal{N}(0,1)\bigg)+2\bigg(\frac{\var(\Delta_n)}{\bar{\sigma}_n^2(X,Y)}\bigg)^{1/3}\lesssim_M \bigg(\frac{1}{n\wedge p\wedge q}\bigg)^{1/6}.
\end{align*}
Next, with the normalization $\var^{1/2}(\dcov_\ast^2(\bm{X},\bm{Y}))$, consider the decomposition
\begin{align*}
\frac{ \dcov_\ast^2(\bm{X},\bm{Y})-\dcov^2(X,Y)}{\var^{1/2}(\dcov_\ast^2(\bm{X},\bm{Y}))} & = \bar{L}_n + \bar{L}_n \bigg(\frac{\bar{\sigma}_n(X,Y)}{\var^{1/2}(\dcov_\ast^2(\bm{X},\bm{Y}))}-1\bigg)\equiv \bar{L}_n+\bar{\Delta}_n.
\end{align*}
By Theorem \ref{thm:variance_expansion}, $\var(\bar{\Delta}_n)\lesssim_M n^{-1}+(p\wedge q)^{-1/2}$. The claim now follows by invoking Lemma \ref{lem:kolmogorov_res}.\qed

\section{Proof of Theorem \ref{thm:local_clt}}\label{section:proof_local_clt}

We write $A(\Sigma)$ as $A$ for simplicity in the proof. Let
\begin{align*}
m_0& \equiv \frac{ \pnorm{\Sigma_{XY}}{F}^2}{\tau_X\tau_Y},\quad \sigma_0^2\equiv \frac{2\pnorm{\Sigma_X}{F}^2\pnorm{\Sigma_Y}{F}^2}{n^2\tau_X^2\tau_Y^2}.
\end{align*}
Then with $\bar{\sigma}_n\equiv \bar{\sigma}_n(X,Y)$ in the proof,
\begin{align*}
\frac{\dcov_\ast^2(\bm{X},\bm{Y})-m_0}{\sigma_0}& \equiv \frac{\dcov_\ast^2(\bm{X},\bm{Y})-\dcov^2(X,Y)}{\bar{\sigma}_n}+\Delta_1+\Delta_2,
\end{align*}
where
\begin{align*}
\Delta_1&\equiv \frac{\dcov_\ast^2(\bm{X},\bm{Y})-\dcov^2(X,Y)}{\bar{\sigma}_n}\bigg(\frac{\bar{\sigma}_n}{\sigma_0}-1\bigg),\quad \Delta_2\equiv \frac{\dcov^2(X,Y)-m_0}{\sigma_0}.
\end{align*}
By Lemma \ref{lem:kolmogorov_res} and Theorem  \ref{thm:non_null_clt},
\begin{align*}
&d_{\mathrm{Kol}}\bigg(\frac{\dcov_\ast^2(\bm{X},\bm{Y})-m_0}{\sigma_0},\mathcal{N}(0,1)\bigg)
\leq C_M \bigg(\frac{1}{n\wedge p\wedge q}\bigg)^{1/6}+ 2 \var^{1/3}(\Delta_1)+\abs{\Delta_2}.
\end{align*}
To give a variance bound for $\Delta_1$, first by using Lemma \ref{lem:SigmaXY_F_bound}, it follows that
\begin{align*}
\bar{\sigma}_n^2&= \mathcal{O}_M\bigg(\frac{\pnorm{\Sigma_{XY}}{F}^2}{n\tau_X^2\tau_Y^2}+\frac{\pnorm{\Sigma_{XY}}{F}^4}{n\tau_X^2\tau_Y^2(n\wedge p\wedge q)}\bigg) +\frac{2\pnorm{\Sigma_X}{F}^2\pnorm{\Sigma_Y}{F}^2}{n^2\tau_X^2\tau_Y^2}.
\end{align*}
So with $\rho(A)\equiv A/(n\wedge p\wedge q)^{1/2}$,
\begin{align}\label{ineq:local_clt_1}
\frac{\bar{\sigma}_n^2}{\sigma_0^2}&=1+\mathcal{O}_M\bigg(\frac{A}{(pq)^{1/2}}+\frac{A^2}{n(n\wedge p\wedge q)}\bigg)=1+\mathcal{O}_M\big(\rho(A)+\rho^2(A)\big).
\end{align}
Therefore combined with Theorem \ref{thm:variance_expansion}, we have
\begin{align*}
\var\big(\Delta_1\big)&=\bigg(\frac{\bar{\sigma}_n}{\sigma_0}-1\bigg)^2\cdot \frac{\var\big(\dcov_\ast^2(\bm{X},\bm{Y})\big)}{\bar{\sigma}_n^2}\lesssim_M \big(\rho(A)+\rho^2(A)\big)^2.
\end{align*}
On the other hand, by Theorem \ref{thm:dcov_expansion} we have
\begin{align*}
\abs{\Delta_2}&\lesssim_M \frac{A}{(p\wedge q)^{1/2}}\lesssim \rho(A).
\end{align*}
The claim follows by collecting the estimates and using $1\wedge t \lesssim 1\wedge t^\alpha$ for $\alpha \in (0,1)$.\qed

\section{Proof of Theorem \ref{thm:dcov_power}}\label{section:proof_power_formula}

\begin{lemma}\label{lem:ratio_dcov}
	Suppose that the spectrum of $\Sigma$ is contained in $[1/M,M]$ for some $M>1$. Then with $\Delta_X$ and $\Delta_Y$ defined by 
	\begin{align*}
	\frac{\dcov_\ast(\bm{X})}{\dcov(X)}\equiv 1+\Delta_X,\quad \frac{\dcov_\ast(\bm{Y})}{\dcov(Y)}\equiv 1+\Delta_Y,
	\end{align*}
	we have $
	\E \Delta_X^2\vee \E \Delta_Y^2\leq C(n\wedge p\wedge q)^{-2}$ for some constant $C=C(M)>0$. 
\end{lemma}
\begin{proof}
	By mean and variance expansions in Theorems \ref{thm:dcov_expansion} and \ref{thm:variance_expansion},
	\begin{align*}
	\dcov^2(X)= \frac{\pnorm{\Sigma_X}{F}^2}{\tau_X^2}\Big[1+\mathcal{O}_M(\tau_X^{-1})\Big],\quad\var\big(\dcov_\ast^2(\bm{X})\big)\lesssim_M \frac{\pnorm{\Sigma_X}{F}^2}{n\tau_X^4}+\frac{\pnorm{\Sigma_X}{F}^4}{n^2\tau_X^4}.
	\end{align*}
	Then
	\begin{align*}
	\E\Delta_X^2 &= \E\Big[\frac{\dcov_\ast^2(\bm{X}) - \dcov^2(X)}{\dcov(X)(\dcov_*(\bm{X}) + \dcov(X))}\Big]^2\\
	&\leq \dcov^{-4}(X)\cdot \E\Big(\dcov_\ast^2(\bm{X}) - \dcov^2(X)\Big)^2\\
	&= \dcov^{-4}(X)\cdot\var\big(\dcov_\ast^2(\bm{X})\big) \asymp_M (n\tau_X^2)^{-1} + n^{-2} \lesssim_M (n\wedge p\wedge q)^{-2}.
	\end{align*}
	Similar bounds can be derived for $\Delta_Y$.
\end{proof}

\begin{proof}[Proof of Theorem \ref{thm:dcov_power}]
	Let $d_{X,Y}^2\equiv \dcov^2(X)\cdot \dcov^2(Y)$, and $\Delta_d$ be defined by $d_{X,Y}\big/\big(\frac{\pnorm{\Sigma_X}{F}}{\tau_X}\cdot \frac{\pnorm{\Sigma_Y}{F}}{\tau_Y}\big) = 1+\Delta_d$. Then by Theorem \ref{thm:dcov_expansion}, $\Delta_d = \mathcal{O}_M\big((\tau_X\wedge \tau_Y)^{-1}\big)$. Recall the definitions of (random) $\Delta_X$ and $\Delta_Y$ in Lemma \ref{lem:ratio_dcov}, and note that by definition $\Delta_d,\Delta_X,\Delta_Y\geq -1$. For any $t \geq 0$,
		\begin{align}\label{ineq:dcov_power_1}
		\mathfrak{p}_n(t)&\equiv\Prob\bigg(\frac{n\cdot\dcov_\ast^2(\bm{X},\bm{Y})}{\sqrt{2\dcov_\ast^2(\bm{X})\cdot \dcov_\ast^2(\bm{Y})}}\leq t\bigg)=\Prob\bigg(\frac{n\cdot\dcov_\ast^2(\bm{X},\bm{Y})}{\sqrt{2} d_{X,Y}}\leq t (1+\Delta_X)(1+\Delta_Y)\bigg)\nonumber\\
		&= \Prob\bigg(\frac{n\cdot \tau_X\tau_Y\dcov_\ast^2(\bm{X},\bm{Y})}{\sqrt{2}\pnorm{\Sigma_X}{F}\pnorm{\Sigma_Y}{F}}\leq t (1+\Delta_X)(1+\Delta_Y)(1+\Delta_d)\bigg).
		\end{align}
	Let $\Delta_1$ be defined by $(1+\Delta_X)(1+\Delta_Y)(1+\Delta_d)=1+\Delta_1+\Delta_d$. Then Lemma \ref{lem:ratio_dcov} yields that $\E \Delta_1^2\lesssim_M (n\wedge p\wedge q)^{-2}$. From (\ref{ineq:dcov_power_1}), it holds for any $\delta_1,\delta_d\in(0,1/2]$, 
		\begin{align}\label{ineq:pn_bound_1}
		&\mathfrak{p}_n(t) \leq \Prob\bigg(\frac{n\cdot \tau_X\tau_Y\dcov_\ast^2(\bm{X},\bm{Y})}{\sqrt{2}\pnorm{\Sigma_X}{F}\pnorm{\Sigma_Y}{F}}\leq t (1+\delta_1 + \delta_d)\bigg) +\Prob(|\Delta_1|\geq \delta_1) + \bm{1}_{|\Delta_d|\geq \delta_d}\nonumber\\
		&\leq \Prob\bigg(\frac{n\cdot \tau_X\tau_Y\dcov_\ast^2(\bm{X},\bm{Y})}{\sqrt{2}\pnorm{\Sigma_X}{F}\pnorm{\Sigma_Y}{F}}\leq t (1+\delta_1 + \delta_d)\bigg) +\frac{C_M(n\wedge p\wedge q)^{-2}}{\delta_1^2} + \bm{1}_{|\Delta_d|\geq \delta_d},
		\end{align}
	using the bound on $\E \Delta_1^2$ in the second inequality. Next we give two estimates for the upper bound of $\mathfrak{p}_n(t)$ useful for different magnitudes of $A\equiv A(\Sigma)$. 
	
	\noindent \textbf{(Estimate 1)}.  By the local central limit theorem in Theorem \ref{thm:local_clt}, with $A$ defined therein, $\rho(A) = A/(n\wedge p\wedge q)^{1/2}$, and
	\begin{align*}
	e(A)\equiv 1\bigwedge \bigg(\frac{1\vee A^2}{n\wedge p\wedge q}\bigg)^{1/6} \leq (n\wedge p\wedge q)^{-1/6} + 1\wedge \rho^{1/3}(A),
	\end{align*}
	we have by using $\Delta_d = \mathcal{O}_M\big((\tau_X\wedge \tau_Y)^{-1}\big)$
	\begin{align*}
	&\mathfrak{p}_n(t)\leq \inf_{\delta_1,\delta_d\in(0,1/2)}\bigg\{\Phi\bigg[ t (1+\delta_1+\delta_d)-\frac{A}{\sqrt{2}}\bigg]+C_M\cdot e(A) + \frac{C_M'(n\wedge p\wedge q)^{-2}}{\delta_1^2} + \bm{1}_{|\Delta_d|\geq \delta_d}\bigg\}\\
	&\leq \Phi\bigg(t-\frac{A}{\sqrt{2}}\bigg)+ C_{t,M}\bigg[\inf_{\delta_1\in(0,1/2]}\Big(\delta_1+\frac{(n\wedge p\wedge q)^{-2}}{\delta_1^2}\Big)+\inf_{\delta_d\in(0,1/2]}(\delta_d + \bm{1}_{|\Delta_d|\geq \delta_d})+e(A)\bigg]\\
	&\leq  \Phi\bigg(t-\frac{A}{\sqrt{2}}\bigg)+ C_{t,M}\cdot \big[ (n\wedge p\wedge q)^{-1/6} + 1\wedge \rho^{1/3}(A)\big].
	\end{align*}
	\noindent \textbf{(Estimate 2)}. Note that by Theorem \ref{thm:dcov_expansion} and (\ref{ineq:local_clt_1}), we have
	\begin{align*}
	\E \bigg[\frac{n\cdot \tau_X\tau_Y\dcov_\ast^2(\bm{X},\bm{Y})}{\sqrt{2}\pnorm{\Sigma_X}{F}\pnorm{\Sigma_Y}{F}}\bigg] &= \frac{A}{\sqrt{2}}\big(1+\mathcal{O}_M((p\wedge q)^{-1/2}\big),\\
	\var\bigg[\frac{n\cdot \tau_X\tau_Y\dcov_\ast^2(\bm{X},\bm{Y})}{\sqrt{2}\pnorm{\Sigma_X}{F}\pnorm{\Sigma_Y}{F}}\bigg] &\asymp_M 1\vee \rho^2(A).
	\end{align*}
	Hence by choosing $u = A/(2\sqrt{2}C_M(1+\rho(A)))$, and $\delta_1 = \delta_d = 1/2$ in (\ref{ineq:pn_bound_1}), we have
	\begin{align*}
	\mathfrak{p}_n(t)&\leq \Prob\big(A/\sqrt{2}-C_M\cdot u(1+\rho(A)) \leq 2t\big)+u^{-2}+C_M(n\wedge p\wedge q)^{-2}\\
	&\leq \bm{1}_{A\leq 4\sqrt{2}t}+C_M\cdot\Big(\big[1+\rho(A)^2\big]/A^2 + (n\wedge p\wedge q)^{-2}\Big).
	\end{align*}
	which implies 
	\begin{align*}
	&\mathfrak{p}_n(t)-\Phi\bigg(t-\frac{A}{\sqrt{2}}\bigg)\leq C_{t,M}\bigg(e^{-A^2/C}+\frac{1+\rho^2(A)}{A^2}+\bm{1}_{A\leq 4\sqrt{2}t}+\frac{1}{(n\wedge p\wedge q)^2}\bigg).
	\end{align*}
	Combining the estimates (1) and (2), we have
	\begin{align*}
	\mathfrak{p}_n(t)-\Phi\bigg(t-\frac{A}{\sqrt{2}}\bigg)&\leq C_{t,M} \bigg[\bigg((n\wedge p\wedge q)^{-1/6} + 1\wedge \rho^{1/3}(A)\bigg)\\
	&\quad\quad\quad\bigwedge \bigg(e^{-A^2/C}+\frac{1}{A^2}+\bm{1}_{A\leq 4\sqrt{2}t}+ \frac{1}{n\wedge p\wedge q}\bigg)\bigg].
	\end{align*}
	The term in the bracket can be bounded by, via balancing the leading terms $\rho^{1/3}(A)$ and $A^{-2}$,
	\begin{align}\label{ineq:dcov_power_2}
	&\mathfrak{p}_n(t)-\Phi\bigg(t-\frac{A}{\sqrt{2}}\bigg)\leq \frac{C_{t,M}}{(n\wedge p\wedge q)^{1/7}}.
	\end{align}
	By applying similar arguments to the other direction (lower bound) and the case $t\leq 0$, we arrive at
	\begin{align*}
	\Big|\mathfrak{p}_n(t)-\Phi\bigg(t-\frac{A}{\sqrt{2}}\bigg)\Big|\leq \frac{C_{t,M}}{(n\wedge p\wedge q)^{1/7}},\quad t\in\R.
	\end{align*}
	Finally, using Theorem \ref{thm:dcov_expansion} and its analogous expansion to $\dcov^2(X),\dcov^2(Y)$, we have
	\begin{align*}
	&n\dcor^2(X,Y)=\frac{n\dcov^2(X,Y)}{\sqrt{\dcov^2(X)\dcov^2(Y)}}\\
	&= \frac{n\pnorm{\Sigma_{XY}}{F}^2}{\pnorm{\Sigma_X}{F}\pnorm{\Sigma_Y}{F}}\Big[1+\mathcal{O}_M\big((\tau_X\wedge \tau_Y)^{-1}\big)\Big]=A \Big[1+\mathcal{O}_M\big((\tau_X\wedge \tau_Y)^{-1}\big)\Big],
	\end{align*}
	the $A$ term in (\ref{ineq:dcov_power_2}) can be replaced by $n\dcor^2(X,Y)$ at the lost of a larger constant $C_{t,M}$ by \cite[Lemma 2.4]{han2021general}.
\end{proof}

\section{Proof of Theorem \ref{thm:minimax_lower}}\label{section:proof_minimax}

Let $\phi_\Sigma$ denote the Lebesgue density for $\mathcal{N}(0,\Sigma)$. For any prior $\Pi$ on $\Theta(\zeta,\Sigma_0)$, let $P_\Pi$ the probability measure corresponding to the density $\phi_\Pi(x)\equiv \int p_\Sigma(x)\,\d{\Pi(\Sigma)}$. Then 
\begin{align*}
&\inf_{\psi}\sup_{\Sigma \in \Theta(\zeta,\Sigma_0)} \big(\E_{\Sigma_0}\psi + \E_{\Sigma}(1-\psi)\big)\geq \inf_{\psi}\big(P_{\Sigma_0}(\psi)+P_{\Pi}(1-\psi)\big)\\
& = 1- d_{\textrm{TV}}(P_\Pi,P_0) \geq 1- \bigg(\int\frac{\phi_\Pi^2}{\phi_{\Sigma_0}}-1\bigg)^{1/2}.
\end{align*}
The last inequality follows from, e.g. \cite[Equation (2.27), pp. 90]{tsybakov2008introduction}.
So the goal is to find the smallest possible $\zeta>0$ so that the right hand side of the above display has a non-trivial lower bound under a constructed prior $\Pi$ on $\Theta(\zeta,\Sigma_0)$. 

First consider the case $\Sigma_0=I$. For any $\tilde{u}\in\R^p$ and $\tilde{v}\in\R^q$, let $u\equiv [\tilde{u}, 0_q]\in\R^{p+q}$ and $v\equiv [0_p, \tilde{v}]\in\R^{p+q}$. For any $a > 0$, let
\begin{align*}
\Sigma_{u,v}\equiv\Sigma_{u,v}(a)\equiv I + a\big(uv^\top + vu^\top\big) = 
\begin{pmatrix}
I & a\tilde{u}\tilde{v}^\top\\
a\tilde{v}\tilde{u}^\top & I
\end{pmatrix}.
\end{align*}
We place independent priors $\tilde{u}_i\sim_{\textrm{i.i.d.}}\sqrt{q}\cdot\textrm{Unif}\{\pm 1\}$ on $\tilde{u}$ and $\tilde{v}_j\sim_{\textrm{i.i.d.}}\sqrt{p}\cdot\textrm{Unif}\{\pm 1\}$ on $\tilde{v}$, so that $\pnorm{\tilde{u}}{}^2=\pnorm{u}{}^2=\pnorm{\tilde{v}}{}^2 = \pnorm{v}{}^2 = pq$ and $\iprod{u}{v} = 0$. Then by direct calculation,
\begin{align*}
&(uv^\top + vu^\top)(u+v) = u\cdot\pnorm{v}{}^2 + v\cdot \pnorm{u}{}^2 = pq\cdot (u+v),\\
&(uv^\top + vu^\top)(u-v) = -u\cdot\pnorm{v}{}^2 + v\cdot \pnorm{u}{}^2 = -pq\cdot (u-v).
\end{align*}
This entails that the rank two matrix $(uv^\top + vu^\top)$ has eigenvalue $pq$ with eigenvector $u+v$ and eigenvalue $-pq$ with eigenvector $u-v$. Hence $\Sigma_{u,v}$ has eigenvalues $1+apq$ with eigenvector $u + v$, $1-apq$ with eigenvector $u-v$, and all rest eigenvalues as $1$, hence $\det(\Sigma_{u,v}) = 1-(apq)^2$. So by writing $\{s_i\}_{i=1}^{p+q-2}$ as a basis for the orthogonal complement of $u + v$ and $u - v$ in $\R^{p+q}$, we have the eigenvalue decomposition
\begin{align*}
\Sigma_{u,v} = \sum_{i=1}^{p+q-2}s_is_i^\top + (1+apq)\frac{(u+v)(u+v)^\top}{\pnorm{u+v}{}^2} + (1-apq)\frac{(u-v)(u-v)^\top}{\pnorm{u-v}{}^2}.
\end{align*}
This implies that
\begin{align*}
\Sigma_{u,v}^{-1} &= \sum_{i=1}^{p+q-2}s_is_i^\top + \frac{1}{1+apq}\frac{(u+v)(u+v)^\top}{\pnorm{u+v}{}^2} + \frac{1}{1-apq}\frac{(u-v)(u-v)^\top}{\pnorm{u-v}{}^2}\\
&= \Big[I - \frac{(u+v)(u+v)^\top}{\pnorm{u+v}{}^2} - \frac{(u-v)(u-v)^\top}{\pnorm{u-v}{}^2}\Big]\\
&\quad\quad + \frac{1}{1+apq}\frac{(u+v)(u+v)^\top}{\pnorm{u+v}{}^2} + \frac{1}{1-apq}\frac{(u-v)(u-v)^\top}{\pnorm{u-v}{}^2}\\
&= I - \frac{apq}{1+apq}\frac{(u+v)(u+v)^\top}{\pnorm{u+v}{}^2} +\frac{apq}{1-apq} \frac{(u-v)(u-v)^\top}{\pnorm{u-v}{}^2}\\
&= I - \frac{a}{2(1+apq)}(u+v)(u+v)^\top + \frac{a}{2(1-apq)}(u-v)(u-v)^\top.
\end{align*}
Let $\phi_\Pi$ be the joint density of $\{(X_i,Y_i)\}_{i=1}^n$ under the prescribed prior on $\tilde{u}$ and $\tilde{v}$, and $\phi_I$ be the joint density under the null. With
\begin{align}\label{ineq:minimax_1}
&\tilde{\Sigma}(u_1,u_2,v_1,v_2) \nonumber\\
&\equiv -\frac{a}{2(1+apq)}(u_1+v_1)(u_1+v_1)^\top +  \frac{a}{2(1-apq)}(u_1-v_1)(u_1-v_1)^\top \nonumber\\
&\qquad - \frac{a}{2(1+apq)}(u_2+v_2)(u_2+v_2)^\top + \frac{a}{2(1-apq)}(u_2-v_2)(u_2-v_2)^\top,
\end{align}
we have
\begin{align*}
\int\frac{\phi_\Pi^2}{\phi_I} &= \int\frac{\Big[\sum_{u,v}2^{-(p+q)}(2\pi)^{-n(p+q)/2}\det^{-n/2}(\Sigma_{u,v})e^{-\frac{1}{2}\cdot\sum_{i=1}^n x_i^\top\Sigma_{u,v}^{-1}x_i}\Big]^2}{(2\pi)^{-n(p+q)/2}\det^{-n/2}(I_{p+q})e^{-\frac{1}{2}\cdot\sum_{i=1}^n\pnorm{x_i}{}^2}}\,\d x\\
&=\sum_{(u_1,v_1),(u_2,v_2)}\frac{2^{-2(p+q)}}{\big[(1-(apq)^2\big]^n}\int (2\pi)^{-n(p+q)/2}e^{-\frac{1}{2}\cdot\sum_{i=1}^nx_i^\top [\Sigma_{u_1,v_1}^{-1} + \Sigma_{u_2,v_2}^{-1} - I]x_i}\,\d x\\
&\equiv \sum_{(u_1,v_1),(u_2,v_2)}\frac{2^{-2(p+q)}}{\big[1-(apq)^2\big]^n}\int (2\pi)^{-n(p+q)/2}e^{-\frac{1}{2}\cdot\sum_{i=1}^nx_i^\top(I+\tilde{\Sigma}(u_1,u_2,v_1,v_2))x_i}\,\d x\\
&\stackrel{(*)}{=} \E_{(u_1,v_1),(u_2,v_2)}\bigg[\frac{1}{\big[1-(apq)^2\big]^n}\det\big(I+\tilde{\Sigma}(u_1,u_2,v_1,v_2)\big)^{-n/2}\bigg].
\end{align*}
Here in $(*)$, $(u_1,v_1)$ and $(u_2,v_2)$ denote two independent copies of the prescribed prior on $(u,v)$.

By Lemma \ref{lem:perturb_eigen} and the notation $\{\lambda_i=\lambda_i(u_1,u_2,v_1,v_2)\}_{i=1}^4$ therein, $I + \tilde{\Sigma}$ has $p+q-4$ eigenvalues of $1$, and the other four eigenvalues being $1+\lambda_i$, $1\leq i\leq 4$. Hence by the second claim of Lemma \ref{lem:perturb_eigen}, for any $a$ such that $apq\leq 1/4$, we have
\begin{align*}
\int\frac{\phi_\Pi^2}{\phi_I}&= \E_{(u_1,v_1),(u_2,v_2)}\Big[\frac{1}{\big[1-(apq)^2\big]^n}\cdot\frac{1}{\big[(1+\lambda_1)(1+\lambda_2)(1+\lambda_3)(1+\lambda_4)\big]^{n/2}}\Big]\\
&= \E_{(u_1,v_1),(u_2,v_2)}\Big[\frac{1}{\big[1-(apq)^2\big]^n\cdot\big[1 + \frac{a^2}{1-(apq)^2}(p^2q^2 - \iprod{u_1}{u_2}\iprod{v_1}{v_2})\big]^n}\Big]\\
&= \E_{(u_1,v_1),(u_2,v_2)}\Big[\frac{1}{\big(1-a^2\iprod{u_1}{u_2}\iprod{v_1}{v_2}\big)^n}\Big]\\
&= 1+ n \E\big(a^2\iprod{u_1}{u_2}\iprod{v_1}{v_2}\big) \\
&\qquad +n(n+1) \E \bigg[\big(a^2 \iprod{u_1}{u_2}\iprod{v_1}{v_2}\big)^2 \int_0^1 \frac{1-t}{\big(1-t a^2 \iprod{u_1}{u_2}\iprod{v_1}{v_2}\big)^{n+2}}\,\d{t}\bigg]. 
\end{align*}
To calculate the right hand side of the above display, first note that the first order term vanishes: $\E(a^2\iprod{u_1}{u_2}\iprod{v_1}{v_2})=0$. Next we compute the second order term. With $\epsilon_i$'s denoting i.i.d. Rademacher random variables,
\begin{align*}
\E \big(a^2 \iprod{u_1}{u_2}\iprod{v_1}{v_2}\big)^4 = a^8 p^4 q^4 \cdot \E\bigg(\sum_{i=1}^p \epsilon_i\bigg)^4\cdot \E\bigg(\sum_{i=1}^q \epsilon_i\bigg)^4\lesssim a^8 p^6 q^6.
\end{align*}
On the other hand, under $a^2p^2q^2\leq 1/2$, there exists some absolute constant $c_0>0$ such that for any $t\in[0,1]$,
\begin{align*}
&\E \bigg[\frac{1}{\big(1-t a^2 \iprod{u_1}{u_2}\iprod{v_1}{v_2}\big)^{2(n+2)}}\bigg]\leq \E e^{c_0 a^2 n\iprod{u_1}{u_2}\iprod{v_1}{v_2}} \\
&= \E e^{c_0 a^2 npq \sum_{i=1}^p \epsilon_i \sum_{j=1}^q \epsilon_j'}\leq \E \exp\bigg[(c_0^2/2) a^4 n^2p^3q^3 \bigg(\frac{1}{\sqrt{q}}\sum_{j=1}^q \epsilon_j'\bigg)^2\bigg].
\end{align*}
Under $c_0^2 a^4 n^2p^3q^3\leq 1/2$, the above display is bounded by an absolute constant. Collecting all the estimates, we have established that if
\begin{align}\label{ineq:minimax_2}
\max\big\{a^2 p^2q^2, c_0^2 a^4 n^2p^3q^3\big\}\leq 1/2,
\end{align}
the following bound holds: for some absolute constant $C>0$,
\begin{align*}
\int\frac{\phi_\Pi^2}{\phi_I}-1\leq C\cdot a^4 n^2p^3q^3.
\end{align*}	
Now for any $\delta > 0$, by taking $a_*= \tau n^{-1/2}p^{-3/4}q^{-3/4}$ for some sufficiently small $\tau = \tau(\delta) > 0$, we have  $\int \phi_\Pi^2/\phi_I - 1\leq \delta$ and (\ref{ineq:minimax_2}) is satisfied. As $\Sigma_{u,v}(a^*)\in\Theta(\zeta)$ for such choice of $a_*$ and any realization of $(u,v)$, this completes the proof for the lower bound when $\Sigma_0=I$ of the order $\pnorm{a_* uv^\top}{F}^2 = a_*^2\pnorm{u}{}^2\pnorm{v}{}^2 = a_*^2p^2q^2 \asymp \sqrt{pq}/n$. For the general case $\Sigma_0$, we may proceed with the above calculations using $\Sigma_0^{1/2}\Sigma_{u,v}\Sigma_0^{1/2}$ instead of $\Sigma_{u,v}$, which gives the same separation rate under the assumption on the spectrum of $\Sigma_X,\Sigma_Y$. \qed

\begin{lemma}\label{lem:perturb_eigen}
	The matrix $\tilde{\Sigma}(u_1,u_2,v_1,v_2)$ defined in (\ref{ineq:minimax_1}) has at most 4 nontrivial eigenvalues of the form:
	\begin{align*}
	\lambda_i &= (c_2-c_1)(pq + \iprod{u_1}{u_2}) - \beta_i(c_1+c_2)(\iprod{v_1}{v_2} + pq),\quad i=1,2,\\
	\lambda_i &= (c_2-c_1)(pq - \iprod{u_1}{u_2}) + \beta_i(c_1+c_2)(\iprod{v_1}{v_2}-pq),\quad i=3,4.
	\end{align*}
	Here $(\beta_1,\beta_2)$ and $(\beta_3,\beta_4)$ are solutions to equations (\ref{eq:eigen_12}) and (\ref{eq:eigen_34}) below, respectively. Furthermore, 
	\begin{align*}
	(1+\lambda_1)(1+\lambda_2) = (1+\lambda_3)(1+\lambda_4) = 1 + \frac{a^2}{1-(apq)^2}(p^2q^2 - \iprod{u_1}{u_2}\iprod{v_1}{v_2}).
	\end{align*}
\end{lemma}
\begin{proof}[Proof of Lemma \ref{lem:perturb_eigen}]
	We write $\tilde{\Sigma}(u_1,u_2,v_1,v_2)$ as $\tilde{\Sigma}$ in the proof for simplicity. Define 
	\begin{align*}
	b &\equiv \alpha_1u_1 + \alpha_2u_2 + \alpha_3v_1 + \alpha_4v_2,\quad c_1\equiv \frac{a}{2(1+apq)},\quad c_2=\frac{a}{2(1-apq)}.
	\end{align*}
	Then 
	\begin{align*}
	\tilde{\Sigma} 
	&\equiv -c_1\Big[(u_1+v_1)(u_1+v_1)^\top + (u_2+v_2)(u_2+v_2)^\top\Big]\\
	&\qquad\qquad + c_2\Big[(u_1-v_1)(u_1-v_1)^\top +(u_2-v_2)(u_2-v_2)^\top \Big],
	\end{align*}
	and
	\begin{align*}
	\tilde{\Sigma}b \equiv T_1u_1 + T_2u_2 + T_3u_3 + T_4u_4,
	\end{align*}
	where
		\begin{align*}
		T_1  &= \alpha_1\cdot (c_2-c_1)pq + \alpha_2\cdot (c_2-c_1)\iprod{u_1}{u_2} - \alpha_3 \cdot (c_1+c_2)pq - \alpha_4\cdot (c_1+c_2)\iprod{v_1}{v_2},\\
		T_2  &= \alpha_1\cdot (c_2-c_1)\iprod{u_1}{u_2} + \alpha_2\cdot (c_2-c_1)pq - \alpha_3 \cdot (c_1+c_2)\iprod{v_1}{v_2} -\alpha_4\cdot(c_1+c_2)pq,\\
		T_3  &= -\alpha_1\cdot (c_1+c_2)pq  -\alpha_2\cdot (c_1+c_2)\iprod{u_1}{u_2} + \alpha_3 \cdot (c_2-c_1)pq + \alpha_4\cdot (c_2-c_1)\iprod{v_1}{v_2},\\
		T_4  &= -\alpha_1\cdot (c_1+c_2)\iprod{u_1}{u_2}  -\alpha_2\cdot (c_1+c_2)pq + \alpha_3 \cdot(c_2-c_1)\iprod{v_1}{v_2} + \alpha_4\cdot (c_2-c_1)pq.
		\end{align*}
	Hence the identity $T_1/\alpha_1 = T_2/\alpha_2 = T_3/\alpha_3 = T_4/\alpha_4$ is equivalent to (denoting $\rho_{ij}\equiv \alpha_i/\alpha_j$)
		\begin{align}
		\tag{i}&\quad\quad\rho_{21}\cdot(c_2-c_1)\iprod{u_1}{u_2} - \rho_{31}\cdot(c_1+c_2)pq - \rho_{41}\cdot(c_1+c_2)\iprod{v_1}{v_2}\\
		\tag{ii}&=\rho_{12}\cdot (c_2-c_1)\iprod{u_1}{u_2} - \rho_{32}\cdot(c_1+c_2)\iprod{v_1}{v_2} - \rho_{42}\cdot(c_1+c_2)pq\\
		\tag{iii}&= -\rho_{13}\cdot(c_1+c_2)pq - \rho_{23}\cdot(c_1+c_2)\iprod{u_1}{u_2} + \rho_{43}\cdot(c_2-c_1)\iprod{v_1}{v_2}\\
		\tag{iv}&= -\rho_{14}\cdot(c_1+c_2)\iprod{u_1}{u_2} - \rho_{24}\cdot(c_1+c_2)pq + \rho_{34}\cdot(c_2-c_1)\iprod{v_1}{v_2}.
		\end{align}
	For the first set of two solutions, let $\alpha_1 = \alpha_2 = 1$ and $\alpha_3 = \alpha_4 = \beta$. Then $(i) = (ii)$ and $(iii) = (iv)$ automatically holds, and $(ii) = (iii)$ is equivalent to
	\begin{align}\label{eq:eigen_12}
	\notag&(c_1+c_2)(\iprod{v_1}{v_2} +pq)\cdot\beta^2 + (c_2-c_1)(\iprod{v_1}{v_2} - \iprod{u_1}{u_2})\cdot \beta\\
	&\quad\quad - (c_1+c_2)(\iprod{u_1}{u_2} + pq) = 0.
	\end{align}
	Note that as $|\iprod{v_1}{v_2}|\leq \pnorm{v_1}{}\pnorm{v_2}{} = pq$ and similarly for $|\iprod{u_1}{u_2}|$, we have
	\begin{align*}
	(c_2-c_1)^2(\iprod{v_1}{v_2} - \iprod{u_1}{u_2})^2 + 4(c_1+c_2)^2(\iprod{v_1}{v_2} +pq)(\iprod{u_1}{u_2} +pq)\geq 0.
	\end{align*}
	Hence (\ref{eq:eigen_12}) has two (possible identical) solutions, and these two solutions denoted by $\beta_1$ and $\beta_2$ lead to two eigenvectors of $\tilde{\Sigma}$. For the second set of two solutions, let $\alpha_1 = -\alpha_2 = 1$ and $\alpha_3 = -\alpha_4 = \beta$. Then again $(i) = (ii)$ and $(iii) = (iv)$ automatically, and $(ii) = (iii)$ is equivalent to
	\begin{align}\label{eq:eigen_34}
	\notag&(c_1+c_2)(\iprod{v_1}{v_2} -pq)\cdot\beta^2 + (c_2-c_1)(\iprod{v_1}{v_2} - \iprod{u_1}{u_2})\cdot \beta\\
	&\quad\quad - (c_1+c_2)(\iprod{u_1}{u_2} - pq) = 0.
	\end{align}
	The two (possible identical) solutions of the above quadratic equation (existence due to similar reason as (\ref{eq:eigen_12})), denoted by $\beta_3$ and $\beta_4$, lead to the other two eigenvectors. Summarizing the above arguments, $\tilde{\Sigma}$ has four eigenvalue
	\begin{align*}
	\lambda_i &= (c_2-c_1)(pq + \iprod{u_1}{u_2}) - \beta_i(c_1+c_2)(\iprod{v_1}{v_2} + pq),\quad i=1,2,\\
	\lambda_i &= (c_2-c_1)(pq - \iprod{u_1}{u_2}) + \beta_i(c_1+c_2)(\iprod{v_1}{v_2}-pq),\quad i=3,4.
	\end{align*}
	
	For the second claim of the lemma, note that (\ref{eq:eigen_12}) implies that
	\begin{align*}
	\beta_1+\beta_2 &= \frac{(c_2-c_1)(\iprod{u_1}{u_2} - \iprod{v_1}{v_2})}{(c_1+c_2)(\iprod{v_1}{v_2} + pq)},\\
	\beta_1\cdot\beta_2 &= -\frac{(c_1+c_2)(\iprod{u_1}{u_2} + pq)}{(c_1+c_2)(\iprod{v_1}{v_2} + pq)}.
	\end{align*}
	Hence we have
	\begin{align*}
	&(1+\lambda_1)(1+\lambda_2)\\
	& = \Big(\big[1 + (c_2-c_1)(pq + \iprod{u_1}{u_2})\big] - \beta_1(c_1+c_2)(\iprod{v_1}{v_2} + pq)\Big)\\
	&\quad\quad\cdot \Big(\big[1 + (c_2-c_1)(pq + \iprod{u_1}{u_2})\big] - \beta_2(c_1+c_2)(\iprod{v_1}{v_2} + pq)\Big)\\
	\quad\quad &=\big[1 + (c_2-c_1)(pq + \iprod{u_1}{u_2})\big]^2\\
	&\quad\quad\quad-(\beta_1+\beta_2)(c_1+c_2)(\iprod{v_1}{v_2} + pq)\big[1 + (c_2-c_1)(pq + \iprod{u_1}{u_2})\big]\\
	&\quad\quad\quad +\beta_1\beta_2\cdot(c_1+c_2)(\iprod{v_1}{v_2} + pq)^2\\
	&= \big[1 + (c_2-c_1)(pq + \iprod{u_1}{u_2})\big]^2\\
	&\quad\quad\quad - (c_2-c_1)(\iprod{u_1}{u_2} - \iprod{v_1}{v_2})\big[1 + (c_2-c_1)(pq + \iprod{u_1}{u_2})\big]\\
	&\quad\quad\quad - (c_1+c_2)^2(\iprod{u_1}{u_2} + pq)(\iprod{v_1}{v_2} + pq)\\
	&= 1 + (c_2-c_1)(2pq + \iprod{u_1}{u_2} + \iprod{v_1}{v_2})- 4c_1c_2(\iprod{u_1}{u_2}+pq)(\iprod{v_1}{v_2}+pq).
	\end{align*}
	By plugging in the definition of $c_1$ and $c_2$, we have
	\begin{align*}
	(1+\lambda_1)(1+\lambda_2) &= 1 + \frac{a^2pq}{1-(apq)^2}(2pq + \iprod{u_1}{u_2} + \iprod{v_1}{v_2})\\
	&\quad\quad\quad - \frac{a^2}{1-(apq)^2}(\iprod{u_1}{u_2}+pq)(\iprod{v_1}{v_2}+pq)\\
	&= 1 + \frac{a^2}{1-(apq)^2}(p^2q^2 - \iprod{u_1}{u_2}\iprod{v_1}{v_2}).
	\end{align*}
	The argument for the expression $(1+\lambda_3)(1+\lambda_4)$ is similar, so the proof is complete.
\end{proof}

\appendix

\section{Proof of Proposition \ref{prop:hoef_decomp}}\label{appendix:kernel_rep}

\begin{lemma}
	Recall $A= \{A_{ij}=\pnorm{X_i-X_j}{}\}$ and $B= \{B_{ij}= \pnorm{Y_i-Y_j}{}\}$ and $A^\ast, B^\ast$ defined in (\ref{def:AB}). Then
	\begin{align*}
	\dcov_\ast^2(\bm{X},\bm{Y}) =\frac{1}{n(n-3)}\sum_{k\neq \ell} A_{k\ell}^\ast B_{k\ell}^\ast= \frac{1}{n(n-3)}\bigg(\tr(AB)+ \frac{\bm{1}^\top A \bm{1} \bm{1}^\top B \bm{1} }{(n-1)(n-2)}-\frac{2\bm{1}^\top AB \bm{1}}{n-2}\bigg).
	\end{align*}
\end{lemma}
This result is known in the literature, see e.g. in the proof of \cite[Lemma 2.1]{yao2018testing}. We provide a complete proof for the convenience of the reader.
\begin{proof}
	Recall $A\equiv \{A_{ij}\}\equiv\{\pnorm{X_i-X_j}{}\}$ and $B\equiv \{B_{ij}\|\equiv \pnorm{Y_i-Y_j}{}$. Note that $\sum_{k\neq \ell}A_{k\ell}^*B_{k\ell}^* = \tr(A^*B^*) - \sum_{k=1}^n (A^*)_{kk}(B^*)_{kk}$. The second term can be computed as follows: As
	\begin{align*}
	A_{kk}^* 
	&= -\frac{2}{n-2}e_k^\top A\bm{1} + \frac{1}{(n-1)(n-2)}\bm{1}^\top A\bm{1},
	\end{align*}
	we have
	\begin{align*}
	&\sum_{k=1}^n A_{kk}^*B_{kk}^* = \sum_{k=1}^n\Big(-\frac{2}{n-2}e_k^\top A\bm{1} + \frac{1}{(n-1)(n-2)}\bm{1}^\top A\bm{1}\Big)\\
	&\qquad\qquad \times \Big(-\frac{2}{n-2}e_k^\top B\bm{1} + \frac{1}{(n-1)(n-2)}\bm{1}^\top B\bm{1}\Big)\\
	&= \frac{4}{(n-2)^2}\bm{1}^\top AB\bm{1} + \Big(-\frac{4}{(n-1)(n-2)^2} + \frac{n}{(n-1)^2(n-2)^2}\Big)\cdot \bm{1}^\top A\bm{1}\bm{1}^\top B\bm{1}\\
	&=  \frac{4}{(n-2)^2}\bm{1}^\top AB\bm{1} - \frac{3n-4}{(n-1)^2(n-2)^2}\bm{1}^\top A\bm{1}\bm{1}^\top B\bm{1}.
	\end{align*}
	For the first term, using the definition of $A^\ast,B^\ast$,
	\begin{align*}
	\tr(A^*B^*) &= \tr\Big[\Big(A - \frac{1}{n-2}\bm{1}\bm{1}^\top A - \frac{1}{n-2}A\bm{1}\bm{1}^\top + \frac{\bm{1}^\top A\bm{1}}{(n-1)(n-2)}\bm{1}\bm{1}^\top\Big)\\
	&\quad\quad\times \Big(B - \frac{1}{n-2}\bm{1}\bm{1}^\top B - \frac{1}{n-2}B\bm{1}\bm{1}^\top + \frac{\bm{1}^\top B\bm{1}}{(n-1)(n-2)}\bm{1}\bm{1}^\top\Big)\Big]\\
	&= \tr(AB) + \frac{8-2n}{(n-2)^2}\bm{1}^\top AB\bm{1} + \frac{n^2-6n +6}{(n-1)^2(n-2)^2}\bm{1}^\top A\bm{1}\bm{1}^\top B\bm{1}.
	\end{align*}
	Combining the two identities above to conclude.
\end{proof}
\begin{proof}[Proof of Proposition \ref{prop:hoef_decomp}]
	By the proceeding lemma,
	\begin{align*}
	&\dcov_\ast^2(\bm{X},\bm{Y}) = \frac{1}{n(n-3)}\bigg(\tr(AB)+ \frac{\bm{1}^\top A \bm{1} \bm{1}^\top B \bm{1} }{(n-1)(n-2)}-\frac{2\bm{1}^\top AB \bm{1}}{n-2}\bigg)\\
	&= \frac{1}{n(n-3)}\bigg(\sum_{i\neq j}A_{ij}B_{ij} + \frac{1}{(n-1)(n-2)}\sum_{i\neq j,k\neq \ell}A_{ij}B_{k\ell} - \frac{2}{n-2}\sum_{i\neq j,i\neq k}A_{ij}B_{ik}\bigg)\\
	&\stackrel{(*)}{=} \frac{1}{n(n-3)}\bigg(\frac{n-3}{n-1}\sum_{i\neq j}A_{ij}B_{ij} \\
	&\qquad\qquad + \frac{1}{(n-1)(n-2)}\sum_{i\neq j\neq k\neq \ell}A_{ij}B_{k\ell} - \frac{2(n-3)}{(n-1)(n-2)}\sum_{i\neq j\neq k}A_{ij}B_{ik}\bigg)\\
	& = \frac{1}{4!\binom{n}{4}}\sum_{i\neq j\neq k\neq \ell} \big(A_{ij}B_{ij}+A_{ij}B_{k\ell}-2 A_{ij}B_{ik}\big) = \frac{1}{\binom{n}{4}}\sum_{i_1<\cdots<i_4}k_0\big(Z_{i_1},Z_{i_2},Z_{i_3},Z_{i_4}\big),
	\end{align*}
	Here $(*)$ follows from the identities:
	\begin{align*}
	\sum_{i\neq j,i\neq k}A_{ij}B_{ik} &= \sum_{i\neq j\neq k}A_{ij}B_{ik} + \sum_{i\neq j}A_{ij}B_{ij},\\
	\sum_{i\neq j,k\neq \ell}A_{ij}B_{k\ell} &= \sum_{i\neq j\neq k\neq \ell}A_{ij}B_{k\ell} + 4\sum_{i\neq j\neq k}A_{ij}B_{ik} + 2\sum_{i\neq j}A_{ij}B_{ij},
	\end{align*}
	and $k_0$ is the symmetrized kernel
	\begin{align*}
	&k_0(Z_1,Z_2,Z_3,Z_4)\\
	& = \frac{1}{4!}\sum_{(i_1,\ldots,i_4)\in \sigma(1,2,3,4)} \pnorm{X_{i_1}-X_{i_2}}{}\big(\pnorm{Y_{i_1}-Y_{i_2}}{}+\pnorm{Y_{i_3}-Y_{i_4}}{}-2\pnorm{Y_{i_1}-Y_{i_3}}{}\big)\\
	& \equiv \E_{\bm{i} \in \sigma_4} (i_1,i_2)_X\big[(i_1,i_2)_Y+(i_3,i_4)_Y-2(i_1,i_3)_Y\big]\\
	&=\E_{\bm{i} \in \sigma_4} (i_1,i_2)_X\big[(i_1,i_2)_Y+(i_3,i_4)_Y-2(i_2,i_4)_Y\big] \quad(1\leftrightarrow 2, 3\leftrightarrow 4),
	\end{align*}
	where we abbreviate $(i,j)_X=(X_i,X_j),(i,j)_Y=(Y_i,Y_j)$ and the expectation $\E_{\bm{i} \in \sigma_4}$ is taken over uniform distribution of permutation group over $\{1,2,3,4\}$. So by averaging the above two equations, we have
	\begin{align*}
	k_0(\bm{Z})& \stackrel{(a)}{=} \E_{\bm{i} \in \sigma_4} (i_1,i_2)_X\big[(i_1,i_2)_Y+(i_3,i_4)_Y-(i_1,i_3)_Y-(i_2,i_4)_Y\big]\\
	& \stackrel{(b)}{=} \E_{\bm{i} \in \sigma_4} (i_3,i_4)_X\big[(i_1,i_2)_Y+(i_3,i_4)_Y-(i_1,i_3)_Y-(i_2,i_4)_Y\big] \quad (a: 1\leftrightarrow 3, 2\leftrightarrow 4)\\
	& \stackrel{(c)}{=}  \E_{\bm{i} \in \sigma_4} (i_1,i_3)_X\big[(i_1,i_3)_Y+(i_2,i_4)_Y-(i_1,i_2)_Y-(i_3,i_4)_Y\big] \quad (a: 2 \leftrightarrow 3)\\
	& =  \E_{\bm{i} \in \sigma_4} (i_2,i_4)_X\big[(i_1,i_3)_Y+(i_2,i_4)_Y-(i_1,i_2)_Y-(i_3,i_4)_Y\big] \quad (c: 1\leftrightarrow 3, 2\leftrightarrow 4).
	\end{align*}
	Using that
	\begin{align*}
	&(i_1,i_2)_X+(i_3,i_4)_X-(i_1,i_3)_X-(i_2,i_4)_X\\
	& = U(X_{i_1},X_{i_2})+U(X_{i_3},X_{i_4})-U(X_{i_1},X_{i_3})-U(X_{i_2},X_{i_4}),
	\end{align*}
	and a similar equation for $(\cdot,\cdot)_Y$ replacing $U$ by $V$, we have
		\begin{align*}
		&k_0(\bm{Z}) = \frac{1}{4\cdot 4!}\sum_{(i_1,\ldots,i_4)\in \sigma(1,2,3,4)}\Big(U(X_{i_1},X_{i_2})+U(X_{i_3},X_{i_4})-U(X_{i_1},X_{i_3})-U(X_{i_2},X_{i_4})\Big)\\
		&\qquad\qquad\qquad \times \Big(V(Y_{i_1},Y_{i_2})+V(Y_{i_3},Y_{i_4})-V(Y_{i_1},Y_{i_3})-V(Y_{i_2},Y_{i_4})\Big)\\
		& = \frac{1}{4\cdot 4!}\sum_{(i_1,\ldots,i_4)\in \sigma(1,2,3,4)} \bigg[U(X_{i_1},X_{i_2})\Big(V(Y_{i_1},Y_{i_2})+V(Y_{i_3},Y_{i_4})-V(Y_{i_1},Y_{i_3})-V(Y_{i_2},Y_{i_4})\Big)\\
		&\qquad + U(X_{i_3},X_{i_4})\Big(V(Y_{i_1},Y_{i_2})+V(Y_{i_3},Y_{i_4})-V(Y_{i_1},Y_{i_3})-V(Y_{i_2},Y_{i_4})\Big)\\
		&\qquad - U(X_{i_1},X_{i_3})\Big(V(Y_{i_1},Y_{i_2})+V(Y_{i_3},Y_{i_4})-V(Y_{i_1},Y_{i_3})-V(Y_{i_2},Y_{i_4})\Big)\\
		&\qquad - U(X_{i_2},X_{i_4})\Big(V(Y_{i_1},Y_{i_2})+V(Y_{i_3},Y_{i_4})-V(Y_{i_1},Y_{i_3})-V(Y_{i_2},Y_{i_4})\Big)\bigg]\\
		& = \frac{1}{4!}\sum_{(i_1,\ldots,i_4)\in \sigma(1,2,3,4)} \bigg[ U(X_{i_1},X_{i_2})V(Y_{i_1},Y_{i_2})+U(X_{i_1},X_{i_2})V(Y_{i_3},Y_{i_4})\\
		&\qquad\qquad- 2 U(X_{i_1},X_{i_2})V(Y_{i_1},Y_{i_3})\bigg]=k(\bm{Z}).
		\end{align*}
	The proof is complete.
\end{proof}

\section{Proofs for Section \ref{section:preliminaries}}\label{appendix:proof_prel}

\subsection{Proof of Lemma \ref{lem:GH_hadamard}}
\begin{proof}[Proof of Lemma \ref{lem:GH_hadamard}]
We first prove (1). Note that under the spectral condition $ \pnorm{\Sigma_X}{F}^2\pnorm{\Sigma_{Y}}{F}^2\asymp_M pq$, so it suffices to show that all three terms on the left side is bounded by the order of $p\vee q$. For the first term, we have
\begin{align*}
\tr(G_{[12]}\circ G_{[12]}) = \sum_{i=1}^{p+q} (\Sigma^{1/2}e_1\Sigma_{XY}e_2^\top\Sigma^{1/2})_{ii}^2 \leq (p+q)\pnorm{\Sigma^{1/2}e_1\Sigma_{XY}e_2^\top\Sigma^{1/2}}{\op}^2 \lesssim_M p\vee q.
\end{align*}
The second term can be controlled similarly. For the third term, using $\max_{i,j}|(H_{[11]})_{i,j}|\vee |(H_{[22]})_{i,j}|\lesssim_M 1$,
\begin{align*}
\pnorm{H_{[11]}\circ H_{[22]}}{F}^2 \lesssim_M \pnorm{H_{[11]}}{F}^2 \wedge \pnorm{H_{[22]}}{F}^2 \lesssim_M p\wedge q.
\end{align*}

Next we prove (2). Let $\Delta \equiv \Sigma^{1/2} - \Sigma_0^{1/2}$, where $\Sigma_0 \equiv [\Sigma_X,0;0,\Sigma_Y]$. Then the proof of Lemma \ref{lem:cross_XY_var} shows that $\pnorm{\Delta}{F}\lesssim_M \pnorm{\Sigma_{XY}}{F}$. Expanding the definitions of $G_{[11]}$ and $G_{[22]}$, we have
\begin{align*}
\pnorm{G_{[11]}\circ G_{[22]}}{F} &= \pnorm{(\Sigma_0^{1/2}+\Delta)\Sigma_{[11]}(\Sigma_0^{1/2}+\Delta)\circ (\Sigma_0^{1/2}+\Delta)\Sigma_{[22]}(\Sigma_0^{1/2}+\Delta)}{F}\\
&\leq \sum_{D_1,\ldots,D_4\in\{\Sigma_0^{1/2}, \Delta\}} \pnorm{D_1\Sigma_{[11]}D_2\circ D_3\Sigma_{[22]}D_4}{F}.
\end{align*}
The claim now follows from the inequality $\pnorm{A\circ B}{F} \leq \pnorm{A}{\op}\pnorm{B}{F}$ for (not necessarily symmetric) matrices $A,B$, and by noting that the summand is zero if at most one of $D_1$-$D_4$ is $\Delta$.
\end{proof}

\subsection{Proof of Lemma \ref{lem:property_h_fcn}}

The bound for $h$ follows by considering two regimes $u \in [-1,-1/2]$ and $u>-1/2$ separately. In particular, for $u \in [-1,-1/2]$, $\abs{h(u)}=\abs{\sqrt{1+u}-1-u/2}\lesssim 1$. For $u>-1/2$, $\inf_{s \in [0,1]}(1+su)^{3/2}\gtrsim 1$ so using the second equality of (\ref{def:h}) leads to the bound $\abs{h(u)}\lesssim u^2$. Combining the two cases to conclude $\abs{h(u)}\lesssim u^2$ for $u>-1$. Next, as 
\begin{align*}
h''(u) = -\frac{1}{4(1+u)^{3/2}},\, h^{(3)}(u) = \frac{3}{8(1+u)^{5/2}},\, h^{(4)}(u) = -\frac{15}{16(1+u)^{7/2}},
\end{align*}
by Taylor expansion,
\begin{align*}
h(u)& = -\frac{1}{4}\int_0^u \frac{u-t}{(1+t)^{3/2}}\,\d{t} \\
& = \frac{h''(0)}{2}u^2+u^3 \int_0^1 h^{(3)}(su)\frac{(1-s)^2}{2}\,\d{s} = -\frac{u^2}{8}+ u^3 \int_0^1 \frac{3(1-s)^2}{16(1+su)^{5/2}}\,\d{s}\\
& = \frac{h''(0)}{2}u^2+\frac{h^{(3)}(0)}{6}u^3+ u^4 \int_0^1 h^{(4)}(su)\frac{(1-s)^3}{6}\,\d{s}\\
& = -\frac{u^2}{8}+\frac{u^3}{16}-u^4 \int_0^1 \frac{5(1-s)^3}{32(1+su)^{7/2}}\,\d{s}.
\end{align*}
The bound for $h'$ follows by noting that $h'(u) = 2^{-1}(1-\sqrt{1+u})/\sqrt{1+u}$ and $|1-\sqrt{1+u}| \lesssim |u|$.\qed

\subsection{Proof of Lemma \ref{lem:sqrt_norm_LX}}

The first two equations follow from the definition. For the second two equations, note that by definition of $U(x_1,x_2)$, 
\begin{align*}
U(x_1,x_2)&=\frac{\tau_X}{2}\bigg[\pnorm{x_1-x_2}{}^2/\tau_X^2-\E\big[\pnorm{x_1-X}{}^2/\tau_X^2\big]\\
&\qquad-\E\big[\pnorm{X-x_2}{}^2/\tau_X^2\big]+\E\big[\pnorm{X-X'}{}^2/\tau_X^2\big]+2\bar{R}_X(x_1,x_2)\bigg]\\
& = - \frac{1}{\tau_X}\Big(x_1^\top x_2-\tau_X^2 \bar{R}_X(x_1,x_2)\Big),
\end{align*}
as desired. Similar derivation applies to $V(y_1,y_2)$.\qed

\subsection{Proof of Lemma \ref{lem:moment_R}}

\noindent (1). By Jensen's inequality, it suffices to consider the case where $s = 2^\ell$ for some $\ell\in\mathbb{N}$, which we will prove by induction. For the baseline case $\ell = 1$, it follows by Lemma \ref{lem:gaussian_4moment}-(1) that
\begin{align*}
\E L_X^2 &= \tau_X^{-4} \E \big(\pnorm{X_1-X_2}{}^2 -\tau_X^2\big)^2 = 4\tau_X^{-4}\cdot \Big(\E \big[\frac{1}{2}(Z_{X,1}-Z_{X,2})^\top \Sigma_X (Z_{X,1} - Z_{X,2})\big]^2 - \frac{1}{4}\tau_X^4\Big)\\
&=4\tau_X^{-4}\cdot(2\pnorm{\Sigma_X}{F}^2 + (\tilde{m}_4 - 3)\tr(\Sigma_X\circ \Sigma_X)) \lesssim_{m_4} \tau_X^{-4}\pnorm{\Sigma_X}{F}^2.
\end{align*}
Here $\tilde{m}_4 \equiv \E \big[\big((Z_{1,X})_1 - (Z_{2,X})_1\big)/\sqrt{2}\big]^4 \lesssim m_4$. For general $\ell$, using the Poincar\'{e} inequality of $Z$, we have
\begin{align*}
\tau_X^{2^{\ell+1}} \E L_X^{2^\ell} &=  \E \big(\pnorm{X_1-X_2}{}^2 -\tau_X^2\big)^{2^\ell} \lesssim_\ell \E \big(Z_X^\top\Sigma_XZ_X -\tr(\Sigma_X)\big)^{2^\ell}\\
&= \Big(\E \big(Z_X^\top\Sigma_XZ_X -\tr(\Sigma_X)\big)^{2^{\ell-1}}\Big)^2 + \var\Big(\big(Z_X^\top\Sigma_XZ_X -\tr(\Sigma_X)\big)^{2^{\ell-1}}\Big)\\
&\leq \Big(\E \big(Z_X^\top\Sigma_XZ_X -\tr(\Sigma_X)\big)^{2^{\ell-1}}\Big)^2 + c_* \E \Big[2^{\ell-1}\big(Z_X^\top\Sigma_XZ_X -\tr(\Sigma_X)\big)^{2^{\ell-1}-1}\cdot 2\pnorm{\Sigma_XZ_X}{}\Big]^2\\
&\leq \Big(\E \big(Z_X^\top\Sigma_XZ_X -\tr(\Sigma_X)\big)^{2^{\ell-1}}\Big)^2 + \frac{1}{2}\cdot \E\Big[\big(Z_X^\top\Sigma_XZ_X -\tr(\Sigma_X)\big)^{2^{\ell}}\Big] + K(\ell, c_*)\E \pnorm{\Sigma_XZ_X}{}^{2^\ell}.
\end{align*}
Here $c_*$ is the Poincar\'{e} constant of $Z$. Rearranging the terms yields that
\begin{align*}
&\E\Big[\big(Z_X^\top\Sigma_XZ_X -\tr(\Sigma_X)\big)^{2^{\ell}}\Big]\\
&\lesssim  \Big(\E \big(Z_X^\top\Sigma_XZ_X -\tr(\Sigma_X)\big)^{2^{\ell-1}}\Big)^2 + K(\ell, c_*)\E \pnorm{\Sigma_XZ_X}{}^{2^\ell}\\
&\lesssim \Big(\E \big(Z_X^\top\Sigma_XZ_X -\tr(\Sigma_X)\big)^{2^{\ell-1}}\Big)^2 + K'(\ell,c_*,M)\E \big(Z_X^\top \Sigma_X Z_X\big)^{2^{\ell-1}}\\
&\lesssim \Big(\E \big(Z_X^\top\Sigma_XZ_X -\tr(\Sigma_X)\big)^{2^{\ell-1}}\Big)^2 + K'(\ell,c_*,M)\Big[\E \big(Z_X^\top \Sigma_X Z_X - \tr(\Sigma)\big)^{2^{\ell-1}} + \tr(\Sigma_X)^{2^{\ell-1}}\Big]\\
&\lesssim K'(\ell,c_*,M)\pnorm{\Sigma_X}{F}^{2^\ell},
\end{align*}
where the last inequality follows from the induction hypotheses. This concludes that $\E L_X^{2^\ell} \lesssim_{\ell,M,c_*} \tau_X^{-2^{\ell+1}}\pnorm{\Sigma_X}{F}^{2^\ell}$.

\noindent (2). As $\abs{h(u)}\lesssim u^2$, we have
\begin{align*}
\E R_X^s(X_1,X_2) = \E h^s\big(L(X_1,X_2)\big) \lesssim \E L_X^{2s}(X_1,X_2) \lesssim_s \tau_X^{-4s}\pnorm{\Sigma_X}{F}^{2s}.
\end{align*}
The bound for $R_Y$ is similar.

\noindent (3). As $\abs{h'(u)}\lesssim \abs{u}/(1+u)^{1/2}$, we have
\begin{align}\label{ineq:chi_inv_moment}
\notag&\E h'(L_X(X_1,X_2))^s \lesssim \E \bigg[\frac{ \abs{L_X(X_1,X_2)}^s}{(1+L_X(X_1,X_2))^{s/2}}\bigg]\\
&\leq \E^{1/2} L_X^{2s}(X_1,X_2)\cdot \E^{1/2} \big(1+L_X(X_1,X_2)\big)^{-s}\lesssim_{M,s} \tau_X^{-2s}\pnorm{\Sigma_{X}}{F}^s.
\end{align}
To calculate the last term, note that
\begin{align}\label{ineq:z_inverse}
\notag\E \pnorm{Z_X}{}^{-2s} &= \E \pnorm{Z_X}{}^{-2s}\bm{1}_{|\pnorm{Z_X}{} - \E\pnorm{Z_X}{}| \leq p^{1/4}} +  \E \pnorm{Z_X}{}^{-2s}\bm{1}_{|\pnorm{Z_X}{} - \E\pnorm{Z_X}{}| \geq p^{1/4}}\\
\notag&\lesssim (\sqrt{p} + \mathcal{O}(p^{1/4}))^{-2s} + \E^{1/2}\pnorm{Z_X}{}^{-4s}\cdot \Prob^{1/2}(|\pnorm{Z_X}{} - \E\pnorm{Z_X}{}| \geq p^{1/4})\\
&\stackrel{(*)}{\lesssim_s} p^{-s} + \epsilon^{-4s}C(\alpha)p^{-\alpha/4} \asymp p^{-s},
\end{align}
by choosing $\alpha \geq 4s$. Here $(*)$ follows from the exponential tail for 1-Lipschitz functionals of distributions satisfying a Poincar\'{e} inequality \cite[Corollary 4.2]{bobkov2009weighted}, and the following calculation: for any $\epsilon > 0$,
\begin{align*}
\E\pnorm{Z_X}{}^{-4s} &= \E\pnorm{Z_X}{}^{-4s}\bm{1}_{\pnorm{Z_X}{}\geq \epsilon} + \E\pnorm{Z_X}{}^{-4s}\bm{1}_{\pnorm{Z_X}{}< \epsilon}\\
&\leq \epsilon^{-4s} + \int_{\pnorm{x}{}\leq \epsilon} \pnorm{x}{}^{-4s} \prod_{i=1}^p f(x_i) \prod_{i=1}^p\d x_i\\
&\lesssim  \epsilon^{-4s} + \big(2\pi\cdot\sup_{x\in(-\epsilon,\epsilon)} f(x)\big)^p \int_{0}^\epsilon r^{p-1-4s} \d r\\
&= \epsilon^{-4s} + \big(2\pi\cdot \sup_{x\in(-\epsilon,\epsilon)} f(x)\big)^p\cdot \frac{\epsilon^{p-4s}}{p-4s}.
\end{align*}
Hence by choosing $\epsilon$ to be a small but fixed constant (depending only on the distribution of $Z_1$) such that $2\pi\epsilon\cdot \sup_{x\in(-\epsilon,\epsilon)} f(x) < 1$, we have $\E\pnorm{Z_X}{}^{-4s} \lesssim \epsilon^{-4s}$. Now using the fact that $\tau_X \asymp_M \sqrt{p}$, (\ref{ineq:z_inverse}) leads to $\E h'(L_X(X_1,X_2))^s \lesssim \tau_X^{-s}$.
\qed

\section{Proofs for Section \ref{section:res_mean}}\label{appendix:proof_res}

\subsection{Proof of Lemma \ref{lem:psi_decomp}}

We only prove the formula for $\psi_X$. Note that
\begin{align*}
\psi_X(X_1,Y_1) &= \E_{X_2,Y_2}\Big[\big(R_X(X_1,X_2)-\E_X R_X(X_1,X)-\E_X R_X(X,X_2)+ (\E R_X)\big)Y_2^\top Y_1\Big]\\
& \stackrel{(\ast)}{=} \E_{X_2,Y_2}\Big[R_X(X_1,X_2)Y_2^\top Y_1\Big] = \E_{X_2,Y_2}\Big[h(L_X(X_1,X_2))Y_2^\top Y_1\Big].
\end{align*}
Here in $(*)$ we use the fact that $\E_{X_1,X_2,Y_2}R_X(X_1,X_2)(Y_2^\top Y_1) = 0$  for fixed $Y_1$, because $R_X(X_1,X_2)(Y_2^\top Y_1)$ has equal distribution as $R_X(-X_1,-X_2)(-Y_2^\top Y_1)$ for any fixed $Y_1$. On the other hand, by Lemma \ref{lem:property_h_fcn} we have
\begin{align*}
&\E_{X_2,Y_2}\Big[h(L_X(X_1,X_2))Y_2^\top Y_1\Big] = -\frac{\E_{X_2,Y_2} \big(L_X^2(X_1,X_2)Y_2^\top Y_1\big)}{8} +\E_{X_2,Y_2}\big[h_3(L_X(X_1,X_2))Y_2^\top Y_1\big].
\end{align*}
To compute the first term, use the representation $X_1 = \Sigma_X^{1/2} \tilde{Z}_{X},Y_1 = \Sigma_{Y\setminus X}^{1/2} \tilde{Z}_{Y}+\Sigma_{YX}\Sigma_X^{-1/2} \tilde{Z}_{X}$ in Lemma \ref{lem:rep}, where $\Sigma_{Y\setminus X}\equiv \Sigma_Y-\Sigma_{YX}\Sigma_X^{-1}\Sigma_{XY}$ and $\tilde{Z} = (\tilde{Z}_X^\top, \tilde{Z}_Y^\top)^\top = OZ$ for some orthogonal matrix $O$ described therein. Then we have
\begin{align*}
&\tau_X^4\cdot \E_{X_2,Y_2} \big(L_X^2(X_1,X_2)Y_2^\top Y_1\big)\\
\notag & = \E_{Z_X,Z_Y} \bigg[\Big(\pnorm{X_1}{}^2+\pnorm{\Sigma_X^{1/2} \tilde{Z}_{X}}{}^2-2X_1^\top \Sigma_X^{1/2} \tilde{Z}_{X}-\tau_X^2\Big)^2 \tilde{Z}_{X}^\top \Sigma_X^{-1/2} \Sigma_{XY} Y_1\bigg]\\
&\quad\quad + \E_{Z_X,Z_Y} \bigg[\Big(\pnorm{X_1}{}^2+\pnorm{\Sigma_X^{1/2} \tilde{Z}_{X}}{}^2-2X_1^\top \Sigma_X^{1/2} \tilde{Z}_{X}-\tau_X^2\Big)^2 \tilde{Z}_{Y}^\top \Sigma_{Y\backslash X}^{1/2} Y_1\bigg]\equiv (I) + (II).
\end{align*}
Using symmetry of $Z$ and Lemma \ref{lem:gaussian_4moment}-(1), we have
\begin{align*}
(I)& = \E_{Z_X,Z_Y} \bigg[-4 \pnorm{X_1}{}^2 X_1^\top \Sigma_X^{1/2} \tilde{Z}_{X} \tilde{Z}_{X}^\top \Sigma_X^{-1/2} \Sigma_{XY} Y_1\\
\notag&\qquad -4\pnorm{\Sigma_X^{1/2} \tilde{Z}_{X}}{}^2 X_1^\top \Sigma_X^{1/2} \tilde{Z}_{X} \tilde{Z}_{X}^\top \Sigma_X^{-1/2} \Sigma_{XY} Y_1 + 4 \tau_X^2 X_1^\top \Sigma_X^{1/2} \tilde{Z}_{X} \tilde{Z}_{X}^\top \Sigma_X^{-1/2} \Sigma_{XY} Y_1 \bigg]\\
\notag& = -4 \pnorm{X_1}{}^2X_1^\top \Sigma_{XY} Y_1-4\Big(\tr(\Sigma_X)\tr(X_1Y_1^\top \Sigma_{YX})+2\tr(X_1Y_1^\top \Sigma_{YX}\Sigma_X) + (m_4 - 3)\tr(H_{[11]}\circ Q_{X,1})\Big)\\
&\qquad + 4\tau_X^2 X_1^\top \Sigma_{XY} Y_1\\
& = -4\Big[\big(\pnorm{X_1}{}^2-\tr(\Sigma_X)\big)X_1^\top \Sigma_{XY} Y_1 +2 X_1^\top \Sigma_X \Sigma_{XY} Y_1 + (m_4 - 3)\tr(H_{[11]}\circ Q_{X,1})\Big],
\end{align*}
where 
\begin{align}\label{def:AB1}
Q_{X,1} = \Sigma^{-1/2}
\begin{pmatrix}
\Sigma_X X_1Y_1^\top \Sigma_{YX} & \Sigma_XX_1Y_1^\top\Sigma_{YX}\Sigma_X^{-1}\Sigma_{XY}\\
\Sigma_{YX}X_1Y_1^\top\Sigma_{YX} & \Sigma_{YX}X_1Y_1^\top \Sigma_{YX}\Sigma_X^{-1}\Sigma_{XY}
\end{pmatrix}\Sigma^{-1/2}.
\end{align}
Similarly, we have
\begin{align*}
(II) &= \E_{Z_X,Z_Y} \bigg[-4 \pnorm{X_1}{}^2 X_1^\top \Sigma_X^{1/2} \tilde{Z}_{X} \tilde{Z}_{Y}^\top \Sigma_{Y\backslash X}^{1/2} Y_1\\
\notag&\qquad -4\pnorm{\Sigma_X^{1/2} \tilde{Z}_{X}}{}^2 X_1^\top \Sigma_X^{1/2} \tilde{Z}_{X} \tilde{Z}_{Y}^\top \Sigma_{Y\backslash X}^{1/2} Y_1 + 4 \tau_X^2 X_1^\top \Sigma_X^{1/2} \tilde{Z}_{X}\tilde{Z}_{Y}^\top \Sigma_{Y\backslash X}^{1/2} Y_1 \bigg]\\
&= -4(m_4 - 3)\tr(H_{[11]}\circ Q_{X,2}),
\end{align*}
where
\begin{align}\label{def:AB2}
Q_{X,2} = \Sigma^{-1/2}
\begin{pmatrix}
0 & \Sigma_XX_1Y_1^\top\Sigma_{Y\backslash X}\\
0 & \Sigma_{YX}X_1Y_1^\top \Sigma_{Y\backslash X}
\end{pmatrix}\Sigma^{-1/2}.
\end{align}
The claim follows by noting that $Q_{X,1} + Q_{X,2} = Q_X$. The collection of identities follow from direct calculation.\qed

\begin{lemma}\label{lem:rep}
Suppose $(X^\top,Y^\top)^\top \stackrel{d}{=} \Sigma^{1/2}Z$, where $\Sigma$ is a p.s.d. matrix in $\R^{(p+q)\times (p+q)}$ and $Z = (Z_X^\top, Z_Y^\top)^\top\in \R^{p+q}$ has i.i.d. components with mean $0$ and variance $1$. Then there exists some orthogonal matrix $O\in\R^{(p+q)\times(p+q)}$ such that with $\tilde{Z}\equiv(\tilde{Z}_X^\top, \tilde{Z}_Y^\top)^\top \equiv OZ$, we have $(X^\top,Y^\top)^\top \stackrel{d}{=} ((\Sigma_X^{1/2}\tilde{Z}_X)^\top,(\Sigma_{YX}\Sigma_X^{-1/2}\tilde{Z}_X + \Sigma_{Y\backslash X}^{1/2}\tilde{Z}_Y)^\top)^\top$.
\end{lemma}
\begin{proof}[Proof of Lemma \ref{lem:rep}]
By taking $O = [\Sigma_X^{1/2}, \Sigma_X^{-1/2}\Sigma_{XY}; 0, \Sigma_{Y\backslash X}^{1/2}]\Sigma^{-1/2}$ which satisfies $OO^\top = O^\top O = I_{p+q}$, we have $(X^\top,Y^\top)^\top \stackrel{d}{=} \Sigma^{1/2}O^\top OZ$, where $\Sigma^{1/2}O^\top =  [\Sigma_X^{1/2}, 0; \Sigma_{YX}\Sigma_X^{-1/2}, \Sigma_{Y\backslash X}^{1/2}]$, as desired. 
\end{proof}

%

\subsection{Proof of Lemma \ref{lem:cross_XY_var}}
Let $\mathcal{M}^d$ be the set of $d\times d$ matrices and $\mathcal{M}_+^d$ the set of p.s.d. matrices of size $d\times d$. We need the following matrix derivative formulae. 
\begin{lemma}\label{lem:mat_sqrt_int_formula}
	Let $I\subset \R$. 
	\begin{enumerate}
		\item Suppose that $S: I\to \mathcal{M}_d$ is a smooth map. Then for any integer $m \in \N$, and $t \in \mathrm{int}(I)$,
		\begin{align}\label{ineq:mat_poly_int_formula}
		\frac{\d}{\d t} S(t)^m = \sum_{k=0}^{m-1} S(t)^k \frac{\d}{\d t} S(t) S(t)^{m-1-k}.
		\end{align}
		\item Suppose that $S: I\to \mathcal{M}_d$ is a smooth map. Then for any $t \in \mathrm{int}(I)$,
		\begin{align}\label{ineq:mat_exp_int_formula}
		\frac{\d}{\d t} e^{S(t)} = \int_0^1 e^{u S(t)}\frac{\d}{\d t} S(t) e^{(1-u) S(t)}\,\d{u}.
		\end{align}
		\item Suppose that $S: I \to \mathcal{M}^d_+$ is a smooth map. Then for any $t \in \mathrm{int}(I)$ such that $S(t)$ is p.d.,
		\begin{align}\label{ineq:mat_sqrt_int_formula}
		\frac{\d}{\d{t}} S^{1/2}(t) = \int_0^\infty e^{-u S^{1/2}(t)} \frac{\d}{\d{t}} S(t) e^{-u S^{1/2}(t)}\,\d{u}.
		\end{align}
	\end{enumerate}
\end{lemma}
These formulae are well-known. We provide a self-contained proof below for the convenience of the reader.
\begin{proof}
	\noindent (1). We prove the claim by induction on $m$. The base case $m=1$ is trivial, so we suppose the claim holds up to $m-1$. Now differentiating on both sides of the identity $S(t)^m = S(t)^{m-1}S(t)$ and using the induction hypothesis, we have
	\begin{align*}
	\frac{\d}{\d t} S(t)^m & = \frac{\d}{\d t} S(t)^{m-1} S(t)+S(t)^{m-1} \frac{\d}{\d t} S(t)\\
	& = \bigg[\sum_{k=0}^{m-2} S(t)^k \frac{\d}{\d t} S(t) S(t)^{m-2-k}\bigg] S(t)+S(t)^{m-1} \frac{\d}{\d t} S(t) = \hbox{RHS of (\ref{ineq:mat_poly_int_formula})}.
	\end{align*}

	\noindent (2). Using the expansion $e^A = \sum_{k= 0}^\infty A^k/k!$,
	\begin{align*}
	\hbox{RHS of (\ref{ineq:mat_exp_int_formula})}&= \int_0^1 \sum_{k,\ell} \frac{1}{k!\ell!} \big(uS(t)\big)^k \frac{\d}{\d t} S(t) \big((1-u) S(t)\big)^\ell\,\d{u}\\
	& \stackrel{(\ast)}{=} \sum_{k,\ell} \frac{1}{(k+\ell+1)!} S(t)^k \frac{\d}{\d t} S(t) S(t)^{\ell}\\
	& = \sum_{m=0}^\infty \frac{1}{m!} \sum_{k=0}^{m-1} S(t)^k \frac{\d}{\d t} S(t) S(t)^{m-1-k} \quad \hbox{(rearraging the summation)}\\
	& = \sum_{m=0}^\infty \frac{1}{m!}\frac{\d}{\d t} S(t)^m  = \hbox{LHS of (\ref{ineq:mat_exp_int_formula})}.
	\end{align*}
	Here in the last line we used (1), whereas in $(\ast)$ we used the simple fact that $\int_0^1 u^k (1-u)^\ell\,\d{u} = k!\ell!/(k+\ell+1)!$ that can be proved by induction on $\ell$:
	\begin{align*}
	\int_0^1 u^k (1-u)^\ell\,\d{u}& = \int_0^1 u^k (1-u)^{\ell-1} (1-u)\,\d{u}\\
	& = \int_0^1 u^k (1-u)^{\ell-1}\,\d{u}-\int_0^1 u^{k+1}(1-u)^{\ell-1}\,\d{u}\\
	& = \frac{k!(\ell-1)!}{(k+\ell)!}-\frac{(k+1)!(\ell-1)!}{(k+\ell+1)!} =\frac{k!\ell!}{(k+\ell+1)!}.
	\end{align*}

	\noindent (3). The derivative $D(t)\equiv \frac{\d}{\d{t}} S^{1/2}(t)$ solves the Sylvester equation
	\begin{align}\label{ineq:mat_sqrt_int_formula_1}
	D(t) S^{1/2}(t)+S^{1/2}(t) D(t) = \frac{\d}{\d{t}} S(t).
	\end{align}
	We will verify $D(t)$ given by the right hand side of (\ref{ineq:mat_sqrt_int_formula}) satisfies (\ref{ineq:mat_sqrt_int_formula_1}). As $S(t)$ is p.d., the integral in (\ref{ineq:mat_sqrt_int_formula}) converges. Now we compute
	\begin{align*}
	&\bigg(\int_0^\infty e^{-u S^{1/2}(t)} \frac{\d}{\d{t}} S(t) e^{-u S^{1/2}(t)}\,\d{u}\bigg) S^{1/2}(t)\\
	&\qquad\qquad\qquad + S^{1/2}(t)\bigg(\int_0^\infty e^{-u S^{1/2}(t)} \frac{\d}{\d{t}} S(t) e^{-u S^{1/2}(t)}\,\d{u}\bigg)\\
	& = -\int_0^\infty \frac{\d}{\d{u}}\bigg[e^{-u S^{1/2}(t)} \frac{\d}{\d{t}} S(t) e^{-u S^{1/2}(t)}\bigg]\,\d{u} = \frac{\d}{\d{t}} S(t),
	\end{align*}  
	as desired.
\end{proof}

\begin{proof}[Proof of Lemma \ref{lem:cross_XY_var}: compact spectrum case]
	We may represent $(X_\ell^\top,Y_\ell^\top)^\top = \Sigma^{1/2}Z_\ell$ for $\ell \in \N$. Let $\bar{Z}_{k,\ell}\equiv Z_k-Z_\ell$, where $k,\ell \in \N$, and 
	\begin{align}\label{def:Sigma_interpolate}
	\Sigma_0 \equiv 
	\begin{pmatrix}
	\Sigma_X & 0\\
	0 & \Sigma_Y
	\end{pmatrix}
	, \quad \Sigma_t \equiv \Sigma_0+ t
	\begin{pmatrix}
	0 & \Sigma_{XY}\\
	\Sigma_{YX} & 0
	\end{pmatrix}=\Sigma_0+t(\Sigma_{[12]}+\Sigma_{[21]}),
	\end{align}
	and $H_{[ij]}(t) \equiv \Sigma_t^{1/2}I_{[ij]}\Sigma_t^{1/2}$. Then
	\begin{align*}
	\psi_{\mathfrak{h}_X,\mathfrak{h}_Y}(t\Sigma_{XY})\equiv \E \bigg[\mathfrak{h}_X\bigg(\frac{\bar{Z}_{k,\ell}^\top H_{[11]}(t) \bar{Z}_{k,\ell}}{\tau_X^2}-1\bigg) \mathfrak{h}_Y\bigg(\frac{\bar{Z}_{k',\ell'}^\top H_{[22]}(t) \bar{Z}_{k',\ell'}}{\tau_Y^2}-1\bigg)\bigg\lvert Z_1\bigg].
	\end{align*}
	By direct calculation, we have
		\begin{align*}
		&\frac{\d}{\d t} \psi_{\mathfrak{h}_X,\mathfrak{h}_Y}(t\Sigma_{XY})\\
		& = \E \bigg[\mathfrak{h}_X'\bigg(\frac{\bar{Z}_{k,\ell}^\top H_{[11]}(t) \bar{Z}_{k,\ell}}{\tau_X^2}-1\bigg) \mathfrak{h}_Y\bigg(\frac{\bar{Z}_{k',\ell'}^\top H_{[22]}(t) \bar{Z}_{k',\ell'}}{\tau_Y^2}-1\bigg)\tau_X^{-2} \bar{Z}_{k,\ell}  \frac{\d}{\d t}H_{[11]}(t) \bar{Z}_{k,\ell}\bigg\lvert Z_1\bigg]\\
		&\qquad + \E \bigg[\mathfrak{h}_X\bigg(\frac{\bar{Z}_{k,\ell}^\top H_{[11]}(t) \bar{Z}_{k,\ell}}{\tau_X^2}-1\bigg) \mathfrak{h}_Y'\bigg(\frac{\bar{Z}_{k',\ell'}^\top H_{[22]}(t) \bar{Z}_{k',\ell'}}{\tau_Y^2}-1\bigg)\tau_Y^{-2} \bar{Z}_{k',\ell'} \frac{\d}{\d t}H_{[11]}(t) \bar{Z}_{k',\ell'}\bigg\lvert Z_1\bigg]\\
		&\equiv B_1(t)+B_2(t).
		\end{align*}
	Now we compute the term $\d H_{[11]}(t)/\d t$ in $B_1$. Using Lemma \ref{lem:mat_sqrt_int_formula}-(3) and the fact that $(\d \Sigma_t/\d t)_{k,\ell} = \Sigma_{[12]}+\Sigma_{[21]}$, we have
	\begin{align*}
	\frac{\d}{\d t}\Sigma_t^{1/2} 
	&= \int_0^\infty e^{-u\Sigma_t^{1/2}}\cdot (\Sigma_{[12]} + \Sigma_{[21]})\cdot e^{-u\Sigma_t^{1/2}}\d u.
	\end{align*}
	This entails
	\begin{align}\label{ineq:res_psi_bound_0}
	\notag\frac{\d}{\d t} H_{[11]}(t) & = \frac{\d}{\d t}\big(\Sigma_t^{1/2} I_{[11]} \Sigma_t^{1/2}\big) = \frac{\d}{\d t}\Sigma_t^{1/2}\cdot I_{[11]} \Sigma_t^{1/2} + \Sigma_t^{1/2}I_{[11]} \cdot\frac{\d}{\d t}\Sigma_t^{1/2}\\
	&=\int_0^\infty \mathcal{T}\bigg[e^{-u \Sigma_t^{1/2}} (\Sigma_{[12]} + \Sigma_{[21]}) e^{-u \Sigma_t^{1/2}} I_{[11]} \Sigma_t^{1/2}\bigg]\,\d{u}.
	\end{align}
	Here for a square matrix $A$, $\mathcal{T}(A)\equiv A+A^\top$. This implies
	\begin{align*}
	\int_0^1 B_{1}(t)\,\d t&= \tau_X^{-2}\int_0^1   \E\bigg[\mathfrak{h}_X'\bigg(\frac{\bar{Z}_{k,\ell}^\top H_{[11]}(t) \bar{Z}_{k,\ell}}{\tau_X^2}-1\bigg) \mathfrak{h}_Y\bigg(\frac{\bar{Z}_{k',\ell'}^\top H_{[22]}(t) \bar{Z}_{k',\ell'}}{\tau_Y^2}-1\bigg)\\
	&\qquad \times \mathcal{T}\bigg(\int_0^\infty \bar{Z}_{k,\ell}^\top e^{-u \Sigma_t^{1/2}} \big(\Sigma_{[12]}+\Sigma_{[21]}\big) e^{-u \Sigma_t^{1/2}} I_{[11]} \Sigma_t^{1/2} \bar{Z}_{k,\ell}\,\d{u}\bigg)\bigg\lvert Z_1 \bigg]\,\d{t}.
	\end{align*}
	To continue calculations, we may write $e^{-u\Sigma_t^{1/2}}$ as follows: by Lemma \ref{lem:mat_sqrt_int_formula}-(2), $A\mapsto e^A$ as a map from $\R^{(p+q)\times (p+q)}$ to $\R^{(p+q)\times (p+q)}$ has Frech\'et derivative $[\nabla (e^\cdot)](A)\cdot H = \int_0^1 e^{vA}He^{(1-v)A}\,\d v$, so a first-order Taylor expansion yields that 
	\begin{align}\label{ineq:res_psi_bound_3}
	\notag &e^{-u\Sigma_t^{1/2}}- e^{-u\Sigma_0^{1/2}}= \int_0^1 \big[\nabla (e^{\cdot})\big]\big(-u\Sigma_0^{1/2}+ s(-u)(\Sigma_t^{1/2} - \Sigma_0^{1/2})\big)\cdot \big[(-u)(\Sigma_t^{1/2} - \Sigma_0^{1/2})\big]\d s\\
	&= (-u) \int_0^1 \int_0^1 e^{-uv[\Sigma_0^{1/2} + s(\Sigma_t^{1/2} - \Sigma_0^{1/2})]}(\Sigma_t^{1/2}-\Sigma_0^{1/2})e^{-u(1-v)[\Sigma_0^{1/2} + s(\Sigma_t^{1/2} - \Sigma_0^{1/2})]}\d v\d s\equiv \mathscr{R}_1(u,t).
	\end{align}
	On the other hand, by Lemma \ref{lem:mat_sqrt_int_formula}-(3), $A\mapsto A^{1/2}$ as a map from p.s.d. matrices in $\R^{(p+q)\times (p+q)}$ to itself has Frech\'et derivative $[\nabla (\cdot)^{1/2}](A)\cdot H = \int_0^\infty e^{-z A^{1/2}} H e^{-z A^{1/2}}\,\d{z}$, so again a first-order Taylor expansion yields that 
	\begin{align}\label{ineq:res_psi_bound_4}
	\Sigma_t^{1/2} - \Sigma_0^{1/2} &= \int_0^1  \big[\nabla(\cdot)^{1/2}\big]\big(\Sigma_0+w(\Sigma_t-\Sigma_0)\big)\cdot (\Sigma_t-\Sigma_0)\,\d{w}\nonumber\\
	& =  \int_0^1  \int_0^\infty e^{-z (\Sigma_0+w(\Sigma_t-\Sigma_0))}(\Sigma_t-\Sigma_0) e^{-z (\Sigma_0+w(\Sigma_t-\Sigma_0))}\,\d{z}\d{w} \equiv \mathscr{R}_2(t).
	\end{align}
	Now using (\ref{ineq:res_psi_bound_3})-(\ref{ineq:res_psi_bound_4}) and that
	\begin{align*}
	&e^{-u \Sigma_0^{1/2}} \big(\Sigma_{[12]}+\Sigma_{[21]}\big) e^{-u \Sigma_0^{1/2}} I_{[11]} \Sigma_0^{1/2} =
	\begin{pmatrix}
	0 & 0\\
	e^{-u\Sigma_Y^{1/2}}\Sigma_{YX} e^{-u\Sigma_X^{1/2}}\Sigma_X^{1/2} & 0
	\end{pmatrix},
	\end{align*}
	we arrive at
	\begin{align*}
	&M(u,t)\equiv  e^{-u \Sigma_t^{1/2}} \big(\Sigma_{[12]}+\Sigma_{[21]}\big) e^{-u \Sigma_t^{1/2}} I_{[11]} \Sigma_t^{1/2}\\
	&= 	\begin{pmatrix}
	0 & 0\\
	e^{-u\Sigma_Y^{1/2}}\Sigma_{YX} e^{-u\Sigma_X^{1/2}}\Sigma_X^{1/2} & 0
	\end{pmatrix} + e^{-u \Sigma_0^{1/2}} \big(\Sigma_{[12]}+\Sigma_{[21]}\big) e^{-u \Sigma_0^{1/2}} I_{[11]} \mathscr{R}_2(t)\\
	&\qquad + \mathscr{R}_1(u,t)\big(\Sigma_{[12]}+\Sigma_{[21]}\big) e^{-u\Sigma_0^{1/2}} I_{[11]} \Sigma_t^{1/2} +e^{-u \Sigma_t^{1/2}} \big(\Sigma_{[12]}+\Sigma_{[21]}\big) \mathscr{R}_1(u,t) I_{[11]} \Sigma_t^{1/2}.
	\end{align*}
	To control the remainder terms, we have 
	\begin{align}\label{ineq:R1R2_bound}
	\pnorm{\mathscr{R}_1(u,t)}{F}\vee \pnorm{\mathscr{R}_2(t)}{F} \lesssim_M  t(u\vee 1) \pnorm{\Sigma_{XY}}{F},
	\end{align}
	where we used the following: as $\inf_{t \in [0,1]} \lambda_{\min}(\Sigma_t)\geq \lambda_{\min}(\Sigma)\wedge \lambda_{\min}(\Sigma_0)\gtrsim_M 1$,
	\begin{itemize}
		\item $\sup_{w\in[0,1]}\int_0^\infty\pnorm{e^{-z (\Sigma_0+w(\Sigma_t-\Sigma_0))}}{\op}^2\d z \lesssim_M 1$.
		\item $\sup_{s,v\in[0,1]}\sup_{u\geq 0}\pnorm{e^{-uv[\Sigma_0^{1/2} + s(\Sigma_t^{1/2} - \Sigma_0^{1/2})]}}{\op}\vee \pnorm{e^{-u(1-v)[\Sigma_0^{1/2} + s(\Sigma_t^{1/2} - \Sigma_0^{1/2})]}}{\op}\lesssim_M 1$.
	\end{itemize}
	Hence by using again $\pnorm{e^{-u\Sigma_0^{1/2}}}{\op}\vee \pnorm{e^{-u\Sigma_t^{1/2}}}{\op} \leq e^{-c_0u}$ for some $c_0 = c_0(M) > 0$, there exist another positive $c_1$ only depending on $M$ such that
	\begin{align*}
	\abs{\tr M(u,t)}&\lesssim_M e^{-c_1 u}\pnorm{\Sigma_{[12]}}{F} \big(\pnorm{\mathscr{R}_1(u,t)}{F}\vee \pnorm{\mathscr{R}_2(t)}{F}\big) \lesssim_M  t(u\vee 1) e^{-c_1 u} \pnorm{\Sigma_{XY}}{F}^2,\\
	\pnorm{M(u,t)}{F}&\lesssim_M e^{-c_1 u}\pnorm{\Sigma_{[12]}}{F} \big(\pnorm{\mathscr{R}_1(u,t)}{\op}\vee \pnorm{\mathscr{R}_2(t)}{\op}\vee 1\big) \lesssim_M  t(u\vee1) e^{-c_1 u} \pnorm{\Sigma_{XY}}{F}(\pnorm{\Sigma_{XY}}{F}\vee 1),
	\end{align*}
	which further leads to, with $\bar{M}(t)\equiv \int_0^\infty M(u,t)\d u$,  
	\begin{align*}
	\big|\tr\big(\bar{M}(t)\big)\big| &\leq \int_0^\infty |\tr(M(u,t))| \lesssim_M t\pnorm{\Sigma_{XY}}{F}^2,\,	\bigpnorm{\bar{M}(t)}{F} \leq \int_0^\infty  \pnorm{M(u,t)}{F}\d u \lesssim_M t\pnorm{\Sigma_{XY}}{F}.     
	\end{align*}
	Therefore by Cauchy-Schwarz,
	\begin{align*}
	&\E \Big(\int_0^1 B_{1}(t)\,\d{t}\Big)^2\lesssim \tau_X^{-4}  \E^{1/4} \mathfrak{h}_X'(L_X(X_1,X_2))^8\cdot \E^{1/4} \mathfrak{h}_Y^8(L_Y(Y_1,Y_2))\cdot \sup_{t \in [0,1]}\E^{1/2}\big(Z_1^\top \bar{M}(t) Z_1\big)^4\\
	&\lesssim\tau_X^{-4}  \E^{1/4} \mathfrak{h}_X'(L_X(X_1,X_2))^8\cdot \E^{1/4} \mathfrak{h}_Y^8(L_Y(Y_1,Y_2))\cdot \sup_{t \in [0,1]}\big[|\tr(\bar{M}(t))|\vee\pnorm{\bar{M}(t)}{F}\big]^2\\
	&\lesssim_M  \tau_X^{-4}  \E^{1/4} \mathfrak{h}_X'(L_X(X_1,X_2))^8\cdot \E^{1/4} \mathfrak{h}_Y^8(L_Y(Y_1,Y_2))\cdot \pnorm{\Sigma_{XY}}{F}^2\big(1\vee \pnorm{\Sigma_{XY}}{F}^2\big).
	\end{align*}
	Flip the roles of $X$ and $Y$ to conclude that
	\begin{align*}
	&\E \Big(\int_0^1 B_2(t)\,\d{t}\Big)^2\lesssim_M  \tau_Y^{-4}  \E^{1/4} \mathfrak{h}_Y'(L_Y(Y_1,Y_2))^8\cdot \E^{1/4} \mathfrak{h}_X^8(L_X(X_1,X_2))\cdot \pnorm{\Sigma_{XY}}{F}^2\big(1\vee \pnorm{\Sigma_{XY}}{F}^2\big).
	\end{align*}
	The claim follows by noting that
	\begin{align*}
	&\psi_{\mathfrak{h}_X,\mathfrak{h}_Y}(\Sigma_{XY}) - \psi_{\mathfrak{h}_X,\mathfrak{h}_Y}(0)= \int_0^1\frac{\d}{\d t} \psi_{\mathfrak{h}_X,\mathfrak{h}_Y}(t\Sigma_{XY})\,\d t = \int_0^1 \big[B_1(t) +B_2(t)\big]\,\d t,
	\end{align*}
	as desired.\end{proof}

\begin{proof}[Proof of Lemma \ref{lem:cross_XY_var}: $X=Y$ case]
	The proof is largely similar so we only sketch the key steps. Let $X_\ell\equiv \Sigma_X^{1/2} Z_\ell$ for $\ell \in \N$. Then with $\mathfrak{h}_X,\mathfrak{h}_Y$ denoted as $ \mathfrak{h},\mathfrak{g}$ to avoid confusion, we have
	\begin{align*}
	\psi_{\mathfrak{h},\mathfrak{g}}(t \Sigma_X) =  \E \bigg[\mathfrak{h}\bigg(\frac{\bar{Z}_{k,\ell}^\top (t\Sigma_X) \bar{Z}_{k,\ell}}{\tau_X^2}-1\bigg) \mathfrak{g}\bigg(\frac{\bar{Z}_{k',\ell'}^\top (t\Sigma_X) \bar{Z}_{k',\ell'}}{\tau_X^2}-1\bigg)\bigg\lvert Z_1\bigg].
	\end{align*}
	So
	\begin{align*}
	&\frac{\d}{\d t}\psi_{\mathfrak{h},\mathfrak{g}}(t\Sigma_{X})\\
	& = \E \bigg[\mathfrak{h}'\bigg(\frac{\bar{Z}_{k,\ell}^\top (t\Sigma_X) \bar{Z}_{k,\ell}}{\tau_X^2}-1\bigg) \mathfrak{g}\bigg(\frac{\bar{Z}_{k',\ell'}^\top (t\Sigma_X) \bar{Z}_{k',\ell'}}{\tau_X^2}-1\bigg)\tau_X^{-2} \bar{Z}_{k,\ell}^\top \Sigma_X \bar{Z}_{k,\ell}\bigg\lvert Z_1\bigg]\\
	&\qquad + \E \bigg[\mathfrak{h}\bigg(\frac{\bar{Z}_{k,\ell}^\top (t\Sigma_X) \bar{Z}_{k,\ell}}{\tau_X^2}-1\bigg) \mathfrak{g}'\bigg(\frac{\bar{Z}_{k',\ell'}^\top (t\Sigma_X) \bar{Z}_{k',\ell'}}{\tau_X^2}-1\bigg)\tau_X^{-2} \bar{Z}_{k',\ell'}^\top \Sigma_X \bar{Z}_{k',\ell'}\bigg\lvert Z_1\bigg]\\
	&\equiv B_{1}(t)+B_{2}(t).
	\end{align*}
	Using Cauchy-Schwarz directly to conclude that 
	\begin{align*}
	&\E \Big(\int_0^1 B_{1}(t)\,\d{t}\Big)^2\lesssim \tau_X^{-4}\cdot \E^{1/4} (\mathfrak{h}'\circ L_X)^8\cdot  \E^{1/4} (\mathfrak{g}\circ L_X)^8\cdot \E^{1/2}(Z_1^\top \Sigma_X Z_1)^4.
	\end{align*}
	As $\E^{1/2}(Z_1^\top \Sigma_X Z_1)^4\lesssim_M \tau_X^4\asymp_M \pnorm{\Sigma_X}{F}^4$, the claim follows by treating $B_2$ in a similar way.
\end{proof}

\subsection{Proof of Proposition \ref{prop:res_psi_bound}}

We first prove a few supporting lemmas for Proposition \ref{prop:res_psi_bound}.
\begin{lemma}\label{lem:cross_moment_LX_improved}
	The following hold:
	\begin{enumerate}
		\item We have
		\begin{align*}
		\bigabs{ \E\big[L_X^3(X_1,X_2) Y_1^\top Y_2\big]}  \lesssim_{M,\kappa} \tau_X^{-4}\pnorm{\Sigma_{XY}}{F}^2.
		\end{align*}
		\item For any $k \in \N$ and real number $\ell >0$, there exists some constant $C=C(k,\ell,M)>0$ such that for $p\geq 8\ell+1$, uniformly in $s \in [0,1]$ and $r>0$, 
		\begin{align*}
		\biggabs{ \E \bigg[L^k_X(X_1,X_2)(Y_1^\top Y_2)\cdot\int_0^{s} \frac{(s -t)^r }{\big(1+t L_X(X_1,X_2)\big)^{\ell}}\,\d{t}\bigg]}\leq C\cdot \tau_X^{-k} \pnorm{\Sigma_{XY}}{F}^2.
		\end{align*}
	\end{enumerate}
\end{lemma}

\begin{proof}
	\noindent (1). 
%
	By symmetry, we have
	\begin{align*}
	\tau_X^6\E\big[L_X^3(X_1,X_2) Y_1^\top Y_2\big] &= \E\pnorm{X_1-X_2}{}^6(Y_1^\top Y_2) + (-3\tau_X^2)\cdot \E\pnorm{X_1-X_2}{}^4(Y_1^\top Y_2)\\
	&\quad\quad\quad + (3\tau_X^4)\cdot \E\pnorm{X_1-X_2}{}^2(Y_1^\top Y_2)\\
	&\equiv L_1 - 3\tau_X^2 L_2 + 3\tau_X^4 L_3.
	\end{align*}	
	We first calculate $L_1$. Using again the symmetry of $Z$ and the identity
	\begin{align}\label{eq:X_diff_fourth}
	\pnorm{X_1-X_2}{}^4 &= \pnorm{X_1}{}^4 + \pnorm{X_2}{}^4 + 4(X_1^\top X_2)^2 + 2\pnorm{X_1}{}^2\pnorm{X_2}{}^2- 4\pnorm{X_1}{}^2(X_1^\top X_2) - 4\pnorm{X_2}{}^2(X_1^\top X_2),
	\end{align}
	we have
	\begin{align*}
	L_1 &= \E\Big[(\pnorm{X_1}{}^2 + \pnorm{X_2}{}^2 - 2X_1^\top X_2)\cdot(Y_1^\top Y_2)\\
	&\quad\cdot\Big(\pnorm{X_1}{}^4 + \pnorm{X_2}{}^4 + 4(X_1^\top X_2)^2 + 2\pnorm{X_1}{}^2\pnorm{X_2}{}^2 - 4\pnorm{X_1}{}^2(X_1^\top X_2) - 4\pnorm{X_2}{}^2(X_1^\top X_2)\Big)\Big]\\
	&= -12\E\pnorm{X_1}{}^4(X_1^\top X_2)(Y_1^\top Y_2) - 12\E\pnorm{X_1}{}^2\pnorm{X_2}{}^2(X_1^\top X_2)(Y_1^\top Y_2)- 8\E(X_1^\top X_2)^3(Y_1^\top Y_2).
	\end{align*}
	Now by direct calculation, we have:
	\begin{itemize}
		\item By first taking expectation of $X_2$, we have
		\begin{align*}
		\E\pnorm{X_1}{}^4(X_1^\top X_2)(Y_1^\top Y_2) = \E\pnorm{X_1}{}^4(X_1^\top \Sigma_{XY}Y_1) = \big(\frac{1}{4}\tau_X^4+\mathcal{O}(p)\big)\pnorm{\Sigma_{XY}}{F}^2.
		\end{align*}
		\item By first taking expectation of $X_1$, we have by Lemma \ref{lem:gaussian_4moment}-(1)
		\begin{align*}
		&\E\pnorm{X_1}{}^2\pnorm{X_2}{}^2(X_1^\top X_2)(Y_1^\top Y_2) = \E (X_1^\top X_1)(X_1^\top \pnorm{X_2}{}^2X_2Y_2^\top X_1)\\
		&= \big[\frac{1}{2}\tau_X^2 + \mathcal{O}(1)\big]\cdot \E(X_2^\top Y_2)\pnorm{X_2}{}^2 = \big(\frac{1}{4}\tau_X^4+\mathcal{O}(p)\big)\pnorm{\Sigma_{XY}}{F}^2.
		\end{align*}
		\item By first taking expectation of $X_1$, we have by Lemma \ref{lem:gaussian_4moment}-(1)
		\begin{align*}
		&\E(X_1^\top X_2)^3(Y_1^\top Y_2) = \E(X_1^\top X_2X_2^\top X_2)(X_1^\top X_2 Y_2^\top X_1)\\
		&\lesssim \E(X_2^\top Y_2)(X_2^\top X_2) = \mathcal{O}(p)\cdot\pnorm{\Sigma_{XY}}{F}^2.
		\end{align*}
	\end{itemize}
	Combining the pieces yields that
	\begin{align*}
	L_1 = -6\big(\tau_X^4+\mathcal{O}(p)\big)\pnorm{\Sigma_{XY}}{F}^2.
	\end{align*}
	On the other hand, direct calculation of $L_2$ via (\ref{eq:X_diff_fourth}) and Lemma \ref{lem:hadamard}, and $L_3$ yields that
	\begin{align*}
	L_2 = -4\big(\tau_X^2+\mathcal{O}(1)\big)\pnorm{\Sigma_{XY}}{F}^2, \quad L_3 = -2\pnorm{\Sigma_{XY}}{F}^2.
	\end{align*}
	Putting together the estimates for $L_1$-$L_3$ concludes that $\tau_X^6\E\big[L_X^3(X_1,X_2) Y_1^\top Y_2\big] = \mathcal{O}(p)\pnorm{\Sigma_{XY}}{F}^2 \lesssim \tau_X^2\cdot \pnorm{\Sigma_{XY}}{F}^2$, as desired.

	\noindent (2). We only prove the Gaussian case, where the dependence on $\Sigma_{XY}$ is directly exposed via the representation $(X,Y)\stackrel{d}{=} (\Sigma_X^{1/2}Z_X, \Sigma_{Y\backslash X}^{1/2}Z_Y + \Sigma_{YX}\Sigma_X^{-1}Z_X)$; the general case could be dealt by the interpolation method adopted in Lemmas \ref{lem:cross_XY_var} and \ref{lem:hadamard}.
	
	The left hand side of the desired inequality equals
	\begin{align*}
	&\E\bigg[L_X^k\big(\Sigma_X^{1/2}Z_{1;X}, \Sigma_X^{1/2}Z_{2;X}\big)Z_{1;X}^\top D Z_{2;X}\cdot\int_0^{s} \frac{(s -t)^r }{\big(1+t L_X(\Sigma_X^{1/2}Z_{1;X},\Sigma_X^{1/2}Z_{2;X})\big)^{\ell}}\,\d{t} \bigg],\\
	&\lesssim_M\E^{1/4} L_X^{4k}(X_1,X_2)\cdot \E^{1/2} (Z_{1;X}^\top D Z_{2;X})^2 \cdot \E^{1/4} \bigg[\int_0^{s} \frac{(s -t)^r }{\big(1+t L_X(X_1,X_2)\big)^{\ell}}\,\d{t}\bigg]^4 \\
	&\stackrel{(\ast)}{\lesssim}_{k,\ell,M} \tau_X^{-2k} \pnorm{\Sigma_X}{F}^k \cdot \pnorm{D}{F} \cdot 1 \lesssim_M \tau_X^{-k} \pnorm{\Sigma_{XY}}{F}^2.
	\end{align*}
	Here in $(\ast)$ we used the calculation: as $1+tL_X(X_1,X_2) = (1-t) + t\pnorm{X_1-X_2}{}^2/\tau_X^2$,
	\begin{align*}
	&\E\bigg[\int_0^{s} \frac{(s -t)^r }{\big(1+t L_X(X_1,X_2)\big)^{\ell}}\,\d{t}\bigg]^4 \lesssim \int_0^s\E\big(1+t L_X(X_1,X_2)\big)^{-4\ell}\d t\\
	&\lesssim \int_0^s (1-t)^{-4\ell} \wedge \big[t^{-4\ell}\E(\pnorm{X_1-X_2}{}^2/\tau_X^2)^{-4\ell}\big]\d t \stackrel{(**)}{\lesssim}_\ell\int_0^s \big[(1-t)^{-4\ell} \wedge t^{-4\ell}\big]\d t \lesssim 1,
	\end{align*}
	using in $(**)$ a similar calculation as in (\ref{ineq:chi_inv_moment}) which holds under $p\geq 8\ell+1$ . 
\end{proof}

\begin{lemma}\label{lem:hadamard}
Recall the definitions of $P_X, P_Y, Q_X,Q_Y$ in (\ref{def:PQ}). Suppose the spectrum of $\Sigma$ is contained in $[M^{-1},M]$ for some $M > 1$. Then 
\begin{align*}
\E \tr(P_X \circ  Q_X) \vee \E \tr(P_Y \circ  Q_Y) \lesssim_M \pnorm{\Sigma_{XY}}{F}^2.
\end{align*} 
\end{lemma}
\begin{proof}[Proof of Lemma \ref{lem:hadamard}]
By Lemma \ref{lem:psi_decomp} and symmetry, it suffices to show that $\E \tr(P_X \circ  Q_X) =  \tr(H_{[11]}\circ G_{[12]}) \lesssim_M \pnorm{\Sigma_{XY}}{F}^2$. The high level proof strategy is the same as Lemma \ref{lem:cross_XY_var} and we present the details below. Define the map $F:[0,1]\rightarrow \R$ by
\begin{align*}
F(t): t\mapsto  \tr(H_{[11]}(t)\circ G_{[12]}(t)) \equiv  \tr(\Sigma_t^{1/2} I_{[11]} \Sigma_t^{1/2} \circ \Sigma_t^{1/2} \Sigma_{[12]}(t) \Sigma_t^{1/2}).
\end{align*}
Here $\Sigma_t$ is given in (\ref{def:Sigma_interpolate}), and $\Sigma_{[12]}(t) = [0,t\Sigma_{XY}; 0,0]$. Hence by using $F(0) = 0$, we have
\begin{align*}
 \tr(H_{[11]}\circ G_{[12]}) = F(1) = \int_0^1 F'(t)\d t = \int_0^1 \tr\Big(\frac{\d}{\d t}H_{[11]}(t)\circ G_{[12]}(t) + H_{[11]}(t)\circ \frac{\d }{\d t}G_{[12]}(t)\Big)\d t.
\end{align*}
For the first term, we use (\ref{ineq:res_psi_bound_0}) to deduce that (recall that $\mathcal{T}(A) = A+A^\top$)
\begin{align*}
	\int_0^1 \tr\Big(\frac{\d}{\d t}H_{[11]}(t)\circ G_{[12]}(t)\Big)\d t
	&=\int_0^1 \int_0^\infty \tr\Big[\mathcal{T}\Big(e^{-u \Sigma_t^{1/2}} (\Sigma_{[12]} + \Sigma_{[21]}) e^{-u \Sigma_t^{1/2}} I_{[11]} \Sigma_t^{1/2}\Big)\circ G_{[12]}(t)\Big]\,\d{u}\d t\\
	&=2\int_0^1 \int_0^\infty \tr\Big[\Big(e^{-u \Sigma_t^{1/2}} (\Sigma_{[12]} + \Sigma_{[21]}) e^{-u \Sigma_t^{1/2}} I_{[11]} \Sigma_t^{1/2}\Big)\circ G_{[12]}(t)\Big]\,\d{u}\d t.
\end{align*}
Next, using the definitions in (\ref{ineq:res_psi_bound_3}) and (\ref{ineq:res_psi_bound_4}), we have
\begin{align*}
e^{-u \Sigma_t^{1/2}} = e^{-u \Sigma_0^{1/2}} + \mathscr{R}_1(u,t),\quad \Sigma_t^{1/2} = \Sigma_0^{1/2} + \mathscr{R}_2(t),
\end{align*}
so we may expand the integrand as $\tr\big(M(u,t) + R(u,t)\big)$, where
\begin{align*}
M(u,t) &= \Big(e^{-u \Sigma_0^{1/2}} (\Sigma_{[12]} + \Sigma_{[21]}) e^{-u \Sigma_0^{1/2}} I_{[11]} \Sigma_0^{1/2}\Big)\circ \Sigma_0^{1/2}\Sigma_{[12]}(t)\Sigma_0^{1/2},\\
R(u,t) &= e^{-u\Sigma_t^{1/2}}(\Sigma_{[12]} + \Sigma_{[21]})e^{-u\Sigma_t^{1/2}}I_{[11]}\Sigma_t^{1/2}\circ \Big(\mathscr{R}_2(t)\Sigma_{[12]}(t)\Sigma_t^{1/2} + \Sigma_0^{1/2} \Sigma_{[12]}(t)\mathscr{R}_2(t)\Big)\\
&\quad + \Big(e^{-u\Sigma_0^{1/2}}(\Sigma_{[12]} + \Sigma_{[21]})e^{-u\Sigma_0^{1/2}}I_{[11]}\mathscr{R}_2(t) + e^{-u\Sigma_t^{1/2}}(\Sigma_{[12]} + \Sigma_{[21]})\mathscr{R}_1(u,t)I_{[11]}\Sigma_t^{1/2}\\
&\quad\quad\quad + \mathscr{R}_1(u,t)(\Sigma_{[12]} + \Sigma_{[21]})e^{-u\Sigma_0^{1/2}}I_{[11]}\Sigma_t^{1/2}\Big)\circ \Sigma_0^{1/2}\Sigma_{[12]}(t)\Sigma_0^{1/2}.
\end{align*}
By direct calculation we have $\tr\big(M(u,t)\big) = 0$. Furthermore, using $\tr(A\circ B)\leq \pnorm{A}{F}\pnorm{B}{F}$ and the bound of $\mathscr{R}_1(u,t), \mathscr{R}_2(t)$ in (\ref{ineq:R1R2_bound}), we have $\tr(R(u,t)) \leq t(u\vee 1)e^{-cu}\pnorm{\Sigma_{XY}}{F}^2$ for some $c = c(M) > 0$. Combining the two estimates yields that
\begin{align*}
\int_0^1 \tr\Big(\frac{\d}{\d t}H_{[11]}(t)\circ G_{[12]}(t)\Big)\d t = 2\int_0^1\int_0^\infty \tr(M(u,t) + R(u,t))\d u\d t \lesssim_M \pnorm{\Sigma_{XY}}{F}^2.
\end{align*}

Next, for the second term, we have
\begin{align*}
\frac{\d}{\d t}G_{[12]}(t) &= \int_0^\infty e^{-u\Sigma_t^{1/2}}(\Sigma_{[12]}+\Sigma_{[21]})e^{-u\Sigma_t^{1/2}}\Sigma_{[12]}(t)\Sigma_t^{1/2}\d u\\
&\quad + \int_0^\infty \Sigma_t^{1/2}\Sigma_{[12]}(t)e^{-u\Sigma_t^{1/2}}(\Sigma_{[12]}+\Sigma_{[21]})e^{-u\Sigma_t^{1/2}}\Sigma_{[12]}(t)\Sigma_t^{1/2}\d u\\
&\quad + \Sigma_t^{1/2}\Sigma_{[12]}\Sigma_t^{1/2} \equiv \partial G_{[12],1}(t) + \partial G_{[12],2}(t) + \partial G_{[12],3}(t).
\end{align*}
Using similar calculations as above, we have
\begin{align*}
\Big|\int_0^1 \tr(H_{[11]}(t)\circ \partial G_{[12],1}(t))\d t\Big| \vee \Big|\int_0^1 \tr(H_{[11]}(t)\circ \partial G_{[12],2}(t))\d t\Big|\lesssim_M \pnorm{\Sigma_{XY}}{F}^2.
\end{align*}
For $\partial G_{[12], 3}$, we may use $\Sigma_t^{1/2} = \Sigma_0^{1/2} + \mathscr{R}_2(t)$ to expand the integrand as $\tr(M(t) + R(t))$, where
\begin{align*}
M(t) &= \Sigma_0^{1/2}I_{[11]}\Sigma_0^{1/2}\circ \Sigma_0^{1/2}\Sigma_{[12]}\Sigma_0^{1/2}, \\
R(t) &= \big(\Sigma_t^{1/2}I_{[11]}\mathscr{R}_2(t) + \mathscr{R}_2(t)I_{[11]}\Sigma_0^{1/2}\big)\circ \Sigma_t^{1/2}\Sigma_{[12]}\Sigma_t^{1/2} \\
&\quad\quad+ \Sigma_0^{1/2}I_{[11]}\Sigma_0^{1/2}\circ \big(\Sigma_t^{1/2}\Sigma_{[12]}\mathscr{R}_2(t) + \mathscr{R}_2(t)\Sigma_{[12]}\Sigma_0^{1/2}\big).
\end{align*}
Direct calculation yields that $\tr(M(t)) = 0$ and $\sup_{t\in[0,1]}|\tr(R(t))|\lesssim_M \pnorm{\Sigma_{XY}}{F}^2$. This concludes that
\begin{align*}
\int_0^1 \tr\Big(H_{[11]}(t)\circ \frac{\d}{\d t}G_{[12]}(t)\Big)\d t \lesssim_M \pnorm{\Sigma_{XY}}{F}^2,
\end{align*}
and hence the proof is complete.
\end{proof}

We are now ready for the proof of Proposition \ref{prop:res_psi_bound}. We divide the proof into four parts: first moment bounds for $\psi_X$ and $\psi_Y$, first moment bound for $\psi_{X,Y}$, second moment bounds for $\psi_{X}$ and $\psi_Y$, and second moment bound for $\psi_{X,Y}$.

\begin{proof}[Proof of Proposition \ref{prop:res_psi_bound}-\textbf{Part 1}: first moment bounds for $\psi_X,\psi_Y$]
	By Lemma \ref{lem:psi_decomp}, we have $\E \psi_X(X_1,Y_1) = \E A_{1,X}(X_1,Y_1) + \E A_{2,X}(X_1,Y_1)$ with $A_{1,X},A_{2,X}$ defined therein. For $\E A_{1,X}(X_1,Y_1)$, using
	\begin{align*}
	\E[\pnorm{X_1}{}^2 X_1^\top \Sigma_{XY} Y_1]= \tr(\Sigma_X)\tr(\Sigma_{XY}\Sigma_{YX})+2\tr(\Sigma_{XY}\Sigma_{YX}\Sigma_X) + \kappa \tr(H_{[11]}\circ \E Q_X),
	\end{align*}
	we have
	\begin{align}\label{ineq:res_psi_bound_2}
	\E A_{1,X}(X_1,Y_1) = \tau_X^{-4}\big[2\tr(\Sigma_{XY}\Sigma_{YX}\Sigma_X) + \kappa \tr(H_{[11]}\circ \E Q_X)\big].
	\end{align}
	Here $Q_X$ is given by (\ref{def:PQ}). The claim $|\E A_{1,X}(X_1,Y_1)|\lesssim_{M,\kappa} \tau_X^{-4}\pnorm{\Sigma_{XY}}{F}^2$ now follows from Lemma \ref{lem:hadamard}. Now we handle the `residual' term $\E A_{2,X}(X_1,Y_1)$: By Lemmas \ref{lem:property_h_fcn} and \ref{lem:cross_moment_LX_improved},
	\begin{align*}
	&\abs{\E A_{2,X}(X_1,Y_1)} \lesssim \bigabs{ \E\big[L_X^3(X_1,X_2)Y_1^\top Y_2\big]}\\
	&\qquad +\biggabs{\E \bigg[L_X^{4}(X_1,X_2) Y_1^\top Y_2\cdot \int_0^{1} \frac{(1-s)^3 }{\big(1+sL_X(X_1,X_2)\big)^{7/2}}\,\d{s}\bigg] }\lesssim_M \tau_X^{-4}\pnorm{\Sigma_{XY}}{F}^2.
	\end{align*}
	The claim follows.
\end{proof}

\begin{proof}[Proof of Proposition \ref{prop:res_psi_bound}-\textbf{Part 2}: first moment bounds for $\psi_{X,Y}$]
By direct calculation, we have
\begin{align*}
\E\psi_{X,Y} = \E R_X(X_1,X_2)R_Y(Y_1,Y_2) - 2\E R_X(X_1,X_2)R(Y_1,Y_3) + \E R(X_1,X_2)R(Y_3,Y_4) \equiv \sum_{i=1}^3 S_i.
\end{align*}
Next we use the same interpolation method as in Lemma \ref{lem:cross_XY_var} to extract the dependence on $\Sigma_{XY}$ of $S_1$-$S_3$. For $S_1$, define the map $F_1:[0,1]\rightarrow\R$ by
\begin{align*}
F_1(t): t\mapsto \E_{(X_1,Y_1)\stackrel{d}{=} \Sigma_t^{1/2}Z_1, (X_2,Y_2)\stackrel{d}{=} \Sigma_t^{1/2}Z_2} R_X(X_1,X_2)R_Y(Y_1,Y_2),
\end{align*}
where $\Sigma_t$ is defined in (\ref{def:Sigma_interpolate}). Define $F_2(t)$ and $F_3(t)$ similarly. Then $\E\psi_{X,Y} = \sum_{i=1}^3 F_i(1)$ and by definition $\sum_{i=1}^3 F_i(0) = 0$. Since with $\bar{Z}_{1,2}\equiv Z_1-Z_2$, $H_{[11]}(t)\equiv \Sigma_t^{1/2}I_{[11]} \Sigma_t^{1/2}$, and a similar definition for $H_{[22]}(t)$, 
\begin{align*}
F_1(t) = \E h\Big(\frac{\bar{Z}_{1,2}^\top H_{[11]}(t)\bar{Z}_{1,2}}{\tau_X^2}-1\Big)h\Big(\frac{\bar{Z}_{1,2}^\top H_{[22]}(t)\bar{Z}_{1,2}}{\tau_Y^2}-1\Big),
\end{align*}
we have $\d F_1(t)/\d t = \Delta_{1,1}(t) + \Delta_{1,2}(t)$, where
\begin{align*}
\tau_X^2\Delta_{1,1}(t) &\equiv \E h'\Big(\frac{\bar{Z}_{1,2}^\top H_{[11]}(t)\bar{Z}_{1,2}}{\tau_X^2}-1\Big)\cdot \bar{Z}_{1,2}^\top\frac{\d}{\d t}H_{[11]}(t)\bar{Z}_{1,2}\cdot h\Big(\frac{\bar{Z}_{1,2}^\top H_{[22]}(t)\bar{Z}_{1,2}}{\tau_Y^2}-1\Big),\\
\tau_Y^2\Delta_{1,2}(t) &\equiv \E h\Big(\frac{\bar{Z}_{1,2}^\top H_{[11]}(t)\bar{Z}_{1,2}}{\tau_X^2}-1\Big)\cdot \bar{Z}_{1,2}^\top\frac{\d}{\d t}H_{[22]}(t)\bar{Z}_{1,2}\cdot h'\Big(\frac{\bar{Z}_{1,2}^\top H_{[22]}(t)\bar{Z}_{1,2}}{\tau_Y^2}-1\Big).
\end{align*}
Now using the formula of $\d H_{[11]}(t)/\d t$ in (\ref{ineq:res_psi_bound_0}), we have the decomposition $\tau_X^2\Delta_{1,1}(t) = M_{1,1} + R_{1,1}(t)$, where
\begin{align*}
M_{1,1} &= \E h'\Big(\frac{\bar{Z}_{1,2}^\top H_{[11]}(0)\bar{Z}_{1,2}}{\tau_X^2}-1\Big)\cdot \bar{Z}_{1,2}^\top\int_0^\infty \mathcal{T}\Big(e^{-u\Sigma_0^{1/2}}(\Sigma_{[12]} + \Sigma_{[21]})e^{-u\Sigma_0^{1/2}}I_{[11]}\Sigma_0^{1/2}\Big)\d u\bar{Z}_{1,2}\\
&\qquad\cdot h\Big(\frac{\bar{Z}_{1,2}^\top H_{[22]}(0)\bar{Z}_{1,2}}{\tau_Y^2}-1\Big)
\end{align*} 
and $R_{1,1}(t)$ consists of terms of the forms
\begin{align*}
R_{1,1,1}(t) &\equiv \E \Big[h'\Big(\frac{\bar{Z}_{1,2}^\top H_{[11]}(t)\bar{Z}_{1,2}}{\tau_X^2}-1\Big) - h'\Big(\frac{\bar{Z}_{1,2}^\top H_{[11]}(0)\bar{Z}_{1,2}}{\tau_X^2}-1\Big)\Big]\\
&\cdot \bar{Z}_{1,2}^\top\int_0^\infty \mathcal{T}\Big(e^{-u\Sigma_0^{1/2}}(\Sigma_{[12]} + \Sigma_{[21]})e^{-u\Sigma_0^{1/2}}I_{[11]}\Sigma_0^{1/2}\Big)\d u\bar{Z}_{1,2}\cdot h\Big(\frac{\bar{Z}_{1,2}^\top H_{[22]}(0)\bar{Z}_{1,2}}{\tau_Y^2}-1\Big),\\
R_{1,1,2}(t) &\equiv \E h'\Big(\frac{\bar{Z}_{1,2}^\top H_{[11]}(0)\bar{Z}_{1,2}}{\tau_X^2}-1\Big)\cdot \bar{Z}_{1,2}^\top\int_0^\infty \mathcal{T}\Big(e^{-u\Sigma_0^{1/2}}(\Sigma_{[12]} + \Sigma_{[21]})e^{-u\Sigma_0^{1/2}}I_{[11]}\big(\Sigma_t^{1/2} - \Sigma_0^{1/2}\big)\Big)\d u\bar{Z}_{1,2}\\
&\qquad\cdot h\Big(\frac{\bar{Z}_{1,2}^\top H_{[22]}(0)\bar{Z}_{1,2}}{\tau_Y^2}-1\Big),\\
R_{1,1,3}(t)&\equiv \E h'\Big(\frac{\bar{Z}_{1,2}^\top H_{[11]}(0)\bar{Z}_{1,2}}{\tau_X^2}-1\Big)\cdot \bar{Z}_{1,2}^\top\int_0^\infty \mathcal{T}\Big(e^{-u\Sigma_0^{1/2}}(\Sigma_{[12]} + \Sigma_{[21]})e^{-u\Sigma_0^{1/2}}I_{[11]}\Sigma_0^{1/2}\Big)\d u\bar{Z}_{1,2}\\
&\qquad\cdot \Big[h\Big(\frac{\bar{Z}_{1,2}^\top H_{[22]}(t)\bar{Z}_{1,2}}{\tau_Y^2}-1\Big) - h\Big(\frac{\bar{Z}_{1,2}^\top H_{[22]}(0)\bar{Z}_{1,2}}{\tau_Y^2}-1\Big)\Big].
\end{align*}
Using Cauchy-Schwarz, a similar calculation as in (\ref{ineq:chi_inv_moment}), and the estimates
\begin{align*}
\pnorm{H_{[11]}(t) - H_{[11]}(0)}{F}\vee \pnorm{\Sigma_t^{1/2}-\Sigma_0^{1/2}}{F}\lesssim_M \pnorm{\Sigma_{XY}}{F}^2, \quad \pnorm{e^{-u\Sigma_t^{1/2}} - e^{-u\Sigma_0^{1/2}}}{F} \lesssim_M e^{-c_0u}\pnorm{\Sigma_{XY}}{F},
\end{align*}
we arrive at
\begin{align*}
\sup_{t\in[0,1]} |R_{1,1,1}(t)|\vee |R_{1,1,2}(t)|\vee |R_{1,1,3}(t)| \lesssim_M \frac{\pnorm{\Sigma_{XY}}{F}^2}{\tau_X\tau_Y^2}.
\end{align*}
Define $R_{i,j,k}$, $i\leq 3,j\leq 2,k\leq 3$ and $M_{i,j}$, $i\leq 3, j\leq 2$ similarly. Flipping the roles of $X,Y$ then yields that
\begin{align*}
\Big|\sum_{i=1}^3\frac{\d F_1(t)}{\d t}\Big| \lesssim_M \Big|\sum_{i=1}^3\sum_{j=1}^2 M_{i,j}\Big| + \frac{\pnorm{\Sigma_{XY}}{F}^2}{\tau_X^2\tau_Y^2(\tau_X\wedge \tau_Y)}.
\end{align*}
The claim now follows by noting that $\sum_{i=1}^3\sum_{j=1}^2 M_{i,j} = 0$.
\end{proof}

\begin{proof}[Proof of Proposition \ref{prop:res_psi_bound}-\textbf{Part 3}: second moment bounds for $\psi_{X},\psi_Y$]
	We write $A_{1,X},A_{2,X}$ as $A_1,A_2$ in the proof. We only need to bound
	\begin{align*}
	\E\psi_{X}^2(X_1,Y_1)\lesssim \big[\E A_{1}^2(X_1,Y_1) + \E A_{2}^2(X_1,Y_1)\big].
	\end{align*}
	Next we bound the above two summands separately.
	
	To bound $\var\big(A_1(X_1,Y_1)\big)$, we write
	\begin{align*}
	A_{1}(X_1,Y_1) \stackrel{d}{=} A_1(Z)= (2\tau_X^4)^{-1}\Big[\big(Z^\top H_{[11]} Z-\tr(H_{[11]})\big)Z^\top \bar{G}_{[1,2]} Z +2 Z^\top \bar{G}_X Z + \kappa (P_X\circ Q_X)\Big].
	\end{align*}
	Here $\bar{G}_{[1,2]} = (G_{[12]} + G_{[21]})/2$ and
	\begin{align}\label{def:bar_G_X}
	\bar{G}_X\equiv \frac{1}{2} \Sigma^{1/2}\begin{pmatrix}
	0 & \Sigma_X\Sigma_{XY}\\
	\Sigma_{YX}\Sigma_X & 0
	\end{pmatrix} \Sigma^{1/2}.
	\end{align}
	Let $(2\tau_X^4)A_1(X_1,Y_1)\equiv A_{1,1}(X_1,Y_1) + A_{1,2}(X_1,Y_1)$, with $A_{1,2}(X_1,Y_1) = \kappa (P_X\circ Q_X)$. For $A_{1,1}(X_1,Y_1)$, by Poincar\'e inequality and Lemma \ref{lem:property_GH}, we have
	\begin{align*}
	\var\big(A_{1,1}(X_1,Y_1)\big)&\lesssim \E(Z^\top \bar{G}_{[1,2]} Z)^2 Z^\top H_{[11]}^2 Z+\E\big(Z^\top H_{[11]} Z-\tr(H_{[11]})\big)^2 Z^\top \bar{G}_{[1,2]}^2 Z +\E Z^\top \bar{G}_{X}^2 Z\\
	&\lesssim \big(\tr^2(\bar{G}_{[1,2]})+ \pnorm{\bar{G}_{[1,2]}}{F}^2\big) \pnorm{H_{[11]}}{F}^2+  \pnorm{\bar{G}_X}{F}^2\\
	&\lesssim_M  \pnorm{\Sigma_X}{F}^2\big(\pnorm{\Sigma_{XY}}{F}^4+\pnorm{\Sigma_{XY}\Sigma_{YX}}{F}^2\big)+ \pnorm{\Sigma_{XY}\Sigma_{YX}}{F}^2\lesssim_M \tau_X^2 \pnorm{\Sigma_{XY}}{F}^4.
	\end{align*}
	For $A_{1,2}(X_1,Y_1)$, note that $Q_X(X_1,Y_1) = Q_X(Z) = H_{[11]}ZZ^\top H_{[22]}$. This implies
	\begin{align*}
	\nabla_Z \tr(P_X\circ Q_X(Z)) = \sum_{i=1}^{p+q} (H_{[11]})_{ii} \big[H_{[11]}e_i(e_i^\top H_{[22]}Z) + H_{[22]}e_i(e_i^\top H_{[11]}Z)\big] \equiv D_1(Z) + D_2(Z). 
	\end{align*}
	Then by the Poincar\'e inequality, we have $\var (A_{1,2}(X_1,Y_1)) \leq c_* \big(\E \pnorm{D_1(Z)}{}^2 + \E \pnorm{D_2(Z)}{}^2\big)$. By direct calculation, we have
	\begin{align*}
	\E \pnorm{D_1(Z)}{}^2 &= \E \sum_{i,j=1}^{p+q} (H_{[11]})_{ii}(H_{[11]})_{jj}e_i^\top H_{[11]}^2 e_j \cdot e_i^\top H_{[22]}ZZ^\top H_{[22]} e_j\\
	&= \sum_{i,j=1}^{p+q} (H_{[11]})_{ii}(H_{[11]})_{jj} (H_{[11]}^2\circ H_{[22]}^2)_{ij}\\
	&\leq \Big(\sum_{i,j=1}^{p+q} (H_{[11]})^2_{ii}(H_{[11]})^2_{jj}\Big)^{1/2} \cdot \Big(\sum_{i,j=1}^{p+q} (H_{[11]}^2\circ H_{[22]}^2)_{ij}^2\Big)^{1/2}\\
	&= \pnorm{H_{[11]}}{F}^2\cdot \pnorm{G_{[11]}\circ G_{[22]}}{F}.
	\end{align*}
	By Lemmas \ref{lem:property_GH} and \ref{lem:GH_hadamard}, we have $\E \pnorm{D_1(Z)}{}^2 \lesssim_M \pnorm{\tau_X}{}^2\pnorm{\Sigma_{XY}}{F}^2$. Hence with a similar bound for $\E \pnorm{D_1(Z)}{}^2$, we have $\var (A_{1,2}(X_1,Y_1)) \lesssim_M \tau_X^2\pnorm{\Sigma_{XY}}{F}^2$.
	This together with (\ref{ineq:res_psi_bound_2}) yields that
	\begin{align}\label{ineq:res_psi_bound_5}
	\E A_1^2(X_1,Y_1) &= \var\big(A_1(X_1,Y_1)\big)+\big(\E A_1(X_1,Y_1)\big)^2 \lesssim_M \tau_X^{-6} \pnorm{\Sigma_{XY}}{F}^2(\pnorm{\Sigma_{XY}}{F}^2\vee 1).
	\end{align}
	Next we handle $A_2$. Similar to the proof of Lemma \ref{lem:cross_moment_LX_improved}-(2), we only show below the proof for the Gaussian case. By definition of $h_3$ in Lemma \ref{lem:property_h_fcn} and the parametrization $Y_2 = \Sigma_{Y\backslash X}^{1/2}Z_{2;Y} + \Sigma_{YX}\Sigma_X^{-1/2}Z_{2;X}$,
	\begin{align*}
	&\E A_2^2(X_1,Y_1) \lesssim \E_{X_1,Y_1} \bigg\{\E_{X_2,Y_2}\bigg[L_X^3(X_1,X_2)Y_1^\top Y_2\int_0^1 \frac{(1-s)^2}{(1+sL_X(X_1,X_2))^{5/2}}\,\d{s}\bigg]\bigg\}^2\\
	&\lesssim \E^{1/4} L_X^{24}(X_1,X_2)\cdot \E^{1/2} (Z_{2;X}^\top \Sigma_X^{-1/2}\Sigma_{XY} Y_1)^4\cdot  \E^{1/4}\bigg(\int_0^1 \frac{(1-s)^2}{(1+sL_X(X_1,X_2))^{5/2}}\,\d{s}\bigg)^8.
	\end{align*}
	The three terms above can be bounded as follows:
	\begin{itemize}
		\item Lemma \ref{lem:moment_R} yields that $\E^{1/4} L_X^{24}(X_1,X_2)\lesssim_M \tau_X^{-6}$.
		\item With $L\equiv \Sigma_X^{-1/2}\Sigma_{XY} \Sigma_Y^{1/2}$, we have
		\begin{align*}
		\E^{1/2} \big(Z_{2;X}^\top\Sigma_X^{-1/2}\Sigma_{XY} Y_1\big)^4 = \E^{1/2}\big(Z_X^\top L Z_Y\big)^4\lesssim \pnorm{L}{F}^2 \lesssim_M \pnorm{\Sigma_{XY}}{F}^2.
		\end{align*}
		\item By a similar calculation as in (\ref{ineq:chi_inv_moment}), for $p$ larger than an absolute constant, the third term is at most of constant order.
	\end{itemize}	
	Combining the estimates yields that 
	\begin{align}\label{ineq:res_psi_bound_6}
	\E A_2^2(X_1,Y_1)&\lesssim_M \tau_X^{-6} \pnorm{\Sigma_{XY}}{F}^2.
	\end{align}
	Combine  (\ref{ineq:res_psi_bound_5})-(\ref{ineq:res_psi_bound_6}) to conclude the estimate for $\E \psi_X^2(X_1,Y_1)$. A similar argument applies to $\E \psi_Y^2(X_1,Y_1)$.
\end{proof}

\begin{proof}[Proof of Proposition \ref{prop:res_psi_bound}-\textbf{Part 4}: second moment bound for $\psi_{X,Y}$]
	Using the notation in the proof of Lemma \ref{lem:cross_XY_var}, and viewing $\psi_{X,Y}$ as a function of $\Sigma_{XY}$, for any $t \in [0,1]$,
	\begin{align*}
	&\psi_{X,Y}(t\Sigma_{XY}) 
	= \E_{Z_2}\bigg[\mathcal{D} h\bigg(\frac{\bar{Z}_{1,2}^\top H_{[11]}(t) \bar{Z}_{1,2}}{\tau_X^2}-1\bigg)\mathcal{D} h\bigg(\frac{\bar{Z}_{1,2}^\top H_{[22]}(t) \bar{Z}_{1,2}}{\tau_Y^2}-1\bigg)\bigg],
	\end{align*}
	where 
	\begin{align*}
	\mathcal{D} g(Z_1,Z_2) \equiv g(Z_1,Z_2)-\E_{Z_1} g(Z_1,Z_2)-\E_{Z_2}g(Z_1,Z_2)+\E_{Z_1,Z_2}g(Z_1,Z_2).
	\end{align*}
	Expanding the product results in 16 terms of form
	\begin{align*}
	B_{(k,\ell),(k',\ell')}(t\Sigma_{XY})\equiv \E \bigg[h\bigg(\frac{\bar{Z}_{k,\ell}^\top H_{[11]}(t) \bar{Z}_{k,\ell}}{\tau_X^2}-1\bigg) h\bigg(\frac{\bar{Z}_{k',\ell'}^\top H_{[22]}(t) \bar{Z}_{k',\ell'}}{\tau_Y^2}-1\bigg)\bigg\lvert Z_1\bigg],
	\end{align*}
	where $k\neq \ell, k'\neq \ell'$. Using Lemma \ref{lem:cross_XY_var} and the moment estimates in Lemma \ref{lem:moment_R}, we have
	\begin{align*}
	\E \big(\psi_{X,Y}(\Sigma_{X,Y})-\psi_{X,Y}(0)\big)^2& \lesssim \sup_{k,k',\ell,\ell'} \E \big(B_{(k,\ell),(k',\ell')}(\Sigma_{XY})-B_{(k,\ell),(k',\ell')}(0)\big)\\
	&\lesssim_M \tau_X^{-4}\tau_Y^{-4} \pnorm{\Sigma_{XY}}{F}^2(1\vee \pnorm{\Sigma_{XY}}{F}^2) (\tau_X\wedge \tau_Y)^{-2}.
	\end{align*}
	The claim follows as $\psi_{X,Y}(0)=0$.
\end{proof}

\section{Proofs for Section \ref{section:hoeffding_decomp}}\label{appendix:proof_hoeff}

\subsection{Proof of Lemma \ref{lem:h1}}

As $\E_{Z_2,Z_3,Z_4}U(X_{i_1},X_{i_2})V(Y_{i_3},Y_{i_4})=0$, Proposition \ref{prop:hoef_decomp} yields that
\begin{align*}
&k_1(z_1) =\frac{1}{4!}\sum_{(i_1,\ldots,i_4)\in \sigma(1,2,3,4)} \bigg[ \E_{Z_2,Z_3,Z_4}U(X_{i_1},X_{i_2})V(Y_{i_1},Y_{i_2})\\
&\qquad\qquad -2 \E_{Z_2,Z_3,Z_4} U(X_{i_1},X_{i_2})V(Y_{i_1},Y_{i_3})\bigg].
\end{align*}
The claim now follows by calculating
\begin{align*}
&\sum_{(i_1,\ldots,i_4)\in \sigma(1,2,3,4)} \E_{Z_2,Z_3,Z_4}\big(U(X_{i_1},X_{i_2})V(Y_{i_1},Y_{i_2})\big) \\
&\qquad  = 12 \big(\E U(x_1, X) V(y_1,Y)+\E U(X,X') V(Y,Y')\big),
\end{align*}
and
\begin{align*}
\E_{Z_2,Z_3,Z_4} U(X_{i_1},X_{i_2})V(Y_{i_1},Y_{i_3})=0,
\end{align*}
using the double-centered property of $U,V$.\qed

\subsection{Proof of Lemma \ref{lem:g1}}

Recall the identities involving $U,V$ in (\ref{eqn:UV}). Then we have
\begin{align*}
&g_1(x_1,y_1) \equiv  \frac{1}{2}\bigg[\E U(x_1,X_2) V(y_1,Y_2)-\dcov^2(X,Y)\bigg]\\
& = \frac{1}{2\tau_X\tau_Y}\bigg[\big(x_1^\top \Sigma_{XY} y_1-\pnorm{\Sigma_{XY}}{F}^2\big)+\big(R_1(x_1,y_1)-\E R_1(X_1,Y_1)\big)\bigg],
\end{align*}
where, with $\psi_X,\psi_Y,\psi_{XY}$ defined in (\ref{def:psi}),
\begin{align*}
R_1(x_1,y_1) \equiv -\tau_X^2 \psi_X (x_1,y_1)-\tau_Y^2\psi_Y(x_1,y_1)+\tau_X^2\tau_Y^2 \psi_{X,Y}(x_1,y_1).
\end{align*}
By Lemma \ref{lem:psi_decomp}, we may write [recall the definitions of $\bar{\psi}_X$ and $\bar{\psi}_Y$ in (\ref{def:bar_psi})]
\begin{align*}
\psi_X(x_1,y_1)
&= \frac{1}{2\tau_X^4}\big(\pnorm{x_1}{}^2-\tr(\Sigma_X)\big)\pnorm{\Sigma_{XY}}{F}^2 + \bar{\psi}_X(x_1,y_1),\\
\psi_Y(x_1,y_1)&= \frac{1}{2\tau_Y^4}\big(\pnorm{y_1}{}^2-\tr(\Sigma_Y)\big)\pnorm{\Sigma_{XY}}{F}^2 + \bar{\psi}_Y(x_1,y_1).
\end{align*}
The claim follows.\qed

\subsection{Proof of Lemma \ref{lem:bar_g1_second_moment}}
Note that
\begin{align*}
&\E \bar{g}_1^2(X,Y) = \frac{1}{4\tau_X^2\tau_Y^2}\Big[\E(X^\top\Sigma_{XY}Y - \pnorm{\Sigma_{XY}}{F}^2)^2 + \E  \mathscr{A}_{1,X}(X,Y)^2 + \E  \mathscr{A}_{1,Y}(X,Y)^2\\
&\qquad +2\E(X^\top\Sigma_{XY}Y) \mathscr{A}_{1,X}(X,Y) + 2\E(X^\top\Sigma_{XY}Y) \mathscr{A}_{1,Y}(X,Y)+ 2\E  \mathscr{A}_{1,X}(X,Y) \mathscr{A}_{1,Y}(X,Y)\Big].
\end{align*}
The terms can be calculated as follows:
\begin{itemize}
	\item By Lemma \ref{lem:gaussian_4moment}-(1), we have
	\begin{align*}
	\E(X^\top\Sigma_{XY}Y - \pnorm{\Sigma_{XY}}{F}^2)^2 = \pnorm{\Sigma_{XY}\Sigma_{YX}}{F}^2 + \tr(\Sigma_{XY}\Sigma_{Y}\Sigma_{YX}\Sigma_X) + \kappa\tr(G_{[12]} \circ G_{[12]}).
	\end{align*}
	\item By Lemma \ref{lem:gaussian_4moment}-(1), we have
	\begin{align*}
	\E  \mathscr{A}_{1,X}(X,Y)^2 &= \frac{\pnorm{\Sigma_{XY}}{F}^4}{4\tau_X^4}\cdot \big[2\pnorm{\Sigma_{X}}{F}^2 + \kappa\tr(H_{[11]}\circ H_{[11]})\big]\\
	&= \frac{\pnorm{\Sigma_{XY}}{F}^4\pnorm{\Sigma_X}{F}^2}{2\tau_X^4} + \frac{\kappa}{4\tau_X^4}\pnorm{\Sigma_{XY}}{F}^4\tr(H_{[11]}\circ H_{[11]}),
	\end{align*}
	and similarly $\E  \mathscr{A}_{1,Y}(X,Y)^2 = \pnorm{\Sigma_{XY}}{F}^4\pnorm{\Sigma_Y}{F}^2/(2\tau_Y^4) + (4\tau_Y^4)^{-1}\kappa\pnorm{\Sigma_{XY}}{F}^4\tr(H_{[22]}\circ H_{[22]})$.
	\item We have
	\begin{align*}
	\E (X^\top\Sigma_{XY}Y) \mathscr{A}_{1,X}(X,Y) &= -\frac{\pnorm{\Sigma_{XY}}{F}^2}{2\tau_X^2}\E\Big[(X^\top\Sigma_{XY} Y)(\pnorm{X}{}^2-\tr(\Sigma_X))\Big] \\
	&= -\frac{\pnorm{\Sigma_{XY}}{F}^2}{\tau_X^2}\tr(\Sigma_{XY}\Sigma_{YX}\Sigma_X) - \frac{\kappa\pnorm{\Sigma_{XY}}{F}^2}{2\tau_X^2}\tr(H_{[11]}\circ G_{[12]}),
	\end{align*}
	and similarly 
	\begin{align*}
	\E (X^\top\Sigma_{XY}Y) \mathscr{A}_{1,Y}(X,Y) = -\frac{\pnorm{\Sigma_{XY}}{F}^2}{\tau_Y^2}\tr(\Sigma_{YX}\Sigma_{XY}\Sigma_Y)- \frac{\kappa\pnorm{\Sigma_{XY}}{F}^2}{2\tau_Y^2}\tr(H_{[22]}\circ G_{[12]}).
	\end{align*}
	\item We have
	\begin{align*}
	\E  \mathscr{A}_{1,X}(X,Y) \mathscr{A}_{1,Y}(X,Y) &= \frac{\pnorm{\Sigma_{XY}}{F}^4}{4\tau_X^2\tau_Y^2}\E\Big[(\pnorm{X}{}^2-\tr(\Sigma_X))(\pnorm{Y}{}^2-\tr(\Sigma_Y))\Big]\\
	&= \frac{\pnorm{\Sigma_{XY}}{F}^4}{4\tau_X^2\tau_Y^2}\cdot \big[2\pnorm{\Sigma_{XY}}{F}^2 + \kappa\tr(H_{[11]}\circ H_{[22]})\big] \\
	&= \frac{\pnorm{\Sigma_{XY}}{F}^6}{2\tau_X^2\tau_Y^2} + \frac{\kappa\pnorm{\Sigma_{XY}}{F}^4}{4\tau_X^2\tau_Y^2}\tr(H_{[11]}\circ H_{[22]}).
	\end{align*}
\end{itemize}
The proof is complete.\qed

\subsection{Proof of Lemma \ref{lem:first_variance_lower}}
\noindent (1). With $(X^\top, Y^\top)^\top\stackrel{d}{=} \Sigma^{1/2}Z$, we may write
\begin{align*}
\bar{g}_1(X,Y)\equald \bar{g}_1(Z)\equiv \frac{1}{4\tau_X\tau_Y}\cdot \Big[Z^\top AZ - \E(Z^\top AZ)\Big],
\end{align*}
where the symmetric matrix $A$ is defined by
\begin{align*}
A\equiv G_{[12]} + G_{[21]} - \frac{\pnorm{\Sigma_{XY}}{F}^2}{\tau_X^2}H_{[11]} - \frac{\pnorm{\Sigma_{XY}}{F}^2}{\tau_Y^2}H_{[22]}.
\end{align*}
So by Lemma \ref{lem:gaussian_4moment}-(1), we have
\begin{align*}
&\E \bar{g}_1^2(X,Y)= \frac{1}{16\tau_X^2\tau_Y^2}\cdot \big(2\pnorm{A}{F}^2 + \kappa \tr(A\circ A)\big).
\end{align*}
Since $Z$ has a Lebesgue density, we have $\kappa = \E Z_1^4 - 3 = \var(Z_1^2) - 2 > c_0 - 2$ for some $c_0 > 0$ that only depends on the distribution of $Z_1$. Hence using $\tr(A\circ A)\leq \pnorm{A}{F}^2$, we have $\E \bar{g}_1^2(X,Y) \geq c_0 \pnorm{A}{F}^2/(16\tau_X^2\tau_Y^2)$, which can be bounded by
\begin{align*}
\pnorm{A}{F}^2& = \biggpnorm{\Sigma^{1/2}
	\begin{pmatrix}
	-\big(\pnorm{\Sigma_{XY}}{F}^2/\tau_X^2\big) I_p & \Sigma_{XY}\\
	\Sigma_{YX} & -\big(\pnorm{\Sigma_{XY}}{F}^2/\tau_Y^2\big) I_q
	\end{pmatrix}
	\Sigma^{1/2}}{F}^2\\
&\gtrsim_M \biggpnorm{
	\begin{pmatrix}
	-\big(\pnorm{\Sigma_{XY}}{F}^2/\tau_X^2\big) I_p & \Sigma_{XY}\\
	\Sigma_{YX} & -\big(\pnorm{\Sigma_{XY}}{F}^2/\tau_Y^2\big) I_q
	\end{pmatrix}
}{F}^2\gtrsim \pnorm{\Sigma_{XY}}{F}^2.
\end{align*}
This concludes the lower bound claim. If the spectrum of $\Sigma$ is bounded above as well, then using Lemma \ref{lem:SigmaXY_F_bound}, we have
\begin{itemize}
	\item $\pnorm{\Sigma_{XY}\Sigma_{YX}}{F}^2\lesssim_M \pnorm{\Sigma_{XY}}{F}^2$;
	\item $\tr(\Sigma_{XY}\Sigma_Y\Sigma_{YX}\Sigma_X)\asymp_M \pnorm{\Sigma_{XY}}{F}^2$;
	\item $\frac{\pnorm{\Sigma_{XY}}{F}^4\pnorm{\Sigma_X}{F}^2}{2\tau_X^4}\vee \frac{\pnorm{\Sigma_{XY}}{F}^4\pnorm{\Sigma_Y}{F}^2}{2\tau_Y^4}\lesssim_M (\tau_X^{-2}\vee\tau_Y^{-2}) \pnorm{\Sigma_{XY}}{F}^2 (\tau_X^2\wedge \tau_Y^2)\lesssim \pnorm{\Sigma_{XY}}{F}^2$;
	\item $\frac{\pnorm{\Sigma_{XY}}{F}^2 \tr(\Sigma_{XY}\Sigma_{YX}\Sigma_X)}{\tau_X^2}\vee \frac{\pnorm{\Sigma_{XY}}{F}^2 \tr(\Sigma_{YX}\Sigma_{XY}\Sigma_Y)}{\tau_Y^2}\lesssim_M \frac{\pnorm{\Sigma_{XY}}{F}^4}{(\tau_X\wedge \tau_Y)^2}\lesssim_M \pnorm{\Sigma_{XY}}{F}^2$;
	\item $\frac{\pnorm{\Sigma_{XY}}{F}^6}{\tau_X^2\tau_Y^2}\lesssim_M \pnorm{\Sigma_{XY}}{F}^2 \frac{\tau_X^4\wedge \tau_Y^4}{\tau_X^2\tau_Y^2}\lesssim \pnorm{\Sigma_{XY}}{F}^2$;
	\item $\tr(G_{[12]}\circ G_{[12]}) \leq \pnorm{G_{[12]}}{F}^2 \lesssim_M \pnorm{\Sigma_{XY}}{F}^2$;
	\item $\frac{\pnorm{\Sigma_{XY}}{F}^4}{4\tau_X^4}\tr(H_{[11]}\circ H_{[11]}) \vee \frac{\pnorm{\Sigma_{XY}}{F}^4}{4\tau_Y^4}\tr(H_{[22]}\circ H_{[22]}) \lesssim_M \frac{\pnorm{\Sigma_{XY}}{F}^4}{\tau_X^2\wedge \tau_Y^2}\lesssim_M \pnorm{\Sigma_{XY}}{F}^2$;
	\item $\frac{\pnorm{\Sigma_{XY}}{F}^2}{2\tau_X^2}\tr(H_{[11]}\circ G_{[12]})\vee \frac{\pnorm{\Sigma_{XY}}{F}^2}{2\tau_Y^2}\tr(H_{[22]}\circ G_{[12]}) \lesssim_M \frac{\pnorm{\Sigma_{XY}}{F}^3}{\tau_X\wedge \tau_Y}\lesssim_M \pnorm{\Sigma_{XY}}{F}^2$;
	\item $\frac{\pnorm{\Sigma_{XY}}{F}^4}{\tau_X^2\tau_Y^2}\tr(H_{[11]}\circ H_{[22]}) \lesssim \frac{\pnorm{\Sigma_{XY}}{F}^4}{\tau_X\tau_Y}\lesssim_M \pnorm{\Sigma_{XY}}{F}^2$.
\end{itemize}
Collecting the bounds we conclude that $\E \bar{g}_1^2(X,Y)\lesssim_M \tau_X^{-2}\tau_Y^{-2}\pnorm{\Sigma_{XY}}{F}^2$. 

\noindent (2). When $X=Y$, we may write with $Z \sim \mathcal{N}(0,I_p)$ that
\begin{align*}
\bar{g}_1(Z)=\bar{g}_1(X,Y) \stackrel{d}{=} \frac{1}{2\tau_X^2} \cdot \Big[Z^\top\bar{A}Z - \E(Z^\top \bar{A}Z)\Big],
\end{align*}
with the symmetric matrix $\bar{A}$ defined by $\bar{A}\equiv  \Sigma_X^{1/2}\Big[\Sigma_X-\pnorm{\Sigma_X}{F}^2/\tau_X^2\Big]\Sigma_X^{1/2}$. Using a similar argument as in the previous case, 
\begin{align*}
\E \bar{g}_1^2(X,Y)& \asymp \tau_X^{-4} \biggpnorm{\Sigma_X^{1/2}\Big[\Sigma_X-\pnorm{\Sigma_X}{F}^2/\tau_X^2\Big]\Sigma_X^{1/2}}{F}^2\\
&\gtrsim_M \tau_X^{-4} \bigpnorm{\Sigma_X-\pnorm{\Sigma_X}{F}^2/\tau_X^2}{F}^2 = \tau_X^{-4}\cdot \pnorm{\Sigma_X}{F}^4 p /\tau_X^4.
\end{align*}
A matching upper bound can be easily proved under the condition $\pnorm{\Sigma_X}{\op}\leq M$. \qed

\subsection{Proof of Lemma \ref{lem:var_bar_psi}}

We only prove the case for $\bar{\psi}_X$. In the definition (\ref{def:bar_psi}) of $\bar{\psi}_X(X,Y)$, we write $\bar{\psi}_X(x_1,y_1)\equiv \bar{A}_{1,X}(x_1,y_1) +A_{2,X}(x_1,y_1)$. Then we have
\begin{align*}
\var\big(\bar{\psi}_X(X,Y)\big)\lesssim \var\big(\bar{A}_{1,X}(X,Y)) + \var\big(A_{2,X}(X,Y)\big).
\end{align*}
Let $(X^\top, Y^\top)^\top \stackrel{d}{=} \Sigma^{1/2}Z$. We may write
\begin{align*}
\bar{A}_{1,X}(X,Y) &= \frac{1}{2\tau_X^4}\bigg[\big(\pnorm{X_1}{}^2-\tr(\Sigma_X)\big)(X^\top\Sigma_{XY}Y - \pnorm{\Sigma_{XY}}{F}^2)+ 2X^\top \Sigma_X \Sigma_{XY}Y\bigg]\\
& \stackrel{d}{=} \frac{1}{2\tau_X^4}\bigg[\big(Z^\top H_{[11]} Z-\tr(H_{[11]})\big)\big(Z^\top \bar{G}_{[1,2]} Z-\tr(\bar{G}_{[1,2]})\big)+ 2Z^\top \bar{G}_X Z\bigg],
\end{align*}
with $\bar{G}_X$ is defined in (\ref{def:bar_G_X}). By the Poincar\'e inequality, we have 
\begin{align*}
&\var \big(\tau_X^4\bar{A}_{1,X}(X,Y)\big)=\var\big(\tau_X^4 \bar{A}_{1,X}(Z)\big)\\
&\lesssim \E \big(Z^\top H_{[11]} Z-\tr(H_{[11]})\big)^2 Z^\top \bar{G}_{[1,2]}^2 Z+\E \big(Z^\top \bar{G}_{[1,2]} Z-\tr(\bar{G}_{[1,2]})\big)^2 Z^\top G_{[11]} Z+ \E Z^\top \bar{G}_{X}^2 Z\\
&\lesssim \pnorm{G_{[12]}}{F}^2 \pnorm{H_{[11]}}{F}^2 + \pnorm{\bar{G}_X}{F}^2 \lesssim_M \tau_X^2 \pnorm{\Sigma_{XY}}{F}^2,
\end{align*}
On the other hand, by the estimate in (\ref{ineq:res_psi_bound_6}), 
\begin{align*}
\var\big(A_{2,X}(X,Y)\big)\lesssim \E A_{2,X}^2(X,Y) \lesssim_M \tau_X^{-6} \pnorm{\Sigma_{XY}}{F}^2. 
\end{align*}
Collecting the bounds to conclude.\qed

\subsection{Proof of Lemma \ref{lem:h2}}

As $\E_{Z_3,Z_4}U(X_{i_1},X_{i_2})V(Y_{i_3},Y_{i_4})=0$, Proposition \ref{prop:hoef_decomp} yields that
\begin{align*}
k_2(z_1,z_2)&= \frac{1}{4!} \sum_{(i_1,\ldots,i_4)\in \sigma(1,2,3,4)} \bigg[ \E_{Z_3,Z_4} U(X_{i_1},X_{i_2})V(Y_{i_1},Y_{i_2}) \\
&\qquad\qquad - 2\E_{Z_3,Z_4} U(X_{i_1},X_{i_2})V(Y_{i_1},Y_{i_3})\bigg].
\end{align*}
The claim now follows by calculating
\begin{align*}
&\sum_{(i_1,\ldots,i_4)\in \sigma(1,2,3,4)} \E_{Z_3,Z_4} U(X_{i_1},X_{i_2})V(Y_{i_1},Y_{i_2})\\
& = 4 U(x_1,x_2) V(y_1,y_2)+8\E U(x_1,X)V(y_1,Y)+8\E U(x_2,X)V(y_2,Y)+ 4\dcov^2(X,Y),
\end{align*}
and
\begin{align*}
&\sum_{(i_1,\ldots,i_4)\in \sigma(1,2,3,4)} \E_{Z_3,Z_4} U(X_{i_1},X_{i_2})V(Y_{i_1},Y_{i_3})\\
& = \bigg[\sum_{\substack{(i_1,\ldots,i_4)\in \sigma(1,2,3,4),\\i_2=1,i_3=2}}+ \sum_{\substack{(i_1,\ldots,i_4)\in \sigma(1,2,3,4),\\i_2=2,i_3=1}}\bigg]\E_{Z_3,Z_4} U(X_{i_1},X_{i_2})V(Y_{i_1},Y_{i_3})\\
& = 2\E U(x_1,X) V(y_2,Y)+2\E U(x_2,X)V(y_1,Y),
\end{align*}
as desired.\qed

\subsection{Proof of Lemma \ref{lem:g2}}

By definition of $g_2$, we have
\begin{align*}
&g_2\big((x_1,y_1),(x_2,y_2) \big) \nonumber \\
&\equiv k_2\big((x_1,y_1),(x_2,y_2)\big) -k_1(x_1,y_1)-k_1(x_2,y_2)+\dcov^2(X,Y)\nonumber\\
&  =\frac{1}{6}\bigg[U(x_1,x_2)V(y_1,y_2)- \E U(x_1,X)V(y_1,Y)-\E U(x_2,X)V(y_2,Y)+\dcov^2(X,Y)\nonumber\\
&\qquad\qquad - \E U(x_1,X)V(y_2,Y) - \E U(x_2,X)V(y_1,Y)\bigg]\nonumber\\
& = \frac{1}{6\tau_X\tau_Y}\bigg[\Big(x_1^\top x_2 y_1^\top y_2 - x_1^\top \Sigma_{XY} y_1-x_2^\top \Sigma_{XY} y_2 + \pnorm{\Sigma_{XY}}{F}^2\Big)-\Big(x_1^\top \Sigma_{XY} y_2+x_2^\top \Sigma_{XY} y_1\Big)\nonumber\\
&\qquad\qquad + \big[\bar{R}_2\big((x_1,y_1),(x_2,y_2)\big) -\E \bar{R}_2\big((X_1,Y_1),(X_2,Y_2)\big)\big] \bigg],
\end{align*}
completing the proof.\qed

\subsection{Proof of Lemma \ref{lem:h3}}
As $\E_{Z_4}U(X_{i_1},X_{i_2})V(Y_{i_3},Y_{i_4})=0$, Proposition \ref{prop:hoef_decomp} yields that
\begin{align*}
k_3(z_1,z_2,z_3) &= \frac{1}{4!} \sum_{(i_1,\ldots,i_4)\in \sigma(1,2,3,4)} \bigg[ \E_{Z_4} U(X_{i_1},X_{i_2})V(Y_{i_1},Y_{i_2}) - 2\E_{Z_4} U(X_{i_1},X_{i_2})V(Y_{i_1},Y_{i_3})\bigg].
\end{align*}
The claim follows by calculating
\begin{align*}
&\sum_{(i_1,\ldots,i_4)\in \sigma(1,2,3,4)} \E_{Z_4} U(X_{i_1},X_{i_2})V(Y_{i_1},Y_{i_2})\\
& = \bigg[\sum_{\substack{(i_1,\ldots,i_4)\in \sigma(1,2,3,4)\\ i_1,i_2\neq 4}} + \sum_{\substack{(i_1,\ldots,i_4)\in \sigma(1,2,3,4)\\ i_1=4}}+\sum_{\substack{(i_1,\ldots,i_4)\in \sigma(1,2,3,4)\\ i_2=4}}\bigg] \E_{Z_4} U(X_{i_1},X_{i_2})V(Y_{i_1},Y_{i_2})\\
& = 2\sum_{1\leq i_1\neq i_2\leq 3}U(x_{i_1},x_{i_2})V(y_{i_1},y_{i_2}) + 4\sum_{1\leq i\leq 3}\E U(X,x_{i}) V(Y,y_i),
\end{align*}
and
\begin{align*}
&\sum_{(i_1,\ldots,i_4)\in \sigma(1,2,3,4)} \E_{Z_4} U(X_{i_1},X_{i_2})V(Y_{i_1},Y_{i_3})\\
& = \bigg[\sum_{\substack{(i_1,\ldots,i_4)\in \sigma(1,2,3,4),\\i_1=4}}+\sum_{\substack{(i_1,\ldots,i_4)\in \sigma(1,2,3,4),\\i_4=4}}\bigg] \E_{Z_4} U(X_{i_1},X_{i_2})V(Y_{i_1},Y_{i_3})\\
& = \sum_{1\leq i_1\neq i_2\leq 3} \E U(X,x_{i_1})V(Y,y_{i_2})+\sum_{(i_1,i_2,i_3) \in \sigma (1,2,3)} U(x_{i_1},x_{i_2}) V(y_{i_1},y_{i_3}),
\end{align*}
as desired.\qed

\section{Proofs for Section \ref{section:clt_truncated_dcov}}\label{appendix:proof_clt_Tn_bar}

\subsection{Proof of Lemma \ref{lem:rep_Tbar}}\label{subsec:truncate_decomposition}

This follows calculations with the definition (\ref{def:bar_T}) of $\bar{T}_n$:
	\begin{align*}
	&\bar{T}_n\big(\bm{X},\bm{Y}\big) = \dcov^2(X,Y)\\
	&+ \frac{2}{\tau_X\tau_Yn}\sum_i \Big(X_i^\top \Sigma_{XY} Y_i-\pnorm{\Sigma_{XY}}{F}^2-\frac{\pnorm{\Sigma_{XY}}{F}^2}{2\tau_X^2}(\pnorm{X_i}{}^2-\tr(\Sigma_X))-\frac{\pnorm{\Sigma_{XY}}{F}^2}{2\tau_Y^2}(\pnorm{Y_i}{}^2-\tr(\Sigma_Y))\Big)\\
	&\qquad + \frac{1}{\tau_X\tau_Y\cdot 2 \binom{n}{2}}\sum_{I_n^2}\bigg[\Big(X_{i_1}^\top X_{i_2} Y_{i_1}^\top Y_{i_2} - X_{i_1}^\top \Sigma_{XY} Y_{i_1}-X_{i_2}^\top \Sigma_{XY} Y_{i_2} + \pnorm{\Sigma_{XY}}{F}^2\Big)\nonumber\\
	&\qquad\qquad\qquad\qquad -\Big(X_{i_1}^\top \Sigma_{XY} Y_{i_2}+X_{i_2}^\top \Sigma_{XY} Y_{i_1}\Big) \bigg]\\
	& =  \dcov^2(X,Y) - \frac{2}{\tau_X\tau_Yn}\sum_{i=1}^n\bigg[\frac{\pnorm{\Sigma_{XY}}{F}^2}{2\tau_X^2}\big(\pnorm{X_i}{}^2-\tr(\Sigma_X)\big)+\frac{\pnorm{\Sigma_{XY}}{F}^2}{2\tau_Y^2}\big(\pnorm{Y_i}{}^2-\tr(\Sigma_Y)\big)\bigg]\\
	&\qquad\qquad+ \frac{1}{\tau_X\tau_Y\cdot 2 \binom{n}{2}}\sum_{I_n^2}\bigg[\Big(X_{i_1}^\top X_{i_2} Y_{i_1}^\top Y_{i_2} - \pnorm{\Sigma_{XY}}{F}^2\Big) -\Big(X_{i_1}^\top \Sigma_{XY} Y_{i_2}+X_{i_2}^\top \Sigma_{XY} Y_{i_1}\Big) \bigg].
	\end{align*}
The claim now follows by the definitions of $\psi_1$-$\psi_3$.\qed

\subsection{Proof of Proposition \ref{prop:var_psi_1}}
\subsubsection{Variance bound}
\begin{proof}[Proof of Proposition \ref{prop:var_psi_1}-(1)]
	With $W_i\equiv X_iY_i^\top$, we have
		\begin{align}\label{ineq:clt_2}
		&\Delta_i \psi_1\big(\bm{X},\bm{Y}\big) = \psi_1\big(\bm{X},\bm{Y}\big)-\psi_1\big(\bm{X}^{\{i\}},\bm{Y}^{\{i\}}\big)\nonumber\\
		&=\sum_{ \substack{i_1\neq i_2,\\ i_1,i_2\neq i}}  \tr\big(W_{i_1}W_{i_2}^\top\big)+ 2\sum_{j\neq i} \tr(W_{i} W_{j}^\top) -\bigg[\sum_{ \substack{i_1\neq i_2,\\ i_1,i_2\neq i}}  \tr\big(W_{i_1}W_{i_2}^\top\big)+ 2\sum_{j\neq i} \tr(W_{i}' W_{j}^\top)\bigg] \nonumber\\
		&= 2\tr\bigg[\big(W_i-W_{i}'\big)\sum_{j\neq i} W_j^\top\bigg]= 2 \bigg[Y_i^\top \Big(\sum_{j\neq i} Y_jX_j^\top\Big) X_i-Y_i'^\top \Big(\sum_{j\neq i} Y_jX_j^\top\Big) X_i'\bigg].
		\end{align}
	So for any $A \subset \{1,\ldots,n\}$ and $i \notin A$,
		\begin{align*}
		&\Delta_i \psi_1\big(\bm{X}^A,\bm{Y}^A\big)=2 \tr\bigg[\big(W_i-W_{i}'\big)\bigg(\sum_{j\notin A \cup \{i\}} W_j+\sum_{j \in A} W_j'\bigg)^\top\bigg]\\
		&=2 \bigg[Y_i^\top \bigg(\sum_{j\notin A\cup\{i\}} Y_jX_j^\top+\sum_{j \in A} Y_j'X_j'^\top\bigg) X_i-Y_i'^\top \bigg(\sum_{j\notin A\cup\{i\}} Y_jX_j^\top+\sum_{j \in A} Y_j'X_j'^\top\bigg)  X_i'\bigg].
		\end{align*}
	Then with $\E'\equiv \E_{\bm{X}',\bm{Y}'}$, we have
		\begin{align*}
		&4^{-1}\E'\big[\Delta_i \psi_1\big(\bm{X},\bm{Y}\big)\Delta_i \psi_1\big(\bm{X}^A,\bm{Y}^A\big)\big]\\
		& = \E'\bigg[Y_i^\top\Big(\sum_{j\neq i} Y_jX_j^\top\Big) X_i \cdot Y_i^\top \bigg(\sum_{j\notin A\cup\{i\}} Y_jX_j^\top+\sum_{j \in A} Y_j'X_j'^\top\bigg) X_i\bigg]\\
		&\qquad - \E'\bigg[Y_i^\top \Big(\sum_{j\neq i} Y_jX_j^\top\Big) X_i \cdot Y_i'^\top \bigg(\sum_{j\notin A\cup\{i\}} Y_jX_j^\top+\sum_{j \in A} Y_j'X_j'^\top\bigg)  X_i'\bigg]\\
		&\qquad -\E'\bigg[Y_i'^\top \Big(\sum_{j\neq i} Y_jX_j^\top\Big) X_i' \cdot Y_i^\top \bigg(\sum_{j\notin A\cup\{i\}} Y_jX_j^\top+\sum_{j \in A} Y_j'X_j'^\top\bigg) X_i\bigg]\\
		&\qquad + \E'\bigg[Y_i'^\top \Big(\sum_{j\neq i} Y_jX_j^\top\Big) X_i'\cdot Y_i'^\top \bigg(\sum_{j\notin A\cup\{i\}} Y_jX_j^\top+\sum_{j \in A} Y_j'X_j'^\top\bigg)  X_i'\bigg]\\
		& \equiv B_1(i,A)-B_2(i,A)-B_3(i,A)+B_4(i,A).
		\end{align*}
	The four terms can be calculated as follows:
	\begin{itemize}
		\item By direct calculation, 
		\begin{align*}
		B_1(i,A) &= Y_i^\top \Big(\sum_{j\neq i} Y_jX_j^\top\Big) X_i \cdot Y_i^\top \Big(\sum_{j\notin A\cup\{i\}} Y_jX_j^\top+|A|\cdot\Sigma_{YX}\Big) X_i\\
		&= \sum_{j\neq i, k\notin A\cup\{i\}}(X_i^\top X_j)(Y_i^\top Y_j)(X_i^\top X_k)(Y_i^\top Y_k) + |A|(X_i^\top\Sigma_{XY}Y_i)\cdot\sum_{j\neq i}(X_i^\top X_j)(Y_i^\top Y_j).
		\end{align*}
		\item Using independence between $(X_i',Y_i')$ and $\{(X_j',Y_j')\}_{j\in A}$ as $i\notin A$, 
			\begin{align*}
			B_2(i,A) &= Y_i^\top \Big(\sum_{j\neq i} Y_jX_j^\top\Big) X_i \cdot \tr\Big[\Big(\sum_{j\notin A\cup\{i\}} Y_jX_j^\top+|A|\cdot\Sigma_{YX}\Big)\Sigma_{XY}\Big]\\
			&= \sum_{j\neq i,k\notin A\cup\{i\}}(X_i^\top X_j)(Y_i^\top Y_j)(X_k^\top \Sigma_{XY}Y_k) + |A|\pnorm{\Sigma_{XY}}{F}^2\cdot\sum_{j\neq i}(X_i^\top X_j)(Y_i^\top Y_j).
			\end{align*}
		\item Again by independence between $(X_i',Y_i')$ and $\{(X_j',Y_j')\}_{j\in A}$ as $i \notin A$,
			\begin{align*}
			B_3(i,A) &= \tr\Big(\Big[\sum_{j\neq i}Y_jX_j^\top\Big]\Sigma_{XY}\Big)\cdot Y_i^\top\Big(\sum_{j\notin A\cup\{i\}}Y_jX_j^\top + |A|\cdot\Sigma_{YX}\Big)X_i\\
			&= \sum_{j\neq i, k\notin A\cup\{i\}}(X_j^\top\Sigma_{XY}Y_j)(X_i^\top X_k)(Y_i^\top Y_k) + |A|(X_i^\top\Sigma_{XY}Y_i)\cdot \sum_{j\neq i} X_j^\top \Sigma_{XY}Y_j.
			\end{align*}
		\item By Lemma \ref{lem:moment_X_Y}-(4), we have
		\begin{align*}
		B_4(i,A) &= \E'\Big[X_i'^\top\Big(\sum_{j\neq i}X_jY_j^\top\Big)Y_i'\cdot X_i'^\top\Big(\sum_{j\notin A\cup\{i\}}X_jY_j^\top + |A|\cdot \Sigma_{XY}\Big)Y_i'\Big]\\
		&= \tr\Big[\Sigma_X\Big(\sum_{j\neq i}X_jY_j^\top\Big)\Sigma_Y\Big(\sum_{j\notin A\cup\{i\}}Y_jX_j^\top + |A|\cdot \Sigma_{YX}\Big)\Big]\\
		&\quad\quad\quad + \tr\Big[\Big(\sum_{j\neq i}X_jY_j^\top\Big)\Sigma_{YX}\Big(\sum_{j\notin A\cup\{i\}}X_jY_j^\top + |A|\cdot \Sigma_{XY}\Big)\Sigma_{YX}\Big]\\
		&\quad\quad\quad + \tr\Big[\Big(\sum_{j\neq i}X_jY_j^\top\Big)\Sigma_{YX}\Big]\cdot\tr\Big[\Big(\sum_{j\notin A\cup\{i\}}X_jY_j^\top + |A|\cdot \Sigma_{XY}\Big)\Sigma_{YX}\Big]\\
		&\quad\quad\quad +\kappa \tr\big(M_1(i,A) \circ M_2(i,A)\big)\\
		&= \Big[\sum_{j\neq i,k\notin A\cup\{i\}}(X_j^\top\Sigma_X X_k)(Y_j^\top\Sigma_Y Y_k) + |A|\cdot\sum_{j\neq i}X_j^\top\Sigma_X\Sigma_{XY}\Sigma_YY_j\Big]\\
		&\quad\quad\quad + \Big[\sum_{j\neq i,k\notin A\cup\{i\}}(X_j^\top\Sigma_{XY}Y_k)(X_k^\top\Sigma_{XY}Y_j) + |A|\cdot\sum_{j\neq i}X_j^\top \Sigma_{XY}\Sigma_{YX}\Sigma_{XY}Y_j\Big]\\
		&\quad\quad\quad + \Big[\sum_{j\neq i,k\notin A\cup\{i\}}(X_j^\top\Sigma_{XY}Y_j)(X_k^\top \Sigma_{XY}Y_k) + |A|\pnorm{\Sigma_{XY}}{F}^2\cdot \sum_{j\neq i}X_j^\top\Sigma_{XY} Y_j \Big]\\
		&\quad\quad\quad +\kappa \tr\big(M_1(i,A) \circ M_2(i,A)\big).
		\end{align*}
		Here $M_1(i,A) \equiv (M_1(i,A))(\bm{X},\bm{Y}) \equiv \Sigma^{1/2}[0,\bar{M}_1(i,A);0,0]\Sigma^{1/2}$ and $M_2(i,A) \equiv (M_2(i,A))(\bm{X},\bm{Y}) \equiv \Sigma^{1/2}[0,\bar{M}_2(i,A);0,0]\Sigma^{1/2}$, with $\bar{M}_1(i,A) \equiv \sum_{j\neq i}X_jY_j^\top$ and $\bar{M}_2(i,A) \equiv \sum_{j\notin A\cup\{i\}}X_jY_j^\top + |A|\cdot \Sigma_{XY}$.
	\end{itemize}
	Hence we have
	\begin{align*}
	B_1(i,A) - B_2(i,A) &= \sum_{j\neq i, k\notin A\cup\{i\}}(X_i^\top X_j)(Y_i^\top Y_j)\big[X_k^\top(X_iY_i^\top -\Sigma_{XY})Y_k\big]\\
	&\quad\quad\quad+ |A|\cdot \big(X_i^\top\Sigma_{XY}Y_i - \pnorm{\Sigma_{XY}}{F}^2\big)\cdot \sum_{j\neq i}(X_i^\top X_j)(Y_i^\top Y_j),
	\end{align*}
	and, by combining $-B_3(i,A)$ with the last term of $B_4(i,A)$,
	\begin{align*}
	&B_4(i,A) - B_3(i,A)\\
	&= \sum_{j\neq i,k\notin A\cup\{i\}}(X_j^\top\Sigma_X X_k)(Y_j^\top\Sigma_Y Y_k) + |A|\cdot\sum_{j\neq i}X_j^\top\Sigma_X\Sigma_{XY}\Sigma_YY_j\\
	&\quad\quad\quad + \sum_{j\neq i,k\notin A\cup\{i\}}(X_j^\top\Sigma_{XY}Y_k)(X_k^\top\Sigma_{XY}Y_j) + |A|\cdot\sum_{j\neq i}X_j^\top \Sigma_{XY}\Sigma_{YX}\Sigma_{XY}Y_j\\
	&\quad\quad\quad- \sum_{j\neq i,k\notin A\cup\{i\}} (X_j^\top\Sigma_{XY}Y_j)\big[X_k^\top(X_iY_i^\top - \Sigma_{XY})Y_k\big]\\
	&\quad\quad\quad - |A|\cdot (X_i^\top\Sigma_{XY}Y_i - \pnorm{\Sigma_{XY}}{F}^2)\cdot\sum_{j\neq i} X_j^\top \Sigma_{XY} Y_j +\kappa \tr\big(M_1(i,A) \circ M_2(i,A)\big).
	\end{align*} 
	This implies that
	\begin{align*}
	&4^{-1}\E'\big[\Delta_i \psi_1\big(\bm{X},\bm{Y}\big)\Delta_i \psi_1\big(\bm{X}^A,\bm{Y}^A\big)\big]\\
	&= \sum_{j\neq i,k\notin A\cup\{i\}}(X_j^\top\Sigma_X X_k)(Y_j^\top\Sigma_Y Y_k) + \sum_{j\neq i,k\notin A\cup\{i\}}(X_j^\top\Sigma_{XY}Y_k)(X_k^\top\Sigma_{XY}Y_j)\\
	&\quad\quad\quad+ \sum_{j\neq i,k\notin A\cup\{i\}} \big[X_j^\top(X_iY_i^\top - \Sigma_{XY})Y_j\big]\big[X_k^\top(X_iY_i^\top - \Sigma_{XY})Y_k\big]\\
	&\quad\quad\quad +|A|\cdot\sum_{j\neq i}X_j^\top\Sigma_X\Sigma_{XY}\Sigma_YY_j + |A|\cdot\sum_{j\neq i}X_j^\top \Sigma_{XY}\Sigma_{YX}\Sigma_{XY}Y_j\\
	&\quad\quad\quad + |A|\cdot (X_i^\top\Sigma_{XY}Y_i - \pnorm{\Sigma_{XY}}{F}^2)\cdot\sum_{j\neq i} X_j^\top(X_iY_i^\top - \Sigma_{XY}) Y_j +\kappa \tr\big(M_1(i,A) \circ M_2(i,A)\big)\\
	& \equiv 4^{-1}\cdot\sum_{\ell=1}^7T_\ell(i,A) . 
	\end{align*}
	By definition, we have
	\begin{align*}
	T_1&\equiv\frac{1}{2}\sum_{A \subsetneq \{1,\ldots,n\}}\sum_{i\notin A}\frac{T_1(i,A)}{\binom{n}{\abs{A}}(n-\abs{A})}\\
	& = 2 \sum_{A \subsetneq \{1,\ldots,n\}}\sum_{i \notin A}\sum_{j \neq i, k \notin A\cup \{i\}} \frac{(X_j^\top\Sigma_X X_k)(Y_j^\top\Sigma_Y Y_k)}{\binom{n}{\abs{A}}(n-\abs{A})}\\
	& = 2\sum_{\substack{1\leq i\leq n,\\ j\neq i, k\neq i}} (X_j^\top\Sigma_X X_k)(Y_j^\top\Sigma_Y Y_k)\cdot \bigg[\sum_{A\cap \{i,k\}=\emptyset}\frac{1}{\binom{n}{\abs{A}}(n-\abs{A})}\bigg]\\
	& \stackrel{(\ast)}{=} \sum_{\substack{1\leq i\leq n,\\ j\neq i, k\neq i}} (X_j^\top\Sigma_X X_k)(Y_j^\top\Sigma_Y Y_k) = (n-2) \sum_{j,k} (X_j^\top\Sigma_X X_k)(Y_j^\top\Sigma_Y Y_k)  + \sum_{j=1}^n (X_j^\top\Sigma_XX_j)(Y_j^\top \Sigma_YY_j).
	\end{align*}
	Here in ($\ast$) we used
	\begin{align}\label{ineq:sum_comb_A_1}
	\sum_{A\cap \{i,k\}=\emptyset}\frac{1}{\binom{n}{\abs{A}}(n-\abs{A})} &= \sum_{a=0}^{n-2} \sum_{A\cap \{i,k\}=\emptyset, \abs{A}=a} \frac{1}{\binom{n}{a}(n-a)} = \sum_{a=0}^{n-2} \frac{\binom{n-2}{a}}{\binom{n}{a}(n-a)}=\sum_{a=0}^{n-2} \frac{n-a-1}{n(n-1)}=\frac{1}{2}.
	\end{align}
	With completely analogous definitions and calculations, we have
	\begin{align*}
	T_2 &= (n-2) \sum_{j,k}(X_j^\top\Sigma_{XY}Y_k)(X_k^\top\Sigma_{XY}Y_j)  + \sum_{j=1}^n (X_j^\top\Sigma_{XY}Y_j)^2,\\
	T_3 &= \sum_{1\leq i\leq n}\bigg[\sum_{j\neq i} X_j^\top(X_iY_i^\top - \Sigma_{XY})Y_j\bigg]^2.
	\end{align*}
	Next, we have
	\begin{align*}
	T_4 &\equiv\frac{1}{2}\sum_{A \subsetneq \{1,\ldots,n\}}\sum_{i\notin A}\frac{T_4(i,A)}{\binom{n}{\abs{A}}(n-\abs{A})}\\
	& = 2 \sum_{A \subsetneq \{1,\ldots,n\}}\sum_{i \notin A}\sum_{j \neq i} \frac{|A|\cdot (X_j^\top\Sigma_X\Sigma_{XY}\Sigma_Y Y_j)}{\binom{n}{\abs{A}}(n-\abs{A})}\\
	& = 2\sum_{1\leq i\leq n, j\neq i} (X_j^\top\Sigma_X\Sigma_{XY}\Sigma_Y Y_j)\cdot \bigg[\sum_{A: i\notin A}\frac{|A|}{\binom{n}{\abs{A}}(n-\abs{A})}\bigg]\\
	& \stackrel{(\ast\ast)}{=} (n-1)\cdot\sum_{1\leq i\leq n, j\neq i} X_j^\top\Sigma_X\Sigma_{XY}\Sigma_Y Y_j = (n-1)^2\cdot\sum_{j=1}^n X_j^\top\Sigma_X\Sigma_{XY}\Sigma_Y Y_j.
	\end{align*}
	Here in $(**)$ we used
	\begin{align*}
	\sum_{A:i\notin A}\frac{|A|}{\binom{n}{\abs{A}}(n-\abs{A})} &= \sum_{a=0}^{n-1} \sum_{A:i\notin A, \abs{A}=a} \frac{a}{\binom{n}{a}(n-a)} = \sum_{a=0}^{n-1} \frac{a\binom{n-1}{a}}{\binom{n}{a}(n-a)}=\sum_{a=0}^{n-1} \frac{a}{n}=\frac{n-1}{2}.
	\end{align*}
	With completely analogous definitions and calculations, we have
	\begin{align*}
	T_5 &= (n-1)^2\cdot\sum_{j=1}^n X_j^\top\Sigma_{XY}\Sigma_{YX}\Sigma_{XY}Y_j,\\
	T_6 &= (n-1)\cdot\sum_{1\leq i\leq n, j\neq i}(X_i^\top \Sigma_{XY}Y_i - \pnorm{\Sigma_{XY}}{F}^2)\cdot\big[X_j^\top(X_iY_i^\top-\Sigma_{XY})Y_j\big],\\
	T_7 &= \kappa\cdot\Big[(n-1)^2\tr\Big(\Sigma^{1/2}
	\begin{pmatrix}
	0 & \sum_{j=1}^n X_jY_j^\top\\
	0 & 0
	\end{pmatrix}
	\Sigma^{1/2} \circ G_{[12]}
	\Big)\\
	&\quad\quad + \sum_{j,k} (n-2+\bm{1}_{j=k})\tr\Big(\Sigma^{1/2}
	\begin{pmatrix}
	0 & X_jY_j^\top\\
	0 & 0
	\end{pmatrix}
	\Sigma^{1/2}\circ \Sigma^{1/2}
	\begin{pmatrix}
	0 & X_kY_k^\top\\
	0 & 0
	\end{pmatrix}\Sigma^{1/2}
	\Big)\Big].
	\end{align*}
	Putting together the pieces, we have $
	\E(T_{\psi_1}|\bm{X},\bm{Y}) = \sum_{i=1}^6 T_i$.
	The proof is now complete by invoking Lemma \ref{lem:psi_1_variances} that gives bounds for the variances of $T_i$'s.
\end{proof}

\begin{lemma}\label{lem:psi_1_variances}
	Recall the terms $T_1$-$T_6$ defined in the proof of Proposition \ref{prop:var_psi_1}-(1) above. The following hold.
	\begin{enumerate}
		\item $\var(T_1)\lesssim_Mn^3\cdot(n\vee \tau_X^2\vee \tau_Y^2)\cdot(\pnorm{\Sigma_X}{F}^2\pnorm{\Sigma_Y}{F}^2 + n\pnorm{\Sigma_{XY}}{F}^2)$.
		\item $\var(T_2)\lesssim_M n^3\cdot\pnorm{\Sigma_{XY}}{F}^6 + n^5\cdot \pnorm{\Sigma_{XY}}{F}^2$.
		\item $\var(T_3)\lesssim_M n^3\cdot \pnorm{\Sigma_X}{F}^4\pnorm{\Sigma_Y}{F}^4 + n^4\cdot(1\vee\pnorm{\Sigma_{XY}}{F}^2)\pnorm{\Sigma_X}{F}^2\pnorm{\Sigma_Y}{F}^2 + n^5\cdot \pnorm{\Sigma_{XY}}{F}^2$.
		\item $\var(T_4)\vee \var(T_5)\lesssim_M n^5\cdot \pnorm{\Sigma_{XY}}{F}^2$.
		\item $\var(T_6)\lesssim_M n^4\cdot  \pnorm{\Sigma_X}{F}^2\pnorm{\Sigma_Y}{F}^2\pnorm{\Sigma_{XY}}{F}^2$.
		\item $\var(T_7) \lesssim n^3\cdot \pnorm{\Sigma_X}{F}^4\pnorm{\Sigma_Y}{F}^4 + n^4\pnorm{\Sigma_{XY}}{F}^2\pnorm{\Sigma_X}{F}^2\pnorm{\Sigma_Y}{F}^2 + n^5\cdot \pnorm{\Sigma_{XY}}{F}^4$.
	\end{enumerate}
	Consequently,
	\begin{align*}
	\sum_{i=1}^6 \var(T_i) &\lesssim_M n^3\cdot \pnorm{\Sigma_X}{F}^4\pnorm{\Sigma_Y}{F}^4 + n^4\cdot(1\vee\pnorm{\Sigma_{XY}}{F}^2)\pnorm{\Sigma_X}{F}^2\pnorm{\Sigma_Y}{F}^2 + n^5\cdot(1\vee\pnorm{\Sigma_{XY}}{F}^2) \pnorm{\Sigma_{XY}}{F}^2.
	\end{align*}
\end{lemma}

\begin{proof}[Proof of Lemma \ref{lem:psi_1_variances}]
	\noindent (\textbf{Variance of $T_1$})  
	By the Poincar\'e inequality, it is easy to show that $\var\big((X^\top \Sigma_X X)(Y^\top\Sigma_Y Y)\big)\lesssim_M (\tau_X^2\vee \tau_Y^2)\pnorm{\Sigma_X}{F}^2\pnorm{\Sigma_Y}{F}^2$, so it remains to prove the bound 
	\begin{align*}
	\var(T_1) \equiv \var\Big((n-2)\sum_{j,k} (X_j^\top\Sigma_XX_k)(Y_j^\top\Sigma_Y Y_k)\Big) \lesssim_M n^3(n\vee\tau_X^2\vee \tau_Y^2)(\pnorm{\Sigma_X}{F}^2\pnorm{\Sigma_Y}{F}^2 + n\pnorm{\Sigma_{XY}}{F}^2),
	\end{align*}
	where we slightly abuse the notation $T_1$.
	Note that
	\begin{align*}
	T_1 &= (n-2)\cdot\sum_{k,\ell} (X_k^\top\Sigma_X X_\ell)(Y_k^\top\Sigma_YY_\ell)= (n-2)\cdot\sum_{k,\ell} (Z_k^\top G_{[11]}Z_\ell)(Z_k^\top G_{[22]}Z_\ell).
	\end{align*}	
	Hence for any $i\in[n]$ and $j\in[p+q]$, we have
	\begin{align*}
	\frac{\partial T_1}{\partial Z_{ij}} &= 2(n-2)\cdot \sum_{k=1}^n\Big[(e_j^\top G_{[11]}Z_k)(Z_i^\top G_{[22]}Z_k) + (Z_i^\top G_{[11]}Z_k)(Z_k^\top G_{[22]}e_j)\Big]\\
	&\equiv 2(n-2)\cdot \big[(T_{11})_{(ij)} + (T_{12})_{(ij)}\big].
	\end{align*}
	By direct calculation, we have
	\begin{align*}
	\E\pnorm{T_{11}}{F}^2 &= \sum_{i,j} \E \bigg[e_j^\top G_{[11]}\Big(\sum_{k=1}^n Z_kZ_k^\top\Big)G_{[22]}Z_iZ_i^\top G_{[22]}\Big(\sum_{k=1}^n Z_kZ_k^\top\Big)G_{[11]}e_j\bigg]\\
	&= n\cdot\sum_{k_1,k_2=1}^n\E\Big[(Z_1^\top G_{[22]}Z_{k_1})(Z_1^\top G_{[22]}Z_{k_2})(Z_{k_1}^\top G_{[11]}^2Z_{k_2})\Big]\\
	&= n\cdot \E \bigg[1\cdot (Z_1^\top G_{[22]}Z_1)(Z_1^\top G_{[22]}Z_1)(Z_1^\top G_{[11]}^2Z_1)\\
	&\qquad\qquad + 2(n-1)\cdot (Z_1^\top G_{[22]}Z_1)(Z_1^\top G_{[22]}Z_2)(Z_1^\top G_{[11]}^2Z_2)\\
	&\qquad\qquad + (n-1)\cdot (Z_1^\top G_{[22]}Z_2)(Z_1^\top G_{[22]}Z_2)(Z_2^\top G_{[11]}^2Z_2)\\
	&\qquad\qquad + (n-1)(n-2)\cdot (Z_1^\top G_{[22]}Z_2)(Z_1^\top G_{[22]}Z_3)(Z_2^\top G_{[11]}^2Z_3)\bigg]\\
	&\lesssim n\cdot\bigg[1\cdot\pnorm{\Sigma_Y}{F}^4\pnorm{\Sigma_X^2}{F}^2 \\
	&\qquad\qquad+ n\cdot \big[\tr(\Sigma_{XY}\Sigma_Y^3\Sigma_{YX}\Sigma_X^3) + \pnorm{\Sigma_Y}{F}^2\cdot\tr(\Sigma_{XY}\Sigma_Y\Sigma_{YX}\Sigma_X^3)\big]\\
	&\qquad\qquad + n\cdot\pnorm{\Sigma_X^2}{F}^2\pnorm{\Sigma_Y}{F}^2 + n^2\cdot \tr(\Sigma_{XY}\Sigma_Y^3\Sigma_{YX}\Sigma_X^3)\bigg]\\
	&\asymp_M n\cdot\Big[\tau_Y^4\tau_X^2 + n\cdot\tau_Y^2\pnorm{\Sigma_{XY}}{F}^2 + n\tau_X^2\tau_Y^2 + n^2\pnorm{\Sigma_{XY}}{F}^2\Big]\\
	&\asymp n\cdot(n\vee \tau_Y^2)\cdot(\tau_X^2\tau_Y^2 + n\pnorm{\Sigma_{XY}}{F}^2).
	\end{align*}
	A similar calculation for $T_{12}$ yields that
	\begin{align*}
	\E\pnorm{T_{11}}{F}^2 \lesssim_M n\cdot(n\vee \tau_X^2)\cdot(\tau_X^2\tau_Y^2 + n\pnorm{\Sigma_{XY}}{F}^2),
	\end{align*}
	so we conclude by the Gaussian-Poincar\'e inequality that
	\begin{align*}
	\var(T_1)\lesssim_Mn^3\cdot(n\vee \tau_X^2\vee \tau_Y^2)\cdot(\tau_X^2\tau_Y^2 + n\pnorm{\Sigma_{XY}}{F}^2).
	\end{align*}
	\medskip
	
	\noindent (\textbf{Variance of $T_2$}) 
	By the Poincar\'e inequality, it is easy to show that $\var\big(\sum_{j=1}^n (X_j^\top \Sigma_{XY}Y_j)^2\big)$ is bounded by the desired order, so we redefine
	\begin{align*}
	T_2 \equiv (n-2)\sum_{j,k} (X_j^\top\Sigma_{XY}Y_k)(X_k^\top\Sigma_{XY}Y_j)
	\end{align*}
	with a slight abuse of notation.
	Note that 
	\begin{align*}
	T_2 = (n-2)\cdot \sum_{k,\ell}  (Z_k^\top G_{[12]}Z_\ell)(Z_k^\top G_{[21]}Z_\ell).
	\end{align*}
	Hence for any $i\in[n]$ and $j\in[p+q]$, 
	\begin{align*}
	\frac{\partial T_2}{\partial Z_{ij}} &= 2(n-2)\cdot \sum_{k} \Big[(e_j^\top G_{[12]}Z_k)(Z_i^\top G_{[21]}Z_k) + (Z_k^\top G_{[12]}e_j)(Z_k^\top G_{[21]}Z_i)\Big]\\
	&\equiv 2(n-2)\cdot \big[(T_{21})_{(ij)} + (T_{22})_{(ij)}\big].
	\end{align*}
	By direct calculation, we have
	\begin{align*}
	\E\pnorm{T_{21}}{F}^2 &= \sum_{i,j} \E\Big[Z_i^\top G_{[21]} \Big(\sum_{k} Z_kZ_k^\top\Big)G_{[21]}e_je_j^\top G_{[12]}\Big(\sum_{k} Z_kZ_k^\top\Big)G_{[12]}Z_i\Big]\\
	&= n \cdot \sum_{k_1,k_2} \E\Big[(Z_1^\top G_{[21]}Z_{k_1})(Z_1^\top G_{[21]}Z_{k_2})(Z_{k_1}^\top G_{[21]}G_{[12]}Z_{k_2})\Big]\\
	&= n\cdot \E \bigg[1\cdot (Z_1^\top G_{[21]}Z_1)(Z_1^\top G_{[21]}Z_1)(Z_1^\top G_{[21]}G_{[12]}Z_1)\\
	&\quad\quad + 2(n-1)\cdot (Z_1^\top G_{[21]}Z_1)(Z_1^\top G_{[21]}Z_2)(Z_1^\top G_{[21]}G_{[12]}Z_2)\\
	&\quad\quad + (n-1)\cdot (Z_1^\top G_{[21]}Z_2)(Z_1^\top G_{[21]}Z_2)(Z_2^\top G_{[21]}G_{[12]}Z_2)\\
	&\quad\quad + (n-1)(n-2)\cdot (Z_1^\top G_{[21]}Z_2)(Z_1^\top G_{[21]}Z_3)(Z_2^\top G_{[21]}G_{[12]}Z_3)\bigg]\\
	&\lesssim_M n\cdot\Big[\pnorm{\Sigma_{XY}}{F}^6 + n\cdot \pnorm{\Sigma_{XY}}{F}^4 + n^2\cdot\pnorm{\Sigma_{XY}}{F}^2\Big]\\
	&\asymp n\cdot\big(\pnorm{\Sigma_{XY}}{F}^6 + n^2\pnorm{\Sigma_{XY}}{F}^2\big).
	\end{align*}
	A similar bound holds for $T_{22}$.
	\medskip
	
	\noindent (\textbf{Variance of $T_3$}) Note that
	\begin{align*}
	T_3 &= \sum_{k=1}^n\Big[\sum_{\ell \neq k}X_\ell^\top(X_kY_k^\top-\Sigma_{XY})Y_\ell\Big]^2= \sum_{k=1}^n\Big[\sum_{\ell\neq k} Z_\ell^\top\Big(H_{[11]}Z_kZ_k^\top H_{[22]} - G_{[12]}\Big)Z_\ell\Big]^2.
	\end{align*}
	Hence for any $i\in[n]$ and $j\in[p+q]$, we have
	\begin{align*}
	2^{-1}\frac{\partial T_3}{\partial Z_{ij}} &= \sum_{k=1}^n \Big[\sum_{\ell\neq k} Z_\ell^\top\Big(H_{[11]}Z_kZ_k^\top H_{[22]} - G_{[12]}\Big)Z_\ell\Big]\\
	& \qquad\qquad \times\sum_{\ell\neq k}\Big[\delta_{i\ell}e_j^\top(H_{[11]}Z_kZ_k^\top H_{[22]} - G_{[12]})Z_\ell \\
	&\qquad\qquad\qquad\qquad + \delta_{i\ell}Z_\ell^\top(H_{[11]}Z_kZ_k^\top H_{[22]} - G_{[12]})e_j\\
	&\qquad\qquad\qquad\qquad +\delta_{ik}\cdot Z_\ell^\top(H_{[11]}e_jZ_k^\top H_{[22]}+H_{[11]}Z_ke_j^\top H_{[22]})Z_\ell\Big]\\
	&= \Big[\sum_{\ell\neq i}Z_\ell^\top\Big(H_{[11]}Z_iZ_i^\top H_{[22]} - G_{[12]}\Big)Z_\ell\Big]\cdot \Big[\sum_{\ell\neq i}(Z_\ell^\top H_{[11]}e_j)(Z_i^\top H_{[22]}Z_\ell)\Big]\\
	&+\Big[\sum_{\ell\neq i}Z_\ell^\top\Big(H_{[11]}Z_iZ_i^\top H_{[22]} - G_{[12]}\Big)Z_\ell\Big]\cdot \Big[\sum_{\ell\neq i}(Z_\ell^\top H_{[11]}Z_i)(e_j^\top H_{[22]}Z_\ell)\Big]\\
	&+\sum_{k\neq i}\Big[\sum_{\ell\neq k}Z_\ell^\top\Big(H_{[11]}Z_kZ_k^\top H_{[22]} - G_{[12]}\Big)Z_\ell\Big]\cdot \Big[e_j^\top\Big(H_{[11]}Z_kZ_k^\top H_{[22]} - G_{[12]}\Big)Z_i\Big]\\
	&+\sum_{k\neq i}\Big[\sum_{\ell\neq k}Z_\ell^\top\Big(H_{[11]}Z_kZ_k^\top H_{[22]} - G_{[12]}\Big)Z_\ell\Big]\cdot \Big[Z_i^\top\Big(H_{[11]}Z_kZ_k^\top H_{[22]} - G_{[12]}\Big)e_j\Big]\\
	&\equiv (T_{31})_{(ij)} + (T_{32})_{(ij)} + (T_{33})_{(ij)} + (T_{34})_{(ij)}.
	\end{align*}
	For $T_{31}$, we have
	\begin{align*}
	&\E\pnorm{T_{31}}{F}^2 = n\cdot \E\Big[\sum_{\ell\neq 1}Z_\ell^\top\Big(H_{[11]}Z_1Z_1^\top H_{[22]} - G_{[12]}\Big)Z_\ell\Big]^2\cdot \biggpnorm{\sum_{\ell\neq 1}(Z_1^\top H_{[22]}Z_\ell)H_{[11]}Z_\ell}{}^2\\
	&\leq n\cdot\E^{1/2}\Big[\sum_{\ell\neq 1}Z_\ell^\top\Big(H_{[11]}Z_1Z_1^\top H_{[22]} - G_{[12]}\Big)Z_\ell\Big]^4 \cdot \E^{1/2}\biggpnorm{\sum_{\ell\neq 1}(Z_1^\top H_{[22]}Z_\ell)H_{[11]}Z_\ell}{}^4.
	\end{align*}
	The two terms can be bounded as follows:
	\begin{itemize}
		\item Let $U_1\equiv H_{[11]}Z_1Z_1^\top H_{[22]} - G_{[12]}$. By first conditioning on $Z_1$, we have
		\begin{align}\label{ineq:estimate_T31_1}
		\E\Big[\sum_{\ell\neq 1}Z_\ell^\top U_1Z_\ell\Big]^4 &\lesssim \E \big[n\cdot\tr(U_1) + n^{1/2}\cdot \pnorm{U_1}{F}\big]^4\nonumber\\
		&\lesssim_M n^4\cdot\pnorm{\Sigma_{XY}}{F}^4 + n^2\cdot \pnorm{\Sigma_X}{F}^4\pnorm{\Sigma_Y}{F}^4,
		\end{align}
		where the last inequality follows from Lemma \ref{lem:U_property}-(4).
		\item Let $M_{-1}\equiv \sum_{k\neq 1}Z_kZ_k^\top$. Then by first taking expectation with respect to $Z_1$, we have
		\begin{align*}
		&\E\biggpnorm{\sum_{\ell\neq 1}(Z_1^\top H_{[22]}Z_\ell)H_{[11]}Z_\ell}{}^4 \asymp \E\tr^2\Big(H_{[22]}M_{-1}H_{[11]}^2M_{-1}H_{[22]}\Big)\\
		&= \E\bigg[\sum_{k_1,k_2\neq 1} (Z_{k_1}^\top H_{[11]}^2Z_{k_2}) (Z_{k_1}^\top H_{[22]}^2Z_{k_2})\bigg]^2\\
		&= \E\sum_{k_1,k_2,k_3,k_4\neq 1}(Z_{k_1}^\top H_{[11]}^2Z_{k_2}) (Z_{k_1}^\top H_{[22]}^2Z_{k_2})(Z_{k_3}^\top H_{[11]}^2Z_{k_4}) (Z_{k_3}^\top H_{[22]}^2Z_{k_4}).
		\end{align*}
		\begin{enumerate}
			\item[(i)] When $k_1$-$k_4$ take one distinct value, there are up to $n$ such terms, and each term has contribution
			\begin{align*}
			&\E(Z_1^\top H_{[11]}^2Z_1)^2(Z_1^\top H_{[22]}^2Z_1)^2 \lesssim \E^{1/2}(Z_1^\top H_{[11]}^2Z_1)^4\cdot  \E^{1/2}(Z_1^\top H_{[22]}^2Z_1)^4\\
			&\lesssim \pnorm{H_{[11]}}{F}^4\pnorm{H_{[22]}}{F}^4 = \pnorm{\Sigma_X}{F}^4\pnorm{\Sigma_Y}{F}^4.
			\end{align*}
			\item[(ii)] When $k_1$-$k_4$ take two distinct values, there are up to $n^2$ such terms, and each term has contribution
			\begin{align*}
			&\E\Big[(Z_1^\top H_{[11]}^2Z_1)(Z_1^\top H_{[22]}^2Z_1)(Z_1^\top H_{[11]}^2Z_2)(Z_1^\top H_{[22]}^2Z_2)\Big]\\
			&\quad = \E\Big[(Z_1^\top H_{[11]}^2Z_1)(Z_1^\top H_{[22]}^2Z_1)(Z_1^\top H_{[11]}^2H_{[22]}^2Z_1)\Big]\\
			&\leq \E^{1/4}(Z_1^\top H_{[11]}^2Z_1)^4 \cdot \E^{1/4}(Z_1^\top H_{[22]}^2Z_1)^4\cdot \E^{1/2}(Z_1^\top H_{[11]}^2H_{[22]}^2Z_1)^2\\
			&\lesssim \pnorm{H_{[11]}}{F}^2\cdot\pnorm{H_{[22]}}{F}^2\cdot\big[\tr(H_{[11]}^2H_{[22]}^2) + \pnorm{H_{[11]}^2H_{[22]}^2}{F}\big]\\
			&\lesssim_M \pnorm{\Sigma_X}{F}^2\cdot \pnorm{\Sigma_Y}{F}^2 \cdot (\pnorm{\Sigma_{XY}}{F}^2 + \pnorm{\Sigma_{XY}}{F}).
			\end{align*}
			\item[(iii)] When $k_1$-$k_4$ take three distinct values in the form of $(k_1 = k_2)\neq k_3 \neq k_4$, there are up $n^3$ such terms, and each term has contribution
			\begin{align*}
			&\E\Big[(Z_1^\top H_{[11]}^2Z_1)(Z_1^\top H_{[22]}^2Z_1)(Z_2^\top H_{[11]}^2Z_3)(Z_3^\top H_{[22]}^2Z_2)\Big]\\
			&= \E\Big[(Z_1^\top H_{[11]}^2Z_1)(Z_1^\top H_{[22]}^2Z_1)\Big] \cdot \tr\big(H_{[11]}^2H_{[22]}^2\big)\\
			&\asymp \pnorm{H_{[11]}}{F}^2\pnorm{H_{[22]}}{F}^2\tr\big(H_{[11]}^2H_{[22]}^2\big)\lesssim_M \pnorm{\Sigma_X}{F}^2\pnorm{\Sigma_Y}{F}^2\pnorm{\Sigma_{XY}}{F}^2.
			\end{align*}
			\item[(iv)] When $k_1$-$k_4$ take three distinct values in the form of $(k_1 = k_3)\neq k_2\neq k_4$, there are up to $n^3$ such terms, and each term has contribution
			\begin{align*}
			&\E\Big[(Z_1^\top H_{[11]}^2Z_2)(Z_1^\top H_{[22]}^2Z_2)(Z_1^\top H_{[11]}^2Z_3)(Z_1^\top H_{[22]}^2Z_3)\Big]\\
			&= \E\big(Z_1^\top H_{[11]}^2H_{[22]}^2 Z_1\big)^2 \asymp_M \pnorm{\Sigma_{XY}}{F}^2 + \pnorm{\Sigma_{XY}}{F}^4.
			\end{align*}
			\item[(v)] When $k_1$-$k_4$ take four distinct values, there are up to $n^4$ such terms, and each term has contribution
			\begin{align*}
			&\E\Big[(Z_1^\top H_{[11]}^2Z_2)(Z_1^\top H_{[22]}^2Z_2)(Z_3^\top H_{[11]}^2Z_4)(Z_3^\top H_{[22]}^2Z_4)\Big]\\
			&= \Big[\E Z_1^\top H_{[11]}^2H_{[22]}^2 Z_1\Big]^2 = \tr^2\big(H_{[11]}^2H_{[22]}^2\big) \lesssim_M \pnorm{\Sigma_{XY}}{F}^4.
			\end{align*}
		\end{enumerate}
		Putting together estimates (i)-(v) yields that
		\begin{align}\label{ineq:estimate_T31_2}
		\notag&\E\biggpnorm{\sum_{\ell\neq 1}(Z_1^\top H_{[22]}Z_\ell)H_{[11]}Z_\ell}{}^4\\
		\notag&\lesssim_M n\cdot \tau_X^4\tau_Y^4 + n^2 \cdot  \tau_X^2\tau_Y^2\pnorm{\Sigma_{XY}}{F} + n^3\cdot \tau_X^2\tau_Y^2\pnorm{\Sigma_{XY}}{F}^2 + n^4\cdot \pnorm{\Sigma_{XY}}{F}^4\\
		&\lesssim n^2\cdot\tau_X^4\tau_Y^4 +  n^4\cdot \pnorm{\Sigma_{XY}}{F}^4.
		\end{align}
		The last inequality follows since (1) when the second term dominates the third term $n\pnorm{\Sigma_{XY}}{F}\lesssim 1$, the overall bound is dominated by the first term; (2) when the third term dominates the second term, the overall bound reduces to the sum of the first, third and fourth terms that can be written as a complete square.
	\end{itemize}
	Combining (\ref{ineq:estimate_T31_1}) and (\ref{ineq:estimate_T31_2}) yields that
	\begin{align}\label{ineq:T31}
	T_{31}&\lesssim_M n^3\cdot (n\pnorm{\Sigma_{XY}}{F}^2 + \pnorm{\Sigma_X}{F}^2\pnorm{\Sigma_Y}{F}^2)^2.
	\end{align}
	A similar bound holds for $T_{32}$.
	
	Next we bound $T_{33}$. For any $k\in[n]$, let $U_k\equiv H_{[11]}Z_kZ_k^\top H_{[22]} - G_{[12]}$ be a centered random matrix. Then by definition, we have
	\begin{align*}
	\E\pnorm{T_{33}}{F}^2 &= n\cdot \sum_{k_1,k_2\neq 1}\E\Big[ \Big(\sum_{\ell\neq k_1}Z_\ell^\top U_{k_1}Z_\ell\Big)\cdot\Big(\sum_{\ell\neq k_2}Z_\ell^\top U_{k_2}Z_\ell\Big)\cdot (Z_1^\top U_{k_1}^\top U_{k_2} Z_1)\Big]\\
	&=n\cdot \sum_{\ell_1,\ell_2=1}\sum_{k_1\notin\{\ell_1,1\},k_2\notin\{\ell_2,1\}}\E \big(Z_{\ell_1}^\top U_{k_1}Z_{\ell_1}\big)(Z_{\ell_2}^\top U_{k_2}Z_{\ell_2})(Z_1^\top U_{k_1}^\top U_{k_2} Z_1).
	\end{align*}
	We now divide the analysis of $T_{33}$ into the following four cases: $\ell_1 = \ell_2 = 1$ (Case 1), $\ell_1 = 1$, $\ell_2 \neq 1$ (Case 2), $(\ell_1=\ell_2)\neq 1$ (Case 3), $\ell_1\neq \ell_2\neq 1$ (Case 4).

	\begin{itemize}
		\item When $\ell_1 = \ell_2 = 1$, 
		\begin{align*}
		S_1&\equiv \sum_{k_1\neq 1,k_2\neq 1}\E \big(Z_1^\top U_{k_1}Z_1\big)(Z_1^\top U_{k_2}Z_1)(Z_1^\top U_{k_1}^\top U_{k_2} Z_1)\\
		&\asymp n\cdot \E(Z_1^\top U_2Z_1)^2\cdot(Z_1^\top U_2^\top U_2 Z_1) + n^2\cdot \E \big(Z_1^\top U_2Z_1\big)(Z_1^\top U_3Z_1)(Z_1^\top U_2^\top U_3 Z_1).
		\end{align*}
		\begin{enumerate}
			\item[(i)] By Lemma \ref{lem:U_property},
			\begin{align*}
			&\E(Z_1^\top U_2Z_1)^2\cdot(Z_1^\top U_2^\top U_2 Z_1) \leq \E^{1/2} (Z_1^\top U_2Z_1)^4 \cdot \E^{1/2}(Z_1^\top U_2^\top U_2 Z_1)^2\\
			&\lesssim \E\big(\tr(U_2) + \pnorm{U_2}{F}\big)^2 \cdot \pnorm{U_2}{F}^2 \lesssim \E\big[ \tr(U_2)^4 + \pnorm{U_2}{F}^4\big]\lesssim_M \pnorm{\Sigma_X}{F}^4\pnorm{\Sigma_Y}{F}^4.
			\end{align*}
			\item[(ii)] We have
			\begin{align}\label{ineq:var_T33_0}
			&\E \big(Z_1^\top U_2Z_1\big)(Z_1^\top U_3Z_1)(Z_1^\top U_2^\top U_3 Z_1)\nonumber\\
			& \leq \E^{1/4}\big(Z_1^\top U_2Z_1\big)^4\cdot \E^{1/4}(Z_1^\top U_3Z_1)\cdot \E^{1/2}(Z_1^\top U_2^\top U_3 Z_1)^2\nonumber\\
			&\lesssim \E \big(|\tr(U_2)| + \pnorm{U_2}{F}\big)\cdot \E\big(|\tr(U_3)| + \pnorm{U_2}{F}\big)\cdot \E\big(|\tr(U_2U_3)| + \pnorm{U_2U_3^\top}{F}\big)\nonumber\\
			&\lesssim (\pnorm{\Sigma_X}{F}\pnorm{\Sigma_Y}{F})^2\cdot \E[\pnorm{U_2}{F}\pnorm{U_3}{F}]\lesssim_M \pnorm{\Sigma_X}{F}^4\pnorm{\Sigma_Y}{F}^4.
			\end{align}
		\end{enumerate}
		Combining (i) and (ii) yields that in this first case,
		\begin{align}\label{ineq:var_T33_1}
		S_1	\lesssim_M n^2\pnorm{\Sigma_X}{F}^4\pnorm{\Sigma_Y}{F}^4.
		\end{align}
		\item When $\ell_1 = 1$ and $\ell_2\neq 1$, we have (by letting $\ell_2 = 2$)
		\begin{align*}
		S_2&\equiv \sum_{k_1\neq 1, k_2\notin \{1,2\}}\E (Z_1^\top U_{k_1}Z_1)(Z_2^\top U_{k_2}Z_2)(Z_1^\top U_{k_1}^\top U_{k_2} Z_1)\\
		&= \sum_{k_2\notin\{1,2\}} \E (Z_1^\top U_2Z_1)(Z_2^\top U_{k_2}Z_2)(Z_1^\top U_2^\top U_{k_2} Z_1)\\
		&\quad\quad +\sum_{k_1,k_2\notin\{1,2\}} \E \big(Z_1^\top U_{k_1}Z_1\big)(Z_2^\top U_{k_2}Z_2)(Z_1^\top U_{k_1}^\top U_{k_2} Z_1)\\
		&\asymp n\cdot \E (Z_1^\top U_2Z_1)(Z_2^\top U_3Z_2)(Z_1^\top U_2^\top U_3 Z_1)\\
		&\quad\quad\quad +n\cdot  \E \big(Z_1^\top U_3Z_1\big)(Z_2^\top U_3Z_2)(Z_1^\top U_3^\top U_3 Z_1)\\
		&\quad\quad\quad +n^2 \cdot \E \big(Z_1^\top U_3Z_1\big)(Z_2^\top U_4Z_2)(Z_1^\top U_3^\top U_4 Z_1).
		\end{align*}
		\begin{enumerate}
			\item[(i)] Using similar argument as in the previous case, we have
			\begin{align*}
			\E (Z_1^\top U_2Z_1)(Z_2^\top U_3Z_2)(Z_1^\top U_2^\top U_3 Z_1) &\lesssim_M \pnorm{\Sigma_X}{F}^4\pnorm{\Sigma_Y}{F}^4,\\
			\E\big(Z_1^\top U_3Z_1\big)(Z_2^\top U_3Z_2)(Z_1^\top U_3^\top U_3 Z_1) &\lesssim_M \pnorm{\Sigma_X}{F}^4\pnorm{\Sigma_Y}{F}^4.
			\end{align*}
			\item[(ii)] For the third summand, by taking expectation with respect to $Z_2$ and then $Z_1$, we have
			\begin{align*}
			&\E \big(Z_1^\top U_3Z_1\big)(Z_2^\top U_4Z_2)(Z_1^\top U_3^\top U_4 Z_1) = \E(Z_1^\top U_3Z_1)(Z_1^\top U_3^\top \tr(U_4)U_4 Z_1)\\
			&= \E\tr(U_3^2 \cdot \tr(U_4)U_4^\top) + \E\tr(U_3 U_3^\top  \tr(U_4)U_4) + \E \tr(U_3)\tr(U_4)\tr(U_3^\top U_4)\\
			&\lesssim_M \tr\Big[\E (U_3^2) \cdot \E\big(\tr(U_4)U_4^\top\big)\Big] + \tr\Big[\E (U_3 U_3^\top) \cdot \E\big(\tr(U_4)U_4\big)\Big] + \pnorm{\Sigma_{XY}}{F}^2.
			\end{align*}
			The last line follows from Lemma \ref{lem:U_property}-(6). By Lemma \ref{lem:U_property}-(1)(2)(3), we have
			\begin{align*}
			&\tr\Big[\E (U_3^2) \cdot \E\big(\tr(U_4)U_4^\top\big)\Big]\\
			&= \tr\Big[\Big(G_{[11]}G_{[22]} + \tr(G_{[12]})G_{[12]} + \kappa H_{[11]}(I\circ G_{[21]})H_{[22]}\Big)\\
			&\quad\quad\quad \cdot \Big(G_{[22]}G_{[11]} + G_{[21]}^2 + \kappa H_{[22]}(I\circ G_{[21]})H_{[11]}\Big)\Big]\\
			&= \tr(G_{[11]}G_{[22]}G_{[22]}G_{[11]}) + \tr(G_{[11]}G_{[22]}G_{[21]}^2)\\
			&\quad\quad\quad+ \tr(G_{[12]})\tr(G_{[12]}G_{[22]}G_{[11]}) + \tr(G_{[12]})\tr(G_{[12]}G_{[21]}^2)\\
			&\quad\quad\quad + \kappa\Big[\tr\big((G_{[11]}G_{[22]} + \tr(G_{[12]})G_{[12]})\cdot H_{[22]}(I\circ G_{[21]})H_{[11]}\big)\\
			&\quad\quad\quad + \tr\big((G_{[22]}G_{[11]} + G_{[21]}^2)\cdot H_{[11]}(I\circ G_{[21]})H_{[22]}\big)\Big]\\
			&\quad\quad\quad +\kappa^2\tr\big(H_{[11]}(I\circ G_{[21]})H_{[22]}\cdot H_{[22]}(I\circ G_{[21]})H_{[11]}\big)\\
			&\lesssim_M \pnorm{\Sigma_{XY}}{F}^2 + \pnorm{\Sigma_{XY}\Sigma_{YX}}{F}^2 + \pnorm{\Sigma_{XY}}{F}^4 + \pnorm{\Sigma_{XY}}{F}^2\pnorm{\Sigma_{XY}\Sigma_{YX}}{F}^2 + \pnorm{\Sigma_{XY}}{F}^2 + \pnorm{\Sigma_{XY}}{F}^4\\
			&\lesssim \pnorm{\Sigma_{XY}}{F}^2(1\vee \pnorm{\Sigma_{XY}}{F}^2),
			\end{align*} 
			and
			\begin{align*}
			&\tr\Big[\E (U_3 U_3^\top) \cdot \E\big(\tr(U_4)U_4\big)\Big]\\
			&= \tr\Big[\Big(G_{[12]}G_{[21]} + G_{[11]}\tr(G_{[22]}) + \kappa H_{[11]}(I\circ G_{[22]})H_{[11]}\Big)\\
			&\quad\quad\quad \cdot \Big(G_{[11]}G_{[22]} + G_{[12]}^2 + \kappa H_{[11]}(I\circ G_{[12]})H_{[22]}\Big)\Big]\\
			&= \tr(G_{[12]}G_{[21]}G_{[11]}G_{[22]}) + \tr(G_{[21]}G_{[12]}^3)\\
			&\quad\quad\quad + \tr(G_{[22]})\tr(G_{[11]}^2G_{[22]}) + \tr(G_{[22]})\tr(G_{[11]}G_{[12]}^2)\\
			&\quad\quad\quad + \kappa\Big[\tr\big((G_{[12]}G_{[21]} + G_{[11]}\tr(G_{[22]}))\cdot H_{[11]}(I\circ G_{[12]})H_{[22]}\big)\\
			&\quad\quad\quad + \tr\big((G_{[11]}G_{[22]} + G_{[12]}^2)\cdot H_{[11]}(I\circ G_{[22]})H_{[11]}\big)\Big]\\
			&\quad\quad\quad + \kappa^2 \cdot \tr(H_{[11]}(I\circ G_{[12]})H_{[22]}\cdot H_{[11]}(I\circ G_{[22]})H_{[11]})\\
			&\asymp_M \pnorm{\Sigma_{XY}\Sigma_{YX}}{F}^2 + \tr\big[(\Sigma_{XY}\Sigma_{YX})^3\big] + \pnorm{\Sigma_Y}{F}^2\pnorm{\Sigma_{XY}}{F}^2 +  \pnorm{\Sigma_Y}{F}^2\pnorm{\Sigma_{XY}\Sigma_{YX}}{F}^2 + \pnorm{\Sigma_Y}{F}^2\pnorm{\Sigma_{XY}}{F}^2\\
			&\asymp_M \pnorm{\Sigma_Y}{F}^2\pnorm{\Sigma_{XY}}{F}^2.
			\end{align*}
			Using Lemma \ref{lem:SigmaXY_F_bound}, we have
			\begin{align*}
			&\E \big(Z_1^\top U_3Z_1\big)(Z_2^\top U_4Z_2)(Z_1^\top U_3U_4^\top Z_1) \\
			&= \E(Z_1^\top U_3Z_1)(Z_1^\top U_3\tr(U_4)U_4^\top Z_1) \lesssim_M (\pnorm{\Sigma_X}{F}^2+\pnorm{\Sigma_Y}{F}^2)\pnorm{\Sigma_{XY}}{F}^2.
			\end{align*}
		\end{enumerate}
		Combining (i) and (ii) yields that in this second case,
		\begin{align}\label{ineq:var_T33_2}
		S_2
		&\lesssim_M n\cdot\pnorm{\Sigma_X}{F}^4\pnorm{\Sigma_Y}{F}^4 + n^2(\pnorm{\Sigma_X}{F}^2+\pnorm{\Sigma_Y}{F}^2)\pnorm{\Sigma_{XY}}{F}^2.
		\end{align}
		A similar estimate holds for the case $\ell_1\neq 1,\ell_2=1$. 
		\item When $(\ell_1=\ell_2) \neq 1$, we have (letting $\ell_1 = \ell_2 = 2$)
		\begin{align*}
		S_3&\equiv \sum_{k_1,k_2\notin \{1,2\}}\E \big(Z_2^\top U_{k_1}Z_2\big)(Z_2^\top U_{k_2}Z_2)(Z_1^\top U_{k_1}^\top U_{k_2} Z_1)\\
		&\asymp n\cdot \E \big(Z_2^\top U_3Z_2\big)^2(Z_1^\top U_3^\top U_3 Z_1) + n^2\cdot \E \big(Z_2^\top U_3Z_2\big)(Z_2^\top U_4Z_2)(Z_1^\top U_3^\top U_4 Z_1)\\
		&= n\cdot \E(Z_2^\top U_3 Z_2)^2\pnorm{U_3}{F}^2 + n^2\E(Z_2^\top U_3Z_2)(Z_2^\top U_4Z_2)\tr(U_3^\top U_4)\\
		&\asymp n\cdot\big[\E\pnorm{U_3}{F}^4 + \E\tr^2(U_3)\pnorm{U_3}{F}^2\big]\\
		&\quad\quad\quad+ n^2\cdot \Big[\E\tr^2(U_3U_4^\top) + \E\tr(U_3U_4)\tr(U_3U_4^\top)  + \E\tr(U_3)\tr(U_4)\tr(U_3^\top U_4)\Big].
		\end{align*}
		By Lemma \ref{lem:U_property} and Lemma \ref{lem:SigmaXY_F_bound},
		\begin{align*}
		\E\pnorm{U_3}{F}^4 + \E\tr^2(U_3)\pnorm{U_3}{F}^2&\lesssim_M \pnorm{\Sigma_X}{F}^4\pnorm{\Sigma_Y}{F}^4+\pnorm{\Sigma_{XY}}{F}^2 \pnorm{\Sigma_X}{F}^2\pnorm{\Sigma_Y}{F}^2\\
		&\asymp_M \pnorm{\Sigma_X}{F}^4\pnorm{\Sigma_Y}{F}^4,
		\end{align*}
		and 
		\begin{align*}
		&\E\tr^2(U_3U_4^\top) + \E\tr(U_3U_4)\tr(U_3U_4^\top)  + \E\tr(U_3)\tr(U_4)\tr(U_3^\top U_4)\\
		&\lesssim_M \pnorm{\Sigma_X}{F}^2\pnorm{\Sigma_Y}{F}^2+ \pnorm{\Sigma_X}{F}\pnorm{\Sigma_Y}{F}\pnorm{\Sigma_{XY}}{F}^2+\pnorm{\Sigma_{XY}}{F}^2\asymp \pnorm{\Sigma_X}{F}^2\pnorm{\Sigma_Y}{F}^2.
		\end{align*}
		Consequently, 
		\begin{align}\label{ineq:var_T33_3}
		S_3
		&\asymp_M n\pnorm{\Sigma_X}{F}^4\pnorm{\Sigma_Y}{F}^4 + n^2\pnorm{\Sigma_X}{F}^2\pnorm{\Sigma_Y}{F}^2.
		\end{align}
		\item When $\ell_1,\ell_2\neq 1$ and $\ell_1\neq \ell_2$, we have (letting $\ell_1 = 2$, $\ell_2 =3$)
		\begin{align*}
		S_4&\equiv \sum_{k_1\notin\{1,2\},k_2\notin\{1,3\}}\E \big(Z_2^\top U_{k_1}Z_2\big)(Z_3^\top U_{k_2}Z_3)(Z_1^\top U_{k_1}^\top U_{k_2} Z_1)\\
		&= \Big[\sum_{k_1=3,k_2=2} +\sum_{k_1=3,k_2\notin\{1,2,3\}}+\sum_{k_1\notin\{1,2,3\},k_2=2} +\sum_{k_1,k_2\notin\{1,2,3\}}\Big](\cdots)\\
		&\asymp \E \big(Z_2^\top U_3Z_2\big)(Z_3^\top U_2Z_3)(Z_1^\top U_3^\top U_2 Z_1) + n\cdot \E \big(Z_2^\top U_3Z_2\big)(Z_3^\top U_4Z_3)(Z_1^\top U_3^\top U_4 Z_1)\\
		&\qquad 
		+ n\cdot \E \big(Z_2^\top U_4Z_2\big)(Z_3^\top U_2Z_3)(Z_1^\top U_4^\top U_2 Z_1) \\
		&\qquad +n^2\cdot \E \big(Z_2^\top U_4Z_2\big)(Z_3^\top U_5Z_3)(Z_1^\top U_4^\top U_5 Z_1)\equiv \sum_{\ell} S_{4,\ell}.
		\end{align*}
		\begin{enumerate}
			\item[(i)] By (\ref{ineq:var_T33_0}), $S_{4,1}\lesssim_M \pnorm{\Sigma_X}{F}^4\pnorm{\Sigma_Y}{F}^4$. 
			\item[(ii)] By first taking expectation with respect to $Z_1$ and $Z_2$, we have by Lemma \ref{lem:U_property}-(4)(5)
			\begin{align*}
			S_{4,2}/n& = \E \tr(U_3)(Z_3^\top U_4Z_3)\tr(U_3^\top U_4)\\
			&\leq \E^{1/4}\tr^4(U_3)\cdot \E^{1/4}(Z_3^\top U_4Z_3)^4 \cdot \E^{1/2}\tr^2(U_3^\top U_4)\\
			&\lesssim_M  \pnorm{\Sigma_{XY}}{F} \pnorm{\Sigma_X}{F}^2\pnorm{\Sigma_Y}{F}^2 \lesssim (1\vee\pnorm{\Sigma_{XY}}{F}^2) \pnorm{\Sigma_X}{F}^2\pnorm{\Sigma_Y}{F}^2.
			\end{align*}
			\item[(iii)] By taking expectation with respect to $Z_1,Z_3$, a similar argument as in (ii) yields that
			\begin{align*}
			S_{4,3}/n &= \E(Z_2^\top U_4Z_2)\tr(U_2)\tr(U_4^\top U_2)\lesssim_M (1\vee\pnorm{\Sigma_{XY}}{F}^2) \pnorm{\Sigma_X}{F}^2\pnorm{\Sigma_Y}{F}^2
			\end{align*}
			\item[(iv)] By taking expectation with respect to $Z_1,Z_2,Z_3$, we have by Lemma \ref{lem:U_property}-(6)
			\begin{align*}
			S_{4,4}/n^2 = \E\tr(U_4)\tr(U_5)\tr(U_4^\top U_5) \lesssim_M \pnorm{\Sigma_{XY}}{F}^2.
			\end{align*}
		\end{enumerate}
		Putting together the estimates in (i)-(iv), we have in this case
		\begin{align}\label{ineq:var_T33_4}
		S_4
		&\lesssim_M \pnorm{\Sigma_X}{F}^4\pnorm{\Sigma_Y}{F}^4 + n\cdot (1\vee\pnorm{\Sigma_{XY}}{F}^2)\pnorm{\Sigma_X}{F}^2\pnorm{\Sigma_Y}{F}^2 + n^2\cdot\pnorm{\Sigma_{XY}}{F}^2.
		\end{align}
	\end{itemize}
	Finally combine (\ref{ineq:var_T33_1}) - (\ref{ineq:var_T33_4}) (and recall the notation $S_1$-$S_4$ defined at the end of each case) to conclude the bound for $\E\pnorm{T_{33}}{F}^2$:
	\begin{align*}
	\E\pnorm{T_{33}}{F}^2 &\lesssim n\cdot S_1 + n^2\cdot S_2 + n^2\cdot S_3 + n^3\cdot S_4\\
	&\lesssim_M n\cdot \big[ n^2\pnorm{\Sigma_X}{F}^4\pnorm{\Sigma_Y}{F}^4\big]\\
	&\quad\quad + n^2\cdot\big[n\cdot\pnorm{\Sigma_X}{F}^4\pnorm{\Sigma_Y}{F}^4 + n^2(\pnorm{\Sigma_X}{F}^2+\pnorm{\Sigma_Y}{F}^2)\pnorm{\Sigma_{XY}}{F}^2\big]\\
	&\quad\quad + n^2\cdot\big[ n\pnorm{\Sigma_X}{F}^4\pnorm{\Sigma_Y}{F}^4 + n^2\pnorm{\Sigma_X}{F}^2\pnorm{\Sigma_Y}{F}^2\big]\\
	&\quad\quad + n^3\cdot\big[\pnorm{\Sigma_X}{F}^4\pnorm{\Sigma_Y}{F}^4 + n\cdot (1\vee\pnorm{\Sigma_{XY}}{F}^2) \pnorm{\Sigma_X}{F}^2\pnorm{\Sigma_Y}{F}^2 + n^2\cdot \pnorm{\Sigma_{XY}}{F}^2\big]\\
	&\asymp n^3\cdot \pnorm{\Sigma_X}{F}^4\pnorm{\Sigma_Y}{F}^4 + n^4\cdot(1\vee\pnorm{\Sigma_{XY}}{F}^2)\pnorm{\Sigma_X}{F}^2\pnorm{\Sigma_Y}{F}^2 + n^5\cdot \pnorm{\Sigma_{XY}}{F}^2.
	\end{align*}
	A similar bound holds for $T_{34}$. The above bound dominates the bound (\ref{ineq:T31}) for $T_{31}$ and $T_{32}$ and hence is a valid bound for $T_3$.
	\medskip
	
	\noindent (\textbf{Variance of $T_4$ and $T_5$}) By direct calculation, we have
	\begin{align*}
	\var(T_4) \asymp n^5\var\big(X_1^\top \Sigma_X\Sigma_{XY}\Sigma_Y Y_1\big)\equiv n^5\var(U).
	\end{align*}
	Then
	\begin{align*}
	\E \pnorm{\nabla_{X_1} (X_1^\top \Sigma_{X}\Sigma_{XY}\Sigma_{Y} Y_1)}{}^2 = \E \pnorm{\nabla_{X_1} (\Sigma_{X}\Sigma_{XY}\Sigma_{Y} Y_1)}{}^2\lesssim_M \pnorm{\Sigma_{XY}}{F}^2.
	\end{align*}
	A similar bound holds for the gradient with respect to $Y_1$. By Gaussian-Poincar\'e inequality and a change of variable, we then have
	\begin{align*}
	\var(T_4)\lesssim_M \pnorm{\Sigma_{XY}}{F}^2.
	\end{align*}
	A similar bound holds for $\var(T_5)$.
	\medskip
	
	\noindent (\textbf{Variance of $T_6$}) Recall the notation $U_k\equiv H_{[11]}Z_kZ_k^\top H_{[22]}-G_{[12]}$. Then
	\begin{align*}
	T_6 
	&= (n-1) \sum_{k\neq \ell} (Z_k^\top G_{[12]}Z_k - \pnorm{\Sigma_{XY}}{F}^2)\cdot (Z_\ell^\top U_\ell Z_\ell).
	\end{align*}
	Then for any $i\in[n]$ and $j\in[p+q]$, we have
	\begin{align*}
	\frac{\partial T_6}{\partial Z_{ij}} &= (n-1)\cdot \sum_{k\neq \ell} \bigg[\delta_{ki}(Z_k^\top G_{[12]}e_j + e_j^\top G_{[12]}Z_k)\cdot (Z_\ell^\top U_k Z_\ell)\\
	&\quad + \delta_{\ell i} (Z_k^\top G_{[12]}Z_k - \pnorm{\Sigma_{XY}}{F}^2)(e_j^\top U_k Z_\ell + Z_\ell^\top U_k e_j)\\
	&\quad + \delta_{ki} (Z_k^\top G_{[12]}Z_k - \pnorm{\Sigma_{XY}}{F}^2)\cdot Z_\ell^\top\big[H_{[11]}(Z_ke_j^\top + e_jZ_k^\top)H_{[22]}\big]Z_\ell\bigg]\\
	&= (n-1)\cdot \sum_{k\neq i} \bigg[(e_j^\top G_{[12]}Z_i)(Z_k^\top U_iZ_k) + (Z_i^\top G_{[12]}e_j)(Z_k^\top U_i Z_k)\\
	&\quad\quad+ (Z_k^\top G_{[12]}Z_k - \pnorm{\Sigma_{XY}}{F}^2)(e_j^\top U_kZ_i) + (Z_k^\top G_{[12]}Z_k - \pnorm{\Sigma_{XY}}{F}^2)(Z_i^\top U_ke_j)\\
	&\quad\quad+ (Z_i^\top G_{[12]}Z_i - \pnorm{\Sigma_{XY}}{F}^2)(Z_k^\top H_{[11]}Z_ie_j^\top H_{[22]}Z_k)\\
	&\quad\quad + (Z_i^\top G_{[12]}Z_i - \pnorm{\Sigma_{XY}}{F}^2)(Z_k^\top H_{[11]}e_jZ_i^\top H_{[22]}Z_k)\bigg]\\
	&\equiv (n-1)\cdot(T_{61} + T_{62} + T_{63} + T_{64} + T_{65} + T_{66})_{ij}.
	\end{align*}
	For $T_{61}$, we have
	\begin{align*}
	\E\pnorm{T_{61}}{F}^2 
	&= n\cdot \sum_{k_1,k_2\neq 1}(Z_{k_1}^\top U_1Z_{k_1})(Z_{k_2}^\top U_1Z_{k_2})(Z_1^\top G_{[21]}G_{[12]}Z_1)\\
	&\asymp n\cdot \Big[n\cdot \E(Z_2^\top U_1Z_2)^2(Z_1^\top G_{[12]}G_{[21]}Z_1)\\
	&\qquad\qquad + n^2\cdot \E(Z_2^\top U_1Z_2)(Z_3^\top U_1Z_3)(Z_1^\top G_{[12]}G_{[21]}Z_1)\Big].
	\end{align*}
	The two terms can be bounded as follows:
	\begin{itemize}
		\item By Lemma \ref{lem:U_property}-(4), we have
		\begin{align*}
		&\E(Z_2^\top U_1Z_2)^2(Z_1^\top G_{[12]}G_{[21]}Z_1) \leq \E^{1/2}(Z_2^\top U_1Z_2)^4\cdot \E^{1/2}(Z_1^\top G_{[12]}G_{[21]}Z_1)^2\\
		&\lesssim  \E \big[\tr^2(U_1) + \pnorm{U_1}{F}^2\big]\cdot \pnorm{G_{[12]}}{F}^2 \asymp_M \pnorm{\Sigma_X}{F}^2\pnorm{\Sigma_Y}{F}^2\pnorm{\Sigma_{XY}}{F}^2+\pnorm{\Sigma_{XY}}{F}^4.
		\end{align*}
		\item Again by Lemma \ref{lem:U_property}-(4), we have
		\begin{align*}
		&\E(Z_2^\top U_1Z_2)(Z_3^\top U_1Z_3)(Z_1^\top G_{[12]}G_{[21]}Z_1) = \E\tr(U_1)^2(Z_1^\top G_{[12]}G_{[21]}Z_1)\\
		&\leq \E^{1/2}\tr(U_1)^4\cdot \E^{1/2}(Z_1^\top G_{[12]}G_{[21]}Z_1)^2 \lesssim \pnorm{\Sigma_{XY}}{F}^2\cdot \pnorm{\Sigma_{XY}}{F}^2 = \pnorm{\Sigma_{XY}}{F}^4.
		\end{align*}
	\end{itemize}
	In summary we have
	\begin{align}\label{ineq:var_T6_1}
	(n-1)\cdot\E\pnorm{T_{61}}{F}^2 \lesssim_M n^3\cdot\pnorm{\Sigma_X}{F}^2\pnorm{\Sigma_Y}{F}^2\pnorm{\Sigma_{XY}}{F}^2 + n^4\cdot  \pnorm{\Sigma_{XY}}{F}^4.
	\end{align}
	A similar bound holds for $T_{62}$.
	
	For $T_{63}$, we have
	\begin{align*}
	\E\pnorm{T_{63}}{F}^2 &= \E\sum_i\biggpnorm{\sum_{k\neq i}(Z_k^\top G_{[12]}Z_k - \tr(G_{[12]}))U_kZ_i}{}^2\\
	&= n\cdot \sum_{k_1,k_2\neq 1}\E (Z_{k_1}^\top G_{[12]}Z_{k_1} - \tr(G_{[12]}))(Z_{k_2}^\top G_{[12]}Z_{k_2} - \tr(G_{[12]}))(Z_1^\top U_{k_1}^\top U_{k_2} Z_1)\\
	&\asymp n\cdot \Big[n\cdot \E (Z_2^\top G_{[21]}Z_2 - \tr(G_{[12]}))^2\pnorm{U_2}{F}^2\\
	&\quad\quad + n^2\cdot \E(Z_2^\top G_{[12]}Z_2 - \tr(G_{[12]}))(Z_3^\top G_{[12]}Z_3 - \tr(G_{[12]}))\tr(U_2^\top U_3)\Big]. 
	\end{align*}
	The two terms can be bounded as follows:
	\begin{itemize}
		\item By Lemma \ref{lem:U_property}-(4),
		\begin{align*}
		&\E (Z_2^\top G_{[21]}Z_2 - \tr(G_{[12]}))^2\pnorm{U_2}{F}^2 \leq \E^{1/2}(Z_2^\top G_{[21]}Z_2 - \tr(G_{[12]}))^4\cdot \E^{1/2}\pnorm{U_2}{F}^4\\
		&\lesssim_M \pnorm{G_{[12]}}{F}^2\cdot \pnorm{\Sigma_{X}}{F}^2\pnorm{\Sigma_Y}{F}^2 = \pnorm{\Sigma_{X}}{F}^2\pnorm{\Sigma_Y}{F}^2\pnorm{\Sigma_{XY}}{F}^2.
		\end{align*}
		\item By Lemma \ref{lem:U_property}-(4)(5), we have
		\begin{align*}
		&\E(Z_2^\top G_{[12]}Z_2 - \tr(G_{[12]}))(Z_3^\top G_{[12]}Z_3 - \tr(G_{[12]}))\tr(U_2^\top U_3)\Big]\\
		& \leq \E^{1/4}(Z_2^\top G_{[12]}Z_2 - \tr(G_{[12]}))^4\cdot \E^{1/4}(Z_3^\top G_{[12]}Z_3 - \tr(G_{[12]}))^4\cdot \E^{1/2}\tr^2(U_2^\top U_3)\\
		&\lesssim_M  \pnorm{\Sigma_X}{F}\pnorm{\Sigma_Y}{F}\pnorm{\Sigma_{XY}}{F}^2.
		\end{align*}
	\end{itemize}
	In summary, we have
	\begin{align}\label{ineq:var_T6_2}
	(n-1)\cdot \E\pnorm{T_{63}}{F}^2 &\lesssim_M n^4\cdot\pnorm{\Sigma_X}{F}^2\pnorm{\Sigma_Y}{F}^2\pnorm{\Sigma_{XY}}{F}^2.
	\end{align}
	A similar bound holds for $T_{64}$.
	
	For $T_{65}$, we have
	\begin{align*}
	\E\pnorm{T_{65}}{F}^2 &= \sum_i \E\biggpnorm{\sum_{k\neq i}(Z_i^\top G_{[12]}Z_i - \tr(G_{[12]}))(Z_k^\top H_{[11]}Z_i)H_{[22]}Z_k}{}^2\\
	&= n\cdot \sum_{k_1,k_2\neq 1}\E(Z_1^\top G_{[12]}Z_1 - \tr(G_{[12]}))^2(Z_{k_1}^\top H_{[11]}Z_1)(Z_{k_2}^\top H_{[11]}Z_1)(Z_{k_1}^\top G_{[22]}Z_{k_2})\\
	&\asymp n\cdot\Big[n\cdot \E(Z_1^\top G_{[12]}Z_1 - \tr(G_{[12]}))^2(Z_2^\top H_{[11]}Z_1)^2(Z_2^\top G_{[22]}Z_2)\\
	&\quad\quad + n^2\cdot \E(Z_1^\top G_{[12]}Z_1 - \tr(G_{[12]}))^2(Z_2^\top H_{[11]}Z_1)(Z_3^\top H_{[11]}Z_1)(Z_2^\top G_{[22]}Z_3)\Big].
	\end{align*}
	The two terms can be bounded as follows
	\begin{itemize}
		\item By first taking expectation with respect to $Z_2$, we have
		\begin{align*}
		&\E(Z_1^\top G_{[12]}Z_1 - \tr(G_{[12]}))^2(Z_2^\top H_{[11]}Z_1)^2(Z_2^\top G_{[22]}Z_2)\\
		&\asymp \E(Z_1^\top G_{[12]}Z_1 - \tr(G_{[12]}))^2(Z_1^\top G_{[11]}Z_1)\tr(G_{[22]})\\
		&\leq \pnorm{\Sigma_Y}{F}^2\cdot \E^{1/2}(Z_1^\top G_{[12]}Z_1 - \tr(G_{[12]}))^4\cdot \E^{1/2}(Z_1^\top G_{[11]}Z_1)^2\\
		&\lesssim_M \pnorm{\Sigma_X}{F}^2\pnorm{\Sigma_Y}{F}^2\pnorm{\Sigma_{XY}}{F}^2.
		\end{align*}
		\item By taking expectation with respect to $Z_2,Z_3$, we have
		\begin{align*}
		&\E(Z_1^\top G_{[12]}Z_1 - \tr(G_{[12]}))^2(Z_2^\top H_{[11]}Z_1)(Z_3^\top H_{[11]}Z_1)(Z_2^\top G_{[22]}Z_3)\Big]\\
		&\leq \E^{1/2}(Z_1^\top G_{[12]}Z_1 - \tr(G_{[12]}))^4\cdot \E^{1/2}(Z_1^\top H_{[11]}G_{[22]}H_{[11]}Z_1)^2\\
		&\lesssim_M \pnorm{\Sigma_{XY}}{F}^2\cdot \pnorm{\Sigma_{XY}}{F}^2 = \pnorm{\Sigma_{XY}}{F}^4.
		\end{align*}
	\end{itemize}
	In summary we have
	\begin{align}\label{ineq:var_T6_3}
	(n-1)\cdot \E\pnorm{T_{65}}{F}^2 \lesssim_M n^3\cdot \pnorm{\Sigma_X}{F}^2\pnorm{\Sigma_Y}{F}^2\pnorm{\Sigma_{XY}}{F}^2 + n^4\cdot \pnorm{\Sigma_{XY}}{F}^4.
	\end{align}
	A similar bound holds for $T_{64}$. Combining (\ref{ineq:var_T6_1}) - (\ref{ineq:var_T6_3}) and Lemma \ref{lem:SigmaXY_F_bound} concludes the proof. 
	
	\medskip
	\noindent (\textbf{Variance of $T_7$}) Note that $T_7$ can be decomposed as $T_{7} = \kappa(T_{7,1} + T_{7,2} + T_{7,3})$, where
	\begin{align*}
	T_{7,1} &\equiv (n-1)^2\tr\Big(\Sigma^{1/2}
	\begin{pmatrix}
	0 & \sum_{j=1}^n X_jY_j^\top\\
	0 & 0
	\end{pmatrix}
	\Sigma^{1/2} \circ G_{[12]}
	\Big),\\
	T_{7,2} &\equiv (n-2)\sum_{j,k} \tr\Big(\Sigma^{1/2}
	\begin{pmatrix}
	0 & X_jY_j^\top\\
	0 & 0
	\end{pmatrix}
	\Sigma^{1/2}\circ \Sigma^{1/2}
	\begin{pmatrix}
	0 & X_kY_k^\top\\
	0 & 0
	\end{pmatrix}\Sigma^{1/2}
	\Big),\\
	T_{7,3} &\equiv \sum_{j=1}^n \tr\Big(\Sigma^{1/2}
	\begin{pmatrix}
	0 & X_jY_j^\top\\
	0 & 0
	\end{pmatrix}
	\Sigma^{1/2}\circ \Sigma^{1/2}
	\begin{pmatrix}
	0 & X_jY_j^\top\\
	0 & 0
	\end{pmatrix}\Sigma^{1/2}
	\Big).
	\end{align*}
	
	To control the variance of $T_{7,1}$, note that
	\begin{align*}
	\var(T_{7,1}) \asymp n^5 \var\Big(\tr\Big(\Sigma^{1/2}
	\begin{pmatrix}
	0 & XY^\top\\
	0 & 0
	\end{pmatrix}\Sigma^{1/2}\circ G_{[12]}\Big)\Big) = n^5 \cdot \var(f(Z)),
	\end{align*}
	where $f(Z) \equiv \tr\big(H_{[11]}ZZ^\top H_{[22]}\circ G_{[12]}\big) = \sum_i (e_i^\top H_{[11]}Z)(e_i^\top H_{[22]}Z)(G_{[12]})_{ii}$. Using $\nabla f(Z) = (A_1+A_2)Z$ with $A_1\equiv \sum_i H_{[11]}e_ie_i^\top H_{[22]} (G_{[12]})_{ii}$ and  $A_2\equiv \sum_i H_{[22]}e_ie_i^\top H_{[11]} (G_{[12]})_{ii}$, it suffices to bound $\pnorm{A_1}{F}$ and $\pnorm{A_2}{F}$ so that by the Poincar\'e inequality, $\var(T_{7,1}) \lesssim n^5(\pnorm{A_1}{F}^2\vee \pnorm{A_2}{F}^2)$. To bound $\pnorm{A_1}{F}^2$, note that $(A_1)_{k\ell} = \sum_i (H_{11})_{ki}(H_{22})_{\ell i}(G_{[12]})_{ii}$, so that
	\begin{align*}
	\pnorm{A_1}{F}^2 &= \sum_{k,\ell} \sum_{i_1,i_2} (H_{[11]})_{k,i_1}(H_{[11]})_{k,i_2}(H_{[22]})_{\ell,i_1}(H_{[22]})_{\ell,i_2} (G_{[12]})_{i_1,i_1}(G_{[12]})_{i_2,i_2}\\
	&= \sum_{i,j} (G_{[11]})_{ij}(G_{[22]})_{ij}(G_{[12]})_{ii}(G_{[12]})_{jj} \leq \sqrt{\sum_{i,j}(G_{[11]})^2_{ij}(G_{[12]})^2_{ii}}\cdot \sqrt{\sum_{i,j}(G_{[22]})^2_{ij}(G_{[12]})^2_{jj}}\\
	&= \sqrt{\sum_i(G_{[11]}^2)_{ii}(G_{[12]})^2_{ii}}\cdot \sqrt{\sum_i(G^2_{[22]})_{ii}(G_{[12]})^2_{ii}} \lesssim_M \pnorm{G_{[12]}}{F}^2 \lesssim_M \pnorm{\Sigma_{XY}}{F}^2.
	\end{align*}
	A similar bound holds for $\pnorm{A_2}{F}^2$ so we have $\var(T_{7,1}) \lesssim n^5\pnorm{\Sigma_{XY}}{F}^2$.
	
	Next we bound $\var(T_{7,2})$. Using again $\Sigma^{1/2}[0, XY^\top; 0, 0]\Sigma^{1/2} = H_{[11]}ZZ^\top H_{[22]}$, we have
	\begin{align*}
	\var(T_{7,2}) \asymp n^2\cdot \var\Big(\sum_i \Big(\sum_j (e_i^\top H_{[11]}Z_j)(e_i^\top H_{[22]}Z_j)\Big)^2\Big) \equiv n^2 \var(f(Z))
	\end{align*}
	for $Z = [Z_1^\top,\ldots, Z_{n}^\top]^\top\in\R^{n\times(p+q)}$. Using
	\begin{align*}
	\frac{\partial f(Z)}{\partial Z_k} = \sum_i \Big(\sum_j (e_i^\top H_{[11]}Z_j)(e_i^\top H_{[22]}Z_j)\Big) (H_{[11]}e_ie_i^\top H_{[22]} + H_{[22]}e_ie_i^\top H_{[11]})Z_k,
	\end{align*}
	we have by symmetry (and with $M_i \equiv H_{[11]}e_ie_i^\top H_{[22]}$)
	\begin{align*}
	\E \pnorm{\nabla f(Z)}{}^2 &\lesssim \E \sum_k \pnorm{\sum_i \Big(\sum_j (e_i^\top H_{[11]}Z_j)(e_i^\top H_{[22]}Z_j)\Big) H_{[11]}e_ie_i^\top H_{[22]} Z_k}{}^2\\
	&= n\cdot  \E \pnorm{\sum_i \Big(\sum_j (e_i^\top H_{[11]}Z_j)(e_i^\top H_{[22]}Z_j)\Big) H_{[11]}e_ie_i^\top H_{[22]} Z_1}{}^2\\
	&\equiv n\cdot \E \pnorm{\sum_i \Big(\sum_j Z_j^\top M_i\Big) M_iZ_1}{}^2 = n\cdot \E \sum_{i_1,i_2} \Big(\sum_j Z_j^\top M_{i_1} Z_j\Big)\Big(\sum_j Z_j^\top M_{i_2}Z_j\Big)Z_1^\top M_{i_1}^\top M_{i_2}Z_1\\
	&=n\cdot \Big[\E\sum_{i_1,i_2} (Z_1^\top M_{i_1}Z_1)(Z_1^\top M_{i_2}Z_1)(Z_1^\top M_{i_1}^\top M_{i_2}Z_1)\\
	&\quad\quad\quad + (2n-1)\E\sum_{i_1,i_2} (Z_1^\top M_{i_1}Z_1)(Z_2^\top M_{i_2}Z_2)(Z_1^\top M_{i_1}^\top M_{i_2} Z_1)\\
	&\quad\quad\quad + (n^2-2n) \E \sum_{i_1,i_2} (Z_2^\top M_{i_1}Z_2)(Z_3^\top M_{i_2}Z_3)(Z_1^\top M_{i_1}^\top M_{i_2}Z_1) \Big]\\
	&\lesssim_M n\pnorm{\Sigma_X}{F}^4\pnorm{\Sigma_Y}{F}^4 + n^2\pnorm{\Sigma_{XY}}{F}^2\pnorm{\Sigma_X}{F}^2\pnorm{\Sigma_Y}{F}^2 + n^3\pnorm{\Sigma_{XY}}{F}^2.
	\end{align*}
	Here the last inequality follows from Cauchy-Schwarz and some simple estimates based on Lemma \ref{lem:gaussian_4moment}. Along with the Poincar\'e inequality, this concludes the desired variance bound for $T_{7,2}$. The variance bound for $T_{7,3}$ is straightforward so we omit it. 
\end{proof}

\begin{lemma}\label{lem:U_property} 
	Let $U\equiv H_{[11]}ZZ^\top H_{[22]} - G_{[12]}$ with $Z\sim\mathcal{N}(0,I_{p+q})$, and $U_1,U_2,\ldots$ be independent copies of $U$. Then the following hold.
	\begin{enumerate}
		\item $\E U^2 = G_{[11]}G_{[22]} + \tr(G_{[12]})G_{[12]}  +\kappa H_{[11]} (I\circ G_{[21]}) H_{[22]}$.
		\item $\E U^\top U =G_{[21]}G_{[12]} + G_{[22]}\tr(G_{[11]}) +\kappa H_{[11]} (I\circ G_{[21]}) H_{[22]}$, $\E UU^\top = G_{[12]}G_{[21]}+G_{[11]}\tr(G_{[22]})+ \kappa H_{[11]}(I\circ G_{[22]}) H_{[11]}$.
		\item $\E \tr(U)U = G_{[11]}G_{[22]} + G_{[12]}^2+ \kappa H_{[11]}(I\circ G_{[12]})H_{[22]}$.
		\item $\E \tr^4(U)\lesssim_M \pnorm{\Sigma_{XY}}{F}^4$ and $\E\pnorm{U}{F}^4\lesssim_M \pnorm{\Sigma_X}{F}^4\pnorm{\Sigma_Y}{F}^4$.
		\item $\E \tr^2(U_1^\top U_2) \lesssim \pnorm{\Sigma_X^2}{F}^2\pnorm{\Sigma_Y^2}{F}^2$, $\E \tr^2(U_1U_2)\lesssim \pnorm{\Sigma_{XY}\Sigma_{YX}}{F}^4$.
		\item $\E\tr(U_1)\tr(U_2)\tr(U_1^\top U_2)\lesssim_M  \pnorm{\Sigma_{XY}}{F}^2$.
	\end{enumerate}
\end{lemma}
\begin{proof}
	\noindent (1). By Lemma \ref{lem:gaussian_4moment},
	\begin{align*}
	\E U^2 &= \E\big(H_{[11]}ZZ^\top H_{[22]}H_{[11]}ZZ^\top H_{[22]}\big) - G_{[12]}^2\\
	&= H_{[11]}\cdot \E(ZZ^\top)(Z^\top H_{[22]}H_{[11]}Z)\cdot H_{[22]} - G^2_{[12]}\\
	&= H_{[11]}\big(H_{[11]}H_{[22]} + H_{[22]}H_{[11]} + \tr(G_{[12]})\cdot I + \kappa I\circ G_{[21]}\big)H_{[22]} - G^2_{[12]}\\
	&= G_{[11]}G_{[22]} + \tr(G_{[12]})G_{[12]} +\kappa H_{[11]} (I\circ G_{[21]}) H_{[22]}.
	\end{align*}
	\noindent (2). By Lemma \ref{lem:gaussian_4moment},
	\begin{align*}
	\E U^\top U &= \E\big(H_{[22]}ZZ^\top H_{[11]}^2 ZZ^\top H_{[22]}\big) - G_{[21]}G_{[12]}\\
	&= H_{[22]}\cdot \E(ZZ^\top)(Z^\top H_{[11]}^2 Z)\cdot H_{[22]} - G_{[21]}G_{[12]}\\
	&= H_{[22]}\big(2H_{[11]}^2 + \tr(G_{[11]})\cdot I +\kappa I\circ G_{[11]}\big)H_{[22]} - G_{[21]}G_{[12]}\\
	&= G_{[21]}G_{[12]} + G_{[22]}\tr(G_{[11]})  +\kappa H_{[11]} (I\circ G_{[21]}) H_{[22]},
	\end{align*}
	and similarly
	\begin{align*}
	\E UU^\top & = G_{[12]}G_{[21]}+G_{[11]} \tr(G_{[22]})+ \kappa H_{[11]}(I\circ G_{[22]}) H_{[11]}.
	\end{align*}
	\noindent (3). We have
	\begin{align*}
	\E \tr(U)U &= \E\Big(Z^\top G_{[12]}Z - \tr(G_{[12]})\Big)\Big(H_{[11]}ZZ^\top H_{[22]} - G_{[12]}\Big)\\
	&= H_{[11]} \cdot \E(ZZ^\top)(Z^\top G_{[12]}Z)\cdot H_{[22]} - \tr(G_{[12]})G_{[12]}\\
	&= H_{[11]}\big(H_{[11]}H_{[22]} + H_{[22]}H_{[11]} + \tr(G_{[12]})\cdot I + \kappa I \circ G_{[12]}\big)H_{[22]} - \tr(G_{[12]})G_{[12]}\\
	&= G_{[11]}G_{[22]} + G_{[12]}^2 + \kappa H_{[11]}(I\circ G_{[12]})H_{[22]}.
	\end{align*}
	\noindent (4). We have
	\begin{align*}
	\notag\E\tr(U)^4 &= \E \big[Z^\top G_{[12]}Z - \E(Z^\top G_{[12]}Z)\big]^4\lesssim \pnorm{G_{[12]}}{F}^4\lesssim_M \pnorm{\Sigma_{XY}}{F}^4,\\
	\notag\E\pnorm{U}{F}^4 &\leq \E\pnorm{H_{[11]}ZZ^\top H_{[22]}}{F}^4 + \pnorm{G_{[12]}}{F}^4\\
	& \lesssim_M \E\pnorm{H_{[11]}Z}{}^4\pnorm{H_{[22]}Z}{}^4 +\pnorm{\Sigma_{XY}}{F}^4\lesssim  \pnorm{\Sigma_X}{F}^4\pnorm{\Sigma_Y}{F}^4 + \pnorm{\Sigma_{XY}}{F}^4.
	\end{align*}
	The claim follows by Lemma \ref{lem:SigmaXY_F_bound}.
	
	\noindent (5).  By definition, we have
	\begin{align*}
	\tr(U_1^\top U_2) &= \tr\bigg(\big[H_{[22]}Z_1 Z_1^\top H_{[11]} - G_{[21]}\big]\big[H_{[11]}Z_2 Z_2^\top H_{[22]} - G_{[12]}\big]\bigg)\\
	&= (Z_1^\top G_{[11]}Z_2)(Z_1^\top G_{[22]}Z_2)  - Z_1^\top G_{[11]}G_{[22]}Z_1 -  Z_2^\top G_{[11]}G_{[22]}Z_2 + \tr(G_{[11]}G_{[22]})\\
	&= Z_1^\top\Big(G_{[11]}Z_2Z_2^\top G_{[22]} - G_{[11]}G_{[22]}\Big)Z_1 - \tr(G_{[11]}Z_2Z_2^\top G_{[22]} - G_{[11]}G_{[22]})\\
	&\equiv Z_1^\top A(Z_2) Z_1 - \E_1(Z_1^\top A(Z_2)Z_1),
	\end{align*}
	where $\E_i$ denotes expectation only with respect to $Z_i$. Hence,
	\begin{align*}
	\E\tr^2(U_1^\top U_2) &= \E_2\var_1(Z_1^\top A(Z_2)Z_1) \asymp \E\pnorm{A(Z)}{F}^2\\
	&= \E\pnorm{G_{[11]}ZZ^\top G_{[22]} - G_{[11]}G_{[22]}}{F}^2\\
	&= \E\Big[(Z^\top G_{[11]}^2 Z)(Z^\top G_{[22]}^2 Z) - 2Z^\top G_{[11]}^2G_{[22]}^2 Z + \tr(G_{[11]}^2G_{[22]}^2)\Big]\\
	&= \tr(G_{[11]}^2G_{[22]}^2) + \pnorm{G_{[11]}}{F}^2 \pnorm{G_{[22]}}{F}^2 + \kappa \tr(G_{[11]}^2\circ G_{[22]}^2)\\
	&\asymp_M \pnorm{G_{[11]}}{F}^2 \pnorm{G_{[22]}}{F}^2 = \pnorm{\Sigma_X^2}{F}^2\pnorm{\Sigma_Y^2}{F}^2 \lesssim_M \pnorm{\Sigma_X}{F}^2\pnorm{\Sigma_Y}{F}^2.
	\end{align*}
	On the other hand, as
	\begin{align*}
	\tr(U_1 U_2) &= \tr\bigg(\big[H_{[11]}Z_1 Z_1^\top H_{[22]} - G_{[12]}\big]\big[H_{[11]}Z_2 Z_2^\top H_{[22]} - G_{[12]}\big]\bigg)\\
	&= (Z_1^\top G_{[12]}Z_2)(Z_1^\top G_{[21]} Z_2)  - Z_1^\top G_{[12]}^2Z_1 -  Z_2^\top G_{[12]}^2Z_2 + \tr(G_{[12]}^2)\\
	&= Z_1^\top\Big(G_{[12]}Z_2Z_2^\top G_{[12]} - G_{[12]}^2\Big)Z_1 - \tr(G_{[12]}Z_2Z_2^\top G_{[12]} - G_{[12]}^2)\\
	&\equiv Z_1^\top B(Z_2) Z_1 - \E_1(Z_1^\top B(Z_2)Z_1),
	\end{align*}
	so using a similar argument as in the case for $\E \tr^2(U_1^\top U_2)$, we have
	\begin{align*}
	\E\tr^2(U_1 U_2) 
	&= \E\pnorm{G_{[12]}ZZ^\top G_{[12]} - G_{[12]}^2}{F}^2\\
	&= \E\Big[(Z^\top G_{[12]}G_{[21]} Z)(Z^\top G_{[21]}G_{[12]} Z) - 2Z^\top G_{[12]}G_{[21]}^2G_{[12]} Z + \tr(G_{[12]}^2G_{[21]}^2)\Big]\\
	&= \tr(G_{[12]}^2G_{[21]}^2) + \pnorm{G_{[12]}}{F}^4 +\kappa \tr(G_{[12]}G_{[21]}\circ G_{[21]}G_{[12]})\lesssim \pnorm{G_{[12]}}{F}^4 =\pnorm{\Sigma_{XY}\Sigma_{YX}}{F}^4.
	\end{align*}
	
	\noindent (6). Using the formula for $\tr(U_1^\top U_2)$ in the proof of (5) above, 
	\begin{align*}
	&\E\tr(U_1)\tr(U_2)\tr(U_1^\top U_2)\\
	&= \E(Z_1^\top G_{[12]}Z_1 - \tr(G_{[12]}))(Z_2^\top G_{[12]}Z_2 - \tr(G_{[12]}))(Z_1^\top A(Z_2) Z_1 - \tr(A(Z_2)))\\
	&= \E(Z_2^\top G_{[12]}Z_2 - \tr(G_{[12]}))\cdot \big[\tr(G_{[12]}A(Z_2) + G_{[12]}A(Z_2)^\top) +\kappa \tr(G_{[12]}\circ A(Z_2))\big]\quad \hbox{(by Lemma \ref{lem:gaussian_4moment}-(1))}\\
	&= \E(Z_2^\top G_{[12]}Z_2 - \tr(G_{[12]}))\\
	&\quad\quad\quad\cdot\big[Z_2^\top\big(G_{[11]}(G_{[12]} + G_{[21]})G_{[22]}\big)Z_2 - \tr\big(G_{[11]}(G_{[12]} + G_{[21]})G_{[22]}\big)\big]\\
	&\quad\quad + \kappa\cdot \E(Z_2^\top G_{[12]}Z_2 - \tr(G_{[12]})) \big(Z_2^\top G_{[11]} (I\circ G_{[12]}) G_{[22]}Z_2 - \tr(G_{[11]} (I\circ G_{[12]})  G_{[22]})\big)\\
	&\asymp \tr\big[G_{[12]}\cdot \big(G_{[11]}(G_{[12]} + G_{[21]})G_{[22]}\big) + \kappa G_{[12]}\circ  \big(G_{[11]}(G_{[12]} + G_{[21]})G_{[22]}\big)\big]\\
	&\quad\quad\quad + \kappa \cdot \tr\big[(G_{[12]}\cdot G_{[11]} (I\circ G_{[12]}) G_{[22]} + \kappa\cdot G_{[12]}\circ (G_{[11]} (I\circ G_{[12]}) G_{[22]})\big]\\
	& \lesssim_M \pnorm{\Sigma_{XY}\Sigma_{YX}}{F}^2 + \tr\big[(\Sigma_{XY}\Sigma_{YX})^3\big] + \pnorm{\Sigma_{XY}}{F}^2\\
	&\asymp_M \pnorm{\Sigma_{XY}}{F}^2.
	\end{align*}
	
	The proof is complete.
\end{proof}

\subsubsection{Third moment bound}

\begin{proof}[Proof of Proposition \ref{prop:var_psi_1}-(2)]
	For any $i\in[n]$, let $M_{-i}\equiv \sum_{k\neq i}X_kY_k^\top$. By definition of $\Delta_i \psi_1(\bm{X},\bm{Y})$, we have
	\begin{align*}
	&\E|\Delta_i\psi_1(\bm{X},\bm{Y})|^3 \lesssim \E^{3/4}(X_i^\top M_{-i}Y_i - X_i'^\top M_{-i}Y_i')^4\\
	&\lesssim \E^{3/4}\big[X_i^\top M_{-i}Y_i - \E_{(X_i,Y_i)}(X_i^\top M_{-i}Y_i)\big]^4 + \E^{3/4}\big[X_i'^\top M_{-i}Y_i' - \E_{(X_i',Y_i')}(X_i'^\top M_{-i}Y_i')\big]^4\\
	&= \E^{3/4}\big[Z_i^\top \tilde{M}_{-i}Z_i - \E_{Z_i}(Z_i^\top \tilde{M}_{-i}Z_i)\big]^4 + \E^{3/4}\big[Z_i'^\top \tilde{M}_{-i}Z_i' - \E_{Z_i'}(Z_i'^\top M_{-i}Z_i')\big]^4.
	\end{align*}
	Here $\tilde{M}_{-i}$ is the random matrix defined by
	\begin{align*}
	\tilde{M}_{-i}\equiv \Sigma^{1/2}
	\begin{pmatrix}
	0 & M_{-i}\\
	0 & 0
	\end{pmatrix}\Sigma^{1/2}.
	\end{align*}
	Hence we have
	\begin{align*}
	\E|\Delta_i\psi_1(\bm{X},\bm{Y})|^3 &\lesssim \E^{3/4}\pnorm{\tilde{M}_{-i}}{F}^4\lesssim_M \E^{3/4}\pnorm{M_{-i}}{F}^4\\
	&\lesssim_M n^{3/2} \tr^{3/2}(\Sigma_X)\tr^{3/2}(\Sigma_Y) + n^3\pnorm{\Sigma_{XY}}{F}^3,
	\end{align*}
	using Lemma \ref{lem:estimate_M} in the last step.
\end{proof}

\begin{lemma}\label{lem:estimate_M}
	For any $i\in[n]$, let $M_{-i}\equiv \sum_{k\neq i}X_kY_k^\top$. Then the mean and variance of $\pnorm{M_{-i}}{F}^2$ satisfy:
	\begin{align*}
	\E \pnorm{M_{-i}}{F}^2 &=n(n-1)\pnorm{\Sigma_{XY}}{F}^2 + (n-1)\tr(\Sigma_X)\tr(\Sigma_Y) + \kappa (n-1)\tr(H_{[11]}\circ H_{[22]}),\\
	\var \big(\pnorm{M_{-i}}{F}^2\big) &\lesssim_M n\cdot\big[\tau_X^2\tau_Y^2(n \vee \tau_X^2 \vee \tau_Y^2) + n^2\pnorm{\Sigma_{XY}}{F}^2\big].
	\end{align*}
	Consequently, $\E \pnorm{M_{-i}}{F}^4\lesssim n^2 \tr^2(\Sigma_X)\tr^2(\Sigma_Y) + n^4\pnorm{\Sigma_{XY}}{F}^4$.
\end{lemma}
\begin{proof}
	Without loss of generality we let $i=1$. Note that
	\begin{align*}
	\pnorm{M_{-1}}{F}^2 & = \sum_{k_1,k_2\neq 1} (X_{k_1}^\top X_{k_2})(Y_{k_1}^\top Y_{k_2})= \sum_{k_1,k_2\neq 1} (Z_{k_1}^\top H_{[11]}Z_{k_2})(Z_{k_1}^\top H_{[22]}Z_{k_2}).
	\end{align*}
	Hence the mean formula is
	\begin{align*}
	\E\pnorm{M_{-1}}{F}^2 &= (n-1)\E(Z_1^\top H_{[11]}Z_1)(Z_1^\top H_{[22]}Z_1) + (n-1)(n-2)\E(Z_1^\top H_{[11]}Z_2)(Z_1^\top H_{[22]}Z_2)\\
	&= (n-1)\cdot \big[2\tr(G_{[12]}) + \tr(H_{[11]})\tr(H_{[22]}) + \kappa \tr(H_{[11]}\circ H_{[22]})\big] + (n-1)(n-2)\tr(G_{[12]})\\
	&= n(n-1)\pnorm{\Sigma_{XY}}{F}^2 + (n-1)\tr(\Sigma_X)\tr(\Sigma_Y) + \kappa (n-1)\tr(H_{[11]}\circ H_{[22]}).
	\end{align*}
	To bound the variance of $\pnorm{M_{-1}}{F}^2$, write
	\begin{align*}
	T(Z)\equiv \pnorm{M_{-1}}{F}^2 = \sum_{k_1,k_2\neq 1} (Z_{k_1}^\top H_{[11]}Z_{k_2})(Z_{k_1}^\top H_{[22]}Z_{k_2}).
	\end{align*}
	So for any $i \neq 1$,
	\begin{align*}
	\nabla_{Z_i} T(Z)&= \sum_{k_1,k_2\neq 1}\Big[(\delta_{k_1, i}\cdot H_{[11]}Z_{k_2}+\delta_{k_2, i}\cdot H_{[11]}Z_{k_1})(Z_{k_1}^\top H_{[22]}Z_{k_2})\\
	&\quad\quad +  (\delta_{k_1,i}\cdot H_{[22]}Z_{k_2} + \delta_{k_2,i}H_{[22]}Z_{k_1})(Z_{k_1}^\top H_{[11]}Z_{k_2})\Big]\\
	&= 2\cdot\sum_{k\neq 1} (Z_k^\top H_{[22]}Z_i)H_{[11]}Z_k + 2\cdot\sum_{k\neq 1} (Z_k^\top H_{[11]}Z_i)H_{[22]}Z_k \equiv [T_{11} + T_{12}]_i.
	\end{align*}
	Hence
	\begin{align*}
	\E\pnorm{T_{11}}{F}^2 &\asymp n\cdot \E\biggpnorm{\sum_{k\neq 1} (Z_k^\top H_{[22]}Z_2)H_{[11]}Z_k }{}^2\\
	&= n\cdot \sum_{k_1,k_2\neq 1}\E(Z_{k_1}^\top H_{[22]}Z_2)(Z_{k_2}^\top H_{[22]}Z_2)(Z_{k_1}^\top G_{[11]}Z_{k_2})\\
	&\asymp n\cdot\Big[\E(Z_2^\top H_{[22]}Z_2)^2(Z_2^\top G_{[11]}Z_2) + n\cdot \E(Z_2^\top H_{[22]}Z_2)(Z_3^\top H_{[22]}Z_2)(Z_2^\top G_{[11]}Z_3)\\
	&\quad + n\cdot \E(Z_3^\top H_{[22]}Z_2)^2(Z_3^\top G_{[11]}Z_3) + n^2\E(Z_3^\top H_{[22]}Z_2)(Z_4^\top H_{[22]}Z_2)(Z_3^\top G_{[11]}Z_4)\Big]\\
	&\lesssim_M n(\tau_X^2\tau_Y^4 + n\tau_X^2\tau_Y^2 + n^2\pnorm{\Sigma_{XY}}{F}^2).
	\end{align*}
	Flipping $X$, $Y$ for $T_{12}$ to conclude that
	\begin{align*}
	\pnorm{\nabla_Z T(Z)}{F}^2\lesssim_M n\cdot\big[\tau_X^2\tau_Y^2(n \vee \tau_X^2 \vee \tau_Y^2) + n^2\pnorm{\Sigma_{XY}}{F}^2\big],
	\end{align*}
	completing the proof.
\end{proof}

\subsection{Proof of Proposition \ref{prop:var_psi_2}}

We only prove the claim for the term $X_i^\top\Sigma_{XY}Y_j$, the bound for the other term is completely analogous. With some abuse of notation, we write in the proof $\psi_2(\bm{X},\bm{Y}) \equiv \sum_{i\neq j} X_i^\top\Sigma_{XY}Y_j$.

\noindent (1). For any $i\in[n]$, as
\begin{align*}
\psi_2(\bm{X},\bm{Y}) = \sum_{\substack{i_1\neq i_2\\ i_1,i_2\neq i}} X_{i_1}^\top\Sigma_{XY}Y_{i_2} + \sum_{i_1\neq i} X_{i_1}^\top\Sigma_{XY}Y_i + \sum_{i_2\neq i}X_i^\top \Sigma_{XY}Y_{i_2}, 
\end{align*}
we have
\begin{align*}
\Delta_i\psi_2(\bm{X},\bm{Y})&\equiv \psi_2(\bm{X},\bm{Y}) - \psi_2(\bm{X}^{\{i\}},\bm{Y}^{\{i\}})\\
& = \Big(\sum_{j\neq i} X_j\Big)^\top \Sigma_{XY}(Y_i-Y_i') + (X_i - X_i')^\top\Sigma_{XY}\Big(\sum_{j\neq i}Y_j\Big).
\end{align*}
Hence for any $A\subset[n]$ such that $i\notin A$,
\begin{align*}
\Delta_i\psi_2(\bm{X}^A,\bm{Y}^A) &= \Big(\sum_{j\in A} X_j' + \sum_{j\notin A\cup\{i\}}X_j\Big)^\top \Sigma_{XY}(Y_i-Y_i')\\
&\quad\quad\quad + (X_i - X_i')^\top\Sigma_{XY}\Big(\sum_{j\in A}Y_j' + \sum_{j\notin A\cup\{i\}}Y_j\Big).
\end{align*}
This implies that
\begin{align*}
&\E'\big[\Delta_i\psi_2(\bm{X},\bm{Y})\Delta_i\psi_2(\bm{X}^A,\bm{Y}^A)\big]\\
& = \E'\Big[\Big(\sum_{j\neq i} X_j\Big)^\top \Sigma_{XY}(Y_i-Y_i')\cdot \Big(\sum_{j\in A} X_j' + \sum_{j\notin A\cup\{i\}}X_j\Big)^\top \Sigma_{XY}(Y_i-Y_i')\Big] \\
&\quad+ \E'\Big[\Big(\sum_{j\neq i} X_j\Big)^\top \Sigma_{XY}(Y_i-Y_i')\cdot (X_i - X_i')^\top\Sigma_{XY}\Big(\sum_{j\in A}Y_j' + \sum_{j\notin A\cup\{i\}}Y_j\Big)\Big] \\
&\quad + \E'\Big[(X_i - X_i')^\top\Sigma_{XY}\Big(\sum_{j\neq i}Y_j\Big)\cdot \Big(\sum_{j\in A} X_j' + \sum_{j\notin A\cup\{i\}}X_j\Big)^\top \Sigma_{XY}(Y_i-Y_i')\Big]\\
&\quad + \E'\Big[(X_i - X_i')^\top\Sigma_{XY}\Big(\sum_{j\neq i}Y_j\Big)\cdot (X_i - X_i')^\top\Sigma_{XY}\Big(\sum_{j\in A}Y_j' + \sum_{j\notin A\cup\{i\}}Y_j\Big)\Big]\\
&= \Big(\sum_{j\neq i} X_j\Big)^\top \Sigma_{XY}(Y_iY_i^\top +\Sigma_Y)\Sigma_{YX}\Big(\sum_{j\notin A\cup\{i\}}X_j\Big)\\
&\quad\quad + \Big(\sum_{j\neq i} X_j\Big)^\top \Sigma_{XY}(Y_iX_i^\top + \Sigma_{YX})\Sigma_{XY}\Big(\sum_{j\notin A\cup\{i\}}Y_j\Big)\\
&\quad\quad + \Big(\sum_{j\neq i} Y_j\Big)^\top\Sigma_{YX}(X_iY_i^\top +\Sigma_{XY})\Sigma_{YX}\Big(\sum_{j\notin A\cup\{i\}}X_j\Big)\\
&\quad\quad + \Big(\sum_{j\neq i} Y_j\Big)^\top\Sigma_{YX}(X_iX_i^\top +\Sigma_X)\Sigma_{XY}\Big(\sum_{j\notin A\cup\{i\}}Y_j\Big)\\
&\equiv S_1(i,A) + S_2(i,A) + S_3(i,A) + S_4(i,A).
\end{align*}
Hence by using (\ref{ineq:sum_comb_A_1}) we have
\begin{align*}
S_1&\equiv \frac{1}{2}\sum_{A\subsetneq[n]}\sum_{i\notin A}\frac{S_1(i,A)}{{n\choose |A|}(n-|A|)}\\
&= \frac{1}{2}\sum_{\substack{1\leq i\leq n\\j\neq i,k\neq i}} \big[X_j^\top\Sigma_{XY}(Y_iY_i^\top + \Sigma_Y)\Sigma_{YX}X_k\big]\cdot \Big[\sum_{A\cap\{i,k\}=\emptyset}\frac{1}{{n\choose |A|}(n-|A|)}\Big]\\
&=\frac{1}{4}\sum_{\substack{1\leq i\leq n\\j\neq i,k\neq i}} \big[X_j^\top\Sigma_{XY}(Y_iY_i^\top + \Sigma_Y)\Sigma_{YX}X_k\big].
\end{align*}
Using similar calculations, we have
\begin{align*}
S_2&\equiv \frac{1}{2}\sum_{A\subsetneq[n]}\sum_{i\notin A}\frac{S_2(i,A)}{{n\choose |A|}(n-|A|)} = \frac{1}{4}\sum_{\substack{1\leq i\leq n\\j\neq i,k\neq i}} \big[X_j^\top\Sigma_{XY}(Y_iX_i^\top + \Sigma_{YX})\Sigma_{XY}Y_k\big],\\
S_3&\equiv\frac{1}{2}\sum_{A\subsetneq[n]}\sum_{i\notin A}\frac{S_3(i,A)}{{n\choose |A|}(n-|A|)} = \frac{1}{4}\sum_{\substack{1\leq i\leq n\\j\neq i,k\neq i}} \big[Y_j^\top\Sigma_{YX}(X_iY_i^\top + \Sigma_{XY})\Sigma_{YX}X_k\big],\\
S_4&\equiv\frac{1}{2}\sum_{A\subsetneq[n]}\sum_{i\notin A}\frac{S_4(i,A)}{{n\choose |A|}(n-|A|)} = \frac{1}{4}\sum_{\substack{1\leq i\leq n\\j\neq i,k\neq i}} \big[Y_j^\top\Sigma_{YX}(X_iX_i^\top + \Sigma_{X})\Sigma_{XY}Y_k\big].
\end{align*}
As $\E(\psi_2|\bm{X},\bm{Y}) = \sum_{i=1}^4 S_i$, the proof is complete by invoking Lemma \ref{lem:psi_2_variances} below.
\medskip

\noindent (2). By definition, we have for each $i\in[n]$,
\begin{align*}
&\E|\Delta_i \psi_2(\bm{X},\bm{Y})|^3 \lesssim \E\Big|\Big(\sum_{j\neq i}X_j)^\top \Sigma_{XY}(Y_i - Y_i')\Big|^3 + \E\Big|\Big(\sum_{j\neq i}Y_j)^\top \Sigma_{YX}(X_i - X_i')\Big|^3\\
&\lesssim \Big(n^2 \E (X_1^\top \Sigma_{XY} Y_2)^4\Big)^{3/4}\leq  n^{3/2}\E^{1/2}(Z_1^\top G_{[12]}Z_2)^6 \asymp n^{3/2}\pnorm{G_{[12]}}{F}^3 \asymp n^{3/2}\pnorm{\Sigma_{XY}}{F}^3.
\end{align*}
The proof is complete.\qed

\begin{lemma}\label{lem:psi_2_variances}
	Recall the terms $S_1$-$S_4$ defined in the proof of Proposition \ref{prop:var_psi_2}. The following variance bound holds.
	\begin{align*}
	\max_{1\leq i\leq 4}\var(S_i) \lesssim_M n^4\pnorm{\Sigma_{XY}}{F}^4.
	\end{align*}
\end{lemma}
\begin{proof}
	We only give a bound for $S_1$; bounds for $S_2$-$S_4$ follow from completely analogous arguments. By definition we have $S_1 = 4^{-1}(S_{11} + S_{12})$, where
	\begin{align*}
	S_{11} &\equiv \sum_{\substack{1\leq i\leq n\\j\neq i,k\neq i}} (X_j^\top\Sigma_{XY}Y_i)(Y_i^\top\Sigma_{YX}X_k)= \sum_{k=1}^n Z_k^\top\Big(\sum_{\ell\neq k}G_{[21]}Z_\ell\Big)\Big(\sum_{\ell\neq k}Z_\ell^\top G_{[12]}\Big)Z_k,\\
	S_{12} &\equiv  \sum_{\substack{1\leq i\leq n\\j\neq i,k\neq i}} X_j^\top\Sigma_{XY}\Sigma_Y\Sigma_{YX}X_k = \sum_{j,k=1}^n\big[(n-2) + \delta_{jk}\big]X_j^\top\Sigma_{XY}\Sigma_Y\Sigma_{YX}X_k.
	\end{align*}
	To control the variance of $S_{11}$, note that for any $\ell \in[n]$, 
	\begin{align*}
	\nabla_{X_\ell} S_{11}
	& = \sum_{\substack{1\leq i\leq n\\j\neq i,k\neq i}} \delta_{j\ell} \Sigma_{XY}Y_i(Y_i^\top \Sigma_{YX} X_k) + \delta_{k\ell} \Sigma_{XY} Y_i (Y_i^\top \Sigma_{YX} X_j)\\
	& =  2\sum_{i\neq \ell, j\neq i} \Sigma_{XY} Y_i(Y_i^\top \Sigma_{YX} X_j) = 2 \sum_j \Sigma_{XY} M_{Y,-\{j,\ell\}} \Sigma_{YX} X_j.
	\end{align*}
	Here $M_{Y,-A}\equiv \sum_{k\notin A} Y_kY_k^\top$ for $A\subset [n]$. Then
	\begin{align*}
	\E \pnorm{\nabla_{X_\ell} S_{11}}{}^2
	& = 4 \sum_j \E\tr\big(\Sigma_{XY} M_{Y,-\{j,\ell\}} \Sigma_{YX} \Sigma_{XY} M_{Y,-\{j,\ell\}} \Sigma_{YX}\big)\\
	& = 4 \sum_j \sum_{k_1,k_2\neq j,\ell} \E (Y_{k_1}^\top \Sigma_{YX}\Sigma_{XY}Y_{k_2})^2 \lesssim_M n^3 \pnorm{\Sigma_{XY}\Sigma_{YX}}{F}^2\leq n^3 \pnorm{\Sigma_{XY}}{F}^4.
	\end{align*}
	On the other hand, as
	\begin{align*}
	\nabla_{Y_\ell} S_{11}&= \sum_{\substack{1\leq i\leq n\\j\neq i,k\neq i}} \delta_{i\ell} \big[\Sigma_{YX} X_j (X_k^\top \Sigma_{XY} Y_i)+\Sigma_{YX} X_k (X_j^\top \Sigma_{XY}Y_i)\big]=2\sum_{j,k\neq \ell} \Sigma_{YX}X_j (X_k^\top \Sigma_{XY}Y_\ell),
	\end{align*}
	we have
	\begin{align*}
	\E \pnorm{\nabla_{Y_\ell} S_{11}}{}^2 &= 4 \sum_{ \substack{j_1,j_2,k_1,k_2 \neq \ell,\\ \abs{\{j_1,j_2,k_1,k_2\}}=1,2}} \E \big(X_{j_1}^\top \Sigma_{XY} \Sigma_Y \Sigma_{YX} X_{j_2}\big)\big(X_{k_1}^\top \Sigma_{XY}  \Sigma_{YX} X_{k_2}\big)\\
	& \lesssim_M n \pnorm{\Sigma_{XY}}{F}^4+n^2 \big(\pnorm{\Sigma_{XY}}{F}^4+\pnorm{\Sigma_{XY}\Sigma_{YX}}{F}^2\big)\asymp n^2 \pnorm{\Sigma_{XY}}{F}^4.
	\end{align*}
	Putting together the two estimates and using the Poincar\'e inequality and a change of variable yield that
	\begin{align*}
	\var(S_{11}) \lesssim_M \sum_\ell \big(\E \pnorm{\nabla_{X_\ell} S_{11}}{}^2+\E \pnorm{\nabla_{Y_\ell} S_{11}}{}^2\big)\lesssim_M n^4\pnorm{\Sigma_{XY}}{F}^4.
	\end{align*}
	The variance bound for $S_{12}$ follows from a similar and simpler argument so we omit the proof. Putting together the estimates for $S_{11}$ and $S_{12}$ concludes the proof for $S_1$.
\end{proof}

\subsection{Proof of Proposition \ref{prop:var_psi_3}}

We will bound the $X$ component, and the bound for the $Y$ component is completely analogous. 

\noindent (1). By direct calculation, we have for each $i\in[n]$,
\begin{align*}
\Delta_i \psi_3(\bm{X},\bm{Y}) =\tau_X^{-2}\pnorm{\Sigma_{XY}}{F}^2 (\pnorm{X_i}{}^2 - \pnorm{X_i'}{}^2).
\end{align*}                                                                   
Hence for any $A\subset [n]$ such that $i\notin A$, we have also $\Delta_i \psi_3(\bm{X}^A,\bm{Y}^A) = \tau_X^{-2}\pnorm{\Sigma_{XY}}{F}^2(\pnorm{X_i}{}^2 - \pnorm{X_i'}{}^2)$, thus by Lemma \ref{ineq:sum_comb_A_1}
\begin{align*}
T_{\psi_3} &= \sum_{A\subsetneq [n]}\sum_{i\notin A}\frac{\tau_X^{-4}\pnorm{\Sigma_{XY}}{F}^4 (\pnorm{X_i}{}^2 - \pnorm{X_i'}{}^2)^2}{{n\choose |A|}(n-|A|)}\\
&= \sum_{i=1}^n \tau_X^{-4}\pnorm{\Sigma_{XY}}{F}^4 (\pnorm{X_i}{}^2 - \pnorm{X_i'}{}^2)^2\cdot \Big[\sum_{A\subsetneq[n]:i\notin A}\frac{1}{{n\choose |A|}(n-|A|)}\Big]\\
&\asymp \sum_{i=1}^n \tau_X^{-4}\pnorm{\Sigma_{XY}}{F}^4(\pnorm{X_i}{}^2 - \pnorm{X_i'}{}^2)^2.
\end{align*}
This means
\begin{align*}
\var\big[\E(T_{\psi_3}|\bm{X},\bm{Y})\big] &\leq \var(T_{\psi_3}) = n\tau_X^{-8}\pnorm{\Sigma_{XY}}{F}^8 \cdot \var\Big[(\pnorm{X}{}^2 - \pnorm{X'}{}^2)^2\Big]\\
&\lesssim n\tau_X^{-8}\cdot\pnorm{\Sigma_{XY}}{F}^8\cdot \E\big(\pnorm{X}{}^2 - \tr(\Sigma_X)\big)^4\lesssim_M n\tau_X^{-4}\cdot \pnorm{\Sigma_{XY}}{F}^8.
\end{align*}

\noindent (2). We have
\begin{align*}
\E|\Delta_j \psi_3(\bm{X},\bm{Y})|^3 &=\tau_X^{-6}\pnorm{\Sigma_{XY}}{F}^6 \cdot \E\abs{\pnorm{X}{}^2 - \pnorm{X'}{}^2}^3 \lesssim \tau_X^{-6}\pnorm{\Sigma_{XY}}{F}^6 \E\abs{\pnorm{X}{}^2 - \tr(\Sigma_X)}^3\\
&\lesssim \tau_X^{-6}\pnorm{\Sigma_{XY}}{F}^6 \cdot \pnorm{\Sigma_X}{F}^3\lesssim_M \tau_X^{-3}\pnorm{\Sigma_{XY}}{F}^6.
\end{align*}
The proof is complete.\qed

\section{Proof of Theorem \ref{thm:non_null_clt_kernel}}\label{sec:proof_kernel}

For notational simplicity, we only prove the case $f_X = f_Y = f$; the general case only requires formal modifications. 

\subsection{Notation and preliminaries}
Let
\begin{align*}
U_{f,\gamma_X}(x_1,x_2)&\equiv f\big(\pnorm{x_1-x_2}{}/\gamma_X\big)-\E f\big(\pnorm{x_1-X}{}/\gamma_X\big)\\
&\qquad\qquad-\E f\big(\pnorm{X-x_2}{}/\gamma_X\big)+ \E f\big(\pnorm{X-X'}{}/\gamma_X\big),\\
V_{f,\gamma_Y}(y_1,y_2)&\equiv f\big(\pnorm{y_1-y_2}{}/\gamma_Y\big)-\E f\big(\pnorm{y_1-Y}{}/\gamma_Y\big)\\
&\qquad\qquad-\E f\big(\pnorm{Y-y_2}{}/\gamma_Y\big)+ \E f\big(\pnorm{Y-Y'}{}/\gamma_Y\big).
\end{align*}
The sample kernel distance covariance $\dcov_\ast^2(\bm{X},\bm{Y};f,\gamma)$ in (\ref{def:dcov_empirical_kernel}) can be represented as a 4-th order $U$-statistics using $U_{f,\gamma_X},V_{f,\gamma_Y}$ similar to Proposition \ref{prop:hoef_decomp}.
\begin{proposition}\label{prop:hoef_decomp_kernel}
	The following holds:
	\begin{align*}
	\dcov_\ast^2(\bm{X},\bm{Y};f,\gamma) & = \frac{1}{\binom{n}{4}}\sum_{i_1<\cdots<i_4} h_{f,\gamma}\big(Z_{i_1},Z_{i_2},Z_{i_3},Z_{i_4}\big),
	\end{align*}
	where
	\begin{align*}
	&h_{f,\gamma}(Z_1,Z_2,Z_3,Z_4) = \frac{1}{4!}\sum_{(i_1,\ldots,i_4)\in \sigma(1,2,3,4)} \bigg[ U_{f,\gamma_X}(X_{i_1},X_{i_2})V_{f,\gamma_Y}(Y_{i_1},Y_{i_2})\\
	&\qquad\qquad+U_{f,\gamma_X}(X_{i_1},X_{i_2})V_{f,\gamma_Y}(Y_{i_3},Y_{i_4})- 2 U_{f,\gamma_X}(X_{i_1},X_{i_2})V_{f,\gamma_Y}(Y_{i_1},Y_{i_3})\bigg].
	\end{align*}
\end{proposition}
\begin{proof}
	The proof follows the same lines as that of Proposition \ref{prop:hoef_decomp}; we omit repetitive details.
\end{proof}

We need the following basic properties of $f_{\sqrt{\cdot}}(\cdot)$ inherited from Assumption \ref{assump:kernel}.
\begin{lemma}\label{lem:f_sqrt_basic}
	Under Assumption \ref{assump:kernel} for $f(\cdot)$, conclusions (1)-(4) therein hold for $f_{\sqrt{\cdot}}(\cdot)$ with different constants $c'_\epsilon, C'_\epsilon, \mathfrak{q}'$. Furthermore, $c'_\epsilon, C'_\epsilon$ only depend on $c_\epsilon,C_\epsilon$ and $\epsilon$, and $\mathfrak{q}'$ only depends on $\mathfrak{q}$.
\end{lemma}
\begin{proof}
	The claims hold by definitions of $f^{(\ell)}_{\sqrt{\cdot}}$, $\ell=0,1,2,3,4$. 
\end{proof}

A useful consequence of Assumption \ref{assump:kernel} is the following.
\begin{lemma}\label{lem:kernel_der_int}
	Suppose that Assumption \ref{assump:kernel} holds, and the spectrum of $\Sigma_X,\Sigma_Y$ is contained in $[1/M,M]$ for some $M>1$. Then for any integer $s \in \N$, there exists $K\equiv K(s,\mathfrak{q})>0$ such that if $p\wedge q\geq K$, 
	\begin{align*}
	\max_{1\leq \ell\leq 4}\sup_{\rho^2 \in [1/M,M]}\sup_{t \in[0,1]} \E\bigabs{f_{\sqrt{\cdot}}^{(\ell)}\big(\rho^2(1+tW)\big)}^{s} \lesssim_M 1,
	\end{align*}
	where $W$ is either $L_X(X_1,X_2)$ or $L_Y(Y_1,Y_2)$ defined in (\ref{def:L_X}).
\end{lemma}
\begin{proof}
	We only prove the case for $W=L_X(X_1,X_2)$, $\ell=1$. By Lemma \ref{lem:f_sqrt_basic}, for all $x\geq 0$, $\abs{f_{\sqrt{\cdot}}^{(1)}(x)}\leq C_f (1\vee x^{-\mathfrak{q'}})$ for some $C_f>0$ and $\mathfrak{q}'\in\mathbb{N}$. This means that 
	\begin{align*}
	&\E\bigabs{f_{\sqrt{\cdot}}^{(1)}\big(\rho^2(1+tL_X(X_1,X_2))\big)}^{s}\lesssim 1+ \rho^{-2s\mathfrak{q'}} \E \big(1+tL_X(X_1,X_2)\big)^{-s\mathfrak{q'}}\\
	&\lesssim 1+ \rho^{-2s\mathfrak{q'}} \min\big\{
	(1-t)^{-s\mathfrak{q'}},t^{-s\mathfrak{q'}} \E \big(\tau_X^{-2}\pnorm{X_1-X_2}{}^2\big)^{-s\mathfrak{q'}}\}.
	\end{align*}
	Now use (\ref{ineq:chi_inv_moment}) to conclude.
\end{proof}

We define the analogue of the $h$ function in (\ref{def:h}): For $u\geq -1$, let
\begin{align*}
h_{f,\rho}(u)&\equiv\frac{f_{\sqrt{\cdot}}(\rho^2(1+u))-f_{\sqrt{\cdot}}(\rho^2)-f_{\sqrt{\cdot}}'(\rho^2)\rho^2 u}{2 f_{\sqrt{\cdot}}'(\rho^2)\rho^2}\\ 
& =\frac{\rho^2}{2f'_{\sqrt{\cdot}}(\rho^2)}\cdot u^2 \int_0^1 f''_{\sqrt{\cdot}}\big(\rho^2(1+su)\big)(1-s)\,\d{s}.
\end{align*}
When $f(x) = x$ so that $f_{\sqrt{\cdot}}(x)=\sqrt{x}$ and $\rho=1$, we recover the $h$ function in (\ref{def:h}).

\begin{lemma}\label{lem:property_h_fcn_kernel}
	Suppose that Assumption \ref{assump:kernel} holds and $1/M\leq \rho^2 \leq M$ for some $M>1$. Then 
	\begin{align*}
	\abs{h_{f,\rho}(u)}\lesssim_M u^2,\quad \abs{h_{f,\rho}'(u)}\lesssim_M \big|u\cdot F_{f,\rho;2}(u)\big|,
	\end{align*}
	where
	\begin{align*}
	F_{f,\rho;2}(u)\equiv u^{-1} \int_0^{u} f_{\sqrt{\cdot}}''(\rho^2(1+t))\,\d{t} = \int_0^1 f_{\sqrt{\cdot}}''(\rho^2(1+us))\,\d{s}.
	\end{align*}
	Furthermore
	\begin{align*}
	&h_{f,\rho}(u)= \frac{f_{\sqrt{\cdot}}''(\rho^2)\rho^2}{4 f_{\sqrt{\cdot}}'(\rho^2)}\cdot u^2 + \frac{\rho^4}{2 f_{\sqrt{\cdot}}'(\rho^2)}\cdot u^3 \int_0^1 f^{(3)}_{\sqrt{\cdot}}\big(\rho^2(1+su)\big)\frac{(1-s)^2}{2}\,\d{s}\\
	& = \frac{f_{\sqrt{\cdot}}''(\rho^2)\rho^2}{4 f_{\sqrt{\cdot}}'(\rho^2)}\cdot u^2 +\frac{f_{\sqrt{\cdot}}^{(3)}(\rho^2)\rho^4}{12 f_{\sqrt{\cdot}}'(\rho^2)}\cdot u^3+ \frac{\rho^6}{2 f_{\sqrt{\cdot}}'(\rho^2)}\cdot u^4\int_0^1 f^{(4)}_{\sqrt{\cdot}}\big(\rho^2(1+su)\big)\frac{(1-s)^3}{6}\,\d{s}\\
	&\equiv \frac{f_{\sqrt{\cdot}}''(\rho^2)\rho^2}{4 f_{\sqrt{\cdot}}'(\rho^2)}\cdot u^2 + h_{f,\rho;3}(u).
	\end{align*}
\end{lemma}
\begin{proof}
	The bound for $h_{f,\rho}$ follows by the following consideration: if $u \in [-1,-1/2]$, it holds by Lemma \ref{lem:f_sqrt_basic}-Conclusion (2) and the fact $\rho \asymp_M 1$ that $|h_{f,\rho}(u)|\lesssim_M 1\lesssim_M u^2$;  if $u>-1/2$, if follows from the lower bound on $|f'_{\sqrt{\cdot}}|$ and upper bound on $|f^{''}_{\sqrt{\cdot}}|$ that
	\begin{align*}
	|h_{f,\rho}(u)|&=\Big|\frac{\rho^2}{2f'_{\sqrt{\cdot}}(\rho^2)}\cdot u^2 \int_0^1 f''_{\sqrt{\cdot}}\big(\rho^2(1+su)\big)(1-s)\,\d{s}\Big|\\
	&\lesssim_M u^2\cdot \sup_{u>-1/2}\sup_{s\in[0,1]} \big|f''_{\sqrt{\cdot}}\big(\rho^2(1+su)\big)\big| \lesssim_M u^2.
	\end{align*}
	On the other hand, by Taylor expansion
	\begin{align*}
	&f_{\sqrt{\cdot}}(\rho^2(1+u))-f_{\sqrt{\cdot}}(\rho^2)-f_{\sqrt{\cdot}}'(\rho^2)\rho^2 u = \int_{0}^{\rho^2 u}f_{\sqrt{\cdot}}''(\rho^2+t)(\rho^2u-t)\,\d{t}\\
	& = \frac{f_{\sqrt{\cdot}}''(\rho^2)}{2}(\rho^2 u)^2+ (\rho^2 u)^3 \int_0^1 f^{(3)}_{\sqrt{\cdot}}\big(\rho^2(1+su)\big)\frac{(1-s)^2}{2}\,\d{s}\\
	& = \frac{f_{\sqrt{\cdot}}''(\rho^2)}{2}(\rho^2 u)^2+ \frac{f_{\sqrt{\cdot}}^{(3)}(\rho^2)}{6}(\rho^2 u)^3+(\rho^2 u)^4 \int_0^1 f^{(4)}_{\sqrt{\cdot}}\big(\rho^2(1+su)\big)\frac{(1-s)^3}{6}\,\d{s}. 
	\end{align*}
	The first equality shows that
	\begin{align*}
	\abs{h_{f,\rho}'(u)}=\biggabs{\frac{1}{2f'_{\sqrt{\cdot}}(\rho^2)}\int_0^{\rho^2 u} f_{\sqrt{\cdot}}''(\rho^2+t)\,\d{t}}\lesssim_M \biggabs{\int_0^{u} f_{\sqrt{\cdot}}''(\rho^2(1+t))\,\d{t}} = \big|u\cdot F_{f,\rho;2}(u)\big|,
	\end{align*}
	as desired.
\end{proof}

\subsection{Residual estimates}
Let
\begin{align*}
R_{X;f}(x_1,x_2)\equiv h_{f,\rho_X}(L_X(x_1,x_2)),\, R_{Y;f}(y_1,y_2)\equiv h_{f,\rho_Y}(L_Y(y_1,y_2)).
\end{align*}
These quantities are analogues of $R_X,R_Y$ defined in (\ref{def:L_X}). Now we may develop an expansion of $U_{f,\gamma_X},V_{f,\gamma_Y}$ similar to Lemma \ref{lem:sqrt_norm_LX}.

\begin{lemma}\label{lem:UV_expansion_kernel}
	With $\bar{R}_{X;f},\bar{R}_{Y;f}$ denoting the double centered versions of ${R}_{X;f},{R}_{Y;f}$, we have
	\begin{align*}
	U_{f,\gamma_X}(x_1,x_2)&=- \frac{2f_{\sqrt{\cdot}}'(\rho_X^2)}{\gamma_X^2}\Big[x_1^\top x_2-\tau_X^2 \bar{R}_{X;f}(L_X(x_1,x_2))\Big],\\
	V_{f,\gamma_Y}(y_1,y_2)&=- \frac{2f_{\sqrt{\cdot}}'(\rho_Y^2)}{\gamma_Y^2}\Big[y_1^\top y_2-\tau_Y^2\bar{R}_{Y;f}(L_Y(y_1,y_2))\Big].
	\end{align*}
\end{lemma}

\begin{proof}
	Using the definition of $R_{X,f}$, we have
	\begin{align*}
	f\bigg(\frac{\pnorm{x_1-x_2}{}}{\gamma_X}\bigg)
	& =f_{\sqrt{\cdot}}\big(\rho_X^2+ \rho_X^2 L_X(x_1,x_2)\big)\\ 
	& =f_{\sqrt{\cdot}}(\rho_X^2)+f_{\sqrt{\cdot}}'(\rho_X^2) \rho_X^2 \Big[L_X(x_1,x_2)+2\cdot R_{X;f}(L_X(x_1,x_2))\Big],
	\end{align*}
	The claim follows by using the fact that double centered version of $L_X(x_1,x_2)$ equals $-2x_1^\top x_2/\tau_X^2$.
\end{proof}

The following moment estimates give an analogy of Lemma \ref{lem:moment_R}.

\begin{lemma}\label{lem:moment_R_kernel}
	Suppose Assumption \ref{assump:kernel} holds, and that (i) the spectrum of $\Sigma_X,\Sigma_Y$ and (ii) $\rho_X,\rho_Y$ are contained in $[1/M,M]$ for some $M>1$. Then:
	\begin{enumerate}
		\item  For any positive integer $s\in \N$,
		\begin{align*}
		\tau_X^{2s} \E R_{X;f}^s(X_1,X_2)\bigvee \tau_Y^{2s} \E R_{Y;f}^s(Y_1,Y_2)\lesssim_{M,s} 1.
		\end{align*}
		\item For any positive integer $s \in \N$,
		\begin{align*}
		\tau_X^s \E h_{f,\rho_X}'(L_X(X_1,X_2))^s \bigvee \tau_Y^s \E h_{f,\rho_Y}'(L_Y(Y_1,Y_2))^s \lesssim_M 1
		\end{align*}
		holds for $p,q$ large (possibly depending on $s,\mathfrak{q}$).
	\end{enumerate}
\end{lemma}
\begin{proof}
	Combining with estimates in Lemma \ref{lem:moment_R}, (1) follows by $\abs{h_{f,\rho}(u)}\lesssim_M u^2$, and (2) follows by $\abs{h'_{f,\rho}(u)}\lesssim_M |u\cdot F_{f,\rho;2}(u)|$ and noting that for any $s \in \N$,
	\begin{align*}
	\E \abs{F_{f,\rho;2}(L_X(X_1,X_2))}^s \lesssim \sup_{t \in [0,1]} \E\bigabs{f_{\sqrt{\cdot}}^{(2)}\big(\rho^2(1+tL_X(X_1,X_2))\big)}^{s}\lesssim_M 1
	\end{align*}
	by Lemma \ref{lem:kernel_der_int}.
\end{proof}

Let
\begin{align*}
\psi_{X;f}(x_1,y_1)&\equiv \E_{X_2,Y_2}[\bar{R}_{X;f}(x_1,X_2)Y_2^\top y_1],\nonumber\\
\psi_{Y;f}(x_1,y_1)&\equiv \E_{X_2,Y_2}[\bar{R}_{Y;f}(y_1,Y_2)X_2^\top x_1],\nonumber\\
\psi_{X,Y;f}(x_1,y_1)& \equiv \E_{X_2,Y_2}[\bar{R}_{X;f}(x_1,X_2)\bar{R}_{Y;f}(y_1,Y_2)].
\end{align*}

\begin{proposition}\label{prop:res_psi_bound_kernel}
	Suppose that Assumption \ref{assump:kernel} holds, and that $\rho_X,\rho_Y$ are contained in $[1/M,M]$ for some $M>1$. Then the bounds in Proposition \ref{prop:res_psi_bound} hold when $\psi_X,\psi_Y,\psi_{X,Y}$ are replaced with $\psi_{X;f},\psi_{Y;f},\psi_{X,Y;f}$ under the same spectrum conditions, for $p,q$ large (possibly depending on $\mathfrak{q}$).
\end{proposition}
\begin{proof}
	The proof essentially proceeds the same way as in the proof of Proposition \ref{prop:res_psi_bound}. We sketch some differences. 
	
	\noindent (\textbf{First moment}) First, using a similar argument as in Lemma \ref{lem:psi_decomp}, we have the decomposition
	\begin{align*}
	\psi_{X;f}(x,y) = A_{1,X;f}(x,y) + A_{2,X;f}(x,y),
	\end{align*}
	where $A_{2,X;f}(x,y) = \E\big[h_{f,\rho_X;3}(L_X(x,X))Y^\top y\big]$ ($h_{f,\rho_X;3}$ defined in Lemma \ref{lem:property_h_fcn_kernel}), and
	\begin{align*}
	A_{1,X;f}(x,y) = -\frac{f''_{\sqrt{\cdot}}(\rho_X^2)\rho_X^2}{f'_{\sqrt{\cdot}}(\rho_X^2)\tau_X^4}\Big[\big(\pnorm{x}{}^2-\tr(\Sigma_X)\big)x^\top\Sigma_{XY}y + 2x^\top\Sigma_X\Sigma_{XY}y\Big].
	\end{align*}
	The bound $|\E A_{1,X;f}(X,Y)|\lesssim_M \tau_X^{-4}\pnorm{\Sigma_{XY}}{F}^2$ follows from Lemma \ref{lem:f_sqrt_basic}, and the bound for $|\E A_{2,X;f}(X_1,Y_1)|$ follows from the following modification: by Lemma \ref{lem:f_sqrt_basic},
	\begin{align*}
	&\abs{\E A_{2,X;f}(X_1,Y_1)} \lesssim_M \bigabs{ \E\big[L_X^3(X_1,X_2)Y_1^\top Y_2\big]}\\
	&\qquad +\biggabs{\E \bigg[L_X^{4}(X_1,X_2) Y_1^\top Y_2\cdot \int_0^{1} (1-s)^3f_{\sqrt{\cdot}}^{(4)}\big(\rho^2(1+sL_X(X_1,X_2))\big)\,\d{s}\bigg] }.
	\end{align*}
	The first term can be handled by Lemma \ref{lem:cross_moment_LX_improved} of order at most $\tau_X^{-4}\pnorm{\Sigma_{XY}}{F}^2$, while by Lemma \ref{lem:kernel_der_int} and a similar argument as in Lemma \ref{lem:cross_moment_LX_improved}, the second term can be bounded, up to a constant depending on $M$, by
	\begin{align*}
	&\E^{1/2} L_X^8(X_1,X_2) \bigg[\int_0^{1} (1-s)^3f_{\sqrt{\cdot}}^{(4)}\big(\rho^2(1+sL_X(X_1,X_2))\big)\,\d{s}\bigg]^2\cdot \pnorm{\Sigma_{XY}}{F}^2\\
	&\lesssim_M \tau_X^{-4}\pnorm{\Sigma_{XY}}{F}^2\cdot \sup_{s\in [0,1]}\E^{1/4} \big[f_{\sqrt{\cdot}}^{(4)}\big(\rho^2(1+sL_X(X_1,X_2))\big)\big]^4\lesssim_{M,\mathfrak{q}} \tau_X^{-4}\pnorm{\Sigma_{XY}}{F}^2.
	\end{align*}
	
	\noindent (\textbf{Second moment}) The proof of the second moment bounds for $\psi_{X;f},\psi_{Y;f}$ requires a modification for controlling $\E A_{2,X;f}^2(X_1,Y_1)$: 
	\begin{align*}
	&\E A_{2,X;f}^2(X_1,Y_1)\\
	&\lesssim_M \E_{X_1,Y_1} \bigg\{\E_{X_2,Y_2}\bigg[L_X^3(X_1,X_2)Y_1^\top Y_2\int_0^{1} (1-s)^2f_{\sqrt{\cdot}}^{(3)}\big(\rho^2(1+sL_X(X_1,X_2))\big)\,\d{s}\bigg]\bigg\}^2\\
	&\lesssim_M  \tau_X^{-6}\pnorm{\Sigma_{XY}}{F}^2\cdot \sup_{s \in[0,1]} \E^{1/4}\big[f_{\sqrt{\cdot}}^{(3)}\big(\rho^2(1+sL_X(X_1,X_2))\big)\big]^8\lesssim_M \tau_X^{-6}\pnorm{\Sigma_{XY}}{F}^2.
	\end{align*}
	using Lemma \ref{lem:kernel_der_int} in the last step. The proof of the second moment bound for $\psi_{X,Y;f}$ proceeds along the same lines upon using Lemma \ref{lem:moment_R_kernel}.
\end{proof}

\subsection{Mean and variance expansion}
Now we may prove the mean and variance expansion formulae:
\begin{proposition}
	Suppose that Assumption \ref{assump:kernel} holds, and that (i) the spectrum of $\Sigma_X,\Sigma_Y$ (ii) $\rho_X,\rho_Y$ are contained in $[1/M,M]$ for some $M>1$.
	\begin{enumerate}
		\item The following mean expansion holds:
		\begin{align*}
		m_{\Sigma;f,\gamma}&\equiv \E \dcov_\ast^2(\bm{X},\bm{Y};f,\gamma)=\dcov^2(X,Y;f,\gamma)\\
		&= \varrho(\gamma)m_\Sigma \Big[1+\mathcal{O}_M\big((\tau_X\wedge \tau_Y)^{-1}\big)\Big],
		\end{align*}
		where $m_\Sigma$ is the mean for distance covariance defined in Theorem \ref{thm:dcov_expansion}.
		\item The following variance expansion holds:
		\begin{align*}
		\var\big(\dcov_\ast^2(\bm{X},\bm{Y};f,\gamma)\big)&=\varrho^2(\gamma)\bar{\sigma}_n^2(X,Y)\Big[1+\mathcal{O}_M\big(n^{-1/2}+(p\wedge q)^{-1/4}\big)\Big].
		\end{align*}
		Here $\bar{\sigma}_n^2(X,Y)$ is defined in Theorem \ref{thm:non_null_clt}.
	\end{enumerate}
\end{proposition}
\begin{proof}
	The mean and variance expansions can be proved using exactly the same lines as in the proof of Theorems \ref{thm:dcov_expansion} and \ref{thm:variance_expansion}, but using Proposition \ref{prop:res_psi_bound_kernel}. The scaling can be checked, e.g., as the leading term in $m_{\Sigma;f,\gamma}$ equals, by Lemma \ref{lem:UV_expansion_kernel},
	\begin{align*}
	\frac{4f_{\sqrt{\cdot}}'(\rho_X^2)f_{\sqrt{\cdot}}'(\rho_Y^2)}{\gamma_X^2\tau_Y^2}\pnorm{\Sigma_{XY}}{F}^2 = 4f_{\sqrt{\cdot}}'(\rho_X^2)f_{\sqrt{\cdot}}'(\rho_Y^2) \frac{\tau_X\tau_Y}{\gamma_X^2\gamma_Y^2}\cdot \frac{\pnorm{\Sigma_{XY}}{F}^2}{\tau_X\tau_Y}=\varrho(\gamma)\cdot \frac{\pnorm{\Sigma_{XY}}{F}^2}{\tau_X\tau_Y},
	\end{align*}
	whereas $\pnorm{\Sigma_{XY}}{F}^2/(\tau_X\tau_Y)$ is the leading term in $m_\Sigma$ as established in Theorem \ref{thm:dcov_expansion}.
\end{proof}

The rest of the proof of Theorem \ref{thm:non_null_clt_kernel} follows exactly the same lines as that of the proof of Theorem \ref{thm:non_null_clt}. Details are omitted.

\section{Auxiliary lemmas}

\begin{lemma}\label{lem:kolmogorov_res}
	For real-valued random variables $W,R$, $N\sim \mathcal{N}(0,1)$ and a deterministic real number $D$, we have
	\begin{align*}
	d_{\mathrm{Kol}}(W+R+D,N)\leq d_{\mathrm{Kol}}(W,N)+2 (\E R^2)^{1/3}+\abs{D}.
	\end{align*}
\end{lemma}
\begin{proof}
	Fix $u>0$, on an event $E_u$ with probability at least $1-u^{-2}$, $\abs{R}\leq u\cdot \E^{1/2} R^2$. Then for any $t \in \R$
	\begin{align*}
	\Prob(W+R+D\leq t)& \leq \Prob(W+D\leq t+ u \E^{1/2} R^2)+ u^{-2}\\
	&\leq \Prob(N\leq  t+ \abs{D}+u \E^{1/2} R^2)+ u^{-2}+d_{\mathrm{Kol}}(W,N)\\
	&\leq \Prob(N\leq t)+ \abs{D}+ u \E^{1/2} R^2+u^{-2}+ d_{\mathrm{Kol}}(W,N).
	\end{align*}
	Optimizing the above display over $u>0$ to conclude one side of the inequality. The reverse direction is similar.
\end{proof}

\begin{lemma}\label{lem:dkov_dW}
	For two random variables $X,Y$, with $\varphi_Y$ denoting the Lebesgue density of $Y$, we have $d_{\mathrm{Kol}}^2(X,Y)\leq 4 \pnorm{\varphi_Y}{\infty} d_{\mathrm{W}}(X,Y)$, where $d_{\mathrm{W}}$ is the Wasserstein distance.
\end{lemma}
\begin{proof}
	Fix $t \in \R$ and $\epsilon>0$. Let $g_{t,\epsilon}(\cdot)\equiv1$ on $(-\infty, t]$, $g_{t,\epsilon}(\cdot)\equiv 0$ on $[t+\epsilon,\infty)$ and be linearly interpolated on $(t,t+\epsilon)$. As $g_{t,\epsilon}$ is $\epsilon^{-1}$-Lipschitz,
	\begin{align*}
	&\Prob(X\leq t)-\Prob(Y\leq t)\leq \E g_{t,\epsilon}(X)-\E g_{t,\epsilon}(Y)+\E g_{t,\epsilon}(Y)-\E\bm{1}_{(-\infty,t]}(Y)\\
	&\leq \epsilon^{-1}d_{\mathrm{W}}(X,Y)+ \Prob\big(t\leq Y\leq t+\epsilon\big)\leq \epsilon^{-1}d_{\mathrm{W}}(X,Y)+\epsilon \pnorm{\varphi_Y}{\infty}.
	\end{align*}
	Optimizing over $\epsilon>0$ obtains the one-sided inequality. The other direction is similar.
\end{proof}

\begin{lemma}\label{lem:gaussian_4moment}
	Suppose $Z,Z_1,Z_2,\ldots$ are independent copies of a random vector in $\R^d$ having i.i.d. symmetric components with mean $0$, variance $1$, and fourth moment $m_4$. Let $\kappa \equiv m_4 - 3$ be the excess kurtosis. The following hold.
	\begin{enumerate}
		\item For any matrices $A,B \in \R^{d\times d}$, 
		\begin{align*}
		&\E (Z^\top A Z) (Z^\top BZ)= \tr(A)\tr(B)+\tr(AB)+\tr(AB^\top) + \kappa\tr(A\circ B),\\
		&\E\big(Z^\top A Z-\tr(A)\big)\big(Z^\top B Z-\tr(B)\big)=\tr(AB)+\tr(AB^\top)+ \kappa\tr(A\circ B). 
		\end{align*}
		\item For any matrices $A,B \in \R^{d\times d}$,
		\begin{align*}
		&\E (Z_1^\top A Z_2 Z_2^\top B Z_1)^2\\
		&=2\big(\tr^2(AB)+\tr(ABAB)+\tr(ABB^\top A^\top) + \tr(A^\top B^\top BA)\big)+ \pnorm{A}{F}^2\pnorm{B}{F}^2\\
		&+ 2\kappa\cdot\big[\tr(BA\circ BA) + \tr(AB \circ AB) + \tr(AA^\top \circ BB^\top\big] + \kappa^2\cdot\tr\big((A\circ A)(B\circ B)\big).
		\end{align*}
		\item For any matrix $A \in \R^{d\times d}$,
		\begin{align*}
		\E (ZZ^\top) (Z^\top A Z) = A+A^\top +\tr(A) I_d +\kappa \cdot I_d \circ A.
		\end{align*}
	\end{enumerate}
\end{lemma}
\begin{proof}
	\noindent (1). Note that
	\begin{align*}
	&\E (Z^\top A Z) (Z^\top BZ) = \E \tr\big(A ZZ^\top B ZZ^\top\big)\\
	& = \E\sum_{i,j,k,\ell} A_{ij}(ZZ^\top)_{jk}B_{k\ell} (ZZ^\top)_{\ell i} = \sum_{i,j,k,\ell} A_{ij}B_{k\ell}\cdot  \E Z_i Z_j Z_k Z_\ell. 
	\end{align*}
		To make the expectation on the right hand side non-vanishing:
	\begin{itemize}
		\item $(i=j)\neq (k=\ell)$. This contributes to $
		\sum_{i \neq k} A_{ii}B_{kk}$. 
		\item $(i=k)\neq (j=\ell)$. This contributes to $
		\sum_{i,j} A_{ij}B_{ij}$.
		\item $(i=\ell)\neq (j=k)$. This contributes to $
		\sum_{i,j} A_{ij}B_{ji}$.
		\item $i=j=k=\ell$. This contributes to $
		m_4\sum_{i} A_{ii}B_{ii} = m_4\tr(A\circ B)$.
	\end{itemize}
    Collecting the terms to conclude the first equality. The second equality follows a direct expansion.

	\noindent (2). Using part (1), we have
	\begin{align*}
	&\E (Z_1^\top A Z_2 Z_2^\top B Z_1)^2= \E \Big[\E \big[Z_1^\top \big(A Z_2 Z_2^\top B\big) Z_1\cdot Z_1^\top \big(A Z_2 Z_2^\top B\big) Z_1\lvert Z_2\big]\Big]\\
	& = \E \tr^2(A Z_2Z_2^\top B)+ \E \tr(A Z_2Z_2^\top B A Z_2Z_2^\top B)+\E\tr(A Z_2Z_2^\top B B^\top Z_2Z_2^\top A^\top)\\
	&\qquad+ \kappa\E \tr(AZ_2Z_2^\top B \circ AZ_2Z_2^\top B).
	\end{align*}
	Using again part (1), the above terms can be calculated as follows:
	\begin{itemize}
	\item The firs term is
	\begin{align*}
	\E \tr^2(A Z_2Z_2^\top B) = \E (Z_2^\top BAZ_2)^2 = \tr^2(BA) + \tr(BABA) + \tr(BAA^\top B^\top) + \kappa\tr(BA\circ BA).
	\end{align*}
	\item Since $\E \tr(A Z_2Z_2^\top B A Z_2Z_2^\top B) = \E (Z_2^\top BA Z_2)$, the second term equals the first term.
	\item The third term is
	\begin{align*}
	\E\tr(A Z_2Z_2^\top B B^\top Z_2Z_2^\top A^\top) &= \E (Z_2^\top BB^\top Z_2)(Z_2^\top A^\top A Z_2)\\
	&=  \pnorm{A}{F}^2\pnorm{B}{F}^2 + 2\tr(BB^\top A^\top A) + \kappa\cdot\tr(BB^\top\circ AA^\top).
	\end{align*}
	\item The fourth term satisfies
	\begin{align*}
	&\E \tr(AZ_2Z_2^\top B \circ AZ_2Z_2^\top B) = \sum_{i=1}^d \E (AZ_2Z_2^\top B)_{ii}^2 = \sum_{i=1}^d \E (Z_2^\top Be_ie_i^\top AZ_2)^2\\
	&= \sum_{i=1}^d \Big[\tr^2(Be_ie_i^\top A) + \tr(Be_ie_i^\top ABe_ie_i^\top A) + \tr(Be_ie_i^\top AA^\top e_ie_i^\top B^\top) + \kappa\cdot\tr(Be_ie_i^\top A\circ Be_ie_i^\top A)\Big]\\
	&= \sum_{i=1}^d \Big[(AB)_{ii}^2 + (AB)_{ii}^2 + (AA^\top)_{ii}(B^\top B)_{ii} + \kappa\cdot\sum_{j=1}^d A_{ij}^2B_{ji}^2\Big]\\
	&= 2\tr(AB\circ AB) + \tr(AA^\top \circ BB^\top) + \kappa\cdot\tr\big((A\circ A)(B\circ B)\big).
	\end{align*}
	\end{itemize}
	\noindent (3). For $i\neq j$,
	\begin{align*}
	\big(\E (ZZ^\top) (Z^\top A Z) \big)_{ij}&= \E \bigg[Z_iZ_j \sum_{k,\ell} Z_k A_{k\ell} Z_\ell\bigg] = A_{ij}+A_{ji}.
	\end{align*}
	For $i=j$, 
	\begin{align*}
	\big(\E (ZZ^\top) (Z^\top A Z) \big)_{ii}&= \E \bigg[Z_i^2 \sum_{k,\ell} Z_k A_{k\ell} Z_\ell\bigg] = \E Z_i^4\cdot A_{ii}+\sum_{k\neq i} (\E Z_i^2) (\E Z_k^2) A_{kk}\\
	& = m_4A_{ii}+\sum_{k\neq i} A_{kk} = (m_4-1)A_{ii} + \tr(A).
	\end{align*}
	The proof is complete.
\end{proof}

\begin{lemma}\label{lem:moment_X_Y}
	The following hold:
	\begin{enumerate}
		\item $\E \pnorm{X_1}{}^2 \pnorm{Y_1}{}^2 = \tr(\Sigma_X)\tr(\Sigma_Y)+ 2\pnorm{\Sigma_{XY}}{F}^2+\kappa\tr(H_{[11]}\circ H_{[22]})$.
		\item $\E(X_1^\top \Sigma_{XY} Y_2)^2 = \tr(\Sigma_X\Sigma_{XY}\Sigma_Y \Sigma_{YX})$. 
		\item For any matrix $A,B$,
		\begin{align*}
		\E(X_1^\top A Y_1)(X_1^\top B Y_1)=\tr(A\Sigma_{YX})\tr(B\Sigma_{YX})+\tr(A\Sigma_{YX}B\Sigma_{YX})+\tr(\Sigma_X A \Sigma_Y B^\top) + \kappa\cdot\tr(A_\Sigma\circ B_\Sigma).
		\end{align*}
		Here $A_\Sigma, B_\Sigma$ are given in (\ref{def:AB_Sigma}) below. In particular, for $A=B=\Sigma_{XY}$, we have $\E(X_1^\top \Sigma_{XY} Y_1)^2 = \pnorm{\Sigma_{XY}}{F}^4+\pnorm{\Sigma_{XY}\Sigma_{YX}}{F}^2+\tr(\Sigma_{XY}\Sigma_Y \Sigma_{YX}\Sigma_{X})  + \kappa\cdot\tr(G_{[12]} \circ G_{[12]})$.
		\item For any matrix $A,B$,
		\begin{align*}
		\E (X_1^\top A X_1) (Y_1^\top B Y_1)=\tr(A\Sigma_{X})\tr(B\Sigma_{Y})+\tr(A\Sigma_{XY}B\Sigma_{YX})+\tr(A\Sigma_{XY}B^\top\Sigma_{YX}) + \kappa\cdot\tr(\bar{A}_\Sigma\circ \bar{B}_\Sigma).
		\end{align*}
		Here $\bar{A}_\Sigma, \bar{B}_\Sigma$ are given in (\ref{def:AB_Sigma_bar}) below.
	\end{enumerate}
\end{lemma}
\begin{proof}
	\noindent (1) follows from
	\begin{align*}
		\E \pnorm{X_1}{}^2\pnorm{Y_1}{}^2& = \E Z_1^\top H_{[11]} Z_1\cdot Z_1^\top H_{[22]} Z_1\\
		& = \tr(H_{[11]})\tr(H_{[22]})+2\tr(H_{[11]}H_{[22]}) + \kappa\tr(H_{[11]}\circ H_{[22]})\\
		&= \tr(\Sigma_X)\tr(\Sigma_Y)+ 2\pnorm{\Sigma_{XY}}{F}^2 +\kappa\tr(H_{[11]}\circ H_{[22]}).
	\end{align*}

	\noindent (2) follows from
	\begin{align*}
	\E(X_1^\top \Sigma_{XY} Y_2)^2 
	& = \E \tr \big(\Sigma_{XY} Y_2Y_2^\top \Sigma_{XY} X_1 X_1^\top\big)=\tr(\Sigma_{XY}\Sigma_Y \Sigma_{YX}\Sigma_X).
	\end{align*}

	
	
	\noindent (4). Let
	\begin{align}\label{def:AB_Sigma}
	A_\Sigma\equiv \Sigma^{1/2}
	\begin{pmatrix}
	0 & A\\
	0 & 0
	\end{pmatrix}
	\Sigma^{1/2},\quad 
	B_\Sigma\equiv \Sigma^{1/2}
	\begin{pmatrix}
	0 & B\\
	0 & 0
	\end{pmatrix}
	\Sigma^{1/2}.
	\end{align}
	Then
	\begin{align*}
	&\E(X_1^\top A Y_1)(X_1^\top B Y_1)= \E (Z_1^\top A_\Sigma Z_1) (Z_1^\top B_\Sigma Z_1)\\
	&=\tr(A_\Sigma)\tr(B_\Sigma)+\tr(A_\Sigma B_\Sigma)+\tr(A_\Sigma B_\Sigma^\top) + \kappa\cdot\tr(A_\Sigma\circ B_\Sigma)\\
	& = \tr(A\Sigma_{YX})\tr(B\Sigma_{YX})+\tr(A\Sigma_{YX}B\Sigma_{YX})+\tr(\Sigma_X A \Sigma_Y B^\top) + \kappa\cdot\tr(A_\Sigma\circ B_\Sigma).
	\end{align*}
	\noindent (5). Let 
	\begin{align}\label{def:AB_Sigma_bar}
	\bar{A}_\Sigma\equiv \Sigma^{1/2}
	\begin{pmatrix}
	A & 0\\
	0 & 0
	\end{pmatrix}
	\Sigma^{1/2},\quad 
	\bar{B}_\Sigma\equiv \Sigma^{1/2}
	\begin{pmatrix}
	0 & 0\\
	0 & B
	\end{pmatrix}
	\Sigma^{1/2}.
	\end{align}
		Then
	\begin{align*}
	&\E(X_1^\top A X_1)(Y_1^\top B Y_1)= \E (Z_1^\top \bar{A}_\Sigma Z_1) (Z_1^\top \bar{B}_\Sigma Z_1)\\
	&=\tr(\bar{A}_\Sigma)\tr(\bar{B}_\Sigma)+\tr(\bar{A}_\Sigma \bar{B}_\Sigma)+\tr(\bar{A}_\Sigma \bar{B}_\Sigma^\top) + \kappa\tr(\bar{A}_\Sigma\circ \bar{B}_\Sigma)\\
	& = \tr(A\Sigma_{X})\tr(B\Sigma_{Y})+\tr(A\Sigma_{XY}B\Sigma_{YX})+\tr(A\Sigma_{XY}B^\top\Sigma_{YX})+ \kappa\tr(\bar{A}_\Sigma\circ \bar{B}_\Sigma).
	\end{align*}
	The proof is complete.
\end{proof}

\begin{lemma}\label{lem:SigmaXY_F_bound}
	It holds that $\pnorm{\Sigma_{XY}}{F}^2\leq \pnorm{\Sigma_X}{F}\pnorm{\Sigma_Y}{F}$. If in addition $\pnorm{\Sigma_X^{-1}}{\op}\vee \pnorm{\Sigma_Y^{-1}}{\op}\leq M$ for some $M>1$, then $\pnorm{\Sigma_{XY}}{F}^2\lesssim_M \tau_X^2\wedge \tau_Y^2$.
\end{lemma}

\begin{proof}
	For the first claim, note that
	\begin{align*}
	\pnorm{\Sigma_{XY}}{F}^2 & = \tr\big(\E X_1 Y_1^\top \E Y_2 X_2^\top\big) = \E (X_1^\top X_2) (Y_1^\top Y_2) \leq \E^{1/2}(X_1^\top X_2)^2\cdot \E^{1/2}(Y_1^\top Y_2)^2 = \pnorm{\Sigma_X}{F}\pnorm{\Sigma_Y}{F},
	\end{align*}
	as desired. For the second claim, as $\Sigma_{X\backslash Y} = \Sigma_X - \Sigma_{XY}\Sigma_{Y}^{-1}\Sigma_{YX}$ is p.s.d. (since it is the conditional variance of $X$ given $Y$ when $(X^\top,Y^\top)^\top \stackrel{d}{=} \mathcal{N}(0,\Sigma)$), we have
	\begin{align*}
	\pnorm{\Sigma_{XY}}{F}^2\lesssim_M \tr(\Sigma_{XY}\Sigma_{Y}^{-1}\Sigma_{YX}) \leq \tr(\Sigma_X) = \tau_X^2/2.
	\end{align*}
	Flipping the role of $X,Y$ to conclude.
\end{proof}

\bibliographystyle{amsalpha}
\bibliography{mybib}

\end{document}